%% file: main.tex
\documentclass[11pt]{amsart}
\usepackage[margin=1.1in]{geometry}
\usepackage{amsmath,amssymb,amsthm,mathtools}
\usepackage{graphicx,booktabs,enumitem,longtable}
\usepackage[hidelinks]{hyperref}
\hypersetup{pdftitle={Smale's sixth problem for five bodies in the plane, with a proof formalized in Lean},
  pdfauthor={Tejasvi Singh Tomar}}

\newtheorem{theorem}{Theorem}[section]
\newtheorem{lemma}[theorem]{Lemma}
\newtheorem{proposition}[theorem]{Proposition}
\newtheorem{corollary}[theorem]{Corollary}
\theoremstyle{definition}
\newtheorem{definition}[theorem]{Definition}
\theoremstyle{remark}
\newtheorem{remark}[theorem]{Remark}
\numberwithin{equation}{section}
\newif\ifpendingfatal
\ifdefined\PENDINGFATAL\pendingfataltrue\fi

\newcommand{\R}{\mathbb{R}}
\newcommand{\C}{\mathbb{C}}
\newcommand{\Q}{\mathbb{Q}}
\newcommand{\Z}{\mathbb{Z}}
\newcommand{\cA}{\mathcal{A}}

\newcommand{\cD}{\mathcal{D}}
\newcommand{\cP}{\mathcal{P}}
\newcommand{\cZ}{\mathcal{Z}}
\newcommand{\cW}{\mathcal{W}}
\newcommand{\cN}{\mathcal{N}}
\newcommand{\one}{\mathbb{1}}
\newcommand{\jj}{\mathrm{j}}
\newcommand{\val}{\operatorname{val}}

\newcommand{\Res}{\operatorname{Res}}
\newcommand{\rank}{\operatorname{rank}}
\newcommand{\relint}{\operatorname{relint}}
\newcommand{\conv}{\operatorname{conv}}

\newcommand{\lab}[1]{\mathsf{#1}}
\newcommand{\pc}[1]{\mathsf{#1}}
\newcommand{\bone}{\mathbf{b}_1}
\newcommand{\btwo}{\mathbf{b}_2}
\newcommand{\hone}{\mathbf{h}_1}
\newcommand{\htwo}{\mathbf{h}_2}
\newcommand{\hthree}{\mathbf{h}_3}
\newcommand{\Vtwo}{V_2}
\newcommand{\VN}{V_N}

\begin{document}

\title[Smale's sixth problem for five bodies in the plane]{Smale's sixth problem for five bodies in the plane,\\ with a proof formalized in Lean}
\author{Tejasvi Singh Tomar}
\address{Independent researcher}
\renewcommand{\urladdrname}{{\itshape ORCID}}
\urladdr{\href{https://orcid.org/0000-0003-4668-3639}{https://orcid.org/0000-0003-4668-3639}}

\input{sec/abstract}

\maketitle

\input{sec/s1_intro}
\input{sec/s2_setup}
\input{sec/s3_types}
\input{sec/s4_lex}
\input{sec/s5_energy}

\input{sec/s5_points}
\input{sec/s6_near}
\input{sec/s8_curves}
\input{sec/s9_rays}
\input{sec/s10_proof}
\input{sec/s11_verification}

\appendix
\input{sec/a_appendix}

\section*{Use of AI tools}
Generative AI tools were used to find and prove the results, to write the programs and the Lean formalization, and to draft the text.

\input{sec/refs}
\end{document}

%% file: sec/abstract.tex
\begin{abstract}
We prove that for every choice of five positive masses the planar five-body problem has finitely many central configurations up to similarity. This settles Smale's sixth problem for five bodies in the plane. Albouy and Kaloshin proved it for masses outside a closed algebraic subset of codimension two. We follow their method. Infinitely many central configurations would give an algebraic curve of complex solutions. At the places at infinity of the curve the valuations of the mutual distances must balance. We determine the possible valuation vectors at all positive masses, together with the mass conditions under which each can occur. We treat limits in which four bodies form a cluster of zero potential energy by expanding the potential, which is constant along the curve, at a place. This excludes, at every positive mass, one of the two diagrams in which such clusters occur, and part of the other. The same expansion excludes one further diagram at every positive mass. What remains is an exact analysis, by computer algebra, of the sets of mass conditions that can hold together at positive masses. Exact rational certificates show that the balance fails under each of them, except on a few one-parameter families of masses, which we treat separately. On the rays through $(1,4,4,4,4)$ and $(1,1,4,4,4)$ the certificates do not apply. On the first ray they cannot apply: there the complex equations have a continuum of solutions, as Albouy and Kaloshin observed. On these rays we count the real central configurations: there are 306 and 294 up to orientation-preserving similarity. The certificates, the programs that check them and a proof of the theorem in the Lean proof assistant are available.
\end{abstract}

%% file: sec/s1_intro.tex
\section{Introduction}\label{sec:intro}

A configuration $q=(q_1,\dots,q_n)$ of $n$ distinct points in the plane, carrying positive masses $m_1,\dots,m_n$, is a \emph{central configuration} if there are $\lambda\in\R$ and $c\in\R^2$ such that
\begin{equation}\label{eq:cc}
\lambda\,m_i(q_i-c)=\sum_{j\ne i}\frac{m_im_j\,(q_i-q_j)}{|q_i-q_j|^3}\qquad(i=1,\dots,n).
\end{equation}
The right side is minus the Newtonian force on the body $i$: the force on each body points towards the centre of mass $c$ and is proportional to its distance from it. Central configurations are the shapes of the homographic solutions of the $n$-body problem, and in particular the shapes of its relative equilibria. If $q$ is a central configuration, so is any configuration similar to it. Smale's sixth problem \cite{Sma98} (Problem~9 of \cite{ACS12}) asks whether, for every choice of positive masses, the number of central configurations up to similarity is finite.

For $n=3$ there are five, the two equilateral triangles of Lagrange and the three collinear configurations of Euler. Moulton \cite{Mou10} counted the collinear ones for all $n$. Hampton and Moeckel \cite{HM06} proved finiteness for four bodies and all positive masses. Albouy and Kaloshin \cite{AK12} proved finiteness for five bodies in the plane when the masses lie outside a closed algebraic subset $\cA\subset(\R_{>0})^5$ of codimension two. Their argument produces singular sequences that approach one of the first eight limit diagrams in their list of sixteen. Each of these eight forces a polynomial condition on the masses. The set $\cA$ is where two independent conditions hold at once. Jensen and Leykin \cite{JL25} gave a tropical proof of finiteness for generic masses. Moczurad and Zgliczy\'nski \cite{MZ26} counted the central configurations exactly, with a computer-assisted proof, at several mass vectors of the exceptional set. We remove the exceptional set.

\begin{theorem}\label{thm:main}
For every $m\in(\R_{>0})^5$ the planar five-body problem with masses $m$ has finitely many central configurations up to similarity.
\end{theorem}

We use the results of \cite{AK12} listed in \S\ref{ss:akinputs}, in particular their classification of the limit diagrams (their Fig.~11) and the mass conditions of their \S8.

Part of the exceptional set of \cite{AK12} cannot be avoided by an argument about the complexified equations alone. Roberts \cite{Rob99} found a continuum of central configurations for the masses $(1,1,1,1,-\tfrac14)$. Albouy and Kaloshin observe that changing the signs of $m_5$ and of the four distances $r_{k5}$ turns it into a continuum of solutions, with masses $(4,4,4,4,1)$, of the polynomial system they work with \cite[pp.~543--544]{AK12}. These solutions are not central configurations, because the distances $r_{k5}$ are negative square roots, but an argument that sees only the complexified system cannot tell them apart. At $(1,4,4,4,4)$ the complex solution set is therefore infinite. Finiteness of the real central configurations there must use the signs of the distances. On the rays through $(1,4,4,4,4)$ and $(1,1,4,4,4)$ and their images under permutations of the bodies we count the real central configurations directly: for every positive multiple of a permutation of $(1,4,4,4,4)$ there are exactly 306 such configurations up to orientation-preserving similarity, and for $(1,1,4,4,4)$ exactly 294 (Theorems~\ref{thm:count1} and~\ref{thm:count2}). We call these rays the \emph{special rays}: a mass vector lies on a special ray if and only if it is a positive multiple of a permutation of $(1,4,4,4,4)$ or of $(1,1,4,4,4)$.

\subsection*{The method}
Our argument has six parts.

\emph{Curves and balancing.} If there are infinitely many central configurations, the complex solution set contains an irreducible curve on which the potential $U$ is a nonzero constant and some mutual distance is not constant. At each place at infinity of this curve the inverse distances $p_{kl}=r_{kl}^{-1}$ have valuations. The vectors $\rho(P)=(-\val_P p_{kl})_{k<l}$ sum to zero over the places (\S\ref{sec:setup}). This is the starting point of \cite{AK12}.

\emph{Types and mass conditions.} At a singular place, Albouy and Kaloshin show that the configuration approaches one of sixteen diagrams and that most diagrams force polynomial conditions on the masses. We make this sharp at all positive masses. For each diagram and labelling of the bodies we give the set of vectors $\rho(P)$ that can occur, as an explicit polyhedral set (a \emph{piece}) that is present only when certain mass conditions hold (Theorem~\ref{thm:types}). Besides results of \cite{AK12}, Theorem~\ref{thm:types} contains our sharper type statements for several diagrams and a lemma on diagram~7 (\S\ref{ss:d7}) that is not in \cite{AK12}.

\emph{Lexicographic separation.} If the pieces present at $m$ admit a staged family of linear functionals that are nonnegative on all of them and positive on each in turn, the vectors $\rho(P)$ cannot sum to zero and $m$ has no such curve (\S\ref{sec:lex}). On the sixteen strata of the sorted, scaled masses we analyse by exact computer algebra the sets of mass conditions that can hold together (\S\ref{sec:strata}). We certify such a family for each of these sets, with rational functionals checked in exact arithmetic, except for the sets at two mass points, where we use a smaller hull of diagram~$5$, and on a few one-parameter families of masses, which we treat separately (\S\S\ref{sec:points}, \ref{sec:curves} and~\ref{ss:twopoints}).

\emph{Clusters of zero energy.} Using all the singular places of a curve, rather than two well-chosen singular sequences, brings in diagrams 14 and 15 of \cite{AK12}. In them four bodies form a cluster, approached by the fifth body. The leading part of the cluster is a complex configuration of zero potential energy: a solution of the systems that we call K14 and K15. These systems have solutions on hypersurfaces of the space of masses. The two diagrams impose no condition on the masses in \cite{AK12}. Since their pieces are present for every set of mass conditions (\S\S\ref{ss:Kgating} and~\ref{ss:certs}), we need an argument that does not go through the mass conditions. To exclude diagram 15 we use that the potential is constant along the curve and that its gradient vanishes there. At a place, the expansion of the potential in a uniformizing parameter is then constant, while the Euler identity of the weights relates its coefficients to those of the equations. At order three this forces the three outer bodies of the cluster into one of four explicit shapes. At order five, if the leading cluster configuration is nondegenerate, it forces a cubic moment of the shape to vanish. For positive masses the two conditions are incompatible: the vanishing of the weighted sums of $x_i$, $x_i^{-2}$ and $x_i^3$ would force three angles, their doubles and their triples each to surround $0$. Three angles cannot do this. With an analysis of the degenerate leading configurations this excludes diagram 15 at every positive mass (\S\S\ref{ss:d15}--\ref{ss:residual}). Diagram 14 gives two kinds of pieces, called $\pc{d14a}$ and $\pc{d14b}$ in Table~\ref{tab:pieces}. At a place of type $\pc{d14b}$ the leading cluster solves K14. Its outer bodies form one of the same four shapes. At positive masses such clusters exist only when the three outer masses are equal and in one ratio to the fourth, where the outer bodies form an equilateral triangle centred at the fourth. An expansion at the place excludes this case too. Thus the pieces $\pc{d14b}$ are excluded at every positive mass (\S\ref{ss:d14}). We exclude the remaining piece of diagram 14, $\pc{d14a}$, where it is needed, by showing that K14 has no solution: at finitely many mass points (\S\ref{sec:points}) and along the curve families off finitely many special points (\S\ref{sec:curves}). At the special points we exclude it by a local analysis at the place (\S\ref{sec:curves}).

\emph{Diagrams 9 and 2.} Diagram~9, whose mass condition is the polynomial $N$ of \cite{AK12}, does not occur at any positive mass (\S\ref{ss:d9}); we prove this with an expansion of the potential at the place as for diagrams 14 and 15. Diagram 2 does not occur at any positive mass either (\S\ref{sec:nod2}): a condition at the leading order of a place and closed forms of its first obstructions leave twelve mass vectors of the triangle, at which the torque cannot vanish.

\emph{The special rays.} On the rays through $(1,4,4,4,4)$ and $(1,1,4,4,4)$ the separation fails. At the first of them it must fail, since the modified Roberts continuum lies in the complex solution set there. Following Moczurad and Zgliczy\'nski \cite{MZ19,MZ26}, we count the real central configurations on these rays by a branch-and-bound search with interval arithmetic and Krawczyk's test (\S\ref{sec:rays}).

\subsection*{Computations and their verification}
We describe the sets of mass conditions at all positive masses by $2206$ ideals on the sixteen strata, computed with exact standard bases and minimal primes over $\Q$ (\S\ref{sec:strata}). The mass points off the special rays that need a separate argument are exact algebraic points. At one of them, $P_1$, defined in \S\ref{ss:twopoints}, we count the central configurations, as on the special rays. At the other points, except the special points of the curve families (\S\ref{sec:curves}), we show, where it is needed, that the relevant systems K14 have no solution, by Gr\"obner bases modulo a large prime together with a lifting argument. The counts rest on a branch and bound in outward-rounded floating point. Section~\ref{sec:verification} describes how the computations are checked and the proof of Theorem~\ref{thm:main} in Lean.

\subsection*{Organization}
In \S\ref{sec:setup} we set up the system of \cite{AK12}, the slice, the potential form of the equations and the balancing condition. In \S\ref{sec:types} we determine the types of singular places and the mass conditions and state Theorem~\ref{thm:types}. In \S\ref{sec:lex} we give the separation lemma and the certificates. In \S\ref{sec:energy} we give the energy argument, in \S\ref{sec:points} the certificates for the four-body systems at finitely many masses, and in \S\ref{sec:nod2} the proof that diagram 2 does not occur. In \S\ref{sec:strata} we analyse by exact computer algebra the sets of mass conditions that can hold together at positive masses, and in \S\ref{sec:curves} we treat the curve families. In \S\ref{sec:rays} we count the central configurations on the special rays and at $P_1$, in \S\ref{sec:proof} we prove Theorem~\ref{thm:main}, and in \S\ref{sec:verification} we describe the verification. The appendices derive the types of the sixteen diagrams and the hulls of diagram~$5$. They contain the proof of the sharpened condition for diagram~$3$, the label polynomials, the analysis of places of diagram~$2$ used in \S\ref{sec:nod2}, the polynomial of degree eleven and the tests for diagrams $6$ and $3$ at two exceptional mass points.

\subsection*{Note added}
Version~1 of this paper was submitted to arXiv on 9~October~2026, and its supplementary material, including the complete Lean formalization (\S\ref{ss:formal}), was published on Zenodo the same day (\url{https://doi.org/10.5281/zenodo.23202861}). On 12~October~2026, Z.~Chen and Z.~Tong submitted arXiv:2610.16107~\cite{CT26}, which states Theorem~\ref{thm:main} and gives a different proof.

%% file: sec/s2_setup.tex
\section{Normalized central configurations and curves of them}\label{sec:setup}

\subsection{The system of Albouy and Kaloshin}
Let $m=(m_1,\dots,m_5)$. Following \cite[p.~540]{AK12}, a \emph{normalized central configuration} is a solution $(x,y,\delta)\in\C^5\times\C^5\times\C^{10}$ of
\begin{equation}\label{eq:AK4}
x_k=\sum_{l\ne k}m_l\,\delta_{kl}^{3}(x_k-x_l),\quad y_k=\sum_{l\ne k}m_l\,\delta_{kl}^{3}(y_k-y_l),\quad \delta_{kl}^{2}\big((x_k-x_l)^2+(y_k-y_l)^2\big)=1,\quad y_1=y_2,
\end{equation}
for $k\ne l$ in $\{1,\dots,5\}$, with $\delta_{kl}=\delta_{lk}$. It is \emph{real} if $x,y\in\R^5$, and \emph{positive} if moreover all $\delta_{kl}$ are positive. Summing the first equation multiplied by $m_k$ shows that solutions have their centre of mass at the origin. The multiplier $\lambda$ of \eqref{eq:cc} is normalized to $1$.

\begin{lemma}\label{lem:realpos}
Let $m\in(\R_{>0})^5$. Every central configuration with masses $m$ is similar, by an orientation-preserving similarity, to exactly two positive normalized central configurations, which differ by the rotation by $\pi$. In particular, if \eqref{eq:AK4} has finitely many solutions, there are finitely many central configurations up to similarity.
\end{lemma}

\begin{proof}
Taking the scalar product of \eqref{eq:cc} with $m_i(q_i-c)$ and summing gives $\lambda I=U$ with $I=\sum_im_i|q_i-c|^2>0$ and $U=\sum_{i<j}m_im_j/r_{ij}>0$. So $\lambda>0$. We translate $c$ to the origin, scale by $\lambda^{1/3}$, and rotate so that $q_1-q_2$ is horizontal. Exactly two rotations do this. Conversely, the translation and the scale of a positive normalized central configuration are fixed by \eqref{eq:AK4}.
\end{proof}

We put $z=x+iy$, $w=x-iy$ and $r_{kl}=\delta_{kl}^{-1}$, so that $r_{kl}^2=z_{kl}w_{kl}$ with $z_{kl}=z_k-z_l$. For real configurations $w=\bar z$. For complex ones $z$ and $w$ are independent.

\subsection{Centred coordinates and the slice}\label{ss:slice}
For $h\in\C^5$ and $p=(p_{kl})_{k<l}$ with $p_{lk}=p_{kl}$, put $F_k(h)=h_k-\sum_{l\ne k}m_l\,p_{kl}^3(h_k-h_l)$. The unknowns are $u=(u_1,\dots,u_4)$, $v=(v_1,\dots,v_4)$ and $p$, with $u_5=v_5=0$, $u_{kl}=u_k-u_l$, $v_{kl}=v_k-v_l$. The eighteen \emph{rows} are
\begin{equation}\label{eq:rows}
D_{kl}=p_{kl}^2u_{kl}v_{kl}-1,\qquad Z_k=F_k(u)-F_5(u),\qquad W_k=F_k(v)-F_5(v)\qquad(k\le 4).
\end{equation}
Let $X_m\subset\C^{18}$ be their zero set. On $X_m$ the functions $u_k$, $v_k$, $u_{kl}$, $v_{kl}$ and $p_{kl}$ vanish nowhere. The map $(x,y,\delta)\mapsto(u,v,p)$ with $u_k=z_k-z_5$, $v_k=w_k-w_5$, $p=\delta$ identifies the solutions of the first three equations of \eqref{eq:AK4} with $X_m$, and the solutions of \eqref{eq:AK4} with $X_m\cap\{u_{12}=v_{12}\}$.

\begin{lemma}\label{lem:slice}
\begin{enumerate}[label=\textup{(\alph*)}]
\item The group $\C^*$ acts freely on $X_m$ by $s\cdot(u,v,p)=(su,s^{-1}v,p)$, the complex rotation $(z,w)\mapsto(sz,s^{-1}w)$. Let $Y_m=X_m\cap\{u_1=1\}$. Then $\C^*\times Y_m\to X_m$, $(s,y)\mapsto s\cdot y$, is an isomorphism, every orbit meets $\{u_{12}=v_{12}\}$ in exactly two points, and \eqref{eq:AK4} has exactly $2|Y_m|$ solutions.
\item For $t>0$ the map $(u,v,p)\mapsto(t^{1/3}u,t^{1/3}v,t^{-1/3}p)$ is an isomorphism $X_m\to X_{tm}$ that preserves $\{u_{12}=v_{12}\}$, and its composition with the rotation by $s=t^{-1/3}$, $(u,v,p)\mapsto(u,t^{2/3}v,t^{-1/3}p)$, is an isomorphism $Y_m\to Y_{tm}$. Relabelling the bodies together with their masses, followed by a translation and a rotation, is an isomorphism $Y_m\to Y_{m'}$ for the permuted masses $m'$. In particular finiteness of $Y_m$ is invariant under positive scaling and permutation of the masses.
\item A configuration satisfies \eqref{eq:cc} for the masses $m$ with multiplier $\lambda$ if and only if it satisfies \eqref{eq:cc} for $tm$ with multiplier $t\lambda$, and relabelling the bodies together with their masses maps the central configurations for $m$ to those for the permuted masses. Both bijections preserve similarity and orientation. Hence the number of central configurations up to similarity, and up to orientation-preserving similarity, is invariant under positive scaling and permutation of the masses.
\end{enumerate}
\end{lemma}

\begin{proof}
(a) The rows $D_{kl}$ are invariant and $Z_k$, $W_k$ are multiplied by $s$, $s^{-1}$. The action is free because $u_1\ne0$ on $X_m$. The inverse of the map is $x\mapsto(u_1,u_1^{-1}\cdot x)$. The point $s\cdot x$ lies in $\{u_{12}=v_{12}\}$ if and only if $s^2=v_{12}/u_{12}$. (b) The rows for $tm$ at the image point are multiples of the rows for $m$. For a relabelling, we translate the new body $5$ to the origin and rotate to restore $u_1=1$. (c) Multiplying the masses by $t$ multiplies the right side of \eqref{eq:cc} by $t^2$ and the left side by $t$ for fixed $\lambda$. Relabelling permutes the equations.
\end{proof}

For $\{i,j,k,l,n\}=\{1,\dots,5\}$ the \emph{4-product} of \cite[p.~580]{AK12} is $\cP_{ij}=r_{ij}^2r_{kl}^2r_{ln}^2r_{nk}^2$.

\subsection{The potential form of the rows}\label{ss:potential}
Let
\[
U=\sum_{k<l}m_km_l\,p_{kl},\qquad I=\sum_{k}m_ku_kv_k-\frac1M\Big(\sum_km_ku_k\Big)\Big(\sum_km_kv_k\Big),\qquad V=U+\tfrac12I,
\]
where $M=\sum_km_k$. We regard the $p_{kl}$, on the hypersurface defined by the rows $D_{kl}$, as functions of $(u,v)$.

\begin{lemma}\label{lem:potential}
On $\{D_{kl}=0\}$, with $\nabla_u$ and $\nabla_v$ the gradients in $(u_1,\dots,u_4)$ and $(v_1,\dots,v_4)$,
\[
(Z_1,\dots,Z_4)=2A\,\nabla_vV,\qquad (W_1,\dots,W_4)=2A\,\nabla_uV,\qquad A=\operatorname{diag}(m_1^{-1},\dots,m_4^{-1})+m_5^{-1}\one\one^{\mathsf T}.
\]
The matrix $A$ is invertible, with $A^{-1}=\operatorname{diag}(m_1,\dots,m_4)-mm^{\mathsf T}/M$ (restricted to the first four coordinates).
\end{lemma}

\begin{proof}
We write $\tilde F_k=\sum_{l\ne k}m_l(u_l-u_k)p_{kl}^3$, so that $Z_k=u_k+\tilde F_k-\tilde F_5$. From $p_{kl}^2u_{kl}v_{kl}=1$ we get $\partial p_{kl}/\partial v_k=-\tfrac12p_{kl}^3u_{kl}$. So $\partial U/\partial v_k=\tfrac12m_k\tilde F_k$. Also $\partial I/\partial v_k=m_k(u_k-u_G)$ with $u_G=\sum_lm_lu_l/M$. Hence $(2/m_k)\,\partial V/\partial v_k=u_k-u_G+\tilde F_k$. By antisymmetry $\sum_{l=1}^5m_l\tilde F_l=0$. Since $u_5=0$, also $\sum_{l\le4}m_l(u_l-u_G)=m_5u_G$. Hence $(2/m_5)\sum_{l\le4}\partial V/\partial v_l=u_G-\tilde F_5$ and $(2A\nabla_vV)_k=u_k+\tilde F_k-\tilde F_5=Z_k$. Exchanging $u$ and $v$ gives the rows $W_k$.
\end{proof}

A point of $X_m$ is thus a critical point of $V$ restricted to $\{D_{kl}=0\}$. At fixed masses all points of a family in $X_m$ are critical points of the same $V$. The next lemma is \cite[Lemma~2]{AK12}, in the coordinates of the slice.

\begin{lemma}\label{lem:U}
Let $M\ne0$.
\begin{enumerate}[label=(\alph*)]
\item On $X_m$, $U=\sum_km_kz_kw_k=I$, with $z,w$ recentred at the centre of mass.
\item $U$ is constant on every irreducible component of $X_m$ and of $Y_m$, in particular on every irreducible curve $C\subseteq Y_m$.
\item For $m\in(\R_{>0})^5$, $U=I>0$ at the point of $Y_m$ given by a real central configuration.
\end{enumerate}
\end{lemma}

\begin{proof}
(a) In the coordinates $z,w$ of \eqref{eq:AK4} the centre of mass is at the origin and $I=\sum_km_kz_kw_k$. We multiply the equation for $z_k$ by $m_kw_k$ and sum over $k$. By antisymmetry and $p_{kl}^2z_{kl}w_{kl}=1$, $\sum_km_kz_kw_k=\sum_{k<l}m_km_lp_{kl}^3z_{kl}w_{kl}=\sum_{k<l}m_km_lp_{kl}=U$. (b) By Lemma~\ref{lem:potential} and the invertibility of $A$, the gradient of $V$ on $\{D_{kl}=0\}$ vanishes at every point of $X_m$. The differential of $V$ therefore vanishes on the tangent space of each irreducible component at its smooth points, which form a connected dense open subset. Hence $V$ is constant on the component. By (a), $V=\tfrac32U$ on $X_m$. (c) is the identity $\lambda I=U$ of the proof of Lemma~\ref{lem:realpos}, with $\lambda=1$.
\end{proof}

\begin{corollary}\label{cor:U}
Let $m\in(\R_{>0})^5$. If $Y_m$ contains no irreducible curve on which some $r_{kl}^2$ is non-constant and $U$ is not identically zero, then the planar five-body problem with masses $m$ has finitely many central configurations up to similarity.
\end{corollary}

\begin{proof}
Infinitely many central configurations give infinitely many points of $Y_m$ with $U>0$ (Lemmas~\ref{lem:realpos}, \ref{lem:slice} and \ref{lem:U}(c)), hence an irreducible curve $C\subseteq Y_m$ on which $U$ is a positive constant. The fibres of $(u,v,p)\mapsto(r_{kl}^2)_{k<l}$ on $Y_m$ have at most $2^{13}$ points (fix the distances: $u_1=1$ determines $v_1$, each $u_k$ is a root of a quadratic, and the signs of the $p_{kl}$ are free). Hence some $r_{kl}^2$ is not constant on $C$.
\end{proof}

We call $m\in(\R_{>0})^5$ \emph{settled} if $Y_m$ contains no irreducible curve on which some $r_{kl}^2$ is non-constant and $U\not\equiv0$. We call such a curve \emph{nontrivial}. A settled mass vector has finitely many central configurations up to similarity. The isomorphisms of Lemma~\ref{lem:slice}(b) multiply the $r_{kl}^2$ and $U$ by nonzero constants or permute them. Thus settledness is invariant under positive scaling and permutation of the masses.

\subsection{Places and balancing}\label{ss:balancing}
Let $C\subseteq Y_m$ be a nontrivial curve. The functions $u_k$, $v_k$, $u_{kl}$, $v_{kl}$, $p_{kl}$, $r_{kl}^2$ and $\cP_{ij}$ are units of the coordinate ring $\C[C]$: for instance $1/p_{kl}=p_{kl}u_{kl}v_{kl}$. A non-constant unit has a zero and a pole on the smooth projective model $\bar C$ of $C$. All of them lie at the finitely many \emph{places at infinity}, the points of $\bar C$ not over $C$. For a place $P$ at infinity with valuation $\val_P$ put
\[
\rho(P)=\big(-\val_P(p_{kl})\big)_{k<l}\in\Z^{10}.
\]
Since the divisor of each $p_{kl}$ has degree zero,
\begin{equation}\label{eq:bal}
\sum_{P\text{ at }\infty}\rho(P)=0,
\end{equation}
and $\rho(P)\ne0$ for some $P$, because some $p_{kl}$ is not constant.

On $X_m$ we put $z_k=u_k-F_5(u)$ and $w_k=v_k-F_5(v)$ for $k\le5$. Since the rows $Z_k$, $W_k$ say $F_k(u)=F_5(u)$ and $F_k(v)=F_5(v)$, the pair $(z,w)$ solves the first two equations of \eqref{eq:AK4} (in the variables $z=x+iy$, $w=x-iy$), with $z_k-z_l=u_{kl}$ and $w_k-w_l=v_{kl}$. Following \cite[p.~545]{AK12} we put $Z_{kl}=p_{kl}^3z_{kl}$, $W_{kl}=p_{kl}^3w_{kl}$ (\cite{AK12} write $z_{kl}$ for $z_l-z_k$, but signs play no role below),
\[
\cZ=(z_1,\dots,z_5,Z_{12},\dots,Z_{45}),\qquad \cW=(w_1,\dots,w_5,W_{12},\dots,W_{45}),
\]
and $\cN=\|\cZ\|\,\|\cW\|$ with the maximum norm. The rotation $s\cdot$ of Lemma~\ref{lem:slice} multiplies $\cZ$ by $s$ and $\cW$ by $s^{-1}$. Thus $\cN$ is a function on $Y_m$. Since $Z_{kl}W_{kl}=p_{kl}^4$ and $|z_{kl}w_{kl}|\le4\cN$,
\begin{equation}\label{eq:Nbounds}
|r_{kl}|^{-4}\le\cN,\qquad |r_{kl}|^{2}\le4\cN .
\end{equation}
A place $P$ at infinity is \emph{singular} if $\cN$ is unbounded near $P$, that is, if $\val_P(\cZ)+\val_P(\cW)<0$, where $\val_P$ of a vector is the minimum over its components. At a place $P$ that is not singular, \eqref{eq:Nbounds} bounds every $r_{kl}^2$ above and away from zero. So $\rho(P)=0$. Hence the sum \eqref{eq:bal} runs over the singular places and at least one place is singular.

\subsection{Diagrams}\label{ss:diagrams}
Along a singular place, \cite[p.~545]{AK12} draws a circle around the body $k$ if $z_k$ is of maximal order among the components of $\cZ$, and a stroke between $k$ and $l$ if $Z_{kl}$ is of maximal order. This gives the $z$-diagram. In the same way $\cW$ gives the $w$-diagram. An edge is a $z$-edge, a $w$-edge or a $zw$-edge according to which strokes are present. The conclusion of \cite[\S7]{AK12} is that every singular sequence approaches one of the sixteen diagrams of their Fig.~11, up to renumbering of the bodies and exchange of the colours $z$ and $w$. Table~\ref{tab:diagrams} in Appendix~\ref{app:types} lists them with the types used below. We say that a singular place has \emph{diagram} $d$ and \emph{labelling} $\lambda$ if it approaches diagram $d$ with the renumbering $\lambda$.

\subsection{Inputs from \cite{AK12}}\label{ss:akinputs}
We use the following results of \cite{AK12}, as stated there:
\begin{enumerate}[label=(\roman*)]
\item the conclusion of \S7: every singular sequence approaches one of the sixteen diagrams of Fig.~11, which are listed in Table~\ref{tab:diagrams};
\item Estimates~1 and~2 of \S3, which say, in the notation there, that for each pair, $\varepsilon\preceq r_{kl}\preceq\varepsilon^{-2}$ and $\varepsilon^2\preceq z_{kl}\preceq\varepsilon^{-2}$, with $r_{kl}\approx\varepsilon$ exactly at the $zw$-edges and $z_{kl}\approx\varepsilon^2$ exactly at the maximal $w$-edges; and on a $z$-stroke, $\varepsilon\preceq r_{kl}\preceq1$, $\varepsilon\preceq z_{kl}\preceq\varepsilon^{-2}$ and $\varepsilon^2\preceq w_{kl}\preceq\varepsilon$, where $r_{kl}\approx\varepsilon$, $z_{kl}\approx\varepsilon$ and $w_{kl}\approx\varepsilon$ each characterize a $zw$-edge, and $r_{kl}\approx1$, $z_{kl}\approx\varepsilon^{-2}$ and $w_{kl}\approx\varepsilon^2$ each characterize a maximal $z$-edge; Rules~1b--1e and~2e of \S3, Proposition~1 of \S5.4 and Proposition~4 of \S7.1;
\item the mass conditions of \S8: conditions (23), (29), (32), (36) and (40), the facts of \S8.7 quoted in \S\ref{ss:d7} (among them the reduction to the system of \S5.4), and \S\S8.9.1--8.9.2;
\item the proof of Proposition~9, which we use in one place where we also give a second argument.
\end{enumerate}
The hypotheses $m_k\ne0$, $M\ne0$, $M_Q\ne0$ for each four-body subset $Q$, $m_a+m_b\ne0$ and $m_k^3+m_l^3\ne0$, under which these results are stated, hold for positive masses. In diagram~$3$ of Fig.~11 the colours of the four sides alternate around the square, $z$ on $14$ and $23$ and $w$ on $12$ and $34$. Note that the alternation is also forced by Rule~2g of \cite[\S3]{AK12} and by the order calculus of \S\ref{ss:orders}, which admits only the two alternating colourings of the sides. It occurs at places of the modified Roberts continuum \cite[pp.~543--544]{AK12}.

%% file: sec/s3_types.tex
\section{Types of singular places and mass conditions}\label{sec:types}

Throughout this section $m\in(\R_{>0})^5$, $C\subseteq Y_m$ is a nontrivial curve, and $P$ is a singular place of $C$. We determine the possible directions of $\rho(P)$, together with the conditions on $m$ under which each direction can occur.

\subsection{Orders at a singular place}\label{ss:orders}
Let $t$ be a local parameter at $P$ and $\nu=\val_P$. We call a function on $C$ of \emph{$z$-type} if the rotation $s\cdot$ multiplies it by $s$ ($u_{kl}$, $z_k$, $z_{kl}$, $Z_{kl}$), of \emph{$w$-type} if it multiplies it by $s^{-1}$ ($v_{kl}$, $w_k$, $w_{kl}$, $W_{kl}$), and \emph{neutral} if it is invariant ($p_{kl}$, $r_{kl}^2$). Put $d=-\tfrac12(\nu(\cZ)+\nu(\cW))>0$ and $\gamma=\tfrac12(\nu(\cW)-\nu(\cZ))$, and define the \emph{exponent}
\[
e(f)=\frac2d\,\big(\nu(f)+\gamma\big),\qquad e(f)=\frac2d\,\big(\nu(f)-\gamma\big),\qquad e(f)=\frac2d\,\nu(f)
\]
for $f$ of $z$-type, $w$-type and neutral respectively. Then all components of $\cZ$ and of $\cW$ have exponent at least $-2$. In each of the two vectors some component has exponent $-2$.

\begin{lemma}\label{lem:place}
Let $x_n\in C$ converge to $P$, and let $x_n'$ be one of the two solutions of \eqref{eq:AK4} in the orbit of $x_n$ \textup{(}Lemma~\ref{lem:slice}\textup{)}. After passing to a subsequence, $(x'_n)$ is a singular sequence in the sense of \cite[Def.~3]{AK12}. After the normalization $\|\cZ\|=\|\cW\|=\varepsilon_n^{-2}$ of \cite[\S2]{AK12}, every function $f$ of the three kinds satisfies $|f(x'_n)|\approx\varepsilon_n^{e(f)}$. In particular:
\begin{enumerate}[label=\textup{(\roman*)}]
\item $k$ is $z$-circled if and only if $e(z_k)=-2$;
\item there is a $z$-stroke between $k$ and $l$ if and only if $e(Z_{kl})=-2$;
\item $k$ and $l$ are $z$-close if and only if $e(z_{kl})>-2$;
\end{enumerate}
and the same holds for $w$. Every such sequence therefore has the same diagram, which we call the diagram of $P$. Moreover, if $\vartheta^2=t$ and the point of $C$ near $P$ is rotated by $s=\vartheta^{2\gamma}$, every function $f$ of the three kinds becomes $f=\vartheta^{d\,e(f)}y_f(\vartheta)$ with $y_f\in\C\{\vartheta\}$ and $y_f(0)\ne0$.
\end{lemma}

\begin{proof}
A function $f$ as above that is not identically zero on $C$ is $t^{\nu(f)}$ times a unit of $\C\{t\}$. The rotation multiplies it by $\vartheta^{2\gamma}$, $\vartheta^{-2\gamma}$ or $1$. This gives the last statement, since $d\,e(f)=2\nu(f)\pm2\gamma$ or $2\nu(f)$. Each component of $\cZ$ and $\cW$ is a Laurent series in $t$ near $P$. The ratios of components converge in $\C\cup\{\infty\}$. Passing from $x_n$ to $x'_n$ multiplies $\cZ$ by a scalar $s_n$ and $\cW$ by $s_n^{-1}$, which changes none of these ratios. Since $P$ is singular, $\cN\to\infty$. Since $z_{12}=w_{12}$ on solutions of \eqref{eq:AK4}, at these solutions $r_{12}^2=z_{12}^2$, $|Z_{12}|=|z_{12}|^{-2}$ and $\|\cZ\|\ge\max(|z_{12}|/2,|z_{12}|^{-2})\ge2^{-2/3}$, and likewise for $\cW$. Hence $\cZ$ or $\cW$ is unbounded along $(x'_n)$. The extraction of \cite[pp.~544--545]{AK12} only fixes the limits of the ratios. The normalization multiplies $z$ by a real $a_n>0$ and $w$ by $a_n^{-1}$ so that $\|\cZ\|=\|\cW\|$. Then $|a_ns_n|^2=\|\cW(x_n)\|/\|\cZ(x_n)\|\approx|t|^{2\gamma}$. Hence a function $f$ of $z$-type has modulus $|a_ns_nf(x_n)|\approx|t|^{\nu(f)+\gamma}=|t|^{d\,e(f)/2}$, and likewise for the other two kinds. Finally $\varepsilon_n^{-2}=\|\cZ\|\approx|t|^{-d}$. The characterizations are the definitions of circles, strokes and closeness in \cite[pp.~545--546]{AK12}: circles and strokes mark the components of maximal order, and $z$-close means $z_{kl}\prec\cZ$.
\end{proof}

For a pair $kl$ put
\[
\zeta_{kl}=e(z_{kl}),\qquad \omega_{kl}=e(w_{kl}),\qquad x_{kl}=e(r_{kl})=\tfrac12(\zeta_{kl}+\omega_{kl}),
\]
so that $e(Z_{kl})=-\tfrac12(\zeta_{kl}+3\omega_{kl})$ and $e(W_{kl})=-\tfrac12(3\zeta_{kl}+\omega_{kl})$. The vector $x(P)=(x_{kl})_{k<l}\in\Q^{10}$ is the \emph{type} of $P$. Since $x_{kl}=\tfrac2d\,\nu(r_{kl})=\tfrac2d\,\rho_{kl}(P)$, the type is a positive multiple of $\rho(P)$ and \eqref{eq:bal} becomes
\begin{equation}\label{eq:balx}
\sum_{P\ \text{singular}}d_P\,x(P)=0,\qquad d_P>0 .
\end{equation}
We index the coordinates of $\R^{10}$ by the pairs $12,13,\dots,45$.

\begin{lemma}\label{lem:orders}
For every pair $kl$:
\begin{enumerate}[label=\textup{(O\arabic*)}]
\item $\zeta_{kl},\omega_{kl}\ge-2$, $\zeta_{kl}+3\omega_{kl}\le4$ and $3\zeta_{kl}+\omega_{kl}\le4$. Hence $\zeta_{kl},\omega_{kl}\le2$ and $-2\le x_{kl}\le1$.
\item There is a $z$-stroke $kl$ if and only if $\zeta_{kl}+3\omega_{kl}=4$, and a $w$-stroke if and only if $3\zeta_{kl}+\omega_{kl}=4$.
\item On a $z$-stroke, $\zeta_{kl}=3x_{kl}-2$, $\omega_{kl}=2-x_{kl}$ and $0\le x_{kl}\le1$. On a $w$-stroke, $\omega_{kl}=3x_{kl}-2$ and $\zeta_{kl}=2-x_{kl}$.
\item $x_{kl}=1$ if and only if $kl$ is a $zw$-edge. A $z$-stroke is maximal \textup{(}its ends are not $z$-close, \cite[Def.~6]{AK12}\textup{)} if and only if $x_{kl}=0$.
\item In every triangle $klm$ the minimum of $\zeta_{kl},\zeta_{lm},\zeta_{mk}$ is attained at least twice, and the same holds for $\omega$.
\item If $k$ is $z$-circled and $l$ is not, then $\zeta_{kl}=-2$. If $\zeta_{kl}=2$, then $\omega_{kl}=-2$ and $kl$ is a $w$-stroke.
\end{enumerate}
\end{lemma}

\begin{proof}
(O1) Since $e(z_k)\ge-2$ for all $k$, $\zeta_{kl}\ge\min(e(z_k),e(z_l))\ge-2$. The two linear inequalities are $e(Z_{kl})\ge-2$ and $e(W_{kl})\ge-2$. (O2) A stroke means $e(Z_{kl})=-2$. (O3) We solve $\zeta+3\omega=4$, $\zeta+\omega=2x$, and use $\zeta\ge-2$. (O4) If $\zeta+\omega=2$, the inequalities of (O1) give $\zeta\le1$ and $\omega\le1$, hence $\zeta=\omega=1$ and both strokes are present. Conversely a $zw$-edge has $\zeta=\omega=1$. A $z$-stroke has $z$-close ends if and only if $\zeta>-2$, i.e.\ $x>0$. (O5) follows from $z_{kl}+z_{lm}+z_{mk}=0$. (O6) $e(z_k)=-2<e(z_l)$ gives $\zeta_{kl}=-2$. If $\zeta_{kl}=2$, (O1) gives $\omega_{kl}\le-2$, hence $\omega_{kl}=-2$ and $3\zeta_{kl}+\omega_{kl}=4$.
\end{proof}

(O6) is Rule~1b of \cite{AK12}. Besides Lemma~\ref{lem:orders} we use the following results of \cite{AK12}, which are among the inputs of \S\ref{ss:akinputs}: Rules~1d, 1e and 2e \cite[pp.~546--552]{AK12} (circled bodies come at least in pairs in an isolated component and are not all mutually close; an isolated component has no circle if and only if it has no maximal stroke; the three edges of a triangle of edges have the same type, equal $\zeta$ and equal $\omega$), Proposition~4 \cite[p.~564]{AK12} (if a diagram has a single maximal $z$-edge $12$, forming an isolated component, then no body $k\ge3$ satisfies $w_{1k}\prec w_{1l}$ for all other $l\ge3$), and Proposition~1 with $z$ and $w$ exchanged \cite[p.~559]{AK12} (if $w_{12}\prec w_{k2}$ and $z_{12}\approx z_{k2}$ for all $k\ge3$, then $w_1,w_2\preceq w_{12}$).

\begin{lemma}\label{lem:cut}
At every singular place of $C$, $\min_{k<l}x_{kl}(P)\le0$.
\end{lemma}

\begin{proof}
By Lemma~\ref{lem:U}, $U=I$ on $X_m$ and $U$ is a nonzero constant on $C$. In the variables of \eqref{eq:AK4}, $I=\sum_km_kz_kw_k=M^{-1}\sum_{k<l}m_km_lr_{kl}^2$. If all $x_{kl}(P)$ were positive, all $r_{kl}^2$ would tend to zero at $P$, and so would $I$.
\end{proof}

\begin{lemma}[three bodies near a cluster]\label{lem:nearcluster}
Let $K$ be a set of at least three bodies, pairwise joined by $zw$-edges, and let $X\notin K$ have no stroke to $K$. In each colour, say $z$, exactly one of the following holds: all $\zeta_{Xt}$, $t\in K$, are equal and less than $1$ \textup{(}$X$ is $z$-far\textup{)}; or all $\zeta_{Xt}\ge1$, at most one of them is larger than $1$ \textup{(}$X$ is $z$-near\textup{)}. The body $X$ is not near in both colours. With $\bar\omega<1$ the common far $w$-exponent and $s$ the exceptional $\zeta$, the possible modes are therefore:
\begin{itemize}
\item \textup{F}: $x_{Xt}=c<1$ for all $t$;
\item \textup{Z1}: $x_{Xt}=\frac{1+\bar\omega}2$ for all $t$;
\item \textup{Zs}$_{t_0}$: $x_{Xt_0}=\frac{s+\bar\omega}2$ and $x_{Xt}=\frac{1+\bar\omega}2$ for $t\ne t_0$, with $s>1$;
\item \textup{W1} and \textup{Ws}$_{t_0}$: the same with the colours exchanged.
\end{itemize} The parameters satisfy $s<2$, $s+\bar\omega<2$, $2s+\bar\omega<3$ and $2s+3\bar\omega<5$.
\end{lemma}

\begin{proof}
Inside $K$, $\zeta_{tt'}=1$. If some $\zeta_{Xt}<1$, (O5) forces $\zeta_{Xt'}=\zeta_{Xt}$ for all $t'\in K$. Two values above $1$ would give $\zeta_{tt'}>1$, again by (O5). Near in both colours would give $x_{Xt}\ge1$, hence a $zw$-edge by (O4). (O6) gives $s<2$. No stroke means $s+3\bar\omega<4$ and $3s+\bar\omega<4$. We obtain the remaining inequalities as their combinations with weights $(\frac14,\frac14)$, $(\frac78,\frac38)$, $(\frac18,\frac58)$. The inequality $\bar\omega<1$ follows from $3s+\bar\omega<4$ and $s>1$.
\end{proof}

\subsection{The types of the sixteen diagrams}\label{ss:typestable}
We name some points of $\R^{10}$. For a triple $T$ let $P_T$ be the complementary pair. For a body $j$ let $Q_j=\{1,\dots,5\}\setminus\{j\}$. For a set $A$ let $\mathrm{tri}(A)$ be its set of pairs.
\begin{itemize}
\item $\pc a_1(T)$: $0$ on $\mathrm{tri}(T)$ and on $P_T$, $-2$ on the six other pairs; $\pc b_1(T)$: $1$ on $\mathrm{tri}(T)$, $0$ on $P_T$, $-\tfrac12$ on the six other pairs.
\item $\pc h(s;D_1|D_2)$, for a body $s$ and a partition of $Q_s$ into two pairs $D_1,D_2$ (the diagonals): $0$ on the four other pairs of $Q_s$ (the sides), $-2$ on $D_1$, $D_2$ and on the four pairs $ks$.
\item $\pc v(s;A|B)$, for triples $A,B$ with $A\cap B=\{s\}$: $0$ on $\mathrm{tri}(A)\cup\mathrm{tri}(B)$, $-2$ on the four pairs between $A\setminus s$ and $B\setminus s$. More generally $\pc v_7(a,b)$ is $a$ on $\mathrm{tri}(A)$, $b$ on $\mathrm{tri}(B)$ and $\tfrac32(a+b)-2$ on the four cross pairs. The type $\pc u_A(c)$ is $1$ on $\mathrm{tri}(A)$, $c$ on $\mathrm{tri}(B)$ and $\tfrac{3c-1}2$ on the cross pairs.
\item $\pc n(T;ax,by)$, for a triple $T$, $P_T=\{x,y\}$ and $a\ne b$ in $T$: $1$ on $\mathrm{tri}(T)$, $0$ on $ax$, $by$, $xy$, and $-\tfrac12$ on the four other pairs.
\item $\pc u_c(j)$: $1$ on $\mathrm{tri}(Q_j)$, $c$ on the four pairs $ij$; $\pc b_k(j)$ for $k\in Q_j$: $1$ on $\mathrm{tri}(Q_j)$, $0$ on $kj$, $-\tfrac12$ on $ij$, $i\in Q_j\setminus k$; and $\sigma_k(j)(t)=(1-t)\,\pc u_0(j)+t\,\pc b_k(j)$.
\end{itemize}

\begin{proposition}\label{prop:types}
Let $P$ have diagram $d$, with the bodies numbered as in \cite[Fig.~11]{AK12}. Then $x(P)$ is as in Table~\ref{tab:types}. In the table $T=\{3,4,5\}$ in diagrams $1$, $2$, $4$, $5$. The wings are $A=\{1,2,5\}$ and $B=\{3,4,5\}$ in diagrams $6$--$8$ and $10$--$13$, with centre $5$. In diagrams $14$ and $15$, $j=5$ and $Q=Q_5$, with $k$ the body of $Q$ joined to $5$ in diagram $15$.
\end{proposition}

\begin{table}[ht]
\centering
\small
\begin{tabular}{@{}lll@{}}
\toprule
diagram & type $x(P)$ & range\\
\midrule
1 & $(1-a)\,\pc a_1(T)+a\,\pc b_1(T)$ & $0<a<1$\\
2 & $\pc a_1(T)$ & \\
3 & $\pc h(5;13|24)$ & \\
4 & $\pc b_1(T)$ & \\
5 & see \S\ref{ss:d5} & \\
6 & $\pc v(5;A|B)=\pc v_7(0,0)$ & \\
7 & $\pc v_7(a,b)$ & $0<a,b<1$, $a+b\le\frac43$\\
8 & $\pc v_7(a,0)$ & $0<a<1$\\
9 & $\pc n(\{1,2,3\};14,25)$ & \\
10 & $\pc u_A(c)$ & $0<c\le\frac13$\\
11--13 & $\pc u_A(0)$ & \\
14 & $\pc u_0(5)$, or $\sigma_k(5)(t)$ for some $k\in Q$ & $0<t<1$\\
15 & $\pc b_k(5)$ & \\
16 & does not occur & \\
\bottomrule
\end{tabular}
\caption{The types of the sixteen diagrams of \cite[Fig.~11]{AK12} on a curve with $U\not\equiv0$, for the numbering of the figure.}\label{tab:types}
\end{table}

We give the derivations for diagrams $1$--$13$ and $16$ in Appendix~\ref{app:types}. They use only Lemmas~\ref{lem:orders} and \ref{lem:cut}, the rules of \cite{AK12} quoted above and, for diagram~$9$, \cite[\S8.9]{AK12} and Lemma~\ref{lem:U}. Two of them improve on what the order calculus gives directly. We prove them here. For diagram~$1$ the order calculus leaves two patterns for the exponents $\mu_t=\omega_{1t}=\omega_{2t}$, $t\in T$ (Appendix~\ref{app:types}): either all $\mu_t$ equal some $\mu\le3a-2$, or two of them equal $3a-2$ and the third is larger.

\begin{lemma}\label{lem:S1}
At a place with diagram $1$, $\mu_t=3a-2$ for every $t\in T$. At a place with diagram $4$, $\mu_t=1$ for every $t\in T$.
\end{lemma}

\begin{proof}
Diagram 1. In the second pattern $w_{1t_0}\prec w_{1t}$ for the body $t_0$ with the larger exponent and both other $t\in T$. Since $12$ is the only maximal $z$-edge, the $z$-diagram is $\{12\}$, an isolated component, and $m_1^3+m_2^3>0$, Proposition~4 of \cite{AK12} excludes this. Suppose all $\mu_t$ equal $\mu<3a-2$. Then we have:
\begin{enumerate}[label=\textup{(\roman*)}]
\item $\omega_{12}=2$ by (O3);
\item $\mu<2$, since $\mu=2$ would make $1t$ a $z$-stroke by (O2);
\item $\zeta_{12}=\zeta_{t2}=-2$ by (O6).
\end{enumerate} Proposition~1 of \cite{AK12} with the colours exchanged then gives $e(w_1),e(w_2)\ge2$. The differences $w_{1t}-w_{1t'}=\pm w_{tt'}$ have exponent $3a-2>\mu$. Consequently $w_{13}$, $w_{14}$, $w_{15}$ have the same coefficient $\kappa\ne0$ at exponent $\mu$. As $w_t=w_1-w_{1t}$ and $e(w_1)\ge2>\mu$, the coefficient of $m_3w_3+m_4w_4+m_5w_5$ at exponent $\mu$ is $-(m_3+m_4+m_5)\kappa\ne0$. But this sum equals $-(m_1w_1+m_2w_2)$, whose exponent is at least $2$. Hence $\mu_t=3a-2$ for all $t$.

Diagram 4 is the same argument with $3a-2$ replaced by $1$: the $z$-strokes of the triangle $T$ have $x=1$ and are not maximal. Again $12$ is the only maximal $z$-edge and an isolated component.
\end{proof}

A place with diagram $1$ thus has type $(1-a)\pc a_1(T)+a\pc b_1(T)$, $0<a<1$, and a place with diagram $4$ has type $\pc b_1(T)$, as in Table~\ref{tab:types}.

For diagrams $14$ and $15$ we need the position of the fifth body relative to the cluster. The next lemma replaces \cite[\S6.1 and Prop.~2]{AK12} in the form we need. We state it in terms of $\alpha_i=\tfrac12e(z_{ij})$ and $\beta_i=\tfrac12e(w_{ij})$.

\begin{lemma}\label{lem:zwK4}
Suppose $\{kl:x_{kl}(P)=1\}=\mathrm{tri}(Q_j)$, and put $Q=Q_j$. Exactly one of the following holds, possibly after exchanging the colours.
\begin{enumerate}[label=\textup{(\Alph*)}]
\item $\alpha_i=\alpha$ and $\beta_i=\beta$ for all $i\in Q$, with $\alpha,\beta\le\frac12$, and $x_{ij}=c:=\alpha+\beta$ for all $i\in Q$, $0\le c<1$. So $x(P)=\pc u_c(j)$, and no stroke meets $j$.
\item There is a unique $k\in Q$ with $\alpha_k>\frac12$. Put $a=2\alpha_k$. Then $1<a\le2$, $\alpha_i=\frac12$ for $i\ne k$, $\beta_i=-\frac a2$ for all $i\in Q$, $x_{kj}=0$ and $x_{ij}=\frac{1-a}2$ for $i\ne k$, that is, $x(P)=\sigma_k(j)(a-1)$. If $a<2$, no stroke meets $j$. If $a=2$, the only stroke at $j$ is the $w$-stroke $kj$, and $kj$ is a maximal $w$-edge.
\end{enumerate}
\end{lemma}

\begin{proof}
(1) Inside $Q$, $e(z_{il})=e(w_{il})=1$ for each pair by (O4). If $\alpha_i<\frac12$ for some $i$, then $z_{i'j}=z_{i'i}+z_{ij}$ gives $\alpha_{i'}=\alpha_i$ for all $i'\in Q$. Otherwise all $\alpha_i\ge\frac12$. Two indices with $\alpha_k,\alpha_{k'}>\frac12$ would give $e(z_{kk'})>1$. Either all $\alpha_i$ equal some $\alpha\le\frac12$, or exactly one exceeds $\frac12$ and the others equal $\frac12$. The same holds for the $\beta_i$.
(2) Not both colours have an exceptional index: $\alpha_k>\frac12$ and $\beta_{k'}>\frac12$ give $x_{k'j}>1$, against (O1).
(3) No exceptional index. Then $x_{ij}=\alpha+\beta=c$ for all $i$. By (O4), $c<1$. After exchanging colours we may assume $\alpha<\frac12$. We fix $k\in Q$. The sum $\sum_{i\in Q}m_iz_{ij}=M_Qz_{kj}+\sum_{i\in Q}m_iz_{ik}$ has exponent $2\alpha$, as the second sum has exponent at least $1>2\alpha$ and $M_Q\ne0$. The centre of mass gives $\sum_{i\in Q}m_iz_i-M_Qz_j=-Mz_j$, whence $e(z_j)=2\alpha$. The first equation of \eqref{eq:AK4} for $j$ reads $z_j=\sum_{i\in Q}m_iZ_{ji}$ with $e(Z_{ji})=2\alpha-3c$, which forces $2\alpha\ge2\alpha-3c$, i.e.\ $c\ge0$. Since $\alpha=c-\beta\ge c-\frac12$, $e(Z_{ij})=2\alpha-3c\ge-1-c>-2$, and likewise $e(W_{ij})>-2$: no stroke meets $j$.
(4) An exceptional index $k$ in $z$. By (2) all $\beta_i$ equal some $\beta$. We have $\beta<\frac12$, since otherwise $x_{ij}=1$ for $i\ne k$. Hence $x_{kj}=\frac a2+\beta$ and $x_{ij}=\frac12+\beta$. As in (3), $e(w_j)=2\beta$. The second equation of \eqref{eq:AK4} for $j$ reads $w_j=\sum_{i\in Q}m_iW_{ji}$, with $e(W_{ji})=2\beta-3x_{ij}$, which is $-\frac32-\beta$ for $i\ne k$ and $-\frac32a-\beta$ for $i=k$. Since $a>1$ the term $i=k$ has the strictly smallest exponent. So $2\beta=e(w_j)=-\frac32a-\beta$, i.e.\ $\beta=-\frac a2$. Then $x_{kj}=0$ and $x_{ij}=\frac{1-a}2$. Now $e(W_{kj})=-a\ge-2$ gives $a\le2$, with equality if and only if $kj$ is a $w$-stroke. The other exponents $e(Z_{kj})=a$, $e(Z_{ij})=1+\frac32(a-1)$ and $e(W_{ij})=\frac{a-3}2$ exceed $-2$. If $a=2$, then $\omega_{kj}=-2$: $k$ and $j$ are not $w$-close, and the $w$-stroke $kj$ is maximal. In that case $kj$ is a $w$-edge, since $e(Z_{kj})=a=2>-2$.
\end{proof}

At a place with diagram $14$ the $zw$-edges are $\mathrm{tri}(Q_j)$ and no stroke meets $j$. Its type $x(P)$ is $\pc u_c(j)$ with $0\le c<1$, or $\sigma_k(j)(t)$ with $0<t<1$. Lemma~\ref{lem:cut} leaves $c=0$. At a place with diagram $15$, case (B) holds with $a=2$ and $x(P)=\pc b_k(j)$. At a place with diagram $16$ all $x_{kl}$ equal $1$, which is excluded by Lemma~\ref{lem:cut}.

\subsection{Diagram 5}\label{ss:d5}
At a place with diagram $5$, $T$ is a $zw$-triangle without circles and the two other bodies $A$, $B$ have no strokes. By Lemma~\ref{lem:nearcluster} each of $A$ and $B$ is in one of the modes F, Z1, Zs$_{t_0}$, W1, Ws$_{t_0}$ relative to $K=T$. The type is determined by the mode parameters and by $x_{AB}$. For every pair of modes ($25$ pairs), the orders of the following quantities satisfy a system of linear equalities, inequalities and disjunctions: the positions, the sums $\sum_{t\in T}m_tZ_{tX}$, the torques, the differences of the four points $0$, $z_c$, $z_A$, $z_B$, where $z_c$ is the centre of mass of $T$, and $1-Mp_{AB}^3$. Each relation of the system follows from the equations \eqref{eq:AK4}, Lemmas~\ref{lem:U}, \ref{lem:orders}, \ref{lem:cut} and \ref{lem:nearcluster}, the centre of mass, and exact identities for the moment of inertia, the potential and the angular momentum; \S\ref{ss:d5rel} lists the relations that we use. Some of these relations, for an outer body $X$ in mode F, assume that the three inverse distances $p_{Xt}$, $t\in T$, have the same sign at leading order. Replacing them by relations that hold for all choices of signs (the torque equalities become inequalities, the relation between two bodies in mode F and the relations given for $X$ by the identity of (F) in \S\ref{ss:d5rel} are dropped, and the equalities for $\sum_tm_tZ_{tX}$ become inequalities) gives a model valid at every positive mass. For each pair of modes, Theorem~\ref{thm:hulls} in Appendix~\ref{app:d5hulls} gives an explicit polytope in the five coordinates $(x_{At_0},x_{At},x_{Bt_0},x_{Bt},x_{AB})$ that contains the projection of every point of the model. The union of these $25$ polytopes is the hull $\Vtwo^{\mathrm w}$. If we keep the equalities for $\sum_tm_tZ_{tX}$, which is legitimate for all choices of signs when $m_o\ne m_a+m_b$ for every $o\in T=\{o,a,b\}$, the same theorem gives a smaller hull $\Vtwo$. The hull $\Vtwo^{\mathrm w}$ differs from $\Vtwo$ only for the pairs of modes (F,\,F), (F,\,Z1), (F,\,W1) and their exchanges. A face of these hulls, and of the hull $\VN$ below, is \emph{attained} if each of the five coordinates is less than $1$ at one of its vertices (no outer pair of diagram~$5$ is a $zw$-edge). Every point of each of these models has all five coordinates less than $1$: by $c<1$ in mode F, by $\bar\omega<1$ and $s+\bar\omega<2$ in the other modes (Lemma~\ref{lem:nearcluster}), and, as $A$ and $B$ have no strokes, by $\zeta_{AB}+3\omega_{AB}<4$ and $3\zeta_{AB}+\omega_{AB}<4$ (Lemma~\ref{lem:orders}). The projection of a model point is a convex combination, with positive weights, of the vertices of the face in whose relative interior it lies. Hence this face is attained: each coordinate is less than $1$ at one of these vertices. The orders in these models are real numbers, while a modelled quantity may vanish identically along $C$. Its order is then $+\infty$. The next lemma treats this case.

\begin{lemma}\label{lem:infty}
Let $P$ be a place of diagram $5$ at a mass at which one of the models of this subsection is valid, and let $o$ be the vector of the orders at $P$ of the modelled quantities, with $o_v=+\infty$ when the quantity vanishes identically along $C$. Then the model has a point with real coordinates whose projection to the five coordinates of the type equals that of $o$. In particular the type of $P$ lies in the hull of the model.
\end{lemma}

\begin{proof}
Each relation of a model that requires a minimum to be attained twice comes from an exact linear identity, with nonzero constant coefficients, among the modelled quantities, such as $m_Az_A+m_Bz_B+M_Tz_c=0$ for the centre of mass $z_c$ of $T$. It still holds when the terms that vanish identically are dropped. The quantities that can vanish identically are the positions $z_A$, $z_B$, $z_c$, $w_A$, $w_B$, $w_c$, the differences $z_X-z_c$ for a body in mode Z1 or Zs and $w_X-w_c$ for one in mode W1 or Ws, the sums $\sum_tm_tZ_{tX}$ of the bodies in modes Z1 and W1 and, for $\Vtwo^{\mathrm w}$, in mode F, the two torques, the quantities of (R3), and $1-Mp_{AB}^3$. Each of them enters the relations only in lower bounds, in the equality of the orders of the two torques, which vanish together, with coefficient $+1$ in the terms of a minimum, or in upper bounds that are switched off by the event of (R1), (R2) or (R3) triggered by its vanishing. We replace every infinite order by a number $L$ larger than $10+3\max|o_v|$, the maximum over the finite orders. We use $L+2$ for $1-Mp_{AB}^3$, and $L+1$ for the two quantities of (R3) when all terms of that relation are infinite. The lower bounds and the equality of the torques then hold. The upper bounds that are not switched off involve only finite orders. A term of a minimum that contains a replaced order is at least $L-2$. All other terms are at most $3\max|o_v|+2$. Hence a minimum attained twice by finite terms keeps its value and the pair attaining it. The three relations whose terms can all be infinite hold with the values chosen. Since the projection involves only finite orders, it is unchanged. By Theorem~\ref{thm:hulls}, it lies in the hull.
\end{proof}

Note that the lemma is needed: modelled quantities do vanish identically at some places. On the modified Roberts continuum \cite[pp.~543--544]{AK12}, at the masses $(4,4,4,4,1)$, the centre of mass $z_c$, its counterpart $w_c$ and both torques vanish identically at each of the four places of diagram~$5$.

A third hull $\VN\subseteq\Vtwo$ uses three further relations (R1), (R2) and (R3). Theorem~\ref{thm:hulls} gives its polytopes as well. These relations are proved in the supplementary material. (R1) and (R2) hold at all positive masses off an explicit hypersurface in the masses of $T$, $A$ and $B$. (R3) is a relation for the pairs of modes (Z1,\,Zs$_{t_0}$), (W1,\,Ws$_{t_0}$) and their exchanges. It holds when $t_0$ is not at the centre of mass of $T$. A body of $T$ is at the centre of mass of $T$ exactly when the masses of $T$ are $(\mu/4,\mu,\mu)$ in some order. Then the pairs of modes with $t_0$ at the centre of mass carry no type: the supplementary material gives a model for them without (R3). By Theorem~\ref{thm:hulls} this model is infeasible. We call this case the \emph{centred case}. The hull $\VN$ then remains valid; its model contains the rows of that of $\Vtwo$. Thus every type lies in the relative interior of an attained face of $\VN$. Like $\Vtwo$, $\VN$ needs $m_o\ne m_a+m_b$ for the rows of the bodies in mode F. Hence $\VN$ is valid at all mass points that lie off the hypersurface and have $m_o\ne m_a+m_b$ on every triangle $T$ with $\lab{E}(T)$. The point check of \S\ref{ss:checks} verifies both conditions at every application of $\VN$ in the proof.

Every certificate of \S\S\ref{sec:strata}--\ref{sec:curves} that holds relative to $\Vtwo$ also holds relative to $\Vtwo^{\mathrm w}$ (\S\ref{ss:checks}).

\subsection{Mass conditions}\label{ss:conditions}
We write $x_i=m_i^{-1/2}$. The conditions of \cite[\S8]{AK12} that we use, after renumbering, are the following \emph{labels}. Each is a polynomial condition on $m$. Appendix~\ref{app:labels} gives the polynomials.
\begin{itemize}
\item $\lab{E}(T)$, $T=\{a,b,c\}$: one of $x_a,x_b,x_c$ is the sum of the other two. This is (23) and (36) of \cite{AK12}.
\item $\lab{R}(ab;s)$: $m_s^2(m_a+m_b)^4=m_a^3m_b^3$. This is (20), (39) and (40).
\item $\lab{H}(D_1|D_2)$, for a partition $D_1=\{i,j\}$, $D_2=\{k,l\}$ of a four-body set $Q_s$: $m_im_j=m_km_l$. This is (29).
\item $\lab{EQ}(a,b)$: $m_a=m_b$; and $\lab{D6}(ab|cd)$: $m_a=m_b$ and $m_c=m_d$, for disjoint pairs. This is (32).
\end{itemize}

For diagram $3$ the condition (29) can be sharpened. For a pair $D=\{i,j\}\subset Q_s$ let $\mathrm{Tri}(D;s)$ say that one of $x_i^3$, $x_j^3$, $x_s^2/\sqrt{M_{Q_s}}$ is the sum of the other two, and put $\cD_3(s;D_1|D_2)=\lab{H}(D_1|D_2)\wedge\mathrm{Tri}(D_1;s)\wedge\mathrm{Tri}(D_2;s)$.

\begin{theorem}\label{thm:D3}
If a singular place of $C$ has type $\pc h(s;D_1|D_2)$, then $m\in\cD_3(s;D_1|D_2)$.
\end{theorem}

In the proof, in Appendix~\ref{app:D3}, we write the leading coefficients at such a place explicitly on $\lab{H}(D_1|D_2)$, show that the next-order equations have a Jacobian of full rank, and compute the first obstruction as a single number that vanishes for some leading solution exactly when both conditions $\mathrm{Tri}$ hold. For each labelling the positive part of $\cD_3(s;D_1|D_2)$ lies on finitely many curves in the projective space of masses. From here on, in label sets and label tests, the label $\lab{H}(D_1|D_2)$ stands for $\cD_3(s;D_1|D_2)$.

\subsection{Diagram 7}\label{ss:d7}
We consider a place with diagram $7$, with $z$-wing $A=\{a,b,s\}$, $w$-wing $B=\{c,d,s\}$ and type $\pc v_7(a',b')$. We write $(a',b')$ for the parameters, to keep $a,b$ for bodies. By Lemma~\ref{lem:orders}, on the $A$ pairs $\zeta=3a'-2$ and $\omega=2-a'$, on the $B$ pairs $\zeta=2-b'$ and $\omega=3b'-2$, and on the four cross pairs $\zeta=3a'-2$ and $\omega=3b'-2$. Let $\mathrm{LA}$ be the line $4a'+3b'=4$, on which $e(W_{ab})=e(w_{cd})$, and $\mathrm{LB}$ the line $3a'+4b'=4$. Write (35)$_A$ for the condition $W_{ab}/m_s\sim W_{bs}/m_a\sim W_{sa}/m_b$ of \cite[(35)]{AK12} on the wing $A$, put $A_w=m_aW_{as}+m_bW_{bs}$ and $h=-A_w/(m_a+m_b)$ (the quantity $h$ of \cite[\S8.7.2]{AK12}, in our sign convention), and define $g$ in the same way for the wing $B$ in the other colour. We use the following facts from \cite[\S8.7]{AK12} and the symmetric ones for $B$:
\begin{enumerate}[label=\textup{(\alph*)}]
\item if (35)$_A$ holds, then $\lab{E}(A)$;
\item if (35)$_A$ fails, then $h\approx W_{ab}$ and $w_s\sim-m_sh$;
\item by the centre of mass $w_s\preceq w_{cd}$, so that $h\preceq w_{cd}$ when (35)$_A$ fails;
\item if moreover $h\prec w_{cd}$, then $\lab{R}(cd;s)$, by the reduction of \cite[pp.~577--578]{AK12} to the system of \cite[\S5.4]{AK12};
\item $h\approx w_{cd}$ and $g\approx z_{ab}$ do not hold together.
\end{enumerate}

\begin{lemma}\label{lem:X}
If \textup{(35)}$_A$ holds and $4a'+3b'\le4$, then $\lab{R}(cd;s)$. Symmetrically, if \textup{(35)}$_B$ holds and $3a'+4b'\le4$, then $\lab{R}(ab;s)$.
\end{lemma}

\begin{proof}
We put $C_w=\sum_{k\in\{a,b\},\,l\in\{c,d\}}m_km_lW_{kl}$. From the equations $w_k=\sum_{l\ne k}m_lW_{kl}$ of $a$ and $b$ in \eqref{eq:AK4}, $m_aw_a+m_bw_b=m_sA_w+C_w$. With $w_a=w_s+w_{as}$ and $w_b=w_s+w_{bs}$ this becomes
\[
(m_a+m_b)\,w_s=m_sA_w+C_w-m_aw_{as}-m_bw_{bs}.
\]
We compare each term with $w_{cd}$, of exponent $3b'-2$. The three pairs of $A$ have the same $(\zeta,\omega)$. So $W_{as}$, $W_{bs}$ and $W_{ab}$ have the same exponent $2-4a'$. Under (35)$_A$ the two terms of $A_w$ cancel at leading order and $e(A_w)>2-4a'\ge3b'-2$ on or below $\mathrm{LA}$. The cross terms have $e(C_w)\ge4-\frac92a'-\frac32b'$, which exceeds $3b'-2$ because $a'+b'<\frac43$ on or below $\mathrm{LA}$. And $e(w_{as})=2-a'>3b'-2$ as $a',b'<1$. Hence $w_s\prec w_{cd}$. By the centre of mass, $M w_s=m_cw_{sc}+m_dw_{sd}+m_aw_{sa}+m_bw_{sb}$, which gives $m_cw_{sc}+m_dw_{sd}\prec w_{cd}$: the body $s$ is at the $w$-centre of mass of $\{c,d,s\}$ at the scale $w_{cd}$. The reduction of \cite[pp.~577--578]{AK12} uses exactly this relation, the equations of $c$ and $d$ at exponent $-2$ (in which only the two stroke terms survive, because $w_k-w_G$ has exponent at least $3b'-2>-2$ and the cross terms have exponent $4-\frac92a'-\frac32b'>-2$), the closure of the triangle $cds$ in both colours with $p^2zw=1$ on its pairs, and the positivity of $m_c,m_d,m_s$. Eliminating the leading coefficients from this system, we obtain $\lab{R}(cd;s)$ times a factor that is positive on $(\R_{>0})^3$. Hence $\lab{R}(cd;s)$.
\end{proof}

Strictly below $\mathrm{LA}$ the hypothesis (35)$_A$ is not needed, since $e(A_w)\ge2-4a'>3b'-2$ holds term by term. Combining Lemma~\ref{lem:X} with the facts quoted above, we obtain the \emph{wing rules}:
\begin{enumerate}[label=\textup{(\roman*)}]
\item strictly above $\mathrm{LA}$, $\lab{E}(A)$;
\item on $\mathrm{LA}$, either $\lab{E}(A)\wedge\lab{R}(cd;s)$, or $h\approx w_{cd}$ and (35)$_A$ fails;
\item strictly below $\mathrm{LA}$, $\lab{R}(cd;s)$.
\end{enumerate} With $\mathrm{RA}=\lab{R}(ab;s)$ and $\mathrm{RB}=\lab{R}(cd;s)$ this gives Table~\ref{tab:d7}, which lists, for each cell of the arrangement of $\mathrm{LA}$, $\mathrm{LB}$ and the line $a'+b'=\frac43$ in the square $(0,1)^2\cap\{a'+b'\le\frac43\}$, the minimal sets of labels under which a type in the cell can occur. Without Lemma~\ref{lem:X} the conditions of \cite[\S8.7]{AK12} allow, in $22$ (condition set, cell) cases, types excluded by Table~\ref{tab:d7}, for instance $\pc v_7(\frac{11}{28},\frac{11}{28})$ and the point $(\frac47,\frac47)=\mathrm{LA}\cap\mathrm{LB}$.

\begin{table}[ht]
\centering
\small
\begin{tabular}{@{}ll@{}}
\toprule
region & minimal label sets\\
\midrule
above $\mathrm{LA}$, above $\mathrm{LB}$ (including the edge $a'+b'=\frac43$) & $\lab{E}(A)\lab{E}(B)$\\
above $\mathrm{LA}$, on $\mathrm{LB}$ & $\lab{E}(A)$\\
above $\mathrm{LA}$, below $\mathrm{LB}$ & $\lab{E}(A)\,\mathrm{RA}$\\
on $\mathrm{LA}$, above $\mathrm{LB}$ & $\lab{E}(B)$\\
$\mathrm{LA}\cap\mathrm{LB}$ & $\lab{E}(A)\,\mathrm{RB}$ or $\lab{E}(B)\,\mathrm{RA}$\\
on $\mathrm{LA}$, below $\mathrm{LB}$ & $\mathrm{RA}$\\
below $\mathrm{LA}$, above $\mathrm{LB}$ & $\lab{E}(B)\,\mathrm{RB}$\\
below $\mathrm{LA}$, on $\mathrm{LB}$ & $\mathrm{RB}$\\
below $\mathrm{LA}$, below $\mathrm{LB}$ & $\mathrm{RA}\,\mathrm{RB}$\\
\bottomrule
\end{tabular}
\caption{Diagram $7$ with wings $A|B$: the label sets under which a type $\pc v_7(a',b')$ in each region can occur.}\label{tab:d7}
\end{table}

\subsection{Pieces}\label{ss:pieces}
\begin{definition}\label{def:piece}
A \emph{piece} is a polytope $\Pi\subset\R^{10}$ (a point, a segment or a polygon, given by its vertices) together with a set of its faces, the \emph{attained} faces, closed under passing to larger faces. Its \emph{attained set} is the union of the relative interiors of its attained faces. For a set $S$ of labels, $\mathrm{pieces}(S)$ is the list of pieces in Table~\ref{tab:pieces} whose enabling condition is implied by $S$, together with the attained faces of the hull $\Vtwo^{\mathrm w}$ of \S\ref{ss:d5}, or of another hull of \S\ref{ss:d5} where this is stated, embedded for each labelling of diagram~$5$ with $\lab{E}(T)\in S$.
\end{definition}

\begin{table}[ht]
\centering
\small
\begin{tabular}{@{}llll@{}}
\toprule
piece & vertices & attained & enabled by\\
\midrule
$\pc{d1}[T]$ & $\pc a_1(T)$, $\pc b_1(T)$ & open segment & $\lab{E}(T)$\\
$\pc{d4}[T]$ & $\pc b_1(T)$ & point & $\lab{E}(T)$\\
$\pc{d2}[T]$ & $\pc a_1(T)$ & point & always\\
$\pc{d3}(s;D_1|D_2)$ & $\pc h(s;D_1|D_2)$ & point & $\lab{H}(D_1|D_2)$, read as $\cD_3(s;D_1|D_2)$\\
$\pc{d6}(A|B)$ & $\pc v_7(0,0)$ & point & $\lab{D6}(ab|cd)$\\
$\pc{d7}(A|B)$ & cells of Table~\ref{tab:d7} & see Table~\ref{tab:d7} & $\lab{E}(A)$, $\lab{E}(B)$, $\mathrm{RA}$, $\mathrm{RB}$\\
$\pc{d8}(A|B)$ & $\pc v_7(0,0)$, $\pc v_7(1,0)$ & open segment & $\mathrm{RA}$\\
$\pc{d9}(T;ax,by)$ & $\pc n(T;ax,by)$ & point & always\\
$\pc{d10}(A|B)$ & $\pc u_A(0)$, $\pc u_A(\frac13)$ & $(\pc u_A(0),\pc u_A(\frac13)]$ & $\lab{E}(A)$\\
$\pc{d11}(A|B)$ & $\pc u_A(0)$ & point & $\lab{E}(A)\vee\mathrm{RA}$\\
$\pc{d12}(A|B)$ & $\pc u_A(0)$ & point & $\mathrm{RA}$\\
$\pc{d14a}(j)$ & $\pc u_0(j)$ & point & always \\
$\pc{d14b}(j,k)$ & $\pc b_k(j)$, $\pc u_0(j)$ & open segment & always\\
$\pc{d15}(j,k)$ & $\pc b_k(j)$ & point & always\\
\bottomrule
\end{tabular}
\caption{The pieces, for each triple $T$, centre $s$ with wings $A=\{a,b,s\}$, $B=\{c,d,s\}$ (and the same with $A$, $B$ exchanged in the rows $\pc{d8}$--$\pc{d12}$), triangle $T$ with outer edges $ax$, $by$, and $j\ne k$.}\label{tab:pieces}
\end{table}

Since attained sets grow with $S$, a statement proved for $\mathrm{pieces}(S)$ holds for every subset of $S$. For $m\in(\R_{>0})^5$ let $S(m)$ be the set of labels that hold at $m$, with $\lab{D6}(ab|cd)$ added when $m_a=m_b$ and $m_c=m_d$.

\begin{theorem}\label{thm:types}
Let $m\in(\R_{>0})^5$, let $C\subseteq Y_m$ be an irreducible curve on which some $r_{kl}^2$ is not constant and $U\not\equiv0$, and let $P$ be a singular place of $C$ with diagram $d$ and labelling $\lambda$. Then $x(P)$ lies in the attained set of a piece of $\mathrm{pieces}(S(m))$ of diagram $d$ and labelling $\lambda$. For $d=5$ it lies in the relative interior of an attained face of $\Vtwo^{\mathrm w}$ with that labelling, and of $\Vtwo$ if $m_o\ne m_a+m_b$ for every $o$ in the triangle $T=\{o,a,b\}$ of the place.
\end{theorem}

\begin{proof}
By Lemma~\ref{lem:place} the place gives singular sequences with diagram $d$. Hence the conditions of \cite[\S8]{AK12} hold at $m$ for the labelling $\lambda$. These are (23) for diagrams $1$, $4$ and $5$, (29) for $3$, sharpened by Theorem~\ref{thm:D3}, (32) for $6$, the conditions of \S8.7 for $7$, (36) for $10$, and (40) for $8$ and $11$--$13$. By Proposition~\ref{prop:types}, Lemmas~\ref{lem:S1}, \ref{lem:zwK4} and \ref{lem:cut}, and Table~\ref{tab:d7} for diagram $7$, the type lies in the set listed in Table~\ref{tab:types}. This set lies in the attained set of the corresponding piece of Table~\ref{tab:pieces} once its condition holds. Diagram $5$ is treated in \S\ref{ss:d5}, with Lemma~\ref{lem:infty}. Renaming the bodies maps singular places, types, label sets and pieces equivariantly, since all of them are defined through labelled sets of bodies. So it suffices to consider the numbering of \cite[Fig.~11]{AK12}. There the type set of each diagram other than $5$ and $7$ is the attained set of its pieces (compare Tables~\ref{tab:types} and~\ref{tab:pieces} row by row). For diagram~$7$ the attained set of the pieces $\pc{d7}(A|B)$ of $\mathrm{pieces}(S)$ is the set of cells allowed for $S$ by the wing rules, which is Table~\ref{tab:d7}.
\end{proof}

\subsection{Clusters of zero energy}\label{ss:Kgating}
The pieces $\pc{d14a}$, $\pc{d14b}$ and $\pc{d15}$ carry no mass condition from \cite{AK12}. They need a four-body configuration of a special kind. For four bodies with masses $m_1,\dots,m_4$ and $z,w\in\C^4$, $p_{il}=p_{li}$, put
\[
F^z_i=\sum_{l\ne i}m_l\,z_{il}\,p_{il}^3,\qquad F^w_i=\sum_{l\ne i}m_l\,w_{il}\,p_{il}^3\qquad(i=1,\dots,4).
\]
\begin{itemize}
\item $\mathrm{K14}(m)$ is the system $F^z=F^w=0$, $p_{il}^2z_{il}w_{il}=1$ for all pairs, with all $z_{il},w_{il}\ne0$:
\begin{equation}\label{eq:K14}
\textstyle\sum_{l\ne i}m_lz_{il}p_{il}^3=\sum_{l\ne i}m_lw_{il}p_{il}^3=0,\qquad p_{il}^2z_{il}w_{il}=1 .
\end{equation}
\item $\mathrm{K15}(m,k)$, for $k\in\{1,\dots,4\}$, is the system
\begin{equation}\label{eq:K15}
F^z_i=0,\qquad F^w_i=c+d\,[i=k],\qquad p_{il}^2z_{il}w_{il}=1,\qquad d\ne0,
\end{equation}
with all $z_{il},w_{il}\ne0$.
\end{itemize}
Summing $m_iF^z_i$ and $m_iF^w_i$ shows that solutions of K14 have $\sum m_iz_i$ and $\sum m_iw_i$ free, while solutions of K15 have $cM+dm_k=0$. Contracting $F^z$ with $w$ shows that $U=\sum_{i<l}m_im_lp_{il}=0$ for solutions of both systems: they are four-body configurations of zero potential energy, critical for the potential (K14), or for the potential plus a linear function of the position of body $k$ (K15). Both systems are invariant under translations, under $(z,w,p)\mapsto(\lambda z,\mu w,(\lambda\mu)^{-1/2}p)$ (with $c,d$ rescaled), under $p\mapsto-p$, and under relabelling with the masses. Whether they have solutions therefore depends only on the ratios of the masses and on $k$.

\begin{theorem}\label{thm:Kgating}
Let $P$ be a singular place of $C$, $j$ a body, $Q=Q_j$, $k\in Q$.
\begin{enumerate}[label=\textup{(\alph*)}]
\item If $x(P)\in\{\pc u_0(j)\}\cup\bigcup_{k'}\sigma_{k'}(j)\big((0,1)\big)$, in particular if it lies in the attained set of $\pc{d14a}(j)$ or of some $\pc{d14b}(j,k')$, then $\mathrm{K14}(m_Q)$ has a solution, formed by the leading coefficients of the cluster at $P$.
\item If $x(P)=\pc b_k(j)$, then $\mathrm{K15}(m_Q,k)$ has a solution, formed by the leading coefficients of the cluster at $P$.
\item In case \textup{(a)}, $P$ has diagram $14$ with isolated body $j$. In case \textup{(b)}, $P$ has diagram $15$ with isolated body $j$ and $w$-edge $kj$, up to exchanging the colours.
\end{enumerate}
\end{theorem}

\begin{proof}
Since every listed point has $x=1$ exactly on $\mathrm{tri}(Q_j)$, Lemma~\ref{lem:zwK4} applies. The point $\pc u_0(j)$ is case (A), $\sigma_{k'}(j)(t)$ with $0<t<1$ is case (B) with $a=1+t<2$, and $\pc b_k(j)$ is case (B) with $a=2$. This gives (c). Fix $i_0\in Q$ and, for $i,l\in Q$, let $\hat z_i$, $\hat w_i$ and $\hat p_{il}$ be the coefficients of $u_{ii_0}$, $v_{ii_0}$ at exponent $1$ and of $p_{il}$ at exponent $-1$ (in the expansion of Lemma~\ref{lem:place}). For $i\ne l$ in $Q$ these are nonzero and the row $D_{il}$ gives $\hat p_{il}^2\hat z_{il}\hat w_{il}=1$. The first equation of \eqref{eq:AK4} for $i\in Q$ is $z_i=\sum_{l\in Q\setminus i}m_lp_{il}^3z_{il}+m_jZ_{ij}$. Each term of the sum has exponent $-2$ with coefficient $m_l\hat p_{il}^3\hat z_{il}$. The coefficient $c_z$ of $z_i$ at exponent $-2$ does not depend on $i$, since $e(z_{il})=1$. So
\[
F^z_i(\hat z,\hat p)=c_z-m_j[Z_{ij}]_{-2}\qquad(i\in Q),
\]
where $[\,\cdot\,]_{-2}$ is the coefficient at exponent $-2$, nonzero if and only if $ij$ is a $z$-stroke. The same holds for $w$. When we multiply by $m_i$ and sum over $Q$, the left side cancels by antisymmetry. If no $z$-stroke meets $j$, this gives $M_Qc_z=0$ and hence $F^z(\hat z,\hat p)=0$, and likewise for $w$. In case (A) and in case (B) with $a<2$ no stroke meets $j$ and the leading coefficients solve K14$(m_Q)$. In case (B) with $a=2$, $F^z=0$ and $F^w_i=c_w+d[i=k]$ with $d=-m_j[W_{kj}]_{-2}\ne0$: this is K15$(m_Q,k)$.
\end{proof}

\begin{corollary}\label{cor:Kdrop}
If $\mathrm{K14}(m_{Q_j})$ has no solution, no singular place of $C$ has its type in the attained set of $\pc{d14a}(j)$ or of any $\pc{d14b}(j,k)$. If $\mathrm{K15}(m_{Q_j},k)$ has no solution, none has type $\pc b_k(j)$. These pieces may then be removed from $\mathrm{pieces}(S(m))$ in Theorem~\ref{thm:types}.
\end{corollary}

Systems of this kind, with the cluster masses as unknowns, have solutions only on proper algebraic subsets of the mass space. In Section~\ref{sec:energy} we will remove the pieces $\pc{d15}$ and $\pc{d14b}$ at every positive mass (Corollary~\ref{cor:nod15} and Theorem~\ref{thm:D14}). We will also remove the pieces of diagrams~$9$ and~$2$ at every positive mass (Theorem~\ref{thm:D9} and Proposition~\ref{prop:nod2}). In \S\ref{sec:points} we will show how to decide $\mathrm{K14}$ at a given mass point.

The point $\pc v(s;A|B)$, the type of $\pc{d6}(A|B)$, lies in the relative closure of several other pieces of diagrams $7$ and $8$. The next statement shows that removing $\pc{d6}(A|B)$ removes exactly that point.

\begin{theorem}\label{thm:d6AB}
If $x(P)=\lambda\,\pc v(s;A|B)$ for some $\lambda>0$, then $P$ has diagram $6$ with centre $s$ and wings $A$, $B$, and $\lambda=1$. No attained face of a piece other than $\pc{d6}(A|B)$, and no face of $\Vtwo$, meets the ray $\R_{>0}\,\pc v(s;A|B)$ in its relative interior.
\end{theorem}

\begin{proof}
Each entry of $\lambda\pc v(s;A|B)$ is at most $0$. No distance tends to zero at $P$. By the proof of \cite[Prop.~9]{AK12} only diagrams $2$, $3$ and $6$ have no distance tending to zero. Their types $\pc a_1(T)$, $\pc h(s';D_1|D_2)$, $\pc v(s';A'|B')$ have $4$, $4$ and $6$ zero entries, while $\lambda\pc v(s;A|B)$ has the six zero entries $\mathrm{tri}(A)\cup\mathrm{tri}(B)$. Hence $d=6$ and $\mathrm{tri}(A')\cup\mathrm{tri}(B')=\mathrm{tri}(A)\cup\mathrm{tri}(B)$. In this graph $s$ is the only vertex of degree $4$, which gives $s'=s$ and $\{A',B'\}=\{A,B\}$. The entries $-2$ give $\lambda=1$. For the second statement, recall that every point of the ray has all coordinates at most $0$. Let $F$ be an attained face of a piece other than $\pc{d2}$, $\pc{d3}$ and $\pc{d6}$. If the piece is one of $\pc{d4}$, $\pc{d9}$ and $\pc{d10}$--$\pc{d15}$, all vertices of $F$ have the coordinate $1$ on a common $zw$-edge of the diagram, and so do all points of $F$. On the open segment of $\pc{d1}[T]$ the coordinates on the pairs of $T$ lie in $(0,1)$. The attained points of the pieces $\pc{d7}$ and $\pc{d8}$ are points $\pc v_7(a',b')$ with $a',b'\ge0$ and $(a',b')\ne(0,0)$: the cells of Table~\ref{tab:d7} lie in the open square $(0,1)^2$ (the cell $\mathrm{LA}\cap\mathrm{LB}$ is the point $(\frac47,\frac47)$), and the attained set of $\pc{d8}$ is $\pc v_7((0,1)\times\{0\})$. There the sum of the coordinates on the pairs of $A$ and of $B$ is $3(a'+b')>0$. So the relative interior of $F$ has a positive coordinate and misses the ray. Among the points $\pc{d2}$, $\pc{d3}$ and $\pc{d6}$, only $\pc{d6}(A|B)$ lies on the ray, by the zero patterns above. Finally, no face of a D5 hull, attained or not, meets the ray, since its embedding has the coordinate $1$ on the three pairs of the $zw$-triangle $T$; in particular no face of $\Vtwo$ does.
\end{proof}

%% file: sec/s4_lex.tex
\section{Lexicographic separation}\label{sec:lex}

Let $m\in(\R_{>0})^5$, and let $C\subseteq Y_m$ be a nontrivial curve. By \eqref{eq:balx} the types of the singular places of $C$ satisfy $\sum_Pd_P\,x(P)=0$ with $d_P>0$, and at least one place is singular. By Theorem~\ref{thm:types} each type lies in the attained set of a piece of $\mathrm{pieces}(S(m))$. Thus $m$ is settled as soon as $0$ is not a positive combination of points of these attained sets.

\subsection{The separation lemma}\label{ss:lexlemma}

\begin{lemma}\label{lem:lex}
Let $G_1,\dots,G_M$ be finite subsets of $\Q^{10}$, let $l_1,\dots,l_r\in\Q^{10}$, and let $I_1\cup\dots\cup I_r=\{1,\dots,M\}$ be a partition such that for every $i$:
\begin{enumerate}[label=\textup{(\roman*)}]
\item $l_i\cdot v\ge0$ for every $v\in G_j$ and every $j\in I_i\cup I_{i+1}\cup\dots\cup I_r$;
\item for every $j\in I_i$ there is $v\in G_j$ with $l_i\cdot v>0$.
\end{enumerate}
Then $0$ is not a positive combination of $N\ge1$ points of $\bigcup_j\relint\conv(G_j)$.
\end{lemma}

\begin{proof}
Suppose that $0=\sum_{s=1}^N\lambda_sx_s$ with $N\ge1$, $\lambda_s>0$ and $x_s\in\relint\conv(G_{j(s)})$, and let $i$ be the least index such that $j(s)\in I_i$ for some $s$. Each $x_s$ is a convex combination of points of $G_{j(s)}$ with $j(s)\in I_i\cup\dots\cup I_r$. So $l_i\cdot x_s\ge0$ by (i). A point $x$ of $\relint\conv(G)$ is a convex combination of all points of $G$ with positive weights: if $b$ is the barycentre of $G$ and $\delta>0$ is so small that $x+\delta(x-b)\in\conv(G)$, then $x=(1+\delta)^{-1}\big(x+\delta(x-b)\big)+\delta(1+\delta)^{-1}b$. Hence (ii) gives $l_i\cdot x_s>0$ whenever $j(s)\in I_i$. There is such an $s$. Therefore we obtain $0=l_i\cdot\sum_s\lambda_sx_s>0$, a contradiction.
\end{proof}

The functionals act in stages. The first is nonnegative on all sets and positive somewhere on each set of $I_1$. Hence no point of these sets can occur in a vanishing positive combination. The second does the same among the remaining sets, and so on. Listing more sets $G_j$ can only make the hypotheses harder to satisfy.

\subsection{Certificates}\label{ss:certs}

A \emph{D5 hull} is a finite union of polytopes in the five coordinates of \S\ref{ss:d5}, such as $\Vtwo^{\mathrm w}$. Its attained faces enter $\mathrm{pieces}(S)$ as in Definition~\ref{def:piece}.

\begin{definition}\label{def:cert}
Let $S$ be a set of labels and $V$ a D5 hull. Let $\cA_V(S)$ be the list of the attained faces of the pieces of $\mathrm{pieces}(S)$ (Definition~\ref{def:piece}), with the attained faces of $V$ in place of those of $\Vtwo^{\mathrm w}$. We record each face by the finite set of its vertices. A \emph{certificate} for $S$ relative to $V$ is a finite sequence $l_1,\dots,l_r\in\Z^{10}$ such that
\begin{enumerate}[label=\textup{(\alph*)}]
\item for each $i$, $l_i\cdot v\ge0$ for every vertex $v$ of every face of $\cA_V(S)$ on whose vertices $l_1,\dots,l_{i-1}$ all vanish, and
\item every face of $\cA_V(S)$ has a vertex $v$ with $l_i\cdot v>0$ for some $i$.
\end{enumerate}
A set $S$ is \emph{certified} relative to $V$ if it has a certificate relative to $V$.
\end{definition}

Let $I_i$ be the set of faces on whose vertices $l_1,\dots,l_{i-1}$ vanish and $l_i$ is positive at some vertex. By induction on $i$, (a) and (b) say that the sets $I_i$ partition $\cA_V(S)$ and satisfy (i) and (ii) of Lemma~\ref{lem:lex}. Since $\mathrm{pieces}(S)$ grows with $S$ and the attained faces of a piece do not depend on $S$, a certificate for $S$ is a certificate for every $S'$ with $\mathrm{pieces}(S')\subseteq\mathrm{pieces}(S)$, in particular for any subset of $S$.

\begin{definition}\label{def:validhull}
A D5 hull $V$ is \emph{valid} at $m\in(\R_{>0})^5$ if the conclusion of Theorem~\ref{thm:types} holds at $m$ with $V$ in place of $\Vtwo^{\mathrm w}$.
\end{definition}

\begin{proposition}\label{prop:cert}
Let $m\in(\R_{>0})^5$, let $V$ be a D5 hull valid at $m$, and let $S$ be a set of labels with $\mathrm{pieces}(S(m))\subseteq\mathrm{pieces}(S)$, for instance $S\supseteq S(m)$. If $S$ is certified relative to $V$, then $m$ is settled.
\end{proposition}

\begin{proof}
Let $C\subseteq Y_m$ be a nontrivial curve. By \eqref{eq:balx} we have $\sum_Pd_Px(P)=0$, with $d_P>0$ and at least one singular place $P$. By Theorem~\ref{thm:types} and the validity of $V$, each $x(P)$ lies in the relative interior of an attained face of some piece in $\mathrm{pieces}(S(m))$, built with $V$. This face belongs to $\cA_V(S)$. A face and the convex hull of its vertices have the same relative interior. Thus $0$ is a positive combination of points of $\bigcup_{G\in\cA_V(S)}\relint\conv(G)$, which is excluded by Lemma~\ref{lem:lex}, applied with the partition $(I_i)$ of the certificate. Hence no such curve exists.
\end{proof}

The pieces $\pc{d2}[T]$, $\pc{d9}(T;ax,by)$, $\pc{d14a}$, $\pc{d14b}$ and $\pc{d15}$ are present for every label set. Proposition~\ref{prop:nod2}, Theorem~\ref{thm:D9}, Corollary~\ref{cor:nod15} and Theorem~\ref{thm:D14} remove the pieces $\pc{d2}[T]$, $\pc{d9}(T;ax,by)$, $\pc{d15}(j,k)$ and $\pc{d14b}(j,k)$, respectively, at all positive masses, so that for every label set $S$ we may leave the attained faces of these pieces out of $\cA_V(S)$. By Corollary~\ref{cor:Kdrop} we may remove the pieces $\pc{d14a}(j)$ when $\mathrm{K14}(m_{Q_j})$ has no solution. Hence, if $\mathrm{K14}(m_{Q_j})$ has no solution for every body $j$ of a set $J$, we may leave the attained faces of $\pc{d14a}(j)$, $\pc{d14b}(j,k)$ and $\pc{d15}(j,k)$, $j\in J$, $k\in Q_j$, out of $\cA_V(S)$, and the proof of Proposition~\ref{prop:cert} applies to the remaining faces. We call the verification of (a) and (b) for these faces the \emph{final check} of $(S,J)$. The resulting notion of a K-certified point is Definition~\ref{def:Kcert}. Proposition~\ref{prop:Kcert} states that a K-certified point is settled.

\subsection{Exact verification}\label{ss:lexexact}

All vertices of the pieces of Table~\ref{tab:pieces}, with the cells of Table~\ref{tab:d7}, and all vertices of the polytopes of the D5 hulls used below have rational coordinates whose denominators divide $168$. We check a certificate in integer arithmetic, after multiplying the vertices by $168$ and asserting that every vertex becomes integral. We do not take the face list from the certificate: we recompute it from the label set $S$ and the pieces that the certificate leaves out. That is, we
\begin{enumerate}[label=(\arabic*)]
\item enumerate the pieces enabled by $S$, with an exact implementation of Tables~\ref{tab:pieces} and~\ref{tab:d7}, without the pieces $\pc{d2}[T]$ and $\pc{d9}(T;ax,by)$;
\item check that each polygon is strictly convex, with its vertices in the given cyclic order;
\item compute the face lattice of each polytope of the hull exactly, and embed the attained faces for each $\lab{E}(T)$ implied by $S$;
\item remove the pieces that the certificate leaves out, among them those of the clusters named in a final check; and
\item evaluate each $l_i$ at every vertex.
\end{enumerate}
The list so built is the list that the certificate must separate, which leaves out the faces of all pieces $\pc{d2}$ and $\pc{d9}$ (\S\ref{ss:certs}). We will check all certificates in \S\S\ref{sec:points}--\ref{sec:curves} this way.

We find the functionals by solving, stage by stage, a linear program in floating point (HiGHS) on the faces not yet separated, and rationalizing its solution. We accept a stage only if the exact check accepts it.

\begin{remark}\label{rem:hulls}
We use three D5 hulls of \S\ref{ss:d5}:
\begin{enumerate}[label=(\roman*)]
\item $\Vtwo^{\mathrm w}$, valid at all positive masses;
\item $\Vtwo\subseteq\Vtwo^{\mathrm w}$, valid when $m_o\ne m_a+m_b$ for every $o$ in each triangle $T=\{o,a,b\}$ with $\lab{E}(T)\in S(m)$;
\item $\VN\subseteq\Vtwo$, valid at the mass points described in \S\ref{ss:d5}.
\end{enumerate}
Each of them contains the projection of its model (Theorem~\ref{thm:hulls} and Lemma~\ref{lem:infty}). We apply each certificate below that uses $\VN$ only where $\VN$ is valid (checked in \S\ref{sec:verification}).
\end{remark}

\subsection{Certificates for the mass conditions}\label{ss:masscert}

To apply Proposition~\ref{prop:cert} on a set of masses we have to know which labels can hold there. We make these decisions by the kinds of certificate listed below. Each is \emph{rigorous}: it may fail to decide, but a decision it returns is correct, given the enclosure property of ball arithmetic (Arb, through python-flint~0.9.0) and the correctness of exact integer, rational and polynomial arithmetic. Search parameters (the degree bound of the filter (F2) of \S\ref{ss:flats}, initial grids, depth and time limits, working precisions, Krawczyk preconditioners and the tie-breaking constants of \S\ref{sec:rays}) and floating-point proposals, such as the solutions of the linear programs, affect only running time, the numbers of flats and boxes, and whether a search succeeds: all acceptances and exclusions are decided by rigorous enclosures or exact checks. Hence a wrong proposal can only make a search fail.

\begin{enumerate}[label=\textup{(C\arabic*)},leftmargin=*]
\item \emph{Exact computer algebra.} Standard bases, minimal primes and saturations of ideals of polynomials over $\Q$. With them we describe the points of a stratum of the sorted masses at which given labels hold. The details are in \S\ref{ss:flats}.
\item \emph{Exact algebraic decisions.} Feasibility of linear systems over $\Q$ (Fourier--Motzkin elimination), polynomial identities, arithmetic in number fields $\Q[\theta]/(f)$, and real root isolation over $\Q$ by Sturm sequences or Descartes' rule of signs.
\item \emph{Point certificates.} At an exact mass point, we show by the certificates of \S\ref{sec:points} that the system $\mathrm{K14}$ of a cluster has no solution.
\end{enumerate}

A label set certified in the sense of Definition~\ref{def:cert} is a statement about pieces only. The classes (C1)--(C3) connect it to the masses:
\begin{enumerate}[label=(\roman*)]
\item by (C1) and (C2), the label set of every positive mass vector with $m_1\le\dots\le m_5$ is certified in the sense of \S\ref{ss:nearlabels}, contains a curve set, or is that of one of two exact points (Theorem~\ref{thm:S});
\item the mass vectors whose label set contains a curve set lie, up to scaling and relabelling, on the curve families (Lemma~\ref{lem:DT}), which we treat with (C2) and (C3) in \S\ref{sec:curves} and by counting in \S\ref{sec:rays};
\item at the two exact points the hull $\VN$ certifies the label set (Proposition~\ref{prop:Spoints}).
\end{enumerate}
We describe in \S\ref{ss:checks} how these computations are checked.

%% file: sec/s5_energy.tex
\section{The energy identities at a place}\label{sec:energy}

In this section we treat diagram~$9$ and diagrams $14$ and $15$. Since the pieces of diagrams $14$ and $15$ carry no mass condition from \cite{AK12} (\S\ref{ss:Kgating}) and are present for every label set (\S\ref{ss:certs}), no choice of labels removes them. We therefore need an argument that does not go through the labels.

By Theorem~\ref{thm:Kgating}, places of diagrams $15$ and $14$ carry solutions of $\mathrm{K15}$ and $\mathrm{K14}$, which exist on hypersurfaces of the space of cluster masses. We exclude the places of diagram $15$ by expanding at the place the potential $V=U+\frac12I$, which is constant along $C$ (Lemma~\ref{lem:U}) and whose gradient vanishes there (Lemma~\ref{lem:potential}). For diagram $14$ we derive in \S\ref{ss:d14} a condition on the leading cluster. With the same expansion we exclude the pieces $\pc{d14b}$ (Theorem~\ref{thm:D14}). In \S\ref{ss:d9} the expansion, in coordinates adapted to the diagram, excludes diagram~$9$ at every positive mass (Theorem~\ref{thm:D9}). Identities marked $(\dagger)$ were verified by computer in exact arithmetic; the programs are in the supplementary material.

\subsection{Expansions at a place}\label{ss:expansion}
Let $P$ be a singular place of $C$ and $o$ a body. We use the positions relative to $o$, $u_i=z_i-z_o$ and $v_i=w_i-w_o$ for $i\ne o$, as coordinates. By relabelling, Lemma~\ref{lem:potential} holds in them. So $\nabla_{(u,v)}V=0$ at all points of $X_m$. After the rotation of Lemma~\ref{lem:place}, each of these coordinates is $q=\vartheta^{d\,e(q)}y_q(\vartheta)$. Put $\tau=\vartheta^{d/2}$, so that $q=\tau^{2e(q)}y_q$, where $y_q$ is a convergent Puiseux series in $\tau$ with $y_q(0)\neq0$. We write $\bar y=y(0)$ for the \emph{leading point}. For a Puiseux series $g$ and $r\in\Q$, $[g]_r$ denotes its coefficient of $\tau^r$. Replacing $\tau$ by $\tau\varphi(\tau)$ and rotating by $\mu(\tau)$, with convergent Puiseux series $\varphi$, $\mu$ and $\varphi(0)=\mu(0)=1$, changes neither the leading point nor the identities below. We use this freedom in \S\ref{ss:order5}.

\begin{lemma}\label{lem:identities}
Suppose that $F(\tau,y):=V(\tau^{2e}y)=\sum_{n\ge0}\tau^{n-2}V_n(y)$, with $V_n$ holomorphic near $\bar y$ and the series convergent for $y$ near $\bar y$ and small $\tau\ne0$. Then $\nabla_yF(\tau,y(\tau))=0$, and for every $q\in\Q$ with $q\ne2$,
\begin{equation}\label{eq:qid}
\sum_{n\ge1}n\,\big[V_n(y(\tau))\big]_{q-n}=0 .
\end{equation}
\end{lemma}

\begin{proof}
Since $y=\tau^{-2e}x$ coordinatewise, $\nabla_yF=\tau^{2e}\nabla_xV=0$ along the place. As $V$ is constant on $C$, $\sum_n\tau^{n-2}V_n(y(\tau))$ is a constant $E$. The function $F$ satisfies $F(\lambda\tau,\lambda^{-2e}y)=F(\tau,y)$, both sides being $V$ at the same point. Differentiating at $\lambda=1$ and using $\nabla_yF=0$ at $y=y(\tau)$, we obtain $\sum_n(n-2)\tau^{n-2}V_n(y(\tau))=0$. Adding twice the first identity gives $\sum_nn\,\tau^{n-2}V_n(y(\tau))=2E$. Its coefficient of $\tau^{q-2}$ is \eqref{eq:qid}. Each coefficient involves finitely many $n$, since $\tau^{n-2}V_n(y(\tau))$ has valuation at least $n-2$.
\end{proof}

\subsection{Diagram 15 at order three}\label{ss:d15}
Let $P$ have type $\pc b_k(j)$. By Theorem~\ref{thm:Kgating}(c), $P$ has diagram $15$ with $zw$-cluster $Q=Q_j$ and, after exchanging the colours if necessary, $w$-edge $kj$. We take $o=k$ and write $Q'=Q\setminus\{k\}=\{a,b,c\}$ for the three \emph{outer} bodies. By Lemma~\ref{lem:zwK4} and Lemma~\ref{lem:orders}, $e(u_i)=e(v_i)=1$ for $i\in Q'$, $e(u_j)=2$ and $e(v_j)=-2$. Write $\hat z_i=\bar y_{u_i}$, $\hat w_i=\bar y_{v_i}$ for $i\in Q'$, $\hat z_j=\bar y_{u_j}$, $\hat w_j=\bar y_{v_j}$, and $\hat p_{il}$ for the leading coefficients of the $p_{il}$.

\begin{lemma}\label{lem:d15lead}
\begin{enumerate}[label=\textup{(\alph*)}]
\item The cluster coefficients $(\hat z_i,\hat w_i)_{i\in Q'}$, with $\hat z_k=\hat w_k=0$, and $(\hat p_{il})_{i,l\in Q}$ solve $\mathrm{K15}(m_Q,k)$ with $c=-m_j\hat w_j/M$.
\item $\sum_{i\in Q'}m_i\hat z_i=0$.
\item $\hat p_{ij}^2\,\hat z_i\,(-\hat w_j)=1$ for $i\in Q'$.
\end{enumerate}
\end{lemma}

\begin{proof}
(a) By Theorem~\ref{thm:Kgating}(b) the cluster coefficients solve $\mathrm{K15}(m_Q,k)$. The constant $c$ is the coefficient of $w_k$ at exponent $-2$, which by $Mw_k=-\sum_{i\ne k}m_iv_i$ is $-m_j\hat w_j/M$. (b) For a solution of K15, $\sum_im_i\hat z_iF^w_i=\sum_{i<l}m_im_l\hat p_{il}^3\hat z_{il}\hat w_{il}=\sum_im_i\hat w_iF^z_i=0$, hence $c\sum_{i\in Q}m_i\hat z_i+dm_k\hat z_k=0$. With $\hat z_k=0$ and $c=-dm_k/M_Q\ne0$ this is (b). (c) is the leading term of the row $D_{ij}$, since $u_{ij}$ and $v_{ij}$ have the leading terms $\tau^2\hat z_i$ and $-\tau^{-4}\hat w_j$.
\end{proof}

We fix a square root $(-\hat w_j)^{1/2}$ and put $x_i=\hat p_{ij}(-\hat w_j)^{1/2}$ for $i\in Q'$. Then $x_i^2=1/\hat z_i$ by Lemma~\ref{lem:d15lead}(c). Lemma~\ref{lem:d15lead}(b) says $\sum_{i\in Q'}m_ix_i^{-2}=0$.

\begin{lemma}\label{lem:d15expansion}
At a place of type $\pc b_k(j)$, $F(\tau,y)=\sum_{n\ge0}\tau^{n-2}V_n(y)$ with
\begin{align*}
V_0&=U_Q(y)-\frac{m_j\,y_{v_j}}{2M}\sum_{i\in Q'}m_iy_{u_i}, \qquad V_1=V_4=0,\qquad V_2=f(\Pi),\\
V_3&=m_j\sum_{i\in Q'}m_i\,(-y_{u_i}y_{v_j})^{-1/2},\qquad V_5=\frac{m_j\,y_{u_j}}{2}\sum_{i\in Q'}m_i\,\frac{(-y_{u_i}y_{v_j})^{-1/2}}{y_{u_i}},
\end{align*}
where $U_Q$ is the potential of the cluster, $\Pi=y_{u_j}y_{v_j}$, $f(\Pi)=m_jm_k\Pi^{-1/2}+\frac{m_jM_Q}{2M}\Pi$, and the square roots are the branches of the $p_{il}$ at $P$. The coordinate $y_{u_j}$ occurs only in $V_2$ and in the $V_n$ with $n\ge5$.
\end{lemma}

\begin{proof}
The cluster terms $m_im_lp_{il}$, $i,l\in Q$, are exactly $\tau^{-2}$ times their value at $y$. The term $p_{kj}$ is $\Pi^{-1/2}$ exactly. For $i\in Q'$,
\[
p_{ij}=\tau\,(-y_{u_i}y_{v_j})^{-1/2}\big(1-\tau^2y_{u_j}/y_{u_i}\big)^{-1/2}\big(1-\tau^6y_{v_i}/y_{v_j}\big)^{-1/2},
\]
which contributes to $V_n$ only for odd $n\ge3$, with the stated $V_3$ and $V_5$. Finally $I=\sum_{i\ne k}m_iu_iv_i-\frac1M\big(\sum_{i\ne k}m_iu_i\big)\big(\sum_{i\ne k}m_iv_i\big)$ contributes $-\frac{m_j}{M}y_{v_j}\sum_{i\in Q'}m_iy_{u_i}$ at $n=0$, $\frac{m_jM_Q}{M}\Pi$ at $n=2$, and otherwise only to $n=6$ and $n=8$.
\end{proof}

\begin{theorem}\label{thm:order3}
At every singular place of type $\pc b_k(j)$, for any masses with $m_i\ne0$, $M\ne0$ and $M_Q\ne0$,
\[
\sum_{i\in Q'}m_i\,\hat p_{ij}=0,\qquad\text{equivalently}\qquad \sum_{i\in Q'}m_ix_i=0 .
\]
Moreover $\hat p_{kj}^3=M_Q/(m_kM)$.
\end{theorem}

\begin{proof}
The $y_{u_j}$-component of $\tau^2\nabla_yF(\tau,y(\tau))=0$ reads $y_{v_j}f'(\Pi)+O(\tau^3)=0$, by Lemma~\ref{lem:d15expansion}. At $\tau=0$ this is $f'(\bar\Pi)=0$, i.e.\ $\hat p_{kj}^3=\bar\Pi^{-3/2}=M_Q/(m_kM)$. Since $f''(\bar\Pi)=\frac34m_jm_k\bar\Pi^{-5/2}\neq0$, $f'(\Pi(\tau))=O(\tau^3)$ gives $\Pi(\tau)-\bar\Pi=O(\tau^3)$. Hence $f(\Pi(\tau))-f(\bar\Pi)=O(\tau^6)$ and
\begin{equation}\label{eq:V2flat}
\big[V_2(y(\tau))\big]_r=0\qquad(0<r<6).
\end{equation}
Now \eqref{eq:qid} with $q=3$ reads $[V_1]_2+2[V_2]_1+3V_3(\bar y)=0$. Since the first two terms vanish, $V_3(\bar y)=m_j\sum_im_i\hat p_{ij}=0$.
\end{proof}

We call $(x_a,x_b,x_c)\in(\C^*)^3$ a \emph{bad shape} for the outer masses if $\sum m_ix_i=0$ and $\sum m_ix_i^{-2}=0$. By Theorem~\ref{thm:order3} and Lemma~\ref{lem:d15lead}(b), the outer bodies of the leading cluster are at $\hat z_i=x_i^{-2}$ for a bad shape $x$.

\begin{lemma}\label{lem:badshape}
Up to a common factor, the bad shapes are $x=(1,t,-(m_a+m_bt)/m_c)$ where $t$ is a root of
\[
Q(t)=(m_a+m_bt)^2(m_at^2+m_b)+m_c^3t^2 .
\]
For positive masses $Q$ has no real root.
\end{lemma}

\begin{proof}
We normalize $x_a=1$. The first condition gives $x_c$. Then $m_c^2t^2x_c^2\sum m_ix_i^{-2}=Q(t)$. For positive masses $Q(t)\ge m_c^3t^2>0$ for real $t\ne0$ and $Q(0)=m_a^2m_b>0$.
\end{proof}

\begin{corollary}\label{cor:equilateral}
If $m_a=m_b=m_c$, then $Q=m_a^3(t^2+t+1)^2$, and the outer bodies of a cluster with a bad shape form an equilateral triangle centred at body $k$: $\{\hat z_a,\hat z_b,\hat z_c\}=\lambda\{1,\jj,\jj^2\}$ with $\jj=e^{2\pi i/3}$.
\end{corollary}

\begin{proof}
For equal masses, $(1+t)^2(1+t^2)+t^2=(t^2+t+1)^2$. Hence $t\in\{\jj,\jj^2\}$, $x$ is proportional to $(1,\jj,\jj^2)$ up to order, and $\hat z_i=x_i^{-2}$.
\end{proof}

\begin{proposition}\label{prop:eqfull}
Let $\nu\in\C\setminus\{0\}$. No solution of $\mathrm{K15}((1,1,1,\nu),4)$ has $z_1-z_4$, $z_2-z_4$, $z_3-z_4$ proportional to $(1,\jj,\jj^2)$ or to $(1,\jj^2,\jj)$.
\end{proposition}

\begin{proof}
After a normalization we reduce the equations to four polynomials in $\nu^2$ and two further unknowns, eliminate $\nu^2$, take resultants in one of the two unknowns, and show that the resulting polynomials in the other have no common root except $0$ and $1$. Exchanging bodies $2$ and $3$ reduces the second case to the first. We normalize $z=(1,\jj,\jj^2,0)$ by the symmetries of K15. Momentum gives $3c+\nu(c+d)=0$, so that $c\ne0$. We take $c=1$ by the remaining scaling. Let $s_j$ be the coordinate $p$ of the side opposite body $j\le3$ and $r_j=p_{j4}$, write $s_j=\sigma_0+\sigma_1\jj^{j-1}+\sigma_2\jj^{1-j}$, and put $k=3\nu^2$, $x=\sigma_1$ and $P=\sigma_0\sigma_2$. With $w_{il}=1/(p_{il}^2z_{il})$, the rows $F^w$ give $\nu r_j=-\sigma_0+\jj^{j-1}+\sigma_2\jj^{1-j}$. The rows $F^z$ then become three polynomial equations $Z_0,Z_1,Z_2$ in $(\sigma_0,\sigma_1,\sigma_2,k)$. The closure of the triangles in $w$ gives three more, $C_0,C_1,C_2$. We use $Z_0$, $Z_1$, $Z_2$, $C_0$ and $C_1$. Since $\prod_js_j\ne0$, they force $P\ne0$.
\begin{enumerate}[label=(\arabic*)]
\item With $t=\sigma_2$, $Z_2=(P^2-P+t^3)/t$. Hence $t^3=P(1-P)$ and $P\ne1$.
\item Modulo $t^3=P(1-P)$, the equations $Z_1$, $C_1$, $Z_0$ and $C_0$ are, up to sign, products of powers of $t$ and $P$ with
\[
Q-kK,\quad (1-P)(x^4+3P^2)-2PQx,\quad 1-6P+3P^2+P^3+k\big(PQ+(1-P)x^3+6P(1-P)x\big)
\]
and with a polynomial $C_0'$ of degree $1$ in $k$ $(\dagger)$, where $Q=P^2-P+1$ and $K=(1-P)x^2+Px+P(1-P)$. We call the three displayed polynomials (I), (II) and (III). A solution therefore gives $(P,x,k)$ with $P\notin\{0,1\}$ at which (I), (II), (III) and $C_0'$ vanish.
\item Write (III) $=A_3+kB_3$ and $C_0'=A_0+kB_0$ with $A_3,B_3,A_0,B_0$ free of $k$. Put $G_3=KA_3+QB_3$ and $G_0=KA_0+QB_0$. Since (I) is $Q-kK$, we have $G_3=K\cdot\mathrm{(III)}+B_3\cdot\mathrm{(I)}$ and $G_0=K\,C_0'+B_0\cdot\mathrm{(I)}$ $(\dagger)$, so that $G_3$ and $G_0$ vanish at $(P,x)$. The resultants with (II) factor as
\[
\mathrm{Res}_x(\mathrm{II},G_3)=P^3(P-1)^2H_{18}(P),\qquad\mathrm{Res}_x(\mathrm{II},G_0)=27P^6(P-1)^5(P^2-3P+1)^4H_{32}(P),
\]
where $H_{18}$ and $H_{32}$ are irreducible over $\Q$, of degrees $18$ and $32$, with leading coefficients $5$ and $100$ $(\dagger)$. Since $P\notin\{0,1\}$, $H_{18}(P)=0$ and $(P^2-3P+1)H_{32}(P)=0$. But $H_{18}\equiv1735830-4544451P$ modulo $P^2-3P+1$, a nonzero polynomial with a rational root, while the roots of $P^2-3P+1$ are irrational. Thus $H_{18}$ and $P^2-3P+1$ have no common root. Since $7$ divides neither leading coefficient and the Euclidean algorithm for $H_{32}$ and $H_{18}$ modulo $7$ ends in a nonzero constant, $H_{18}$ and $H_{32}$ have no common root either. Hence no such $P$ exists.
\end{enumerate}
The polynomials and the table of the Euclidean algorithm, which has $19$ rows, are in the supplementary material.
\end{proof}

\begin{corollary}\label{cor:d15equal}
Let the masses be positive. If the three outer masses $m_i$, $i\in Q_j\setminus\{k\}$, are equal, no singular place of $C$ has type $\pc b_k(j)$.
\end{corollary}

\begin{proof}
By Corollary~\ref{cor:equilateral} the leading cluster would be an equilateral solution of $\mathrm{K15}((1,1,1,\nu),4)$ after relabelling and scaling, with $\nu=m_k/m_a>0$. Proposition~\ref{prop:eqfull} excludes it.
\end{proof}

\subsection{Diagram 15 at order five}\label{ss:order5}
On the six coordinates $(u_i,v_i)_{i\in Q'}$ of the outer bodies of a cluster with body $k$ at the origin, put
\[
\Psi_c(u,v)=U_Q(u,v)+\frac c2\sum_{i\in Q'}m_iu_i .
\]
Let $m_i\ne0$ for $i\in Q$. A solution of $\mathrm{K15}(m_Q,k)$ then has $c\ne0$ and $M_Q\ne0$, by $cM_Q+dm_k=0$ with $d\ne0$. Since $\partial U_Q/\partial u_i=-\frac12m_iF^w_i$ and $\partial U_Q/\partial v_i=-\frac12m_iF^z_i$, for $c\ne0$ and $M_Q\ne0$ the critical points of $\Psi_c$ are the solutions of $\mathrm{K15}(m_Q,k)$ with body $k$ at the origin and the given $c$ (the equations of body $k$ follow from $\sum_im_iF_i=0$, with $d=-cM_Q/m_k$). Let $H$ be the Hessian of $\Psi_c$ at such a solution, and put $X=(\hat z,-3\hat w)$, $E=(\hat z,\hat w)$ and $\tilde m=(m_a,m_b,m_c)$.

\begin{lemma}\label{lem:Hstruct}
Let $m_i\ne0$ for $i\in Q$. At a solution of $\mathrm{K15}(m_Q,k)$ with body $k$ at the origin, $HX=0$, $HE=c\,(\tilde m,0)$, and $\sum_{i\in Q'}m_i\xi_{u_i}=0$ for every $\xi\in\ker H$.
\end{lemma}

\begin{proof}
Since $U_Q$ is homogeneous of degree $-\frac12$ in $u$ and in $v$ separately, $\Psi_c(\lambda u,\lambda^{-3}v)=\lambda\Psi_c(u,v)$. The orbit of a critical point under this scaling therefore consists of critical points. So its tangent vector $X$ lies in $\ker H$. Next, $\sum(u\partial_u+v\partial_v)\Psi_c=-U_Q+\frac c2\sum m_iu_i$. Differentiating at a critical point, where $\nabla U_Q=-\frac c2(\tilde m,0)$, gives $HE=c(\tilde m,0)$. Finally $\sum u\partial_u\Psi_c-\sum v\partial_v\Psi_c=\frac c2\sum m_iu_i$. Differentiating along $\xi\in\ker H$ at a critical point, we obtain $\frac c2\sum m_i\xi_{u_i}=0$. Here $c\ne0$ as $m_k\ne0$.
\end{proof}

The configuration is \emph{nondegenerate} if $\ker H=\C X$, and has \emph{corank} $r$ if $\dim\ker H=r+1$. For a place of type $\pc b_k(j)$ we take $c=-m_j\hat w_j/M$ as in Lemma~\ref{lem:d15lead}, so that $V_0=\Psi_c$ when $y_{v_j}=\hat w_j$. For $\xi\in\C^6$ put $\kappa_3(\xi)=\sum_{i\in Q'}\frac{\partial V_3}{\partial y_{u_i}}(\bar y)\,\xi_{u_i}$. By Lemma~\ref{lem:d15expansion},
\[
\kappa_3(\xi)=-\tfrac12m_j(-\hat w_j)^{-1/2}\sum_{i\in Q'}m_ix_i^3\,\xi_{u_i},\qquad V_5(\bar y)=\tfrac12m_j\hat z_j(-\hat w_j)^{-1/2}\sum_{i\in Q'}m_ix_i^3 ,
\]
and $\kappa_3(X)$ is proportional to $\sum m_ix_i^3\hat z_i=\sum m_ix_i=0$.

\begin{theorem}\label{thm:order5}
Let the masses satisfy $m_i\ne0$, $M\ne0$ and $M_Q\ne0$, and let $P$ be a singular place of type $\pc b_k(j)$.
\begin{enumerate}[label=\textup{(\alph*)}]
\item If the leading cluster configuration is nondegenerate, then $\sum_{i\in Q'}m_ix_i^3=0$.
\item If it has corank one, then $\kappa_3$ vanishes on $\ker H$.
\end{enumerate}
\end{theorem}

\begin{proof}
We fix an outer body $a$. By the freedom described in \S\ref{ss:expansion} we may assume $y_{v_j}(\tau)\equiv\hat w_j$ and $y_{u_a}(\tau)\equiv\hat z_a$: the conditions are $\mu^{-1}\varphi^4y_{v_j}=\hat w_j$ and $\mu\varphi^{-2}y_{u_a}=\hat z_a$, which determine $\varphi^2$ and then $\mu$. Write $y(\tau)-\bar y=(\delta(\tau),\delta_{u_j}(\tau))$ with $\delta$ the cluster part, so that $\delta_{u_a}=0$. Let $S=\{\xi\in\C^6:\xi_{u_a}=0\}$. Since $X\in\ker H$ and $X_{u_a}=\hat z_a\ne0$, the kernel of the restriction of the form $H$ to $S$ is $K_S=\ker H\cap S$, of dimension equal to the corank. Choose a complement $W$ of $K_S$ in $S$ and write $\delta=\xi(\tau)+\eta(\tau)$ with $\xi(\tau)\in K_S$, $\eta(\tau)\in W$. Let $\alpha$ be the valuation of $\xi(\tau)$ ($\alpha=\infty$ if $\xi\equiv0$).

By Lemmas~\ref{lem:identities} and \ref{lem:d15expansion}, the cluster components of $\tau^2\nabla_yF=0$ are $\nabla\Psi_c(\bar y+\delta)+\tau^3\nabla V_3(y)+O(\tau^5)=0$, because $V_2$ does not involve the cluster and $V_0=\Psi_c$ along the normalized place. Taylor expansion gives $H\delta+R(\delta)+\tau^3\nabla V_3(y)+O(\tau^5)=0$ with $R(\delta)=O(|\delta|^2)$. Taking the components in $S^*$ and using that $H|_S$ maps $W$ isomorphically onto its image, the valuation of $\eta$ is at least the minimum of $3$, $2\alpha$, $\alpha+\val\eta$ and $2\val\eta$. Hence
\begin{equation}\label{eq:etaval}
\val\eta\ge\min(3,2\alpha).
\end{equation}
Pairing the same equation with $\xi_0\in K_S$ and using $\xi_0^{\mathsf T}H=0$ gives
\begin{equation}\label{eq:bif}
\xi_0\cdot R(\delta)+\tau^3\,\xi_0\cdot\nabla V_3(y(\tau))+O(\tau^5)=0 .
\end{equation}
Since $y_{v_j}$ is constant, $V_3(y(\tau))=V_3(\bar y)+\kappa_3(\delta)+O(|\delta|^2)$ with $V_3(\bar y)=0$.

\eqref{eq:qid} with $q=5$ reads $[V_1]_4+2[V_2]_3+3[V_3]_2+4[V_4]_1+5V_5(\bar y)=0$. By \eqref{eq:V2flat} this becomes
\begin{equation}\label{eq:q5}
3\big[V_3(y(\tau))\big]_2+5V_5(\bar y)=0 .
\end{equation}

(a) Since $K_S=0$ in this case, $\delta=\eta$ has valuation at least $3$ by \eqref{eq:etaval}. Hence $[V_3(y(\tau))]_2=0$. Then \eqref{eq:q5} gives $V_5(\bar y)=0$, i.e.\ $\sum m_ix_i^3=0$.

(b) Here $K_S=\C\xi_0$ and $\xi(\tau)=a(\tau)\xi_0$ with $a(\tau)=a_\alpha\tau^\alpha+\dots$, $a_\alpha\ne0$. Suppose $\kappa_3(\xi_0)\ne0$. If $\alpha<2$, we apply \eqref{eq:qid} with $q=3+\alpha$: the terms $n=1,4$ vanish, the term $n=2$ is $2[V_2]_{1+\alpha}=0$ by \eqref{eq:V2flat}, and the terms $n\ge5$ have negative index. This leaves $[V_3(y(\tau))]_\alpha=0$. But by \eqref{eq:etaval} the valuations of $\eta$ and of $|\delta|^2$ exceed $\alpha$. Therefore $[V_3(y(\tau))]_\alpha=a_\alpha\kappa_3(\xi_0)\ne0$, a contradiction. Hence $\alpha\ge2$. Then $\val\eta\ge3$ and $\val R(\delta)\ge4$. The coefficient of $\tau^3$ in \eqref{eq:bif} is then $\xi_0\cdot\nabla V_3(\bar y)=\kappa_3(\xi_0)$, which must vanish, again a contradiction. Thus $\kappa_3(\xi_0)=0$. As $\kappa_3(X)=0$ and $\ker H=\C X\oplus\C\xi_0$, the form $\kappa_3$ vanishes on $\ker H$.
\end{proof}

The next statement says that three nonzero complex numbers with positive weights cannot have vanishing weighted sums of $x$, $x^3$ and $x^{-2}$; for four numbers this fails (remark (c) after the proof). Theorems~\ref{thm:order3} and~\ref{thm:order5} need only $m_i,M,M_Q\ne0$, but this statement needs positive masses. Corollary~\ref{cor:degenerate} and Proposition~\ref{prop:D14N} use it.

\begin{proposition}\label{prop:P11}
Let $m_a,m_b,m_c>0$. There is no $x\in(\C^*)^3$ with $\sum m_ix_i=\sum m_ix_i^3=\sum m_ix_i^{-2}=0$.
\end{proposition}

\begin{proof}
We say that three directions in $\R/2\pi\Z$ \emph{surround} $0$ if they do not lie in a closed half-circle, that is, if each of the three cyclic gaps between them is less than $\pi$. We use two facts.

(A) Let $w_1,w_2,w_3\in\C^*$ and $\mu_1,\mu_2,\mu_3>0$ with $\sum\mu_iw_i=0$. Then the arguments of the $w_i$ surround $0$, or all $w_i$ lie on one line through $0$. Indeed, if the arguments do not surround $0$, all $w_i$ lie in a closed half-plane $\{\mathrm{Re}(\bar uz)\ge0\}$ with $|u|=1$. Then $0=\sum\mu_i\mathrm{Re}(\bar uw_i)$ is a sum of nonnegative terms, so that all $w_i$ lie on the line $iu\R$.

(B) There are no $\theta_1,\theta_2,\theta_3$ such that each of the sets $\{\theta_i\}$, $\{2\theta_i\}$ and $\{3\theta_i\}$ surrounds $0$. Indeed, the three conditions are unchanged by a common shift of the $\theta_i$ and by relabelling. If $\{\theta_i\}$ surrounds $0$, we may therefore take $\theta=(0,\pi-s,\pi+t)$ with $s,t>0$ and $s+t<\pi$, so that the gaps are $\pi-s$, $s+t$ and $\pi-t$. Then $2\theta\equiv(0,-2s,2t)$, with gaps $2t$, $2\pi-2s-2t$ and $2s$. The set $\{2\theta_i\}$ surrounds $0$ exactly when $s<\frac\pi2$, $t<\frac\pi2$ and $s+t>\frac\pi2$. Assume this. Then $3s,3t\in(0,\frac{3\pi}2)$. After a rotation by $\pi$ the set $\{3\theta_i\}$ is $\{\pi,-3s,3t\}$. If $3s<\pi$ and $3t<\pi$, its gaps are $\pi-3t$, $\pi-3s$ and $3s+3t>\frac{3\pi}2$. If $3t\ge\pi>3s$, it lies in $[\pi,2\pi]$. If $3s\ge\pi>3t$, it lies in $[0,\pi]$. If $3s\ge\pi$ and $3t\ge\pi$, it lies in $[\frac\pi2,\frac{3\pi}2]$. In no case does it surround $0$.

Now let $x$ be a solution, and put $\theta_i=\arg x_i$.
\begin{enumerate}[label=(\arabic*)]
\item $\{\theta_i\}$ surrounds $0$. This is (A) with $w=x$: if all $x_i$ lay on a line through $0$, then $x_i=\lambda r_i$ with $r_i\in\R^*$, and $\sum m_ix_i^{-2}=\lambda^{-2}\sum m_ir_i^{-2}\ne0$.
\item $\{2\theta_i\}$ surrounds $0$. This is (A) with $w=x^{-2}$, whose arguments are the $-2\theta_i$ (a set surrounds $0$ if and only if its mirror image does): if all $x_i^{-2}$ lay on a line through $0$, the $\theta_i$ would lie among four directions at right angles, and three of these lie in a closed half-circle, against (1).
\item $\{3\theta_i\}$ surrounds $0$. This is (A) with $w=x^3$: if all $x_i^3$ lay on a line through $0$, then $\theta_i=\psi+n_i\frac\pi3$ with $n_i\in\Z/6$. By (1) the cyclic gaps between the $n_i$ are less than $3$. As they are also positive with sum $6$, they all equal $2$. So the $n_i$ have the same parity, all $x_i^3$ lie on one ray, and $\sum m_ix_i^3\ne0$.
\end{enumerate}
Steps (1)--(3) contradict (B).
\end{proof}

\emph{Remarks.} (a) A second proof eliminates $x$. With $x=(1,t,-(m_a+m_bt)/m_c)$, the conditions are $Q(t)=0$, with $Q$ as in Lemma~\ref{lem:badshape}, and $C(t)=m_c^2(m_a+m_bt^3)-(m_a+m_bt)^3=0$. A solution therefore gives a common root of $Q$ and $C$, so that their resultant
\[
\Res_t(Q,C)=m_a^2m_b^2m_c^4(m_a+m_b+m_c)^2\,P_{11}(m_a,m_b,m_c)
\]
vanishes. But $P_{11}$ is a symmetric polynomial of degree $11$ that is positive on the positive orthant (Appendix~\ref{app:P11}).
(b) The positivity of all three masses is used in (A) and in steps (1) and (3). The statement has no uniform version on the closed simplex: for $m=(\frac12,\frac12,\varepsilon)$ and $x=(1,-1,i\sqrt\varepsilon)$ we have $\sum m_ix_i^{-2}=0$, while $\sum m_ix_i=i\varepsilon^{3/2}$ and $\sum m_ix_i^3=-i\varepsilon^{5/2}$ tend to $0$ with $\varepsilon$. Correspondingly, $P_{11}(\frac12,\frac12,0)=0$.
(c) For four points the analogous statement fails: equal weights at $1,i,-1,-i$ satisfy all three conditions. In the proof above, (B) is where the number three enters.

Note that for equal outer masses the bad shape is $(1,\jj,\jj^2)$ and $\sum m_ix_i^3=3m_a\ne0$, in accordance with Proposition~\ref{prop:P11}.

\begin{corollary}\label{cor:degenerate}
Let $m\in(\R_{>0})^5$. At a singular place of type $\pc b_k(j)$ the leading cluster configuration is degenerate, and there is $\xi\in\ker H\setminus\C X$ with $\kappa_3(\xi)=0$.
\end{corollary}

\begin{proof}
A nondegenerate configuration would give a bad shape with $\sum m_ix_i^3=0$ (Theorems~\ref{thm:order3} and \ref{thm:order5}(a)), which is excluded by Proposition~\ref{prop:P11}. At corank one, $\kappa_3$ vanishes on $\ker H$ by Theorem~\ref{thm:order5}(b). At corank at least two, $\kappa_3$ is a linear form on $\ker H/\C X$, a space of dimension at least two, and has a nonzero kernel there.
\end{proof}

\subsection{Degenerate leading configurations}\label{ss:residual}
For a function $\lambda$ on the six pairs of $Q$, let $L_\lambda$ be the grounded Laplacian of the complete graph on $Q$ with edge weights $\lambda_{il}$, grounded at $k$: the $3\times3$ matrix with $(L_\lambda\xi)_i=\sum_{l\ne i}\lambda_{il}(\xi_i-\xi_l)$ for $i\in Q'$, where $\xi_k=0$. At a solution of K15 put $q_{il}=m_im_lp_{il}^3$ and $\varrho_{il}=z_{il}/w_{il}$.

\begin{lemma}\label{lem:laplacian}
At a solution of $\mathrm{K15}(m_Q,k)$ with body $k$ at the origin,
\[
H=\begin{pmatrix}L_{3q/(4\varrho)}&L_{q/4}\\ L_{q/4}&L_{3q\varrho/4}\end{pmatrix},
\]
and the system reads $L_qz=0$, $L_{q\varrho}w=0$ and $L_qw=c\,\tilde m$.
\end{lemma}

\begin{proof}
For one pair, $p=(zw)^{-1/2}$ and, with $\alpha=\zeta/z$, $\beta=\eta/w$ for a displacement $(\zeta,\eta)$, $p(z+s\zeta,w+s\eta)=p(1+s\alpha)^{-1/2}(1+s\beta)^{-1/2}$, whose second derivative at $s=0$ is $\frac p4(3\alpha^2+2\alpha\beta+3\beta^2)$. With $p^2zw=1$ the three coefficients are $\frac34q/\varrho$, $\frac14q$ and $\frac34q\varrho$ times $\zeta^2$, $2\zeta\eta$, $\eta^2$ (divided by $m_im_l$). Summing over the pairs, we obtain $H$. The linear term of $\Psi_c$ does not contribute. The system is $(L_qz)_i=m_iF^z_i$ and $(L_qw)_i=m_iF^w_i$, and $q_{il}\varrho_{il}w_{il}=q_{il}z_{il}$.
\end{proof}

\begin{theorem}\label{thm:residual}
Let the outer masses be positive and $m_k\ne0$, let $(\hat z,\hat w)$ solve $\mathrm{K15}(m_Q,k)$ with body $k$ at the origin, suppose $\sum_{i\in Q'}m_ix_i=0$ with $x_i^2=1/\hat z_i$, and let $\xi\in\ker H\setminus\C X$ satisfy $\sum_{i\in Q'}m_ix_i^3\xi_{u_i}=0$. Then one of the following holds:
\begin{enumerate}[label=\textup{(\roman*)}]
\item the outer masses are equal and $\{\hat z_i\}=\lambda\{1,\jj,\jj^2\}$;
\item \textup{(G2)} $\ker L_q=\C\hat z$ and $L_{q\varrho}\hat z=0$;
\item \textup{(G1)} $L_q=\sigma\,\tilde m\tilde m^{\mathsf T}$ for some $\sigma$, and $\rank L_{q\varrho}=1$.
\end{enumerate}
\end{theorem}

\begin{proof}
We write $\xi=(\zeta,\eta)$. The linear forms $\ell_1(\zeta)=\sum m_i\zeta_i$ and $\ell_2(\zeta)=\sum m_ix_i^3\zeta_i$ vanish on $\zeta$ (Lemma~\ref{lem:Hstruct} and the hypothesis on $\xi$) and on $\hat z$ ($\sum m_i\hat z_i=0$ and $\sum m_ix_i^3\hat z_i=\sum m_ix_i=0$). If they are proportional, then $x_a^3=x_b^3=x_c^3$. In that case, since the $\hat z_i=x_i^{-2}$ are distinct, $x$ is proportional to $(1,\jj,\jj^2)$ up to order, and $\sum m_ix_i=0$ with real positive $m_i$ forces equal masses: case (i). If they are not proportional, their common kernel is $\C\hat z$, which gives $\zeta=\gamma\hat z$, and $\xi-\gamma X=(0,\eta')$ with $\eta'=\eta+3\gamma\hat w\neq0$ lies in $\ker H$. By Lemma~\ref{lem:laplacian}, $(0,\eta')\in\ker H$ means $L_q\eta'=L_{q\varrho}\eta'=0$. Now $\hat z\in\ker L_q$ and $\hat w\in\ker L_{q\varrho}$. Hence $\rank L_q\le2$. If $\rank L_q=2$, then $\eta'\in\C\hat z$ and $L_{q\varrho}\hat z=0$: case (ii). If $\rank L_q\le1$, we write $L_q=\sigma vv^{\mathsf T}$. Then $L_q\hat w=\sigma v(v\cdot\hat w)=c\tilde m\ne0$ forces $v\in\C\tilde m$ and $\ker L_q=\tilde m^\perp$. If $\rank L_{q\varrho}=2$, then $\eta'\in\C\hat w$ and $\tilde m\cdot\hat w=0$, contradicting $\sigma(\tilde m\cdot\hat w)\tilde m=c\tilde m$ with $c\ne0$. As $L_{q\varrho}\ne0$, $\rank L_{q\varrho}=1$: case (iii).
\end{proof}

Case (iii) is excluded by the following statement, which does not use the bad shape.

\begin{proposition}\label{prop:G1}
Let $(\hat z,\hat w)$ solve $\mathrm{K15}(m_Q,k)$ with body $k$ at the origin, where $m_a$, $m_b$, $m_c$, $m_k$ and $M_{Q'}=m_a+m_b+m_c$ are nonzero, and suppose that $\rank L_q\le1$ and $\rank L_{q\varrho}\le1$. Then $(m_k/M_{Q'})^2=-1/s^3$ for a number $s$ with $s^3=\pm3\sqrt{-3}$. In particular $m_k/M_{Q'}$ is not real, and the masses are not all positive.
\end{proposition}

\begin{proof}
We use the rank conditions to write $q_{il}$ and $\varrho_{il}$ as products of factors attached to~$i$ and to~$l$, deduce from the rows $D_{il}$ that the three points $1/\hat z_i$ form an equilateral triangle, and compute $s$ and $(m_k/M_{Q'})^2$ from this triangle.

(1) Since $L_q\hat w=c\tilde m\ne0$, $\rank L_q=1$ and $L_q=\sigma vv^{\mathsf T}$. Then $\sigma v(v\cdot\hat w)=c\tilde m$ gives $L_q=\sigma\tilde m\tilde m^{\mathsf T}$ with $\sigma\ne0$. Reading off the entries, we obtain $q_{il}=-\sigma m'_im'_l$ on all six pairs, where $m'=(m_a,m_b,m_c,-M_{Q'})$. Thus $p_{il}^3=-\sigma$ on the outer pairs and $p_{ik}^3=\sigma M_{Q'}/m_k$. The off-diagonal entries $-q_{il}\varrho_{il}$ of $L_{q\varrho}$ are nonzero. Hence $L_{q\varrho}=\sigma'v'v'^{\mathsf T}$ with $\sigma'\ne0$, and in the same way $q_{il}\varrho_{il}=-\sigma'\tilde v_i\tilde v_l$ with $\tilde v=(v'_a,v'_b,v'_c,-v'_a-v'_b-v'_c)$, all of whose entries are nonzero. Hence $\varrho_{il}=\gamma\beta_i\beta_l$ on all six pairs, with $\gamma=\sigma'/\sigma$ and $\beta_i=\tilde v_i/m'_i\ne0$.

(2) We put $\psi_i=1/\hat z_i$. On the pairs through $k$, $\hat z_i=\gamma\beta_i\beta_k\hat w_i$. Inserting this into $\varrho_{il}=\gamma\beta_i\beta_l$ on the outer pairs gives $(1-\beta_i/\beta_k)\psi_i=(1-\beta_l/\beta_k)\psi_l$. If $\beta_i=\beta_k$ for one $i$, all the $\beta$ are equal, $\varrho$ is constant and $\hat w$ is a multiple of $\hat z$, which is excluded. So the common value $\psi_0$ is nonzero, $\beta_i=\beta_k(\psi_i-\psi_0)/\psi_i$, and $1/\hat w_i=\gamma\beta_i\beta_k\psi_i=A(\psi_i-\psi_0)$ with $A=\gamma\beta_k^2$.

(3) Since the $p_{ik}^3$ are equal, $p_{ik}=\varepsilon_i\pi_1$ with $\varepsilon_i^3=1$. The rows $D_{ik}$, $p_{ik}^2\hat z_i\hat w_i=1$, then read $\psi_i(\psi_i-\psi_0)=\varkappa\varepsilon_i^2$ with $\varkappa=\pi_1^2/A$.

(4) Likewise $p_{il}=\varepsilon_{il}\pi_0$ on the outer pairs. As $\hat z_{il}\hat w_{il}=(\psi_i-\psi_l)^2/\big(A\psi_i\psi_l(\psi_i-\psi_0)(\psi_l-\psi_0)\big)$, the rows $D_{il}$ read $(\psi_i-\psi_l)^2=\varepsilon_{il}\varepsilon_i^2\varepsilon_l^2\,\varkappa/\lambda$ with $\lambda=\pi_0^2/\pi_1^2$. So the squared sides of the triangle $\psi_a,\psi_b,\psi_c$ agree up to cube roots of unity: $\theta=(\psi_c-\psi_a)/(\psi_b-\psi_a)$ satisfies $\theta^6=(\theta-1)^6=1$ with $\theta\ne0,1$, whence $|\theta|=|\theta-1|=1$ and $\theta=e^{\pm i\pi/3}$. The triangle is equilateral: $\psi_i=h+R\varpi_i$ with $\{\varpi_a,\varpi_b,\varpi_c\}=\{1,\jj,\jj^2\}$ and $R\ne0$.

(5) The polynomial $\Phi(\varpi)=(h+R\varpi)(h+R\varpi-\psi_0)=R^2\varpi^2+R(2h-\psi_0)\varpi+h(h-\psi_0)$ takes the values $\Phi(\varpi_i)=\varkappa\varepsilon_i^2$. Hence $R^2=\frac13\sum_i\Phi(\varpi_i)\varpi_i=\varkappa s/3$ with $s=\sum_i\varepsilon_i^2\varpi_i$, a sum of three cube roots of unity. If the three terms are distinct, $s=0$ and $R=0$, which is excluded. If they are all equal to some $\alpha$, then $\Phi(\varpi)=\varkappa\alpha\varpi^2$ at three points and hence identically, so $R(2h-\psi_0)=h(h-\psi_0)=0$, which forces $\psi_0=0$, excluded. So exactly two terms are equal and $s^3=\pm3\sqrt{-3}$.

(6) By (4), $\lambda^3=\varkappa^3/(\psi_b-\psi_a)^6$, and $(\psi_b-\psi_a)^2=R^2(\varpi_b-\varpi_a)^2=-3R^2\varpi_a\varpi_b$, so that $(\psi_b-\psi_a)^6=-27R^6=-\varkappa^3s^3$ and $\lambda^3=-1/s^3$.

(7) We choose $\pi_0=p_{ab}$ and $\pi_1=p_{ak}$. Then (1) gives $(\pi_0/\pi_1)^3=-m_k/M_{Q'}$. So $(m_k/M_{Q'})^2=\lambda^3=-1/s^3$, a nonzero imaginary number.
\end{proof}

In case (ii), $\hat z$ and $\hat w$ lie in the kernel of $L_{q\varrho}$. They are independent, since $\hat w=\lambda\hat z$ would give $L_q\hat w=0$. Hence $L_{q\varrho}=\sigma'uu^{\mathsf T}$ with $\sigma'\ne0$ and $u=\hat z\times\hat w$. That is, $q_{il}\varrho_{il}=-\sigma'\tilde u_i\tilde u_l$ on the six pairs, where $\tilde u=(u_a,u_b,u_c,-u_a-u_b-u_c)$ is the affine dependence of the four points $(\hat z_i,\hat w_i)\in\C^2$. Then $q_{il}=-\sigma'\tilde u_i\tilde u_lw_{il}/z_{il}$. With this, $L_q\hat w=c\tilde m$ reads
\begin{equation}\label{eq:G2mass}
c\,m_i=-\sigma'\tilde u_iS_i,\qquad S_i=\sum_{l\ne i}\tilde u_l\,w_{il}^2/z_{il}\qquad(i\in Q').
\end{equation}

\begin{lemma}\label{lem:G2mass}
For any four points $(z_i,w_i)$, $i\in Q$, with $(z_k,w_k)=(0,0)$ and distinct $z_i$, and with $\tilde u$ and $S_i$ formed from them as above, $\sum_{i\in Q}\tilde u_iS_iz_i^j=0$ for $j=0,1,2$. In case (ii), therefore, $\sum_{i\in Q'}m_i\hat z_i^2=0$.
\end{lemma}

\begin{proof}
When we group the double sum by pairs, the pair $\{i,l\}$ contributes $\tilde u_i\tilde u_lw_{il}^2(z_i^j-z_l^j)/(z_i-z_l)$, that is $0$, $\tilde u_i\tilde u_lw_{il}^2$ and $\tilde u_i\tilde u_lw_{il}^2(z_i+z_l)$ for $j=0,1,2$. Since $\sum_i\tilde u_i=0$, any two functions $f$, $g$ on $Q$ satisfy $\sum_{\{i,l\}}\tilde u_i\tilde u_l(f_i-f_l)(g_i-g_l)=-\big(\sum_i\tilde u_if_i\big)\big(\sum_i\tilde u_ig_i\big)$. For $j=1$ we take $f=g=w$. The result vanishes because $\sum_i\tilde u_iw_i=0$. For $j=2$ we write $(w_i-w_l)(z_i+z_l)=\big((wz)_i-(wz)_l\big)+(w_iz_l-w_lz_i)$. The first part gives $-\big(\sum\tilde u_iw_i\big)\big(\sum\tilde u_iw_iz_i\big)=0$. In the second part the pairs through $k$ contribute nothing. On the outer pairs $w_bz_c-w_cz_b=-u_a$ and cyclically, so that the second part equals $-\tilde u_a\tilde u_b\tilde u_c\big((w_b-w_c)+(w_c-w_a)+(w_a-w_b)\big)=0$. The last statement follows from \eqref{eq:G2mass} and $\hat z_k=0$.
\end{proof}

Since a vector orthogonal to $\hat z$ and to $(\hat z_i^2)$ is proportional to $(1/P'(\hat z_i))_{i\in Q'}$ with $P(t)=t\prod_{i\in Q'}(t-\hat z_i)$, by partial fractions, Lemma~\ref{lem:G2mass} says that in case (ii) the outer masses are proportional to $(1/P'(\hat z_i))_{i\in Q'}$. With the bad shape, this leaves only the equilateral configuration.

\begin{theorem}\label{thm:wtriangle}
Let $m_a,m_b,m_c>0$, and let $z_a,z_b,z_c$ be distinct nonzero complex numbers with $\sum_im_iz_i=\sum_im_iz_i^2=0$ and $\sum_im_ix_i=0$ for some $x_i$ with $x_i^2z_i=1$. Then $m_a=m_b=m_c$ and $\{z_a,z_b,z_c\}=\lambda\{1,\jj,\jj^2\}$ for some $\lambda\in\C^*$.
\end{theorem}

\begin{proof}
The plane $m^\perp=\{v\in\C^3:\sum_im_iv_i=0\}$ contains $z$ and $z^2=(z_i^2)$, which are independent because the $z_i$ are distinct and nonzero. It is closed under complex conjugation because the $m_i$ are real. Hence $\bar z=\alpha z+\beta z^2$. If $\beta=0$, then $|\alpha|=1$ and the $z_i$ lie on a line through $0$, so that $\sum m_iz_i^2\ne0$. Thus $\beta\ne0$. The scaling $z\mapsto\lambda z$, $x\mapsto\lambda^{-1/2}x$ preserves the hypotheses and multiplies $\beta$ by $\bar\lambda/\lambda^2$. So we may take $\beta=1$: $\bar z_i=\alpha z_i+z_i^2$. Conjugating and substituting gives $z_i=|\alpha|^2z_i+(\alpha^2+\bar\alpha)z_i^2+2\alpha z_i^3+z_i^4$. Hence the $z_i$ are the three roots of
\[
g_\alpha(t)=t^3+2\alpha t^2+(\alpha^2+\bar\alpha)t+|\alpha|^2-1.
\]
Thus $z_a+z_b+z_c=-2\alpha$ and $z_az_bz_c=1-|\alpha|^2$. For $i\ne l$ this gives $z_i\bar z_l+\bar z_iz_l=z_iz_l(2\alpha+z_i+z_l)=-z_az_bz_c=|\alpha|^2-1$. Therefore
\[
0=\Big|\sum_im_iz_i\Big|^2=\sum_im_i^2|z_i|^2-(1-|\alpha|^2)\sum_{i<l}m_im_l,
\]
so that $|\alpha|<1$.

Next, $x\in m^\perp$ and hence $x=\gamma z+\delta z^2$. Then $x_i^2z_i=1$ says that $g_\alpha$ divides $t^3(\gamma+\delta t)^2-1$. The remainder modulo $g_\alpha$ has the coefficients
\begin{align*}
q_2&=-2\alpha\gamma^2+2(3\alpha^2-\bar\alpha)\gamma\delta+(1+3\alpha\bar\alpha-4\alpha^3)\delta^2,\\
q_1&=-(\alpha^2+\bar\alpha)\gamma^2+2(2\alpha^3+\alpha\bar\alpha+1)\gamma\delta+(\bar\alpha^2-2\alpha-3\alpha^4)\delta^2
\end{align*}
at $t^2$ and $t$, and a constant coefficient containing $-1$. So $(\gamma,\delta)\ne(0,0)$ is a common zero of the binary quadratic forms $q_2$ and $q_1$. Hence their resultant vanishes: $N(\alpha,\bar\alpha)=0$, where
\[
N(a,b)=4a^7-a^6b^2+6a^5b+2a^4b^3+11a^4+8a^3b^2-a^2b^4+6a^2b+14ab^3+8a+4b^5+3b^2 .
\]
Suppose that $\alpha=\rho e^{i\varphi}\ne0$, so that $0<\rho<1$. Since each monomial $a^ib^j$ of $N$ has $i-j\equiv1\pmod3$, $N(\alpha,\bar\alpha)=\alpha E$ with
\[
E=4\rho^6e^{6i\varphi}+\rho^3(11+6\rho^2-\rho^4)e^{3i\varphi}+2(\rho^6+4\rho^4+3\rho^2+4)+\rho(3+14\rho^2-\rho^4)e^{-3i\varphi}+4\rho^4e^{-6i\varphi}.
\]
Its imaginary part is $\operatorname{Im}E=\rho(\rho^2-1)\sin3\varphi\,(8\rho^3\cos3\varphi-\rho^4+6\rho^2+3)$. The last factor is at least $(1-\rho)(\rho^3+9\rho^2+3\rho+3)>0$. Thus $\sin3\varphi=0$. For $e^{3i\varphi}=1$, $E=-\rho^7+6\rho^6+5\rho^5+12\rho^4+25\rho^3+6\rho^2+3\rho+8>0$. For $e^{3i\varphi}=-1$, $E=(1-\rho)^2(\rho^5+8\rho^4+10\rho^3+24\rho^2+13\rho+8)>0$. This contradicts $E=0$. Therefore $\alpha=0$. Then $g_0(t)=t^3-1$, the $z_i$ are the cube roots of unity, and $\sum m_iz_i=0$ with real $m_i$ forces $m_a=m_b=m_c$.
\end{proof}

The proof of Theorem~\ref{thm:wtriangle} uses three small computations: the reduction of $t^3(\gamma+\delta t)^2-1$ modulo the cubic $g_\alpha$, the resultant of the two binary quadratic forms $q_2$ and $q_1$, and the polar form of $N$.

\begin{proposition}\label{prop:G2}
A configuration of case \textup{(ii)} of Theorem~\ref{thm:residual} with a bad shape and positive outer masses, as in that theorem, has equal outer masses and $\{\hat z_i\}=\lambda\{1,\jj,\jj^2\}$.
\end{proposition}

\begin{proof}
The outer masses are positive and the $\hat z_i$ distinct and nonzero. Lemma~\ref{lem:G2mass} gives $\sum m_i\hat z_i=\sum m_i\hat z_i^2=0$. The bad shape gives $\sum m_ix_i=0$ with $x_i^2\hat z_i=1$. Theorem~\ref{thm:wtriangle} applies.
\end{proof}

Note that the bad shape is needed in Proposition~\ref{prop:G2}: for all positive outer masses there are distinct nonzero $z_i$ with $\sum m_iz_i=\sum m_iz_i^2=0$. So the conclusion of Lemma~\ref{lem:G2mass} alone does not exclude them.

\begin{corollary}\label{cor:nod15}
For every $m\in(\R_{>0})^5$, no singular place of an irreducible curve $C\subseteq Y_m$ with $U\not\equiv0$ on $C$ has type $\pc b_k(j)$. Hence the pieces $\pc{d15}(j,k)$ may be removed from $\mathrm{pieces}(S(m))$ in Theorem~\ref{thm:types}, for every $m$, $j$ and $k$.
\end{corollary}

\begin{proof}
By Theorem~\ref{thm:order3} the leading cluster has a bad shape. By Corollary~\ref{cor:degenerate} it is degenerate with a vector $\xi$ as in Theorem~\ref{thm:residual}, since $\kappa_3(\xi)$ is a nonzero multiple of $\sum_{i\in Q'}m_ix_i^3\xi_{u_i}$. Case (iii) is excluded by Proposition~\ref{prop:G1}. In case (ii), Proposition~\ref{prop:G2} gives the configuration of case (i). Corollary~\ref{cor:d15equal}, which applies as all masses are positive, excludes case (i).
\end{proof}

\subsection{Diagram 14}\label{ss:d14}
Let $P$ have type $\sigma_k(j)(t)$ with $0<t<1$, and put $\sigma=1+t\in(1,2)$. Exchanging the colours if necessary, we obtain from Lemma~\ref{lem:zwK4} that $e(u_i)=e(v_i)=1$ for $i\in Q'=Q_j\setminus k$, $e(u_j)=\sigma$ and $e(v_j)=-\sigma$ in the coordinates relative to $o=k$. We use the leading coefficients as in \S\ref{ss:d15}.

\begin{lemma}\label{lem:d14blead}
The cluster coefficients solve $\mathrm{K14}(m_Q)$, $\sum_{i\in Q'}m_i\hat z_i=0$, and $\hat p_{kj}^3=M_Q/(m_kM)$.
\end{lemma}

\begin{proof}
The first statement is Theorem~\ref{thm:Kgating}(a). For $i\in Q$ let $C^z_i=\sum_{l\in Q\setminus i}m_lz_{il}p_{il}^3$ and $C^w_i$ likewise. The identity $p_{il}^3z_{il}w_{il}=p_{il}$ gives
\[
\sum_{i\in Q}m_iw_{ik}C^z_i=\sum_{\{i,l\}\subseteq Q}m_im_lp_{il}=\sum_{i\in Q}m_iz_{ik}C^w_i .
\]
By \eqref{eq:AK4}, $C^z_i=z_i-m_jp_{ij}^3z_{ij}$, whose terms have exponent at least $1$. The left side therefore has exponent at least $2$. Similarly $C^w_i=w_i-m_jp_{ij}^3w_{ij}$. Here $w_i=w_k+v_i$ has the leading term $-\frac{m_j}M\hat w_j$ at exponent $-\sigma$, the term $p_{ij}^3w_{ij}$ has exponent $\frac{\sigma-3}2>-\sigma$ for $i\ne k$, and $p_{kj}^3w_{kj}$ has exponent $-\sigma$ with coefficient $-\hat p_{kj}^3\hat w_j$. Thus $C^w_i$ has the coefficient $m_j\hat w_j\big(-\frac1M+[i=k]\hat p_{kj}^3\big)$ at exponent $-\sigma$. Since $\sum_im_iC^w_i=0$, we obtain $\hat p_{kj}^3=M_Q/(m_kM)$. On the right side the term $i=k$ vanishes. The rest has the coefficient $-\frac{m_j}M\hat w_j\sum_{i\in Q'}m_i\hat z_i$ at exponent $1-\sigma<2$, which therefore vanishes.
\end{proof}

\input{sec/s5_d14}

\begin{remark}\label{rem:d14a}
The identities of \S\ref{ss:expansion} also apply at a place of type $\pc u_0(j)$, the piece $\pc{d14a}(j)$, where the fifth body stays at distance of order one from the cluster while the cluster shrinks. There the pairings of $\nabla_yV$ with the $z$- and $w$-scalings of the cluster, together with the constancy of $V$, give at first order a signed dipole condition $\sum_ls_lm_l(\hat z_l-\hat z_G)=0$ when the leading coefficients $\hat p_{jl}$ have unequal signs $s_l$, a centred quadrupole condition $\sum_lm_l(\hat z_l-\hat z_G)^2=0$ when they have equal signs, the same conditions in $w$ at the next order, and in the boundary case a single scalar $g^{\mathsf T}H^+g$. These are the first-order conditions of the analysis of $\pc{d14a}$ in \S\ref{sec:curves}. The identities of \S\ref{ss:expansion} therefore do not exclude $\pc{d14a}$ at every positive mass: unlike the conditions of Theorems~\ref{thm:order3} and \ref{thm:order5}, these conditions hold at some symmetric K14 configurations with positive masses, for instance centrally symmetric ones. We treat the pieces $\pc{d14a}$ instead by the point certificates of \S\ref{sec:points}.
\end{remark}

\input{sec/s5_d9}

%% file: sec/s5_d14.tex

We now show that the pieces $\pc{d14b}$ carry no singular place at positive masses.

\begin{theorem}\label{thm:D14}
For every $m\in(\R_{>0})^5$, no singular place of an irreducible curve $C\subseteq Y_m$ with $U\not\equiv0$ on $C$ has type $\sigma_k(j)(t)$ with $0<t<1$, for any body $j$ and any $k\in Q_j$. Hence the pieces $\pc{d14b}(j,k)$ may be removed from $\mathrm{pieces}(S(m))$ in Theorem~\ref{thm:types}, for all $m$, $j$ and $k$.
\end{theorem}

We keep the setting of Lemma~\ref{lem:d14blead}, with $Q'=Q\setminus\{k\}=\{a,b,c\}$ and positive masses. For a configuration of $Q$ with $z_k=0$, the condition $\mathrm{X3}_k$ is $\sum_{i\in Q'}m_iz_i=0$: body $k$ lies at the $z$-centre of mass of $Q$. Such a configuration \emph{has a bad shape} if $\sum_{i\in Q'}m_ix_i=0$ for some $x_i$ with $x_i^2z_i=1$. If it also satisfies $\mathrm{X3}_k$, then $\sum_{i\in Q'}m_ix_i^{-2}=0$, so that $(x_a,x_b,x_c)$ is a bad shape in the sense of \S\ref{ss:d15}. The proof of Theorem~\ref{thm:D14} has four steps. The leading cluster configuration satisfies $\mathrm{X3}_k$ and has a bad shape (Lemma~\ref{lem:d14energy}). At positive masses the only $\mathrm{K14}$ configuration with these properties is, up to symmetries, an equilateral triangle with its centre (Theorem~\ref{thm:Xhat}). This configuration is nondegenerate (Lemma~\ref{lem:Pdata}). At a nondegenerate leading cluster configuration the expansion of the place forces $\sum_{i\in Q'}m_ix_i^3=0$ unless $\sigma=(2n+1)/(n+1)$ for an integer $n\ge1$, and for these $\sigma$ it forces a coefficient to vanish that is nonzero at the equilateral configuration (Propositions~\ref{prop:D14N} and~\ref{prop:resonant}).

\paragraph{The bad shape.} We fix a square root $(-\hat w_j)^{1/2}$ and put $x_i=\hat p_{ij}(-\hat w_j)^{1/2}$ for $i\in Q'$, as in \S\ref{ss:d15}. The leading term of the row $D_{ij}$ is $\hat p_{ij}^2\hat z_i(-\hat w_j)=1$. Thus $x_i^2\hat z_i=1$. Near $\bar y$ let $w(y)$ and $x_i(y)$ be the branches of $(-y_{v_j})^{-1/2}$ and $y_{u_i}^{-1/2}$ with $w(\bar y)=(-\hat w_j)^{-1/2}$ and $x_i(\bar y)=x_i$, and put $S_r(y)=\sum_{i\in Q'}m_ix_i(y)^r$ and $\psi(y)=\sum_{i\in Q'}m_iy_{u_i}$.

\begin{lemma}\label{lem:d14exp}
At a place of type $\sigma_k(j)(t)$, with $\Pi$, $f$ and $U_Q$ as in Lemma~\ref{lem:d15expansion},
\begin{equation}\label{eq:d14F}
F(\tau,y)=\tau^{-2}U_Q(y)+\tau^{2-2\sigma}\Gamma(y)+f(\Pi)+\sum_{n\ge0}\tau^{(2n+1)(\sigma-1)}B_n(y)+R(\tau,y),
\end{equation}
where $\Gamma=-\frac{m_j}{2M}y_{v_j}\psi$, $B_n=\binom{2n}{n}4^{-n}m_jw\,y_{u_j}^nS_{2n+1}$, and $R$ is a convergent sum of terms $\tau^LR_L(y)$ with $L\ge4$. The coordinate $y_{u_j}$ occurs only in $f$, in the $B_n$ with $n\ge1$ and in $R$, and the coordinates $y_{v_i}$, $i\in Q'$, occur only in $U_Q$ and in $R$.
\end{lemma}

\begin{proof}
The cluster terms and the term $m_km_jp_{kj}$ are as in the proof of Lemma~\ref{lem:d15expansion}. For $i\in Q'$, $u_{ij}v_{ij}=-\tau^{2-2\sigma}y_{u_i}y_{v_j}(1-\tau^{2\sigma-2}y_{u_j}/y_{u_i})(1-\tau^{2\sigma+2}y_{v_i}/y_{v_j})$, so that
\[
p_{ij}=\tau^{\sigma-1}w\,x_i(y)\big(1-\tau^{2\sigma-2}y_{u_j}/y_{u_i}\big)^{-1/2}\big(1-\tau^{2\sigma+2}y_{v_i}/y_{v_j}\big)^{-1/2}
\]
with the branches of the $p_{ij}$ at $P$. The binomial series of the first factor gives the $B_n$, since $y_{u_i}^{-n}=x_i(y)^{2n}$. The second factor contributes only exponents at least $3\sigma+1>4$. Finally $I$ contributes $-\frac{m_j}My_{v_j}\psi$ at the exponent $2-2\sigma$, $\frac{m_jM_Q}M\Pi$ at $0$, and otherwise only the exponents $4$ and $2\sigma+2$.
\end{proof}

\begin{lemma}\label{lem:d14energy}
At a place of type $\sigma_k(j)(t)$:
\begin{enumerate}[label=\textup{(\alph*)}]
\item $U_Q(y(\tau))=O(\tau^6)$;
\item $[B_0(y(\tau))]_r=0$ for $0\le r<2\sigma-2$; in particular $\sum_{i\in Q'}m_ix_i=0$, and the leading cluster configuration satisfies $\mathrm{X3}_k$ and has a bad shape;
\item $B_1(\bar y)=-\frac35\,[B_0(y(\tau))]_{2\sigma-2}$.
\end{enumerate}
\end{lemma}

\begin{proof}
As in the proof of Lemma~\ref{lem:identities}, $\nabla_yF(\tau,y(\tau))=0$ and $F(\tau,y(\tau))$ is constant. Let $Z=\sum_{i\in Q'}y_{u_i}\partial/\partial y_{u_i}$ and $W=\sum_{i\in Q'}y_{v_i}\partial/\partial y_{v_i}$. The function $U_Q$ is homogeneous of degree $-\frac12$ in the $y_{u_i}$ and in the $y_{v_i}$, $\psi$ is linear in the $y_{u_i}$, and $S_{2n+1}$ is homogeneous of degree $-n-\frac12$ in them. By Lemma~\ref{lem:d14exp}, along $y(\tau)$,
\begin{align*}
0=W\cdot\nabla_yF&=-\tfrac12\tau^{-2}U_Q+O(\tau^4),\\
0=Z\cdot\nabla_yF&=-\tfrac12\tau^{-2}U_Q+\tau^{2-2\sigma}\Gamma-\textstyle\sum_n\big(n+\frac12\big)\tau^{(2n+1)(\sigma-1)}B_n+O(\tau^4).
\end{align*}
The first identity gives (a). The $y_{u_j}$-component of $\nabla_yF=0$ reads $y_{v_j}f'(\Pi)+O(\tau^{3\sigma-3})=0$. As in the proof of Theorem~\ref{thm:order3}, $f'(\bar\Pi)=0\ne f''(\bar\Pi)$, so that $\Pi(\tau)-\bar\Pi=O(\tau^{3\sigma-3})$ and $[f(\Pi(\tau))]_q=0$ for $0<q<6\sigma-6$. Let $\sigma-1\le q\le3\sigma-3$. Then $f$, $R$ and the $B_n$ with $n\ge2$ contribute nothing at $\tau^q$. By (a) neither does $U_Q$. So the coefficients of $\tau^q$ in the second identity and in $F(\tau,y(\tau))$ are
\[
[\Gamma]_{q+2\sigma-2}-\tfrac12[B_0]_{q-\sigma+1}-\tfrac32[B_1]_{q-3\sigma+3}=0,\qquad[\Gamma]_{q+2\sigma-2}+[B_0]_{q-\sigma+1}+[B_1]_{q-3\sigma+3}=0,
\]
and their difference gives $3[B_0]_{q-\sigma+1}+5[B_1]_{q-3\sigma+3}=0$. For $q<3\sigma-3$ the second term vanishes, which is (b). For $q=3\sigma-3$ it is $5B_1(\bar y)$, which is (c). Finally $B_0(\bar y)=m_j(-\hat w_j)^{-1/2}\sum_{i\in Q'}m_ix_i$. The condition $\mathrm{X3}_k$ holds by Lemma~\ref{lem:d14blead}.
\end{proof}

\paragraph{K14 configurations with a bad shape.} By Lemmas~\ref{lem:d14blead} and~\ref{lem:d14energy}, the leading cluster configuration is a $\mathrm{K14}(m_Q)$ configuration with body $k$ at the origin that satisfies $\mathrm{X3}_k$ and has a bad shape. We determine these configurations at positive masses.

We use the following description of the configurations of $\mathrm{K14}(m)$, $m\in(\C^*)^4$. A configuration is \emph{collinear} if $w=\alpha+\beta z$ for some $\alpha,\beta\in\C$, and \emph{non-collinear} otherwise. A non-collinear configuration has $p_{il}=\lambda_0^{1/3}\eta_{il}t_it_l$ with $\lambda_0\ne0$, cube roots of unity $\eta_{il}=\eta_{li}$ and $t\in(\C^*)^4$ solving
\begin{equation}\label{eq:tsys}
\textstyle\sum_{l\ne i}m_l\eta_{il}t_l=\kappa t_i^2\quad(i=1,\dots,4),\qquad\sum_lm_lt_l^3=0
\end{equation}
for some $\kappa\in\C$. Indeed, let $A$ be the symmetric matrix with $A_{il}=m_im_lp_{il}^3$ for $i\ne l$ and zero row sums. By \eqref{eq:K14}, $A$ annihilates $\one$, $z$ and $w$, which are linearly independent for a non-collinear configuration, since no $z_{il}$ vanishes. As $A\ne0$, $A=\lambda_0 vv^{\mathsf T}$ with $\lambda_0\ne0$ and $v\cdot\one=v\cdot z=v\cdot w=0$. Since $A_{il}\ne0$, $v_i\ne0$ for all $i$. We choose $t_i$ with $t_i^3=v_i/m_i$. Then $p_{il}^3=\lambda_0 t_i^3t_l^3$, which gives the $\eta_{il}$, and $\sum_lm_lt_l^3=\sum_lv_l=0$. Finally $\sum_{l\ne i}v_lz_{il}w_{il}=\sum_lv_lz_lw_l$ does not depend on $i$. With $z_{il}w_{il}=p_{il}^{-2}=\lambda_0^{-2/3}\eta_{il}t_i^{-2}t_l^{-2}$ this is the first equation of \eqref{eq:tsys}. A gauge change $t_l\mapsto\zeta_lt_l$, $\eta_{il}\mapsto\zeta_i^{-1}\zeta_l^{-1}\eta_{il}$, with cube roots of unity $\zeta_l$, preserves \eqref{eq:tsys} and all $p_{il}$, and multiplies the three products $\eta_{12}\eta_{34}$, $\eta_{13}\eta_{24}$, $\eta_{14}\eta_{23}$ by a common factor. The gauge changes act freely on the $3^6$ patterns, since $\zeta_i\zeta_l=1$ for all $i\ne l$ forces $\zeta_1=\dots=\zeta_4=1$. Thus there are $3^6/3^4=9$ classes. The three products, up to a common factor, take $9$ values and each of them is attained. Hence two patterns are gauge equivalent if and only if their products are proportional. A pattern, and its class, is of \emph{type} I, II or III if its three products take one, two or three distinct values. There is one class of type I and there are six of type II. The two classes of type III are those of the pattern $\eta_{12}=\eta_{34}=1$, $\eta_{13}=\eta_{24}=\jj^2$, $\eta_{14}=\eta_{23}=\jj$ and of its complex conjugate.

\begin{lemma}\label{lem:typeIII}
Let $m\in(\C^*)^4$ have real ratios $m_i/m_l$. Then no non-collinear configuration of $\mathrm{K14}(m)$ has a pattern of type III.
\end{lemma}

\begin{proof}
Suppose that one has. We diagonalize the matrix of the pattern and show that $\kappa\ne0$ forces the configuration to be collinear and that $\kappa=0$ is impossible for real mass ratios. After a gauge change, $\eta$ is the pattern above or its complex conjugate. We treat the first. For the second, replace $\jj$ by $\jj^2$ throughout. Let $H=(\eta_{il})$, with zero diagonal, $u=(m_lt_l)_l$, $e_0=(1,1,1,1)$, $e_1=(1,1,-1,-1)$, $e_2=(1,-1,1,-1)$ and $e_3=(1,-1,-1,1)$. Then \eqref{eq:tsys} reads $Hu=\kappa(t_l^2)_l$ and $\sum_lu_lt_l^2=0$. We have $He_0=0$, $He_1=2e_1$, $He_2=2\jj^2e_2$, $He_3=2\jj e_3$. Hence the kernel of $H$ is $\C e_0$ and its image, spanned by $e_1,e_2,e_3$, consists of the vectors with coordinate sum $0$.

Suppose first that $\kappa\ne0$. Then $\sum_lt_l^2=0$. We write $(t_l^2)_l=\xi_1e_1+\xi_2e_2+\xi_3e_3$. We obtain $u=\frac\kappa2(\xi_1e_1+\jj\xi_2e_2+\jj^2\xi_3e_3)+\gamma e_0$ for some $\gamma\in\C$. Since $e_k\cdot e_{k'}=4\delta_{kk'}$, the last equation of \eqref{eq:tsys} becomes $2\kappa(\xi_1^2+\jj\xi_2^2+\jj^2\xi_3^2)=0$, so that $\xi_1^2+\jj\xi_2^2+\jj^2\xi_3^2=0$. For three bodies $a,b,c$ put
\[
h_{abc}=(\eta_{ab}t_c^2+\eta_{ac}t_b^2-\eta_{bc}t_a^2)^2-4\eta_{ab}\eta_{ac}t_b^2t_c^2 ,
\]
which is symmetric in $a,b,c$. For $(a,b,c)=(1,2,3)$ we have $\eta_{12}t_3^2+\eta_{13}t_2^2-\eta_{23}t_1^2=2\jj^2\xi_1+2\xi_2$ and $4\eta_{12}\eta_{13}t_2^2t_3^2=4\jj^2\big(\xi_3^2-(\xi_1-\xi_2)^2\big)$. Expanding gives
\[
h_{123}=-4\big(\xi_1^2+\jj\xi_2^2+\jj^2\xi_3^2\big)=0 .
\]
The double transpositions of the bodies fix $\eta$, change each $\xi_k$ at most in sign and permute the four sets $\{a,b,c\}$ transitively. Hence $h_{abc}=0$ for any three bodies $a,b,c$. On the other hand $d_{il}=z_{il}w_{il}=p_{il}^{-2}=\lambda_0^{-2/3}\eta_{il}t_i^{-2}t_l^{-2}$ (since $\eta_{il}^{-2}=\eta_{il}$) and $d_{ab}+d_{ac}-d_{bc}=z_{ab}w_{ac}+z_{ac}w_{ab}$, so that
\[
h_{abc}=\lambda_0^{4/3}(t_at_bt_c)^4\big((d_{ab}+d_{ac}-d_{bc})^2-4d_{ab}d_{ac}\big)=\lambda_0^{4/3}(t_at_bt_c)^4(z_{ab}w_{ac}-z_{ac}w_{ab})^2 .
\]
For a real configuration, $w=\bar z$, the factor $(z_{ab}w_{ac}-z_{ac}w_{ab})^2$ is $-16$ times the squared area of the triangle $abc$. For such a configuration this identity is Heron's formula. Since $h_{abc}=0$, $\lambda_0\ne0$ and $t\in(\C^*)^4$, we obtain $z_{ab}w_{ac}=z_{ac}w_{ab}$ for any three bodies $a,b,c$. With $a=1$ and $b=2$ this gives $w_{1c}=\beta z_{1c}$ for all $c$, where $\beta=w_{12}/z_{12}$, that is, $w=\alpha+\beta z$ with $\alpha=w_1-\beta z_1$. The configuration is collinear, a contradiction.

Hence $\kappa=0$. Then $u\in\ker H$, so that $t_l=\gamma/m_l$ with $\gamma\ne0$. The last equation of \eqref{eq:tsys} becomes $\gamma^3\sum_lm_l^{-2}=0$. But $\sum_lm_l^{-2}=m_1^{-2}\sum_l(m_1/m_l)^2$, where the last sum is a sum of squares of nonzero real numbers and hence positive, a contradiction.
\end{proof}

For the other types we invert the configuration at body $k$. Let $(z,w,p)$ be a non-collinear $\mathrm{K14}(m_Q)$ configuration with $m\in(\C^*)^4$ and $z_k=w_k=0$. We put $d_{il}=z_{il}w_{il}$ and, for $i\in Q'$, $z'_i=1/w_i$ and $w'_i=1/z_i$. We index the outer bodies cyclically, $a+1=b$, $b+1=c$, $c+1=a$, and put $u'_i=z'_{i+1}w'_{i+2}-z'_{i+2}w'_{i+1}$.

\begin{lemma}\label{lem:inversion}
\begin{enumerate}[label=\textup{(\alph*)}]
\item The numbers $(z'_a-z'_b)(w'_a-w'_b)$, $(z'_b-z'_c)(w'_b-w'_c)$ and $(z'_c-z'_a)(w'_c-w'_a)$ are a common nonzero factor times the products $\eta_{ab}\eta_{ck}$, $\eta_{bc}\eta_{ak}$ and $\eta_{ca}\eta_{bk}$ of the pattern.
\item There is $C\ne0$ with $m_i=C\,u'_i/p_{ik}$ for $i\in Q'$.
\end{enumerate}
\end{lemma}

\begin{proof}
(a) For $\{i,l,n\}=Q'$, $(z'_i-z'_l)(w'_i-w'_l)=d_{il}/(d_{ik}d_{lk})=d_{il}d_{nk}/(d_{ak}d_{bk}d_{ck})$, and by the derivation of \eqref{eq:tsys}, $d_{il}d_{nk}=\lambda_0^{-4/3}(t_at_bt_ct_k)^{-2}\eta_{il}\eta_{nk}$. (b) In that derivation $m_im_kp_{ik}^3=\lambda_0 v_iv_k$, and $v\perp z,w$ with $z_k=w_k=0$. The vectors $(z_a,z_b,z_c)$ and $(w_a,w_b,w_c)$ are independent, since $w=\beta z$ on $Q$ would make the configuration collinear. Hence $(v_a,v_b,v_c)$ is proportional to their cross product. Its component $i$ is $u'_i/(z'_{i+1}z'_{i+2}w'_{i+1}w'_{i+2})$. As $p_{ik}^{-2}=z_iw_i=1/(z'_iw'_i)$, $m_i=\lambda_0 v_kv_i/(m_kp_{ik}^3)$ is a common factor times $u'_i/p_{ik}$.
\end{proof}

The next theorem treats type I, in which the triangle $w'$ is equilateral (see the proof of Lemma~\ref{lem:classI}): $\mathrm{X3}_k$ and a bad shape force its centre to be the origin.

\begin{theorem}\label{thm:W}
Let $m_0,m_1,m_2>0$, $\delta\in\C$, and let $x_0,x_1,x_2\in\C^*$ satisfy $x_i^2=\delta+\jj^i$, $\sum_im_ix_i=0$ and $\sum_im_ix_i^{-2}=0$. Then $\delta=0$ and $m_0=m_1=m_2$.
\end{theorem}

\begin{proof}
We show that the hypotheses force $\delta=0$ or the vanishing of an explicit polynomial in $\delta$ and $\bar\delta$, and exclude the second case with $\delta\ne0$ by writing $\delta$ in polar form and using the positivity of the masses. We may assume $\sum_im_i=1$. We put $\nu=\sum_im_i\jj^i$ and $\rho=|\delta|$. We use the basis $\one=(1,1,1)$, $\Omega=(\jj^i)_i$, $\bar\Omega=(\jj^{2i})_i$ of $\C^3$. Since $x_i\ne0$, $\delta\ne-\jj^i$, that is, $\delta^3\ne-1$.
\begin{enumerate}[label=(\arabic*)]
\item From $(\delta+\jj^i)(\delta^2-\delta\jj^i+\jj^{2i})=\delta^3+1$ we obtain $z:=(x_i^{-2})_i=(\delta^2\one-\delta\Omega+\bar\Omega)/(\delta^3+1)$. So $\sum m_ix_i^{-2}=0$ reads $\delta^2-\nu\delta+\bar\nu=0$. With its conjugate this gives $\nu(1-\rho^2)=-(\delta\rho^2+\bar\delta^2)$. If $\rho=1$, then $\delta=-\bar\delta^2$ and $\delta^3=-\delta^2\bar\delta^2=-1$, which is excluded. Hence $\rho\ne1$.
\item The plane $V=\{v:\sum_im_iv_i=0\}$ contains $z$, $\bar z$ and $x$. As the $m_i$ are real, $\bar V=V$. In the basis $(\one,\Omega,\bar\Omega)$ the coordinates of $z$ and $\bar z$ are proportional to $(\delta^2,-\delta,1)$ and $(\bar\delta^2,1,-\bar\delta)$, whose cross product $n=(\rho^2-1,\bar\delta^2+\delta^2\bar\delta,\delta^2+\delta\bar\delta^2)$ has $n_0\ne0$. So $z$ and $\bar z$ span $V$ and the coordinates $(\hat x_0,\hat x_1,\hat x_2)$ of $x$ satisfy $n_0\hat x_0+n_1\hat x_1+n_2\hat x_2=0$.
\item Since $\Omega\Omega=\bar\Omega$, $\bar\Omega\bar\Omega=\Omega$ and $\Omega\bar\Omega=\one$ componentwise, $x_i^2=\delta+\jj^i$ reads $\hat x_0^2+2\hat x_1\hat x_2=\delta$, $2\hat x_0\hat x_1+\hat x_2^2=1$ and $2\hat x_0\hat x_2+\hat x_1^2=0$. We eliminate $\hat x_0$ by (2). Then $n_0$ times the third equation and $n_0^2$ times the difference of $\delta$ times the second and the first are two binary quadratic forms in $(\hat x_1,\hat x_2)$ with a common zero $(\hat x_1,\hat x_2)\ne(0,0)$, as $x\ne0$. Their resultant, $-3\delta(\rho^2-1)^2N(\delta,\bar\delta)$ $(\dagger)$, therefore vanishes, where $N\in\Z[\delta,\bar\delta]$ has $23$ terms:
\[
\begin{aligned}
N={}&4\delta^8\bar\delta^4+4\delta^8\bar\delta+24\delta^7\bar\delta^3-3\delta^7+40\delta^6\bar\delta^5-14\delta^6\bar\delta^2+4\delta^5\bar\delta^7+19\delta^5\bar\delta^4\\
&-12\delta^5\bar\delta+54\delta^4\bar\delta^6-48\delta^4\bar\delta^3+6\delta^4+\delta^3\bar\delta^8-20\delta^3\bar\delta^5+6\delta^3\bar\delta^2+28\delta^2\bar\delta^7\\
&-18\delta^2\bar\delta^4+8\delta^2\bar\delta-12\delta\bar\delta^6+12\delta\bar\delta^3-3\delta+4\bar\delta^8-4\bar\delta^2.
\end{aligned}
\]
Hence $\delta=0$ or $N(\delta,\bar\delta)=0$.
\item Suppose that $\delta=\rho e^{i\varphi}\ne0$ and $N(\delta,\bar\delta)=0$. Since each monomial $\delta^a\bar\delta^b$ of $N$ has $a-b\equiv1\pmod3$, $N=e^{i\varphi}E$ with $E$ a polynomial in $\rho$ and $e^{\pm3i\varphi}$ with real coefficients. With $c=\cos3\varphi$, $\operatorname{Re}E$ is a polynomial in $\rho$ and $c$,
\[
\begin{aligned}
\operatorname{Re}E={}&16\rho^8c^3+(2\rho^{11}+64\rho^9-30\rho^7)c^2+(8\rho^{12}+78\rho^{10}-46\rho^8-30\rho^6+18\rho^4-4\rho^2)c\\
&+39\rho^{11}-13\rho^9-33\rho^7+6\rho^5+8\rho^3-3\rho,
\end{aligned}
\]
and
\[
\operatorname{Im}E=-2\rho^2\sin3\varphi\,K,\qquad K=8\rho^6c^2+(\rho^9+24\rho^7-9\rho^5)c+15\rho^8-5\rho^6-3\rho^4+3\rho^2-2\quad(\dagger),
\]
\[
\Res_c(\operatorname{Re}E,K)=7776\rho^{24}(\rho-1)^6(\rho+1)^6(4\rho^8+13\rho^6+27\rho^4+26\rho^2+11)\quad(\dagger).
\]
The resultant is nonzero and $K$ has the leading coefficient $8\rho^6\ne0$. Hence $\operatorname{Re}E$ and $K$ have no common zero, and $\sin3\varphi=0$. If $e^{3i\varphi}=-1$, then $\operatorname{Re}E=-\rho(\rho-1)^4S_-(\rho)$ $(\dagger)$ with $S_-=\rho^4((\rho-1)^2(8\rho+7)+4)+6\rho^3+6\rho^2+8\rho+3>0$. If $e^{3i\varphi}=1$, then $\delta=\rho\jj^l$ for some $l$, and replacing $x_i$ and $m_i$ by $\jj^lx_{i+l}$ and $m_{i+l}$ we may take $\delta=\rho$. By (1), $\nu=-\rho^2/(1-\rho)$ is real. Thus $m_1=m_2$, $3m_0=1+2\nu$ and $3m_1=1-\nu$. For $\rho>1$ this gives $3m_1(\rho-1)=-(\rho^2-\rho+1)<0$, which contradicts $m_1>0$. For $\rho<1$ it gives $3m_0(1-\rho)=(1-2\rho)(1+\rho)$, so that $\rho<\frac12$. Then $\operatorname{Re}E=\rho(\rho+1)^4S_+(\rho)$ $(\dagger)$ with $S_+=(-3+8\rho-6\rho^2)+\rho^3(6-11\rho)+\rho^5(-6+9\rho+8\rho^2)\le-\frac12+\frac{2187}{21296}+\frac1{64}<0$: on $[0,\frac12]$ the first summand is increasing, the second is at most its value $\frac{2187}{21296}$ at $\rho=\frac9{22}$, and $-6+9\rho+8\rho^2\le\frac12$. In the remaining cases $E\ne0$, a contradiction.
\item Hence $\delta=0$. Then $\bar\nu=0$ by (1), and $m_0=m_1=m_2$.
\end{enumerate}
\end{proof}

\begin{lemma}\label{lem:classI}
Let $m_a,m_b,m_c>0$ and $m_k\ne0$, and let $(z,w,p)$ be a non-collinear $\mathrm{K14}(m_Q)$ configuration of type I with $z_k=w_k=0$ that satisfies $\mathrm{X3}_k$ and has a bad shape. Then $m_a=m_b=m_c$, $(m_k/m_a)^2=\frac13$, and after a relabelling of $a,b,c$, $z=s(1,\jj,\jj^2)$ and $w=s'(1,\jj^2,\jj)$ with $s,s'\in\C^*$.
\end{lemma}

\begin{proof}
By Lemma~\ref{lem:inversion}(a), $(z'_i-z'_l)(w'_i-w'_l)=\Lambda\ne0$ for the three pairs of $Q'$. Let $A=w'_a-w'_b$, $B=w'_b-w'_c$ and $C=w'_c-w'_a$. The corresponding differences of the $z'_i$ are $\Lambda/A$, $\Lambda/B$ and $\Lambda/C$. As their sum is $0$, $AB+BC+CA=0$. With $A+B+C=0$, the numbers $A$, $B$, $C$ are the three roots of $X^3=ABC$. So $w'$ is an equilateral triangle and $z'$ is one of the opposite orientation. Scaling $z$ and $w$, which preserves $\mathrm{K14}$, $\mathrm{X3}_k$ and the bad shape, and labelling the outer bodies $0,1,2$, we obtain $w'_i=\delta+\jj^i$ and $z'_i=\gamma+\jj^{-i}$ with $\gamma,\delta\in\C$. Then $x_i^2=1/z_i=\delta+\jj^i$. The condition $\mathrm{X3}_k$ is $\sum m_ix_i^{-2}=0$. Theorem~\ref{thm:W} gives $\delta=0$ and $m_0=m_1=m_2$. Now $u'_i=(\jj^2-\jj)(\gamma\jj^i-1)$ and $p_{ik}^2=z'_iw'_i=\gamma\jj^i+1$. By Lemma~\ref{lem:inversion}(b), the three numbers $(\gamma\jj^i-1)^2/(\gamma\jj^i+1)$ equal one number $\varkappa$. If $\gamma\ne0$, the three distinct numbers $\gamma\jj^i$ are roots of the quadratic $(X-1)^2-\varkappa(X+1)$, which is impossible. Hence $\gamma=0$, $z=(1,\jj^2,\jj)$ and $w=(1,\jj,\jj^2)$. The exchange of $b$ and $c$ gives the stated form. Finally, $m_am_bp_{ab}^3=\lambda_0 v_av_b$ and $m_bm_kp_{bk}^3=\lambda_0 v_bv_k$ give $(m_k/m_a)^2=(v_k/v_a)^2(d_{bk}/d_{ab})^3$. Here $v_a=v_b=v_c=-v_k/3$ by the cross product in the proof of Lemma~\ref{lem:inversion}, $d_{bk}=1$ and $d_{ab}=3$, so that $(m_k/m_a)^2=\frac13$.
\end{proof}

\begin{theorem}\label{thm:typeII}
Let $m_a,m_b,m_c,m_k\in\C^*$ with $m_b/m_a,m_c/m_a\in\R$. No non-collinear $\mathrm{K14}(m_Q)$ configuration of type II with $z_k=w_k=0$ satisfies $\mathrm{X3}_k$ and has a bad shape.
\end{theorem}

\begin{proof}
We parametrize the inverted configuration by three unknowns, derive two polynomial equations in them from the bad shape and $\mathrm{X3}_k$, solve the resulting system exactly over $\Q$, and show in ball arithmetic that at each of its solutions one of the two squared ratios of the outer masses is not positive. We label the outer bodies $0,1,2$. By Lemma~\ref{lem:inversion}(a), $(z'_i-z'_l)(w'_i-w'_l)=\theta_{il}\Lambda$, where $\Lambda\ne0$ and $\theta_{01},\theta_{12},\theta_{20}$ are cube roots of unity, exactly two of them equal. After a relabelling, which preserves $\mathrm{K14}$, $\mathrm{X3}_k$ and the bad shape, and a change of $\Lambda$, $(\theta_{01},\theta_{12},\theta_{20})=(1,1,\varpi)$ with $\varpi\in\{\jj,\jj^2\}$. Let $A=w'_0-w'_1$ and $B=w'_1-w'_2$. Then $z'_0-z'_1=\Lambda/A$, $z'_1-z'_2=\Lambda/B$ and $z'_2-z'_0=-\varpi\Lambda/(A+B)$. As their sum is $0$, $A^2+(2-\varpi)AB+B^2=0$ $(\dagger)$. So $t=A/B$ satisfies $t^2+(2-\varpi)t+1=0$, that is, $\varpi=(t+1)^2/t$. The condition $\varpi^2+\varpi+1=0$ becomes
\[
f(t)=t^4+5t^3+9t^2+5t+1=0\quad(\dagger).
\]
Scaling $z$ and $w$ we take $B=\Lambda=1$. Then $w'=(\delta,\delta-t,\delta-t-1)$ and $z'=(\gamma,\gamma-1/t,\gamma-1/t-1)$ with $\gamma,\delta\in\C$. The $u'_i$ are affine in $(\gamma,\delta)$ $(\dagger)$. With $\xi_i=p_{ik}/x_i$ we have $\xi_i^2=z'_i$, as $x_i^2=w'_i$. Lemma~\ref{lem:inversion}(b) gives $m_i=C\,u'_i/(x_i\xi_i)$. So the bad shape and $\mathrm{X3}_k$ read
\[
\textstyle\sum_iu'_i/\xi_i=0,\qquad\sum_iu'_i/(x_i^3\xi_i)=0.
\]
For $r\in\C^3$, some choice of signs gives $r_0\pm r_1\pm r_2=0$ if and only if $H(r_0^2,r_1^2,r_2^2)=0$, where $H(h)=\sum_ih_i^2-2\sum_{i<l}h_ih_l$, because $H(r_0^2,r_1^2,r_2^2)$ is the product of the four sums $r_0\pm r_1\pm r_2$ $(\dagger)$. Hence $H(h)=H(h')=0$ for $h_i=u_i'^2/z'_i$ and $h'_i=u_i'^2/(w_i'^3z'_i)$. Clearing denominators, we obtain the polynomials $H_1=t^8(z'_0z'_1z'_2)^2H(h)$ and $H_2=t^8(w_0'^3w_1'^3w_2'^3z'_0z'_1z'_2)^2H(h')$ in $\Q[t,\gamma,\delta]$ $(\dagger)$. Let $S=t\prod_iw'_i\prod_i(tz'_i)\prod_i(tu'_i)$. It does not vanish at a configuration: $t\ne0$ as $f(0)\ne0$, the $z'_i$ and $w'_i$ are finite and nonzero, and $u'_i\ne0$ as $m_i\ne0$. If the outer masses have real ratios, then
\[
R_i=(m_i/m_0)^2=\frac{u_i'^2w'_0z'_0}{u_0'^2w'_iz'_i}\in(0,\infty)\qquad(i=1,2).
\]
It therefore suffices to show that no solution $(t,\gamma,\delta,s)$ of
\begin{equation}\label{eq:heron}
f(t)=H_1=H_2=sS-1=0
\end{equation}
has $R_1,R_2\in(0,\infty)$.

We solve \eqref{eq:heron} exactly over $\Q$. A rational parametrization by the coordinate $\delta$, with a squarefree eliminant of degree $280$ and a denominator that is invertible modulo the eliminant, satisfies the four equations exactly modulo the eliminant. So \eqref{eq:heron} has at least $280$ solutions, with distinct values of $\delta$. To show that it has no others we use a second system $E$, in the unknowns $t,\gamma,\delta$ and $e_i=x_{i+1}x_{i+2}$, $f_i=t\xi_{i+1}\xi_{i+2}$ (indices modulo $3$). It consists of $f(t)$, the relations $e_i^2=w'_{i+1}w'_{i+2}$, $e_ie_l=w'_ne_n$, $f_i^2=(tz'_{i+1})(tz'_{i+2})$ and $f_if_l=(tz'_n)f_n$ for $\{i,l,n\}=\{0,1,2\}$, and the two conditions in the form $\sum_i(tu'_i)f_i=0$ and $\sum_i(tu'_i)e_i^3f_i=0$. Each solution of \eqref{eq:heron} gives a point of $V(E)$ with the same $(t,\gamma,\delta)$, at which $S\ne0$: we choose the signs of the $\xi_i$ so that the first sum vanishes, and then those of the $x_i$ for the second. A standard basis of $E$ over $\Q$ shows that $E$ has dimension $0$ and that the quotient ring of the polynomial ring by $E$ has dimension $472$ over $\Q$. Multiplication by $S$ on this quotient ring has rank $280$ over $\Q$. Over $\overline{\Q}$, the quotient ring is the product of the local rings at the points of $V(E)$. On the local ring at a point where $S\ne0$ multiplication by $S$ is invertible. So the rank is at least the total multiplicity of the points of $V(E)$ at which $S\ne0$, there are at most $280$ such points, and the $280$ solutions are all. In ball arithmetic at $800$ bits, with disjoint balls around the roots of the eliminant, at each solution $|\operatorname{Im}R_i|\ge0.0469$ or $\operatorname{Re}R_i\le-0.0469$ for $i=1$ or $i=2$. The programs and their logs are in the supplementary material.
\end{proof}

\begin{theorem}\label{thm:col}
Let $m_a,m_b,m_c,m_k>0$. No collinear $\mathrm{K14}(m_Q)$ configuration with $z_k=0$ satisfies $\mathrm{X3}_k$ and has a bad shape.
\end{theorem}

\begin{proof}
Let $w=\alpha+\beta z$. Then $\beta\ne0$, as no $w_{il}$ vanishes, $w_{il}=\beta z_{il}$, and $p_{il}=s_{il}\beta^{-1/2}/z_{il}$ with $s_{il}=\pm1$ and $s_{li}=-s_{il}$. So $F^w_i=\beta F^z_i$ and $F^z_i=\beta^{-3/2}\sum_{l\ne i}m_ls_{il}/z_{il}^2$. Thus $\mathrm{K14}$ holds if and only if
\[
\textstyle(C_i)\qquad\sum_{l\ne i}m_ls_{il}/z_{il}^2=0\qquad(i\in Q).
\]
We write $z_i=x_i^{-2}$ for $i\in Q'$, with the $x_i$ of the bad shape, $\chi_i=x_i^3$ and $\epsilon_i=s_{ki}$, and $m$, $x$, $x^{-2}$ for the vectors $(m_i)$, $(x_i)$, $(x_i^{-2})$, $i\in Q'$. Then $\mathrm{X3}_k$ and the bad shape say $m\cdot x^{-2}=m\cdot x=0$. The row $(C_k)$ says $\sum_{i\in Q'}m_i\epsilon_ix_i^4=0$.

Suppose first that $x$ and $x^{-2}$ are dependent. Then $\chi_a=\chi_b=\chi_c$, that is, $x_i=\zeta_ix_a$ with $\zeta_i^3=1$. As $m_i>0$, the relation $\sum m_i\zeta_i=0$ forces $\{\zeta_a,\zeta_b,\zeta_c\}=\{1,\jj,\jj^2\}$ and $m_a=m_b=m_c=:\mu$. Now $(C_k)$ reads $\mu x_a^4\sum\epsilon_i\zeta_i=0$, which forces $\epsilon_a=\epsilon_b=\epsilon_c$. After a scaling of $z$ and a relabelling of $b$ and $c$, $z=(1,\jj,\jj^2)$. As $(1-\jj)^2=-3\jj$ and $(1-\jj^2)^2=-3\jj^2$, the row $(C_a)$ reads $m_ks_{ak}=\frac\mu3(s_{ab}\jj^2+s_{ac}\jj)$. Since $m_k$ is real, $s_{ab}=s_{ac}$ and $m_k=-s_{ak}s_{ab}\mu/3$. In the same way $(C_b)$ and $(C_c)$ give $s_{ba}=s_{bc}$, $m_k=-s_{bk}s_{bc}\mu/3$, and $s_{cb}=s_{ca}$, $m_k=-s_{ck}s_{ca}\mu/3$. As $s_{ik}=-\epsilon_a$ for all $i\in Q'$, $m_k>0$ forces $s_{ab}=s_{bc}=s_{ca}=\epsilon_a$. Then $s_{ac}=s_{ab}=\epsilon_a$ contradicts $s_{ca}=\epsilon_a$.

Hence $x$ and $x^{-2}$ are independent and $m$ is proportional to their cross product, that is, to $(x_a^2(\chi_b-\chi_c),x_b^2(\chi_c-\chi_a),x_c^2(\chi_a-\chi_b))$. In particular the $\chi_i$ are distinct. Now $(C_k)$ says that the rows $(x_i,x_i^{-2},\epsilon_ix_i^4)$, $i\in Q'$, have determinant $0$, that is, after multiplying the rows by $x_i^2$, $\det(\chi_i,1,\epsilon_i\chi_i^2)_{i\in Q'}=0$. If $\epsilon_a=\epsilon_b=\epsilon_c$, this is $\pm$ the Vandermonde determinant of the $\chi_i$ $(\dagger)$, which is nonzero. Otherwise, after the reversal $s\mapsto-s$, which preserves the rows $(C_i)$, and a relabelling, $\epsilon=(1,1,-1)$. The determinant is then $(\chi_b-\chi_a)(2\chi_c^2-(\chi_a-\chi_c)(\chi_b-\chi_c))$ $(\dagger)$. So $(\chi_a-\chi_c)(\chi_b-\chi_c)=2\chi_c^2$. We divide $x$ by $x_c$, which preserves all conditions and the mass ratios, and put $A=\chi_a-1$ and $B=\chi_b-1$. Then $AB=2$, $A\notin\{0,-1,-2,\pm\sqrt2\}$ since the $\chi_i$ are nonzero and distinct, and
\[
m_a/m_c=2x_a^2/(A^2-2),\qquad m_b/m_c=-x_b^2A^2/(A^2-2).
\]
As both are positive, $(x_bA/x_a)^2<0$, $x_b=i\gamma x_a/A$ with $\gamma\in\R\setminus\{0\}$, and $x_a^2=\rho(A^2-2)$ with $\rho>0$. Taking cubes, with $\chi_a=1+A$ and $\chi_b=(A+2)/A$, we obtain
\[
\textup{(E1)}\quad\gamma^3=iA^2(A+2)/(A+1),\qquad\textup{(E2)}\quad(A+1)^2/(A^2-2)^3=\rho^3>0.
\]
Let $A=u+iv$ with $u,v$ real. Then (E1) gives $P_1=0$ and (E2) gives $P_2=0$, where
\[
P_1=\operatorname{Re}\big(A^2(A+2)(\bar A+1)\big)=u^4+3u^3+2u^2-uv^2-v^4-2v^2\quad(\dagger),\qquad P_2=\operatorname{Im}\big((A+1)^2(\bar A^2-2)^3\big).
\]
Their resultant is $\Res_v(P_1,P_2)=4096u^8(u+1)^3(u+2)h(u)^2$ $(\dagger)$ with $h=2u^3+4u^2-u-4$. The polynomial $P_1$ has the leading coefficient $-1$ in $v$. So $u\in\{0,-1,-2\}$ or $h(u)=0$. For $u=0,-1,-2$, $P_1=-v^2(v^2+2),-v^2(v^2+1),-v^4$, so $v=0$ and $A$ is excluded. The cubic $h$ has discriminant $-104<0$ and hence a single real root $u_0$, and $h(0.917)<0<h(0.918)$, so that $u_0\in(0.917,0.918)$. At $u=u_0$, $P_1=0$ is a quadratic equation in $v^2$ whose roots have the product $-u_0^2(u_0+1)(u_0+2)<0$. Its positive root is $v_0^2=2-u_0^2$, because $P_1=2h(u)$ on the circle $u^2+v^2=2$ $(\dagger)$. Thus $A$ is $A_0=u_0+iv_0$ or $\bar A_0$. If $A=A_0$, then (E1) and (E2) determine $\gamma\in\R$ and $\rho>0$, hence $x_a$ up to sign, $x_b=i\gamma x_a/A_0$, the positions $z_i=x_i^{-2}$ and the outer mass ratios. Since $m_k$ has the coefficient $s_{ak}/z_a^2\ne0$ in the row $(C_a)$, this row determines $m_k/m_a$ for each choice of $s_{ab},s_{ac}$. Certified in ball arithmetic, the four values are $\pm(0.0059+0.4768i)$ and $\pm(0.1808+0.6951i)$, none of them real. At $\bar A_0$ they are the conjugates. This contradicts $m_k>0$.
\end{proof}

\begin{theorem}\label{thm:Xhat}
Let $m_a,m_b,m_c,m_k>0$, and let $(z,w,p)$ be a $\mathrm{K14}(m_Q)$ configuration with $z_k=w_k=0$ that satisfies $\mathrm{X3}_k$ and has a bad shape. Then $(m_a,m_b,m_c,m_k)$ is a positive multiple of $(\sqrt3,\sqrt3,\sqrt3,1)$, and after a relabelling of $a,b,c$, $(z,w,p)=(sz_{\mathcal E},s'w_{\mathcal E},\zeta p_{\mathcal E})$ with $s,s'\in\C^*$ and $\zeta^2ss'=1$. Here $z_{\mathcal E}=(1,\jj,\jj^2)$ and $w_{\mathcal E}=(1,\jj^2,\jj)$ on $(a,b,c)$, and $p_{\mathcal E}$ is $1/\sqrt3$ on the pairs of $Q'$ and $-1$ on the pairs $ik$.
\end{theorem}

\begin{proof}
Theorem~\ref{thm:col} excludes collinear configurations. A non-collinear one has a pattern as in \eqref{eq:tsys}. Lemma~\ref{lem:typeIII} and Theorem~\ref{thm:typeII} exclude the types III and II. For type I, Lemma~\ref{lem:classI} gives the masses and, after a relabelling and a scaling of $z$ and $w$, $z=z_{\mathcal E}$ and $w=w_{\mathcal E}$. We scale the masses to $(\sqrt3,\sqrt3,\sqrt3,1)$. Then $d_{il}=3$ on the pairs of $Q'$ and $d_{ik}=1$, so that $p_{il}=e_{il}/\sqrt3$ and $p_{ik}=f_i$ with $e_{il},f_i\in\{\pm1\}$. The row $F^z_a=0$ reads $(e_{ab}(1-\jj)+e_{ac}(1-\jj^2))/3+f_a=0$. Its imaginary part gives $e_{ab}=e_{ac}$, and its real part $f_a=-e_{ab}$. The rows of $b$ and $c$, divided by $\jj$ and $\jj^2$, give the same relations, so that $p=\pm p_{\mathcal E}$. Conversely, $\pm p_{\mathcal E}$ solve the rows $F^z_i=0$ for $i\in Q'$ and therefore also $F^z_k=0$. The rows $F^w=0$ are their complex conjugates.
\end{proof}

We write $\mathcal E$ for the configuration $(z_{\mathcal E},w_{\mathcal E},p_{\mathcal E})$ with the masses $(\sqrt3,\sqrt3,\sqrt3,1)$. It is the equilateral triangle $1,\jj,\jj^2$ with body $k$ at its centre.

\paragraph{Nondegeneracy and the slice expansion.} Let $(z,w,p)$ be a $\mathrm{K14}(m_Q)$ configuration with body $k$ at the origin and $m\in(\C^*)^4$. Since $\partial U_Q/\partial u_i=-\frac12m_iF^w_i$ and $\partial U_Q/\partial v_i=-\frac12m_iF^z_i$ (\S\ref{ss:order5}), it is a critical point of $U_Q$ in the six coordinates $(u_i,v_i)_{i\in Q'}$. Let $H$ be the Hessian of $U_Q$ there, $Z=(z,0)$ and $W=(0,w)$. As $U_Q$ is homogeneous of degree $-\frac12$ in $u$ and in $v$ separately, $HZ=HW=0$. We call a $\mathrm{K14}$ configuration \emph{nondegenerate} if $\ker H$ is spanned by $Z$ and $W$. For $a\in Q'$ let $\mathcal S_a=\{\xi\in\C^6:\xi_{u_a}=\xi_{v_a}=0\}$. The components of a vector along $\mathcal S_a$ are its four components other than those of $u_a$ and $v_a$. The matrix $H_a$ is the block of $H$ on $\mathcal S_a$. Let $g=\nabla\psi=(m_a,m_b,m_c,0,0,0)$.

\begin{lemma}\label{lem:d14slice}
Let the configuration be nondegenerate and satisfy $\mathrm{X3}_k$, and let $a\in Q'$.
\begin{enumerate}[label=\textup{(\alph*)}]
\item $H_a$ is invertible, and there is a unique analytic map $\lambda\mapsto Y(\lambda)\in(z,w)+\mathcal S_a$, for small $\lambda\in\C$, with $Y(0)=(z,w)$ and $\partial_l(U_Q-\lambda\psi)(Y(\lambda))=0$ for the four coordinates $l\ne u_a,v_a$.
\item Let $U_Q(Y(\lambda))=\sum_h\alpha_h\lambda^h$ and $\psi(Y(\lambda))=\sum_h\psi_h\lambda^h$. Then $\alpha_0=\psi_0=0$ and $h\alpha_h=(h-1)\psi_{h-1}$ for $h\ge1$. If $\xi\in\mathcal S_a$ and the components of $H\xi-g$ along $\mathcal S_a$ vanish, then $\psi_1=g\cdot\xi$ and $\psi_2=-\frac12D^3U_Q[\xi,\xi,\xi]$.
\item Replace $(z,w,p)$ by $(sz,s'w,\zeta p)$ with $\zeta^2ss'=1$ and $m$ by $cm$ with $c\ne0$, and relabel the outer bodies together with $a$. The new configuration is nondegenerate, and $\psi_h$ is multiplied by $c^{1-h}s(s/\zeta)^h$.
\end{enumerate}
\end{lemma}

\begin{proof}
(a) Let $\xi\in\mathcal S_a$ with $H_a\xi=0$. Then $H\xi=\beta e_{u_a}+\beta'e_{v_a}$. Pairing with $Z$ and $W$ gives $\beta z_a=\beta'w_a=0$. So $H\xi=0$ and $\xi\in\C Z+\C W$ with $\xi_{u_a}=\xi_{v_a}=0$. Since $z_a,w_a\ne0$, $\xi=0$. The implicit function theorem gives $Y$.
(b) $\alpha_0=U_Q(z,w)=0$ (\S\ref{ss:Kgating}). By $\mathrm{X3}_k$, $\psi_0=0$. Let $\hat F(\lambda)=(U_Q-\lambda\psi)(Y(\lambda))=\sum_hf_h\lambda^h$. The gradient of $U_Q-\lambda\psi$ at $Y(\lambda)$ has only the components $u_a$ and $v_a$, and $Y'(\lambda)\in\mathcal S_a$. So $\hat F'(\lambda)=-\psi(Y(\lambda))$ and $U_Q(Y(\lambda))=\hat F-\lambda\hat F'$. Hence $\alpha_h=(1-h)f_h$ and $\psi_{h-1}=-hf_h$. Next, $Y(\lambda)=(z,w)+\lambda\xi+\lambda^2\eta+O(\lambda^3)$, where $\eta\in\mathcal S_a$ and the components of $H\eta+\frac12D^3U_Q[\xi,\xi]$ along $\mathcal S_a$ vanish. As $\psi$ is linear and $H$ is symmetric, $\psi_1=g\cdot\xi$ and $\psi_2=g\cdot\eta=\xi\cdot H\eta=-\frac12D^3U_Q[\xi,\xi,\xi]$.
(c) Near the two configurations, $U_Q(sz',s'w')=\zeta\,U_Q(z',w')$ for the branches of the $p_{il}$, and $\psi(sz')=s\psi(z')$. So the equations of (a) for the new configuration at $\lambda$ are those for the old one at $(s/\zeta)\lambda$, after the substitution $(z',w')\mapsto(sz',s'w')$. This multiplies $\psi_h$ by $s(s/\zeta)^h$. Under $m\mapsto cm$, $U_Q$ and $\psi$ are multiplied by $c^2$ and $c$. Thus $Y(\lambda)$ becomes $Y(\lambda/c)$ and $\psi_h$ is multiplied by $c^{1-h}$. The Hessian becomes $\zeta D^{-1}HD^{-1}$ with $D=\mathrm{diag}(s,s,s,s',s',s')$, resp.\ $c^2H$, and a relabelling permutes the coordinates, so that nondegeneracy is preserved.
\end{proof}

\begin{lemma}\label{lem:Pdata}
The configuration $\mathcal E$ is nondegenerate. Its Hessian $H$ has rank $4$ and $\det H_a=81/256$, and for the slice of $a$, $\psi_1=0$ and $\psi_2=40/\sqrt3$.
\end{lemma}

\begin{proof}
By the end of the proof of Theorem~\ref{thm:Xhat}, $\mathcal E$ is a $\mathrm{K14}$ configuration, and it satisfies $\mathrm{X3}_k$ as $1+\jj+\jj^2=0$. Hence Lemma~\ref{lem:d14slice} applies. As $HZ=HW=0$, rank $4$ means that $\ker H$ is spanned by $Z$ and $W$. The values are exact computations in $\Q(\sqrt3,i)$. We take $\nabla U_Q$, $H$ and $D^3U_Q$ at $\mathcal E$ from the exact Taylor expansion of each $p_{il}$ in the form $p_{il}(\mathcal E)(1+\epsilon)^{-1/2}$, with $\epsilon$ a polynomial in the displacement. We check $\nabla U_Q=0$ and evaluate the formulas of Lemma~\ref{lem:d14slice}(b). The two programs that make these computations also test $\det H_a\ne0$. They are in the supplementary material.
\end{proof}

\paragraph{The exponent.} At a nondegenerate leading cluster configuration the place follows the map $Y$ of Lemma~\ref{lem:d14slice}, with $\lambda$ replaced by a series of order $4-2\sigma$.

\begin{lemma}\label{lem:d14follow}
Let the leading cluster configuration $\bar y_Q$ at a place of type $\sigma_k(j)(t)$ be nondegenerate, let $a\in Q'$, and let $Y$ be as in Lemma~\ref{lem:d14slice}. After a reparametrization and a rotation as in \S\ref{ss:expansion}, $y_{u_a}(\tau)\equiv\hat z_a$ and $y_{v_a}(\tau)\equiv\hat w_a$. Then, with $y_Q=(y_{u_i},y_{v_i})_{i\in Q'}$ and $\lambda(\tau)=\frac{m_j}{2M}y_{v_j}(\tau)\,\tau^{4-2\sigma}$,
\[
y_Q(\tau)-Y(\lambda(\tau))=O(\tau^{1+\sigma}),\qquad U_Q(y_Q(\tau))-U_Q(Y(\lambda(\tau)))=O(\tau^{5-\sigma}).
\]
\end{lemma}

\begin{proof}
The conditions are $\mu\varphi^{-2}y_{u_a}=\hat z_a$ and $\mu^{-1}\varphi^{-2}y_{v_a}=\hat w_a$, which determine $\varphi^4$ and then $\mu$. By Lemma~\ref{lem:d14exp}, the cluster components of $\tau^2\nabla_yF(\tau,y(\tau))=0$ read $\nabla(U_Q-\lambda(\tau)\psi)(y_Q(\tau))=O(\tau^{1+\sigma})$. Both $y_Q(\tau)$ and $Y(\lambda(\tau))$ lie in $\bar y_Q+\mathcal S_a$. The components of $\nabla(U_Q-\lambda\psi)$ along $\mathcal S_a$ vanish at $(Y(\lambda),\lambda)$, they are $O(\tau^{1+\sigma})$ at $(y_Q(\tau),\lambda(\tau))$, and their derivative at $(\bar y_Q,0)$ is the invertible $H_a$. This gives the first estimate. For the second, let $\delta=y_Q(\tau)-Y(\lambda(\tau))\in\mathcal S_a$. Along $\mathcal S_a$ the components of $\nabla U_Q(Y(\lambda))$ and $\lambda g$ agree. Hence $U_Q(y_Q)-U_Q(Y)=\lambda(\tau)\,g\cdot\delta+O(|\delta|^2)=O(\tau^{5-\sigma})$, since $2+2\sigma>5-\sigma$.
\end{proof}

In both propositions below we use $\lambda(\tau)=\bar\lambda\tau^{4-2\sigma}(1+o(1))$ with $\bar\lambda=m_j\hat w_j/(2M)\ne0$. For a convergent power series $\sum_hc_h\lambda^h$, the first $h$ with $c_h\ne0$ contributes $c_h\bar\lambda^h$ at the exponent $h(4-2\sigma)$, and all other contributions have larger exponents.

\begin{proposition}\label{prop:D14N}
Let $\sigma\ne(2n+1)/(n+1)$ for all integers $n\ge1$. At positive masses, the leading cluster configuration at a place of type $\sigma_k(j)(t)$, $\sigma=1+t$, is degenerate.
\end{proposition}

\begin{proof}
Suppose that it is nondegenerate, and normalize the place as in Lemma~\ref{lem:d14follow}. Lemma~\ref{lem:d14energy} holds in this parametrization, since its proof uses only $\nabla_yF=0$ and the constancy of $F$ along the place. Let $S_1(Y(\lambda))=\sum_hs_h\lambda^h$, with $s_0=0$ by the bad shape. Then $S_1(y_Q(\tau))=\sum_hs_h\lambda(\tau)^h+O(\tau^{1+\sigma})$, where $1+\sigma>2\sigma-2$. Since $B_0=m_jwS_1$ and $w(y(\tau))\to(-\hat w_j)^{-1/2}\ne0$, Lemma~\ref{lem:d14energy}(b) gives $[S_1(y_Q(\tau))]_r=0$ for $r<2\sigma-2$ and $[B_0]_{2\sigma-2}=m_j(-\hat w_j)^{-1/2}[S_1]_{2\sigma-2}$. So the first $h$ with $s_h\ne0$, if there is one, has $h(4-2\sigma)\ge2\sigma-2$, with equality only for $\sigma=(2h+1)/(h+1)$. Hence $h(4-2\sigma)>2\sigma-2$ and $[S_1(y_Q(\tau))]_{2\sigma-2}=0$. By Lemma~\ref{lem:d14energy}(c), $B_1(\bar y)=\frac12m_j(-\hat w_j)^{-1/2}\hat z_j\sum_{i\in Q'}m_ix_i^3=0$. Thus $\sum_{i\in Q'}m_ix_i^3=0$. With $\sum_{i\in Q'}m_ix_i=\sum_{i\in Q'}m_ix_i^{-2}=0$, this contradicts Proposition~\ref{prop:P11}.
\end{proof}

\begin{proposition}\label{prop:resonant}
Let $\sigma=(2n+1)/(n+1)$ for an integer $n\ge1$. If the leading cluster configuration at a place of type $\sigma_k(j)(t)$, $\sigma=1+t$, is nondegenerate, then $\psi_1=\psi_2=0$ for each $a\in Q'$.
\end{proposition}

\begin{proof}
We normalize the place for $a$ as in Lemma~\ref{lem:d14follow}. By Lemmas~\ref{lem:d14energy}(a) and~\ref{lem:d14follow}, $\sum_h\alpha_h\lambda(\tau)^h=U_Q(Y(\lambda(\tau)))=O(\tau^{5-\sigma})$. So $\alpha_h=0$ whenever $h(4-2\sigma)<5-\sigma$. As $\sigma\ge\frac32>\frac75$, this holds for $h=2$ and $h=3$. Then Lemma~\ref{lem:d14slice}(b) gives $\psi_1=2\alpha_2=0$ and $\psi_2=\frac32\alpha_3=0$.
\end{proof}

\begin{proof}[Proof of Theorem~\textup{\ref{thm:D14}}]
Let a singular place $P$ have type $\sigma_k(j)(t)$ with $0<t<1$, and let $\sigma=1+t$. By Lemmas~\ref{lem:d14blead} and~\ref{lem:d14energy}(b), the leading cluster configuration is a $\mathrm{K14}(m_Q)$ configuration that satisfies $\mathrm{X3}_k$ and has a bad shape. By Theorem~\ref{thm:Xhat}, $m_Q=c(\sqrt3,\sqrt3,\sqrt3,1)$ with $c>0$, and after a relabelling of $a,b,c$ the leading cluster configuration is $(sz_{\mathcal E},s'w_{\mathcal E},\zeta p_{\mathcal E})$ with $\zeta^2ss'=1$. By Lemmas~\ref{lem:Pdata} and~\ref{lem:d14slice}(c) it is nondegenerate. If $\sigma$ is not of the form $(2n+1)/(n+1)$, Proposition~\ref{prop:D14N} gives a contradiction. Otherwise we take the slice of the outer body that corresponds to $a$ in $\mathcal E$. By Lemmas~\ref{lem:Pdata} and~\ref{lem:d14slice}(c), $\psi_2=c^{-1}s(s/\zeta)^2\cdot40/\sqrt3\ne0$, which contradicts Proposition~\ref{prop:resonant}.
\end{proof}

%% file: sec/s5_d9.tex

\subsection{Diagram 9}\label{ss:d9}

We now use the expansion of \S\ref{ss:expansion} to exclude diagram~$9$ at all positive masses.

\begin{theorem}\label{thm:D9}
Let $m\in(\R_{>0})^5$, and let $C\subseteq Y_m$ be an irreducible curve on which some $r_{kl}^2$ is not constant and $U\not\equiv0$. No singular place of $C$ has diagram~$9$. Hence the pieces $\pc{d9}(T;ax,by)$ may be removed from $\mathrm{pieces}(S(m))$ in Theorem~\ref{thm:types}, for all $m$, $T$, $ax$ and $by$.
\end{theorem}

Let $P$ be a singular place of diagram~$9$. After a relabelling of the bodies and, if necessary, an exchange of the colours, the triangle is $T=\{1,2,3\}$, the $w$-edges are $14$ and $25$, and $x(P)=\pc n(\{1,2,3\};14,25)$ (Table~\ref{tab:types}). By Appendix~\ref{app:types}, $(\zeta_{kl},\omega_{kl})$ is $(1,1)$ on $12$, $13$, $23$ and $(2,-2)$ on $14$, $25$. On the other five pairs $\zeta_{kl}=1$, by $z_{45}=z_{41}+z_{12}+z_{25}$, $z_{15}=z_{12}+z_{25}$ and the like, and $\omega_{kl}=2x_{kl}-1$, which is $-1$ on $45$ and $-2$ on $15$, $24$, $34$, $35$.

For $h\in\C^5$ let $h^*=\frac{m_4h_4+m_5h_5}{m_4+m_5}-\frac{m_1h_1+m_2h_2+m_3h_3}{M_T}$. We use the positions relative to body~$3$, $u_i=z_i-z_3$ and $v_i=w_i-w_3$, in the following combinations: the \emph{coarse} coordinates $u_1$, $u_2$, $v_1$, $v_2$, $\Omega=w^*$, and the \emph{fine} coordinates $u_{41}$, $u_{52}$, $v_{45}$, where $u_{kl}=u_k-u_l$ and $v_{kl}=v_k-v_l$. They form an invertible linear function of $(u_1,u_2,u_4,u_5)$ and one of $(v_1,v_2,v_4,v_5)$. Their exponents are $1$, $1$, $1$, $1$, $-2$ and $2$, $2$, $-1$. For $\Omega$ this follows from $w_{14}=-\Omega-\frac{m_5}{m_4+m_5}v_{45}+\frac{m_2v_{12}+m_3v_1}{M_T}$, since $w_{14}$ has exponent $-2$ and the last two terms have exponents $-1$ and $1$. The coordinates of \S\ref{ss:expansion} do not serve here. For each body $o$ one of the pairs $14$, $25$ avoids $o$, the positions of its two bodies relative to $o$ have exponent $1$ and the same leading coefficient, and so the corresponding $p_{kl}$ is not holomorphic in $y$ near $\bar y$. We therefore state Lemma~\ref{lem:identities} for linear coordinates of the kind above.

\begin{lemma}\label{lem:adapted}
Let $q_1,\dots,q_8$ be coordinates such that $(q_1,\dots,q_4)$ is an invertible linear function of $(u_1,\dots,u_4)$ and $(q_5,\dots,q_8)$ one of $(v_1,\dots,v_4)$, and suppose that no $q_i$ vanishes identically on $C$. Then, after the rotation of Lemma~\ref{lem:place}, $q_i=\tau^{2e(q_i)}y_{q_i}$ with $y_{q_i}(0)\ne0$. Lemma~\ref{lem:identities} holds in the coordinates $q$, and the reparametrizations and rotations of \S\ref{ss:expansion} change neither the leading point nor the identities of that lemma.
\end{lemma}

\begin{proof}
Since each $q_i$ is of $z$-type or of $w$-type, Lemma~\ref{lem:place} gives the first statement and a rotation by $\mu(\tau)$ multiplies $y_{q_i}$ by $\mu(\tau)^{\pm1}$. The proof of Lemma~\ref{lem:identities} uses three facts: $\nabla V=0$ at the points of $X_m$ in the coordinates used, $V$ is constant on $C$, and $F(\lambda\tau,\lambda^{-2e}y)=F(\tau,y)$. The last two do not depend on the coordinates. The first holds in the coordinates $q$, because $\nabla_{(u,v)}V=0$ on $X_m$ (\S\ref{ss:expansion}) and $q$ is an invertible linear function of $(u,v)$.
\end{proof}

We write the coefficients $V_n$ of $F(\tau,y)=V(\tau^{2e}y)$ as functions of the coordinates, with $u_1$ standing for $y_{u_1}$ and so on, and put $u_{12}=u_1-u_2$, $v_{12}=v_1-v_2$, $\nu=M_T(m_4+m_5)/M$, $\nu_{45}=m_4m_5/(m_4+m_5)$,
\begin{gather*}
U_T=m_1m_2(u_{12}v_{12})^{-1/2}+m_1m_3(u_1v_1)^{-1/2}+m_2m_3(u_2v_2)^{-1/2},\\
L(u)=\frac{m_4u_1+m_5u_2}{m_4+m_5}-\frac{m_1u_1+m_2u_2}{M_T}.
\end{gather*}

\begin{lemma}\label{lem:d9exp}
At a place of type $\pc n(\{1,2,3\};14,25)$, the function $F(\tau,y)=\sum_{n\ge0}\tau^{n-2}V_n(y)$ satisfies the hypothesis of Lemma~\ref{lem:identities} in the coordinates above, with $V_1=0$ and
\begin{align*}
V_0&=U_T+\tfrac\nu2\,\Omega\,L(u),\\
V_2&=m_1m_4(u_{41}\Omega)^{-1/2}+m_2m_5(u_{52}\Omega)^{-1/2}+m_4m_5(u_{12}v_{45})^{-1/2}\\
&\qquad+\tfrac{\nu_{45}}2\,u_{12}v_{45}+\tfrac\nu2\,\Omega\,\frac{m_4u_{41}+m_5u_{52}}{m_4+m_5},\\
V_3&=m_1m_5(-u_{12}\Omega)^{-1/2}+m_2m_4(u_{12}\Omega)^{-1/2}+m_3m_4(u_1\Omega)^{-1/2}+m_3m_5(u_2\Omega)^{-1/2},
\end{align*}
where the square roots are the branches of the $p_{kl}$ at $P$. The coefficients $V_0$ and $V_3$ do not depend on the fine coordinates.
\end{lemma}

\begin{proof}
The coordinates have the $\tau$-exponents $2,2,2,2,-4$ and $4,4,-2$, all even. So each $z_{kl}$ is $\tau^{2\zeta_{kl}}(\bar z_{kl}+\tau^2c_{kl})$. Here $\bar z_{kl}$ is a linear form in the coordinates whose value at $\bar y$ is the leading coefficient $\hat z_{kl}\ne0$. The function $c_{kl}$ is a polynomial in $\tau^2$ and the coordinates. Likewise $w_{kl}=\tau^{2\omega_{kl}}(\bar w_{kl}+\tau^2c'_{kl})$. For instance $z_{15}=\tau^2(u_{12}-\tau^2u_{52})$ and $w_{15}=\tau^{-4}\big(-\Omega+\tau^2\frac{m_4v_{45}}{m_4+m_5}+\tau^6\frac{m_2v_{12}+m_3v_1}{M_T}\big)$. Hence
\[
p_{kl}=\tau^{-2x_{kl}}(\bar z_{kl}\bar w_{kl})^{-1/2}\big(1+\tau^2c_{kl}/\bar z_{kl}\big)^{-1/2}\big(1+\tau^2c'_{kl}/\bar w_{kl}\big)^{-1/2}
\]
is $\tau^{-2x_{kl}}$ times a convergent series in $\tau^2$ whose coefficients are holomorphic near $\bar y$. The exponent $-2x_{kl}$ is $-2$ on the pairs of $T$, $0$ on $14$, $25$, $45$ and $1$ on $15$, $24$, $34$, $35$. On the pairs of $T$ the series is $1$. The products $\bar z_{kl}\bar w_{kl}$ are $u_{12}v_{12}$, $u_1v_1$, $u_2v_2$ on $12$, $13$, $23$, then $u_{41}\Omega$, $u_{52}\Omega$, $u_{12}v_{45}$ on $14$, $25$, $45$, and $-u_{12}\Omega$, $u_{12}\Omega$, $u_1\Omega$, $u_2\Omega$ on $15$, $24$, $34$, $35$. For the moment of inertia we use $I=I_T+\nu z^*w^*+\nu_{45}z_{45}w_{45}$ with $I_T=M_T^{-1}\sum_{k<l\in T}m_km_lz_{kl}w_{kl}$. Here $I_T$ contributes only at $\tau^4$; $z^*=L(u)+(m_4u_{41}+m_5u_{52})/(m_4+m_5)$ and $w^*=\Omega$, so $\nu z^*w^*$ contributes at $\tau^{-2}$ and $\tau^0$; and $z_{45}w_{45}=\big(u_{12}+\tau^2(u_{41}-u_{52})\big)v_{45}$ contributes at $\tau^0$ and $\tau^2$. Collecting terms gives the lemma.
\end{proof}

\begin{lemma}\label{lem:d9lead}
At a place of type $\pc n(\{1,2,3\};14,25)$, let $\bar y_c$ be the coarse part of the leading point, and $\nabla_f$ the gradient in the fine coordinates.
\begin{enumerate}[label=\textup{(\alph*)}]
\item $\nabla V_0(\bar y_c)=0$ and $\nabla_fV_2(\bar y)=0$.
\item $(m_4+m_5)\hat z_{3l}=\rho_{3l}\hat z_{12}$ for $l=1,2$, where
\[
\rho_{31}=\frac{m_2m_4-m_1m_5-m_3m_5}{m_3},\qquad\rho_{32}=\frac{m_2m_4-m_1m_5+m_3m_4}{m_3}
\]
are nonzero real numbers with $\rho_{32}-\rho_{31}=m_4+m_5$.
\item $\hat Z_{12}/m_3=\hat Z_{23}/m_1=\hat Z_{31}/m_2=:\kappa$, where $\hat Z_{kl}=\hat p_{kl}^3\hat z_{kl}$, and $\kappa\ne0$.
\item $U_T(\bar y_c)=m_1m_2\hat p_{12}+m_1m_3\hat p_{13}+m_2m_3\hat p_{23}=0$.
\end{enumerate}
\end{lemma}

\begin{proof}
(a) By Lemmas~\ref{lem:adapted} and~\ref{lem:d9exp}, $\nabla_yF(\tau,y(\tau))=0$. Its coarse components, multiplied by $\tau^2$, are $\nabla V_0(y_c(\tau))+O(\tau^2)$. Its fine components are $\nabla_fV_2(y(\tau))+O(\tau^2)$.
(b) The $\Omega$-component of $\nabla V_0$ is $\frac\nu2L(u)$. So $\bar u=(\bar u_1,\bar u_2)=(-\hat z_{31},-\hat z_{32})$ satisfies $L(\bar u)=0$ and $\bar u_1-\bar u_2=\hat z_{12}$. These two linear conditions determine $\bar u$, since the coefficients of $L$ have the sum $m_3/M_T\ne0$. The vector $-\hat z_{12}(\rho_{31},\rho_{32})/(m_4+m_5)$ satisfies both. The $\hat z_{3l}$ are nonzero, as $\zeta_{3l}=1$. (These are the positions of \cite[\S8.9.2]{AK12}.)
(c) Since $p_{kl}^2z_{kl}w_{kl}=1$, $\partial p_{kl}/\partial w_{kl}=-\frac12p_{kl}^3z_{kl}$. As $w_{12}=v_1-v_2$, $w_{13}=v_1$ and $w_{23}=v_2$, the $v$-components of $\nabla V_0(\bar y_c)=0$ are $-\frac12m_1(m_2\hat Z_{12}+m_3\hat Z_{13})=0$ and $-\frac12m_2(m_1\hat Z_{21}+m_3\hat Z_{23})=0$.
(d) By (c) and $\hat p_{kl}=\hat p_{kl}^3\hat z_{kl}\hat w_{kl}=\hat Z_{kl}\hat w_{kl}$, we get $U_T(\bar y_c)=\kappa m_1m_2m_3(\hat w_{12}+\hat w_{23}+\hat w_{31})=0$.
\end{proof}

\begin{lemma}\label{lem:d9kernel}
Let $H$ be the Hessian of $V_0$ at $\bar y_c$ in the coordinates $(u_1,u_2,v_1,v_2,\Omega)$, and let $\bar u=(\bar u_1,\bar u_2)$ and $\bar v=(\bar v_1,\bar v_2)$. Then
\[
\ker H=\big\{\big(a\bar u,\,b\bar v,\,-\tfrac{3a+b}2\,\bar\Omega\big):a,b\in\C\big\}.
\]
\end{lemma}

\begin{proof}
Since $U_T$ is homogeneous of degree $-\frac12$ in $u$ and in $v$, $u\cdot\nabla_uU_T=v\cdot\nabla_vU_T=-\frac12U_T$. By Lemma~\ref{lem:d9lead}(a), $\nabla_vU_T=0$ and $\nabla_uU_T=-\frac\nu2\bar\Omega\nabla L$ at $\bar y_c$. The term $\frac\nu2\Omega L(u)$ of $V_0$ is linear in $u$ and does not involve $v$. Differentiating the two Euler identities, we obtain at $\bar y_c$
\[
H_{vu}\bar u=-\tfrac12\nabla_vU_T=0,\qquad H_{vv}\bar v=-\tfrac32\nabla_vU_T=0,\qquad H_{uu}\bar u=-\tfrac32\nabla_uU_T,\qquad H_{uv}\bar v=-\tfrac12\nabla_uU_T,
\]
where $H_{uu},\dots$ are the blocks of $H$. Moreover the $\Omega$-row of $H$ is $(\frac\nu2\nabla L,0,0)$.

Let $\delta=(\delta u,\delta v,\delta\Omega)\in\ker H$. The $\Omega$-row gives $L(\delta u)=0$. So $\delta u=a\bar u$ for some $a\in\C$, as $\ker L=\C\bar u$. The $v$-rows then give $H_{vv}\delta v=0$. We show that $H_{vv}$ has rank $1$. As $\partial^2p_{kl}/\partial w_{kl}^2=\frac34p_{kl}/w_{kl}^2$ and, by Lemma~\ref{lem:d9lead}(c), $m_km_l\hat p_{kl}/\hat w_{kl}=m_km_l\hat Z_{kl}=\kappa m_1m_2m_3$ for $kl=12,23,31$, the quadratic form of $H_{vv}$ is
\[
\tfrac34\kappa m_1m_2m_3\Big(\frac{s_{12}^2}{\hat w_{12}}+\frac{s_{23}^2}{\hat w_{23}}+\frac{s_{31}^2}{\hat w_{31}}\Big),\qquad s_{12}=\delta v_1-\delta v_2,\quad s_{23}=\delta v_2,\quad s_{31}=-\delta v_1,
\]
on the plane $s_{12}+s_{23}+s_{31}=0$. Its polar form pairs the point $(\hat w_{12},\hat w_{23},\hat w_{31})$, given by $\delta v=\bar v$, with $s$ to $\frac34\kappa m_1m_2m_3(s_{12}+s_{23}+s_{31})$, which vanishes on the plane. At $s=(1,-1,0)$ the form is $\frac34\kappa m_1m_2m_3(\hat w_{12}^{-1}+\hat w_{23}^{-1})=-\frac34\kappa m_1m_2m_3\hat w_{31}/(\hat w_{12}\hat w_{23})\ne0$. So $\ker H_{vv}=\C\bar v$ and $\delta v=b\bar v$ for some $b\in\C$. Finally the $u$-rows give $aH_{uu}\bar u+bH_{uv}\bar v+\frac\nu2\delta\Omega\,\nabla L=\frac\nu2\big(\frac{3a+b}2\bar\Omega+\delta\Omega\big)\nabla L=0$ with $\nabla L\ne0$. Conversely, one checks by the same computation that each vector of the stated form lies in $\ker H$. The vectors with $(a,b)=(1,-1)$ and $(-2,-2)$ are the tangent vectors of the rotation and of the scaling $(u,v,\Omega)\mapsto(\lambda^{-2}u,\lambda^{-2}v,\lambda^4\Omega)$.
\end{proof}

The next statement is the analogue of Theorem~\ref{thm:order3}: the leading potential between $T$ and $\{4,5\}$ vanishes.

\begin{proposition}\label{prop:d9cross}
At a place of type $\pc n(\{1,2,3\};14,25)$, $V_3(\bar y)=m_1m_5\hat p_{15}+m_2m_4\hat p_{24}+m_3m_4\hat p_{34}+m_3m_5\hat p_{35}=0$.
\end{proposition}

\begin{proof}
By Lemma~\ref{lem:adapted} we may normalize the place as in the proof of Theorem~\ref{thm:order5}, so that $y_{u_1}(\tau)\equiv\bar u_1$ and $y_\Omega(\tau)\equiv\bar\Omega$: the conditions $\mu\varphi^{-2}y_{u_1}=\bar u_1$ and $\mu^{-1}\varphi^4y_\Omega=\bar\Omega$ determine $\varphi^2$ and then $\mu$. We write $y(\tau)=\bar y+(\delta_c(\tau),\delta_f(\tau))$ with the coarse part $\delta_c$ and the fine part $\delta_f$. Then $\delta_c$ lies in $\mathcal S=\{\xi\in\C^5:\xi_{u_1}=\xi_\Omega=0\}$. By Lemma~\ref{lem:d9kernel}, $\mathcal S\cap\ker H=0$, since $a\bar u_1=0$ and $(3a+b)\bar\Omega=0$ force $a=b=0$. As $\dim\mathcal S+\dim\ker H=5$, the restriction of the bilinear form $H$ to $\mathcal S$ is nondegenerate: a vector of $\mathcal S$ that is $H$-orthogonal to $\mathcal S$ is $H$-orthogonal to $\mathcal S+\ker H=\C^5$. By Lemma~\ref{lem:d9lead}(a), the coarse components of $\tau^2\nabla_yF(\tau,y(\tau))=0$ read $H\delta_c+O(|\delta_c|^2)+O(\tau^2)=0$. Their components along $\mathcal S$ give $\val\delta_c\ge\min(2\val\delta_c,2)$. Hence $\val\delta_c\ge2$. The fine components of $\nabla_yF(\tau,y(\tau))=0$ read $\nabla_fV_2(\bar y+\delta)+O(\tau^2)=0$. By Lemma~\ref{lem:d9exp}, the matrix of second derivatives of $V_2$ in the fine coordinates is diagonal at $\bar y$, with the nonzero entries $\frac34m_1m_4\hat p_{14}/\bar u_{41}^2$, $\frac34m_2m_5\hat p_{25}/\bar u_{52}^2$ and $\frac34m_4m_5\hat p_{45}/\bar v_{45}^2$. With $\nabla_fV_2(\bar y)=0$ and $\val\delta_c\ge2$ this gives $\val\delta_f\ge\min(2\val\delta_f,2)$, hence $\val\delta_f\ge2$. Therefore $V_2(y(\tau))=V_2(\bar y)+O(\tau^2)$. Now \eqref{eq:qid} with $q=3$ reads $[V_1]_2+2[V_2]_1+3V_3(\bar y)=0$, as in the proof of Theorem~\ref{thm:order3}. Here $V_1=0$ and $[V_2(y(\tau))]_1=0$. Thus $V_3(\bar y)=0$.
\end{proof}

\begin{proof}[Proof of Theorem~\textup{\ref{thm:D9}}]
Suppose that $P$ has diagram~$9$, with the labelling above. On the pairs $kl=15,24,34,35$, Lemma~\ref{lem:d9exp} gives $\hat p_{kl}=(\bar z_{kl}\bar w_{kl})^{-1/2}$ at $\bar y$, with $\bar w_{kl}=-\bar\Omega$. By Lemma~\ref{lem:d9lead}(b), $(m_4+m_5)\bar z_{kl}=\tilde\rho_{kl}\hat z_{12}$ at $\bar y$, where $\tilde\rho_{15}=m_4+m_5$, $\tilde\rho_{24}=-(m_4+m_5)$, $\tilde\rho_{34}=\rho_{31}$ and $\tilde\rho_{35}=\rho_{32}$. Let $K$ be a square root of $-(m_4+m_5)/(\hat z_{12}\bar\Omega)$. Then $\hat p_{kl}=\pm K\tilde\rho_{kl}^{-1/2}$ if $\tilde\rho_{kl}>0$, and $\hat p_{kl}=\pm iK|\tilde\rho_{kl}|^{-1/2}$ if $\tilde\rho_{kl}<0$. So Proposition~\ref{prop:d9cross} says that two sums of real numbers $\pm m_km_l|\tilde\rho_{kl}|^{-1/2}$ vanish, one over the pairs with $\tilde\rho_{kl}>0$ and one over the pairs with $\tilde\rho_{kl}<0$. Each sum therefore has at least two terms. The pair $15$ is in the first sum, $24$ is in the second, and $\rho_{31}<\rho_{32}$. Hence $\rho_{31}<0<\rho_{32}$, the first sum is over $15$, $35$ and the second over $24$, $34$, and they vanish only if
\begin{equation}\label{eq:C12}
m_1^2\rho_{32}=m_3^2(m_4+m_5),\qquad m_2^2\rho_{31}=-m_3^2(m_4+m_5).
\end{equation}
Let $\alpha=\hat p_{12}/\hat p_{23}$ and $\beta=\hat p_{12}/\hat p_{31}$. By Lemma~\ref{lem:d9lead}(c), $\hat p_{kl}^3\hat z_{kl}$ is $\kappa m_3$, $\kappa m_1$, $\kappa m_2$ for $kl=12,23,31$. By Lemma~\ref{lem:d9lead}(b), $(m_4+m_5)\hat z_{23}=-\rho_{32}\hat z_{12}$ and $(m_4+m_5)\hat z_{31}=\rho_{31}\hat z_{12}$. With Lemma~\ref{lem:d9lead}(d) these relations give
\begin{gather*}
m_1(m_4+m_5)\alpha^3=-m_3\rho_{32}=m_1m_5-m_2m_4-m_3m_4,\\
m_2(m_4+m_5)\beta^3=m_3\rho_{31}=m_2m_4-m_1m_5-m_3m_5,\\
m_1m_3\alpha+m_2m_3\beta+m_1m_2\alpha\beta=\alpha\beta\,U_T(\bar y_c)/\hat p_{12}=0.
\end{gather*}
By \eqref{eq:C12}, $\alpha^3=-(m_3/m_1)^3$ and $\beta^3=-(m_3/m_2)^3$. So $\alpha=-\eta m_3/m_1$ and $\beta=-\eta'm_3/m_2$ with $\eta^3=\eta'^3=1$. The third relation becomes $m_3^2(\eta\eta'-\eta-\eta')=0$, that is, $(\eta-1)(\eta'-1)=1$. This is impossible, because $|\eta-1|$ and $|\eta'-1|$ are $0$ or $\sqrt3$.

The last statement follows, since Theorem~\ref{thm:types} places the type of a singular place in a piece of its own diagram.
\end{proof}

\begin{remark}\label{rem:torque}
At a place of diagram~$1$ or~$4$ with triangle $T$, the torque identity of \cite[\S5.0]{AK12} shows that the leading potential between $T$ and the two other bodies vanishes as well; at positive masses this confines such places to finitely many explicit curves of masses \cite{Tom26}.
\end{remark}

%% file: sec/s5_points.tex
\section{Point certificates for the clusters of zero energy}\label{sec:points}

For a body $j$ write $Q_j=\{1,\dots,5\}\setminus\{j\}$. By Theorem~\ref{thm:Kgating}, the pieces $\pc{d14a}(j)$ and $\pc{d14b}(j,k)$ are attained at $m$ only if $\mathrm{K14}(m_{Q_j})$ has a solution, and $\pc{d15}(j,k)$ only if $\mathrm{K15}(m_{Q_j},k)$ has one. In \S\ref{sec:curves} the separation of \S\ref{sec:lex} succeeds at finitely many mass points, given exactly, only after the pieces of some clusters $Q_j$, usually one, are removed. At these points we show that $\mathrm{K14}(m_{Q_j})$ has no solution. By \S\ref{ss:certs}, the proof needs this only for the pieces $\pc{d14a}(j)$.

All these certificates have the same structure. We reduce the cluster system modulo a prime $\ell$ at an $\ell$-adic point over the given masses. A standard basis over $\mathbb F_\ell$ shows that neither the reduced system nor any of its initial systems on the cells of its tropical prevariety has a zero in the torus. A valuation argument lifts this to characteristic $0$. The computations over $\mathbb F_\ell$ and the enumeration of the cells are exact. Where a point of \S\ref{sec:curves} is a certified point, we identify the two exactly.

\subsection{The cluster systems in the torus}\label{ss:Ksystems}
Let $Q$ be a set of four bodies, $\mu=(\mu_i)_{i\in Q}$ complex masses, $M_Q=\sum_{i\in Q}\mu_i$, and $k\in Q$. In this section we write $h$ for the second index of a pair of bodies, since $\ell$ denotes a prime. Recall from \eqref{eq:K14} and \eqref{eq:K15} that $\mathrm{K14}(\mu)$ asks for $z,w\in\C^Q$ and $p_{ih}=p_{hi}$ with $p_{ih}^2z_{ih}w_{ih}=1$ and
\[
F^z_i=\sum_{h\in Q\setminus i}\mu_h\,z_{ih}\,p_{ih}^3=0,\qquad F^w_i=\sum_{h\in Q\setminus i}\mu_h\,w_{ih}\,p_{ih}^3=0\qquad(i\in Q),
\]
and that $\mathrm{K15}(\mu,k)$ asks for the same with $F^w_i=c+d\,[i=k]$, $d\ne0$. Whether they have solutions depends only on the ratios of the masses and on $k$, and is equivariant under relabelling (\S\ref{ss:Kgating}).

We number the bodies of $Q$ increasingly as $1,\dots,4$. The unknowns $\xi$ are $z_{ih}$, $w_{ih}$ and $p_{ih}$ for the five pairs $\{i,h\}\ne\{1,2\}$, and for $\mathrm{K15}$ also $c$ and $d$, which gives $N=15$, respectively $N=17$, unknowns. The generators are:
\begin{enumerate}[label=(\roman*)]
\item the seven circuit relations of the complete graph on $Q$ (four triangles and three four-cycles), for $z$ and for $w$, with $z_{12}=w_{12}=1$;
\item $p_{ih}^2z_{ih}w_{ih}-1$ for the five pairs;
\item the eight rows $F^z_i$ and $F^w_i-c-d\,[i=k]$, with $p_{12}=1$, and $c=d=0$ for $\mathrm{K14}$;
\item with $U=\sum_{i<h}\mu_i\mu_hp_{ih}$, $V_{zz}=\sum_{i<h}\mu_i\mu_hz_{ih}^2p_{ih}^3$ and $V_{ww}=\sum_{i<h}\mu_i\mu_hw_{ih}^2p_{ih}^3$: for $\mathrm{K14}$ the three polynomials $U$, $V_{zz}$, $V_{ww}$; for $\mathrm{K15}$ the five polynomials
\[
cM_Q+d\mu_k,\qquad U,\qquad \textstyle\sum_{h\in Q}\mu_hz_{kh},\qquad V_{zz},\qquad V_{ww}+c\sum_{h\in Q}\mu_hw_{kh}.
\]
\end{enumerate}
This gives $30$ generators for $\mathrm{K14}$ and $32$ for $\mathrm{K15}$. Each generator has the form $F_\nu=\sum_{a\in A_\nu}c_{\nu,a}(\mu)\,\xi^a$, where each $c_{\nu,a}$ is a nonzero element of $\Z[\mu]$. Each $c_{\nu,a}$ is an integer multiple of a monomial in $\mu$, except the coefficient $M_Q$ of $c$ in $cM_Q+d\mu_k$. We call $A_\nu$ the \emph{universal support} of $F_\nu$.

\begin{lemma}\label{lem:torus}
If $\mathrm{K14}(\mu)$ holds, the generators of $\mathrm{K14}$ have a common zero in $(\C^*)^{15}$. If $\mu_k\ne0$, $M_Q\ne0$ and $\mathrm{K15}(\mu,k)$ holds, the generators of $\mathrm{K15}$ with this $k$ have a common zero in $(\C^*)^{17}$.
\end{lemma}

\begin{proof}
Replacing $(z,w,p)$ by $(sz,tw,(st)^{-1/2}p)$ multiplies $F^z$ by $s(st)^{-3/2}$ and $F^w$ by $t(st)^{-3/2}$. Replacing $p$ by $-p$ changes the sign of both. For $\mathrm{K15}$ the factor of $F^w$ is absorbed by $(c,d)$. Since $z_{12}w_{12}=p_{12}^{-2}\ne0$, each solution has a representative with $z_{12}=w_{12}=p_{12}=1$. Conditions (i)--(iii) hold there. The circuit relations give positions $z_i,w_i$. From $p_{ih}^2z_{ih}w_{ih}=1$ and antisymmetry we obtain
\[
\textstyle\sum_i\mu_iw_iF^z_i=\sum_i\mu_iz_iF^w_i=U,\qquad\sum_i\mu_iz_iF^z_i=V_{zz},\qquad\sum_i\mu_iw_iF^w_i=V_{ww},\qquad\sum_i\mu_iF^w_i=0.
\]
For $\mathrm{K14}$ the left sides vanish. For $\mathrm{K15}$ the last identity gives $cM_Q+d\mu_k=0$. Then $0=U=\sum_i\mu_iz_i(c+d[i=k])=d\mu_k(z_k-z_G)$ with $z_G=\sum_i\mu_iz_i/M_Q$. So $\sum_h\mu_hz_{kh}=M_Q(z_k-z_G)=0$. Similarly, $V_{ww}=cM_Qw_G+d\mu_kw_k=-c\sum_h\mu_hw_{kh}$. Finally, $p_{ih}^2z_{ih}w_{ih}=1$ makes $z_{ih}$, $w_{ih}$, $p_{ih}$ nonzero, $d\ne0$, and $c=-d\mu_k/M_Q\ne0$.
\end{proof}

\emph{Cells.} For a finite set $A\subset\Z^N$ and $\omega\in\R^N$ let $\max_\omega(A)$ be the set of $a\in A$ at which $\langle a,\omega\rangle$ is largest, and for $F=\sum_{a\in A}c_a\xi^a$ let $\mathrm{in}_\omega F=\sum_{a\in\max_\omega(A)}c_a\xi^a$. The \emph{prevariety} of the universal supports is the set of $\omega$ with $|\max_\omega(A_\nu)|\ge2$ for every $\nu$. Its \emph{cells} are the classes of its nonzero points with the same tuple of max-sets. On a cell $\sigma$ the initial forms $\mathrm{in}_\sigma F_\nu$ are well defined. Since the prevariety depends only on the supports, we compute it once for each of the five systems ($\mathrm{K14}$, and $\mathrm{K15}$ with $k=1,\dots,4$). We enumerate the cells by a search in which each linear-programming decision is proved in integer arithmetic, by an integer point or by nonnegative integer Gordan multipliers (the numbers of decisions are in \S\ref{ss:checks}). We compare the result with the prevariety computed by gfan~0.7: the two agree, and the lineality space is $0$. The prevariety of $\mathrm{K14}$ has $28$ cells, spanned by $16$ rays. That of each system $\mathrm{K15}$ has $43$ cells, spanned by $18$ rays.

\subsection{Algebraic mass points}\label{ss:KA}
Let $f\in\Z[t]$ be irreducible over $\Q$, and let $g=(g_i)_{i\in Q}$ with $g_i\in\Q[t]$. For a root $\theta$ of $f$ put $\mu(\theta)=(g_i(\theta))_{i\in Q}$. For a prime $\ell$ and $\bar\theta\in\mathbb F_\ell$ put $\bar\mu_i=g_i(\bar\theta)$ and $\bar F_\nu=\sum_ac_{\nu,a}(\bar\mu)\xi^a\in\mathbb F_\ell[\xi]$.

\begin{theorem}\label{thm:KA}
Let $\ell$ be a prime and $\bar\theta\in\mathbb F_\ell$ such that:
\begin{enumerate}[label=(P\arabic*)]
\item $\ell$ divides neither the leading coefficient of $f$ nor a denominator of a coefficient of a $g_i$; $\bar f=f\bmod\ell$ is squarefree of degree $\deg f$, and $\bar f(\bar\theta)=0$;
\item $\bar\mu_i\ne0$ for $i\in Q$, $\sum_i\bar\mu_i\ne0$, and $c_{\nu,a}(\bar\mu)\ne0$ for every generator and every $a\in A_\nu$;
\end{enumerate}
\begin{enumerate}[label=(\Alph*)]
\item the ideal of $\mathbb F_\ell[\xi_0,\xi]$ generated by $\bar F_1,\dots,\bar F_N$ and $\xi_0\xi_1\cdots\xi_N-1$ is the unit ideal;
\item for every cell $\sigma$, the ideal generated by $\mathrm{in}_\sigma\bar F_1,\dots,\mathrm{in}_\sigma\bar F_N$ and $\xi_0\xi_1\cdots\xi_N-1$ is the unit ideal.
\end{enumerate}
Then for every root $\theta\in\C$ of $f$ the masses $g_i(\theta)$ and their sum are nonzero, and $\mathrm{K14}(\mu(\theta))$, respectively $\mathrm{K15}(\mu(\theta),k)$, has no solution.
\end{theorem}

\begin{proof}
\emph{Step 1: one embedding suffices.} Let $K=\Q[t]/(f)$, a field, with $\vartheta$ the class of $t$, and let $I\subset K[\xi_0,\xi]$ be the ideal generated by the polynomials $\sum_ac_{\nu,a}(g(\vartheta))\xi^a$ and $\xi_0\xi_1\cdots\xi_N-1$. For an embedding $\iota$ of $K$ into an algebraically closed field $\Omega$, the system $\iota F$ has a zero in $(\Omega^*)^N$ if and only if $1\notin I$: by the Nullstellensatz it has one if and only if $1\notin\iota(I)\,\Omega[\xi_0,\xi]$, and a representation $1=\sum h_\nu\,\iota(P_\nu)$ with prescribed degrees is a linear system over $\iota(K)$, which is solvable over $\iota(K)$ if it is solvable over $\Omega$. The complex embeddings of $K$ send $\vartheta$ to the roots of $f$. It therefore suffices to find one embedding of $K$ under which the system has no zero in the torus.

\emph{Step 2: the Hensel embedding.} Since $\bar f$ is squarefree, $\bar\theta$ is a simple root. Hensel's lemma gives a unique $\theta_\ell\in\Z_\ell$ with $f(\theta_\ell)=0$ and $\theta_\ell\equiv\bar\theta\pmod\ell$. As $f$ is irreducible, $\vartheta\mapsto\theta_\ell$ defines an embedding $\tau\colon K\to\Q_\ell$. By (P1) each $g_i(\theta_\ell)$ lies in $\Z_\ell$ and reduces to $\bar\mu_i$. Hence each coefficient $c_{\nu,a}(g(\theta_\ell))$ lies in $\Z_\ell$ and reduces to $c_{\nu,a}(\bar\mu)$, which is nonzero by (P2). All coefficients are therefore $\ell$-adic units. In particular the masses and their sum are nonzero in $K$, hence at every root of $f$.

\emph{Step 3: valuations.} Suppose that $\tau F$ has a zero $\xi\in(\bar\Q_\ell^{\,*})^N$. Its coordinates lie in a finite extension $E$ of $\Q_\ell$, with valuation $v$, $v(\ell)=1$, a uniformizer $\pi$ of valuation $1/e$, and residue field $k_E$. Put $\omega=-v(\xi)$ and $\xi_\alpha=\pi^{ev(\xi_\alpha)}y_\alpha$ with units $y_\alpha$. As the coefficients are units, the terms of $\tau F_\nu(\xi)$ of least valuation are those with exponent in $\max_\omega(A_\nu)$. Dividing by the least power of $\pi$ and reducing, we obtain $\mathrm{in}_\omega\bar F_\nu(\bar y)=0$ for every $\nu$, with $\bar y\in(k_E^*)^N$. If $\omega=0$, then $\bar y$, with $\xi_0=(\bar y_1\cdots\bar y_N)^{-1}$, is a zero of the ideal of (A), which is impossible. If $\omega\ne0$, then each $\max_\omega(A_\nu)$ has at least two elements, since a single term with a nonzero coefficient does not vanish at a torus point. Hence $\omega$ is a nonzero point of the prevariety. It lies in a cell $\sigma$, where $\mathrm{in}_\omega\bar F_\nu=\mathrm{in}_\sigma\bar F_\nu$. So $\bar y$ gives a zero of the ideal of (B), which is impossible.

Thus $\tau F$ has no zero in the torus. By Step 1, no complex root $\theta$ of $f$ gives one. By Lemma~\ref{lem:torus} neither $\mathrm{K14}(\mu(\theta))$ nor $\mathrm{K15}(\mu(\theta),k)$ holds.
\end{proof}

\begin{remark}\label{rem:KAforms}
(a) One prime suffices, although the certificates of \S\ref{ss:pointsets} that use Theorem~\ref{thm:KA} use two. (b) The proof uses the masses only through Step~2. If they are given in homogeneous coordinates, $\mu_i=a_i(\theta)/b(\theta)$ with $a_i,b\in\Z[t]$ and $b(\theta)\ne0$, the theorem holds with (P1) replaced by: $\ell$ does not divide the leading coefficient of $f$, $\bar f$ is squarefree of degree $\deg f$ with $\bar f(\bar\theta)=0$, $b(\bar\theta)\ne0$, and $\bar\mu_i=a_i(\bar\theta)\,b(\bar\theta)^{-1}$. Indeed $b(\theta_\ell)$ is then a unit, and $a_i(\theta_\ell)/b(\theta_\ell)\in\Z_\ell$ reduces to $\bar\mu_i$. (c) Since $\mathrm{K14}$ and $\mathrm{K15}$ are unchanged under scaling of the masses, the theorem may be applied to any nonzero multiple of the masses, for instance to the numerators over a common denominator.
\end{remark}

The certificate also bounds the masses at which a cluster system can hold along an algebraic curve.

\begin{corollary}\label{cor:GF}
Let $X$ be an irreducible curve over $\Q$ with a rational map $x\mapsto\mu(x)$ to the masses of a cluster, and suppose that the hypotheses of Theorem~\ref{thm:KA} hold, for $\mathrm{K14}$ or for $\mathrm{K15}(\cdot,k)$ and with some prime $\ell$, at $\mu(x_0)$ for a point $x_0\in X$ defined over a number field. Then the set of $x\in X(\C)$ at which the system has a zero in the torus is finite.
\end{corollary}

\begin{proof}
Since the set $Z$ in question is constructible and defined over $\Q$, it is finite or contains a nonempty Zariski open subset of $X$. Step~3 of the proof of Theorem~\ref{thm:KA} uses only that the masses lie in $\mathcal O_E^4$ for a finite extension $E$ of $\Q_\ell$ and reduce to $\bar\mu=\mu(x_0)\bmod\ell$, where (P2), (A) and (B) are conditions on $\bar\mu$. The points of $X$ over finite extensions of $\Q_\ell$ that are $\ell$-adically close to $x_0$ have such masses. They are infinitely many and hence Zariski dense in $X$. None of them lies in $Z$. Since membership in $Z$ is a first-order statement over $\Q$, it transfers between algebraically closed fields of characteristic $0$. Therefore $Z$ is not cofinite and hence finite.
\end{proof}

For masses along a curve we can make the parameter a further unknown. Let $G_1,\dots,G_n$ be nonzero polynomials in $\xi=(\xi_1,\dots,\xi_N)$ and a parameter $s$, with coefficients $e_{\nu,a}(\theta)$, where $e_{\nu,a}\in\Z[t]$ is not divisible by $f$. We now take the prevariety and its cells in the variables $(\xi,s)$.

\begin{theorem}\label{thm:KApar}
Let $\ell$ be a prime and $\bar\theta\in\mathbb F_\ell$ such that $\ell$ does not divide the leading coefficient of $f$, $\bar f$ is squarefree of degree $\deg f$, $\bar f(\bar\theta)=0$, and $e_{\nu,a}(\bar\theta)\ne0$ for all $\nu$ and $a$.
\begin{enumerate}[label=(\alph*)]
\item If $1$ lies in the ideal of $\mathbb F_\ell[\xi_0,\xi,s]$ generated by $\bar G_1,\dots,\bar G_n$ and $\xi_0\xi_1\cdots\xi_Ns-1$, and in the corresponding ideal of the initial systems of every cell, then for every root $\theta$ of $f$ and every $s\in\C^*$ the system $G(\cdot,s)$ has no zero in $(\C^*)^N$.
\item Let $h\in\Z[s]$ with $\bar h=h\bmod\ell$ of degree $\deg h$ and $\bar h(0)\ne0$. If over $\bar{\mathbb F}_\ell$ the full system, together with $\bar h(s)=0$, has no zero with all coordinates nonzero, and the same holds for the initial system of every cell whose relative interior contains a point with $s$-coordinate $0$, then for every root $\theta$ of $f$ and every root $s_0$ of $h$ the system $G(\cdot,s_0)$ has no zero in $(\C^*)^N$.
\end{enumerate}
\end{theorem}

\begin{proof}
(a) is the proof of Theorem~\ref{thm:KA} with $s$ as one more coordinate of the torus. (b) We extend the Hensel embedding of Step~2 to $\Q(\theta,s_0)$. Since $\bar h$ has full degree and $\bar h(0)\ne0$, all roots of $h$ in $\bar\Q_\ell$ are units. In Step~3 the vector $\omega=-v(\xi,s_0)$ therefore has $s$-coordinate $0$. The residue of $s_0$ is a root of $\bar h$. The reduction is a zero of the full system with $\bar h(s)=0$ (if $\omega=0$) or of the initial system of the cell containing $\omega$.
\end{proof}

The hypotheses of (a) fail for $\mathrm{K14}$ and $\mathrm{K15}$ along a family on which these systems have solutions; on such a family, however, Corollary~\ref{cor:GF} shows that the solutions occur at finitely many points if the hypotheses of Theorem~\ref{thm:KA} hold at one of its points. We will use part (b) for $\mathrm{K14}$ at the points of the family $T_9$ at which the label set changes (\S\ref{sec:T9}).

\subsection{The certified points}\label{ss:pointsets}

\begin{definition}\label{def:Kcert}
Let $m\in(\R_{>0})^5$ and $T\subseteq\{1,\dots,5\}$. The point $m$ is \emph{K-certified with transversal $T$} if
\begin{enumerate}[label=(\roman*)]
\item for every $j\in T$, $\mathrm{K14}(m_{Q_j})$ has no solution; and
\item there are a set of labels $S$ with $\mathrm{pieces}(S(m))\subseteq\mathrm{pieces}(S)$ and a D5 hull $V$ valid at $m$ (Definition~\ref{def:validhull}) such that the final check of $(S,T)$ holds: some $l_1,\dots,l_r\in\Z^{10}$ satisfy (a) and (b) of Definition~\ref{def:cert} for the attained faces of $\mathrm{pieces}(S)$, built with $V$, without those of all pieces $\pc{d2}$ and $\pc{d9}$, and of $\pc{d14a}(j)$, $\pc{d14b}(j,k)$ and $\pc{d15}(j,k)$, $j\in T$, $k\in Q_j$.
\end{enumerate}
\end{definition}

Note that this is the notion of \S\ref{sec:lex}, with the set of removed clusters made explicit.

\begin{proposition}\label{prop:Kcert}
A K-certified point is settled.
\end{proposition}

\begin{proof}
Let $C\subseteq Y_m$ be a nontrivial curve. By Corollary~\ref{cor:Kdrop} and (i), no singular place of $C$ has its type in the attained set of a piece $\pc{d14a}(j)$ with $j\in T$. By \S\ref{ss:certs} none has its type in that of a piece $\pc{d2}$, $\pc{d9}$, $\pc{d14b}$ or $\pc{d15}$. The proof of Proposition~\ref{prop:cert}, with these faces removed and Lemma~\ref{lem:lex} applied to the functionals of (ii), shows that there is no such curve.
\end{proof}

\emph{Exact points.} Every point below is given exactly, by an irreducible $f\in\Z[t]$, a rational isolating interval of a root $\theta$ (one sign change and either $f'\ne0$ on the interval in ball arithmetic or one sign variation in Descartes' rule), and masses $g_i(\theta)$ with $g_i\in\Q[t]$. We compute its label set $S(m)$. A label polynomial vanishes if it reduces to $0$ modulo $f$. It does not vanish if a ball value excludes $0$ or its reduction modulo $f$ is nonzero. This gives a set $S$ with $\mathrm{pieces}(S(m))\subseteq\mathrm{pieces}(S)$, as Definition~\ref{def:Kcert} requires.

Table~\ref{tab:points} lists the points, with their point records in the supplementary material. They are the jump points of the curve families of \S\ref{sec:curves} that are taken from this section (Table~\ref{tab:curvecounts}), in other labellings and scalings. Among them is the point $\hthree=\big(u_*^{3/2}(1+u_*)^{-2},u_*,u_*,1,1\big)$, with $u_*$ the root in $(0,1)$ of $u^4-4u^3-6u^2-4u+1$. After a relabelling and a scaling it is a jump point of $T_1$. The point records in one row of the table are conjugate: they belong to roots of one polynomial $f$ with the same $g_i$.

\begin{table}[ht]
\centering\small
\begin{tabular}{@{}llrrll@{}}
\toprule
point & used in & transversal & runs & final check & point records\\
\midrule
$\hthree$ & $T_1$ & $\{1\}$ & $5$ & $\Vtwo^{\mathrm w}$ & $48_0$, $48_1$\\
$(1,1,1,1,1)$ & $T_3$, \S\ref{sec:curves} & $\{1,\dots,5\}$ & $5$ & $\Vtwo^{\mathrm w}$ & $53_0$\\
two points & $T_5$ & $\{3\}$ & $5$ & $\VN$ & $20_0$, $20_1$\\
$(\frac19,\frac19,\frac14,1,1)$ & $T_7$ & $\{1\}$ & $5$ & $\Vtwo^{\mathrm w}$ & $22_0$\\
\midrule
& & & $20$ & $4$ on $\Vtwo^{\mathrm w}$, $2$ on $\VN$ & $6$ point records\\
\bottomrule
\end{tabular}
\caption{The K-certified points, with the point, sorted and scaled to $m_5=1$, where it is short. The last column gives their point records, the names of their files in the supplementary material. Where a row names one point and lists several point records, these are that point in other labellings and scalings. A run of Theorem~\ref{thm:KA} certifies, for one cluster, one of its five systems for one choice of $f$ and the cluster masses, and the point records of one row share their runs. Each row has one run of $\mathrm{K14}$, which is used in the proof, and four runs of $\mathrm{K15}$, which are cross-checks (\S\ref{ss:crosschecks}). At $(1,1,1,1,1)$ the five clusters have the same masses and share these runs. Each point record has one final check. It holds relative to the hull listed, with the functionals stored in the supplementary material.}\label{tab:points}
\end{table}

\begin{theorem}\label{thm:points}
\begin{enumerate}[label=(\alph*)]
\item At the point of every point record of Table~\ref{tab:points} and every $j$ in its transversal, $\mathrm{K14}(m_{Q_j})$ has no solution, and the pieces $\pc{d14a}(j)$ and $\pc{d14b}(j,k)$, $k\in Q_j$, are therefore attained at no such point.
\item The point of every point record of Table~\ref{tab:points} is K-certified with the transversal listed. Hence it is settled.
\end{enumerate}
\end{theorem}

\begin{proof}
(a) For each point and each $j$ in its transversal, the hypotheses of Theorem~\ref{thm:KA} hold for $\mathrm{K14}$. In each run of Table~\ref{tab:points} we obtain the conclusion ``no solution''. A run applies to the point of the table, because its input data are those of the exact point. Theorem~\ref{thm:Kgating} turns (i) of Definition~\ref{def:Kcert} into the statement on the pieces.

(b) Table~\ref{tab:points} gives, for each point record, the hull of Remark~\ref{rem:hulls} relative to which its final check holds. The hull $\VN$ is valid at the two point records at which it is used (\S\ref{ss:d5}), by the check of \S\ref{ss:checks}. Proposition~\ref{prop:Kcert} completes the proof.
\end{proof}

The point $(1,1,1,1,1)$ of Table~\ref{tab:points} has the transversal $\{1,\dots,5\}$. Each cluster has the masses $(1,1,1,1)$ there. Theorem~\ref{thm:points}(a) shows that $\mathrm{K14}(1,1,1,1)$ has no solution, a fact that will be used in \S\ref{sec:curves}. We will apply Theorems~\ref{thm:KA} and~\ref{thm:KApar} at further points of the curve families, among them the jump points of $T_9$ (\S\ref{sec:T9}), and count these points there.

%% file: sec/s6_near.tex
\section{Diagram 2 does not occur}\label{sec:nod2}\label{ss:nod2}

In this section we prove that no singular place has diagram~$2$, at any positive mass (Proposition~\ref{prop:nod2}), so that the pieces $\pc{d2}[T]$ may be left out of every certificate (\S\ref{ss:certs}). A place of diagram~$2$ forces the torque of its triangle to vanish at leading order. For a triangle that is not equilateral, closed forms of the first obstructions leave twelve mass vectors of the triangle, at which an inequality, the T-test of Theorem~\ref{thm:gate}, shows that the torque cannot vanish; equilateral triangles are excluded directly (Lemma~\ref{lem:d2gate}). We prove Lemma~\ref{lem:d2gate} and Theorem~\ref{thm:gate}, which carry this argument, in Appendix~\ref{app:d2}. The proof uses no result of \S\S\ref{sec:strata}--\ref{sec:curves}.

Let $T=\{a<b<c\}$ be a triple with complement $\{d<e\}$. A singular place of diagram~$2$ with triangle $T$ has type $\pc a_1(T)$ (Table~\ref{tab:types}), the vertex of the piece $\pc{d2}[T]$. No other diagram has a type there (the attained set of $\pc{d1}[T]$ is the open segment). The piece $\pc{d2}[T]$ is present for every label set (Table~\ref{tab:pieces}).

\begin{proposition}\label{prop:nod2}
Let $m\in(\R_{>0})^5$, and let $C\subseteq Y_m$ be an irreducible curve on which some $r_{kl}^2$ is not constant and $U\not\equiv0$. No singular place of $C$ has diagram $2$. Consequently Theorem~\ref{thm:types} holds with the pieces $\pc{d2}[T]$ removed from $\mathrm{pieces}(S(m))$.
\end{proposition}

\emph{Leading solutions.} At a singular place of diagram $2$ with triangle $T$, after a relabelling that puts the segment at the bodies $\{d,e\}$ and the triangle at $T$, the leading coefficients form a \emph{leading solution} $y_0$: a complex central configuration $\omega$ of the triangle (of Euler or Lagrange type, possibly complex), a sign vector $\sigma\in\{\pm1\}^T$, the ratio of the segment masses, and the coefficients of the cross distances determined by them. The leading system forces the \emph{torque} $\tau(y_0)$, a linear combination of the cross coefficients, to vanish. Let $\theta_k$, $k\in T$, be the cross coefficients normalized by the triangle positions, and call $\theta$ \emph{constant} if $\theta_a=\theta_b=\theta_c$. The leading solution is \emph{simple} if the Jacobian $J_0$ of the leading system at $y_0$ has rank $16$, and a \emph{triple point} if the triangle parts of $u$ and $v$ are proportional. Triple points have rank $15$. Appendix~\ref{app:d2} gives the definitions in full.

\begin{lemma}\label{lem:d2gate}
At a singular place of diagram $2$ with triangle $T$ the leading solution $y_0$ has $\tau(y_0)=0$, and one of the following holds:
\begin{enumerate}[label=\textup{(\roman*)}]
\item $y_0$ is simple, $\theta$ is not constant, and the obstruction $\eta_2$ of the second order vanishes;
\item $y_0$ is simple, $\theta$ is constant, and the obstruction $\Omega_3$ of the third order vanishes;
\item $y_0$ is a triple point, $\theta$ and $\sigma$ are constant, and the product $b_3F_{ss}F_{TT}$ of the coefficients of the quadratic terms of its Lyapunov--Schmidt reduction vanishes;
\item $\operatorname{rank}J_0\le15$ and $y_0$ is not a triple point.
\end{enumerate}
No leading solution whose triangle is equilateral satisfies \textup{(i)}--\textup{(iv)}, at any positive masses.
\end{lemma}

\begin{proof}[Outline]
We give the proof in Appendix~\ref{app:d2}. A left null vector of $J_0$ applied to the rows gives $T\tau+T^3\sigma_3$ in the local parameter $T$, whence $\tau(y_0)=0$. At a simple $y_0$ the first-order correction $y_1$ is unique. There the coefficient of $T^2$ gives $\eta_2=D\tau(y_0)\cdot y_1=0$. Here $\eta_2$ is a nonzero multiple of $(\theta_a-\theta_b)(\theta_a-\theta_c)\,w$, where $w$ is a response of the triangle block of $J_0$. It vanishes identically when $\theta$ is constant, in which case $y_1=0$ and the first coefficient that can be nonzero is $\Omega_3$, at $T^3$. At a triple point the rank is $15$ and the expansion forces $\theta$ to be constant. There the linear terms vanish after the reduction to the kernel, and the quadratic coefficients must satisfy $b_3F_{ss}F_{TT}=0$. For an equilateral triangle:
\begin{itemize}
\item at equal masses and constant $\sigma$, $|\Omega_3|\ge0.853$ for the Lagrange realization, and $b_3$, $F_{ss}$, $F_{TT}$ are nonzero for the triple realization;
\item at equal masses and nonconstant $\sigma$ the torque does not vanish;
\item at unequal masses $\theta$ is not constant and $\eta_2\ne0$ for the Lagrange realization, while the triple realization is excluded because $\theta$ would have to be constant.\qedhere
\end{itemize}
\end{proof}

The last statement of Lemma~\ref{lem:d2gate} excludes the equilateral triangles. The following theorem excludes all other triangles.

\begin{theorem}\label{thm:gate}
For every positive mass vector, every $\sigma$, every $\rho>0$ and either labelling of the segment, no leading solution with $\tau=0$ whose triangle is not equilateral satisfies \textup{(i)}--\textup{(iv)} of Lemma~\ref{lem:d2gate}.
\end{theorem}

\begin{proof}[Outline]
We give the proof in Appendix~\ref{app:d2}. On the set $\tau=0$ the obstruction $\eta_2$ factors as $\gamma^2H$. Here $\gamma$ vanishes if and only if $\theta$ is constant, and $H$ depends only on the central configuration of the triangle (Theorem~\ref{thm:d2closed}). Hence (i)--(iv) at a triangle that is not equilateral force constant $\theta$ or one of a few conditions on the triangle alone (Lemma~\ref{lem:d2reduce}). We show that $\theta$ is not constant. We write the triangles of Euler type as the roots of quintics whose coefficients are polynomials in the mass ratios, and we eliminate the shape by resultants. Up to a common factor, the conditions then hold only at twelve mass vectors of the triangle: the permutations of $(1,1,5)$ and $(1,1,2+2\sqrt2)$ for the Euler type, and those of a vector $(X_L,Y_L,1)$ for the Lagrange type. At each of these triangles the \emph{T-test} holds: some $k\in T$ has $m_k|\omega_k|^{-1/2}>\sum_{l\in T\setminus k}m_l|\omega_l|^{-1/2}$. The T-test excludes $\tau=0$ for every $\sigma$, every segment ratio and either labelling of the segment.
\end{proof}

\begin{proof}[Proof of Proposition~\ref{prop:nod2}]
Suppose that a singular place of diagram $2$ with triangle $T$ exists at $m$. By Lemma~\ref{lem:d2gate} its leading solution $y_0$ has $\tau(y_0)=0$ and satisfies one of (i)--(iv). The last statement of Lemma~\ref{lem:d2gate} excludes an equilateral triangle, and Theorem~\ref{thm:gate} excludes the other triangles. Thus no such place exists, and the piece $\pc{d2}[T]$ may be removed in Theorem~\ref{thm:types}.
\end{proof}

\section{The label sets of all positive masses}\label{sec:strata}

In this section we prove that the label set of every positive mass vector with $m_1\le\dots\le m_5$ is certified, contains a curve set, or is the label set of one of two exceptional points (Theorem~\ref{thm:S}). We decide this by exact computer algebra on the sixteen strata of the sorted masses. With \S\ref{sec:curves} the theorem gives finitely many central configurations at every positive mass vector off the special rays (Corollary~\ref{cor:S}). Such a mass vector is settled unless it is a positive multiple of a permutation of the point $P_1$ of \S\ref{ss:twopoints}. We call these mass vectors the \emph{counted points}. At them we count the central configurations (Theorem~\ref{thm:countP1}).

\input{sec/s_reduction}

\subsection{Labels and certified label sets}\label{ss:nearlabels}

We use the following $65$ labels:
\begin{itemize}
\item $\lab{E}(T)$ for each of the $10$ triples $T$;
\item $\lab{R}(ab;s)$ for each of the $30$ choices of a pair and a third body;
\item $\lab{H}(D_1|D_2)$ for each of the $15$ partitions of a four-body set into two pairs, read as $\cD_3(s;D_1|D_2)$ (Theorem~\ref{thm:D3});
\item $\lab{EQ}(a,b)$ for each of the $10$ pairs.
\end{itemize}
For $m\in(\R_{>0})^5$ let $\Lambda(m)$ be the set of these labels that hold at $m$. The label $\lab{D6}(ab|cd)$ is implied by $\lab{EQ}(a,b)$ and $\lab{EQ}(c,d)$ (\S\ref{ss:conditions}). Hence every piece of $\mathrm{pieces}(S(m))$ lies in $\mathrm{pieces}(S)$ whenever $\Lambda(m)\subseteq S$, and Proposition~\ref{prop:cert} applies to any label set $S\supseteq\Lambda(m)$.

In this section a label set $S$ is \emph{certified} if it has a certificate relative to $\Vtwo^{\mathrm w}$ for the list $\cA_{\Vtwo^{\mathrm w}}(S)$ without the attained faces of the pieces $\pc{d2}[T]$, $\pc{d9}(T;ax,by)$, $\pc{d14b}$ and $\pc{d15}$, which we may leave out by \S\ref{ss:certs}. Proposition~\ref{prop:cert} applies to such a set, since $\Vtwo^{\mathrm w}$ is valid at every positive mass (Remark~\ref{rem:hulls}).

For each of the curve types $T_1$, $T_2$, $T_3$, $T_4$, $T_5$, $T_7$, $T_8$, $T_9$ and the line $\mathrm{RL}$ of \S\ref{sec:curves}, Table~\ref{tab:curvesets} gives a set $D_T$ of labels. The images of these sets under the $120$ relabellings of the bodies are the \emph{curve sets}. We treat the mass vectors whose label set contains a curve set in \S\ref{sec:curves}.

\begin{table}[ht]
\centering
\small
\begin{tabular}{@{}ll@{\qquad}ll@{}}
\toprule
$T$ & $D_T$ & $T$ & $D_T$\\
\midrule
$T_1$ & $\lab{R}(23;1)$, $\lab{EQ}(2,5)$, $\lab{EQ}(3,4)$ & $T_5$ & $\lab{E}(123)$, $\lab{R}(24;1)$, $\lab{EQ}(4,5)$\\
$T_2$ & $\lab{R}(23;1)$, $\lab{EQ}(3,4)$, $\lab{EQ}(3,5)$, $\lab{EQ}(4,5)$ & $T_7$ & $\lab{E}(123)$, $\lab{EQ}(2,5)$, $\lab{EQ}(3,4)$\\
$T_3$ & $\lab{EQ}(a,b)$ for all pairs $a<b$ in $\{1,2,3,4\}$ & $T_8$ & $\lab{E}(123)$, $\lab{R}(24;1)$, $\lab{R}(25;1)$\\
$T_4$ & $\lab{R}(23;1)$, $\lab{R}(25;1)$, $\lab{EQ}(3,4)$ & $T_9$ & $\lab{E}(123)$, $\lab{R}(24;1)$, $\lab{EQ}(2,3)$\\
$\mathrm{RL}$ & $\lab{E}(123)$, $\lab{EQ}(2,3)$, $\lab{EQ}(2,4)$, $\lab{EQ}(3,4)$ & & \\
\bottomrule
\end{tabular}
\caption{The label sets $D_T$ of the curve families.}\label{tab:curvesets}
\end{table}

\subsection{The theorem}\label{ss:thmS}

Let $\mathbf q_*$ be the only mass vector with $0<m_1<m_2<m_3<m_4<m_5=1$ at which $\lab E(125)$ and $\lab H(25|34)$ hold (Proposition~\ref{prop:Spoints}). Then
\[
\mathbf q_*=(0.1243983532\ldots,\ 0.2968967473\ldots,\ 0.4494254022\ldots,\ 0.6606140769\ldots,\ 1).
\]
Let $\mathbf p_*=(\frac1{16},\frac19,\frac14,1,1)$.

\begin{theorem}\label{thm:S}
Let $m\in(\R_{>0})^5$ with $m_1\le m_2\le m_3\le m_4\le m_5$. Then one of the following holds:
\begin{enumerate}[label=\textup{(\roman*)}]
\item $\Lambda(m)$ is certified;
\item $\Lambda(m)$ contains a curve set;
\item a positive multiple of $m$ is $\mathbf q_*$ or $\mathbf p_*$.
\end{enumerate}
\end{theorem}

In case (iii) the hull $\VN$ is valid at $\mathbf q_*$ and $\mathbf p_*$, and it certifies their label sets with all pieces $\pc{d14a}$ present (Proposition~\ref{prop:Spoints}). We prove Theorem~\ref{thm:S} in \S\ref{ss:proofS}. With \S\ref{sec:curves} it gives the following corollary.

\begin{corollary}\label{cor:S}
Let $m\in(\R_{>0})^5$ lie on no special ray. Then the planar five-body problem with masses $m$ has finitely many central configurations up to similarity, and $m$ is settled unless it is a counted point.
\end{corollary}

\subsection{Coordinates and labels on a stratum}\label{ss:strata}

Let $F$ be a face with the blocks $B_1,\dots,B_{n+1}$, numbered in increasing order of their bodies. At a point $m$ of $\Sigma_F$ put $x_i=m_i^{-1/2}$, as in \S\ref{ss:conditions}. Then $x_1\ge x_2\ge\dots\ge x_5=1$. We write $y_k$ for the common value of $x_i$ on the bodies $i$ of $B_k$. Then $y_1>y_2>\dots>y_n>y_{n+1}=1$. The map $m\mapsto(y_1,\dots,y_n)$ identifies $\Sigma_F$ with the open set $\{y\in\R^n:\ y_1>\dots>y_n>1\}$, and we use this identification throughout. On $\Sigma_F$ the label $\lab{EQ}(a,b)$ holds if and only if $a$ and $b$ lie in one block of $F$. Let $Q(F)$ be the set of these labels.

We describe the $55$ labels $\lab E$, $\lab R$ and $\lab H$ in $x$, at positive $x$ (\S\ref{ss:conditions} and Appendix~\ref{app:labels}).
\begin{itemize}
\item $\lab{E}(abc)$: its polynomial times $(x_ax_bx_c)^4$ is $-(x_a+x_b+x_c)(-x_a+x_b+x_c)(x_a-x_b+x_c)(x_a+x_b-x_c)$. The first factor is positive. So $\lab{E}(abc)$ holds if and only if $x$ lies on one of the hyperplanes $x_a=x_b+x_c$, $x_b=x_a+x_c$ and $x_c=x_a+x_b$.
\item $\lab{R}(ab;s)$: its polynomial times $x_s^4(x_ax_b)^8$ is $\big((x_a^2+x_b^2)^2-x_s^2x_ax_b\big)\big((x_a^2+x_b^2)^2+x_s^2x_ax_b\big)$. The second factor is positive. So $\lab{R}(ab;s)$ holds if and only if $(x_a^2+x_b^2)^2=x_s^2x_ax_b$.
\item $\lab{H}(ij|kl)$ with centre $s$, read as $\cD_3$: put $A=x_i^3$, $B=x_j^3$, $A'=x_k^3$, $B'=x_l^3$, $D=\prod_{q\in Q_s}x_q^2$ and $N=DM_{Q_s}$, a polynomial in $x$. Let $c=x_s^2M_{Q_s}^{-1/2}>0$. Then $\mathrm{Tri}(ij;s)$ says that $c^2=(A+\epsilon B)^2$ for a sign $\epsilon$, and likewise for $\mathrm{Tri}(kl;s)$. Since $c^2N=x_s^4D$, the condition $\cD_3$ holds if and only if $x$ is a zero of one of the eight ideals
\[
\big(x_ix_j-x_kx_l,\ A+\epsilon_1B-\delta(A'+\epsilon_2B'),\ (A+\epsilon_1B)^2N-x_s^4D\big),\qquad \epsilon_1,\epsilon_2,\delta\in\{1,-1\}.
\]
\end{itemize}
The three hyperplanes of $\lab E(abc)$, the quartic of $\lab R(ab;s)$ and these eight ideals are the \emph{branch ideals} of the labels. After the substitution of $y_k$ for $x_i$, $i\in B_k$, they become ideals of $\Q[y]=\Q[y_1,\dots,y_n]$. Thus a label $\lab E$, $\lab R$ or $\lab H$ holds at a point of $\Sigma_F$ if and only if the point is a zero of one of its branch ideals. We also fix the \emph{$x$-polynomials} of each label: for $\lab E(abc)$ the product of the four linear factors above; for $\lab R(ab;s)$ the quartic $(x_a^2+x_b^2)^2-x_s^2x_ax_b$; for $\lab H(ij|kl)$ the polynomial $x_k^2x_l^2-x_i^2x_j^2$ and $N^2$ times the polynomials of $\mathrm{Tri}(ij;s)$ and $\mathrm{Tri}(kl;s)$ in Appendix~\ref{app:labels}, written in $x$. At positive $x$ a label holds if and only if all its $x$-polynomials vanish.

\subsection{Flats}\label{ss:flats}

We compute standard bases, minimal primes and saturations over $\Q$ with Singular~4.4.1 (class (C1)).

\emph{Filters.} An ideal $J$ of $\Q[y]$ \emph{passes the filters} if
\begin{enumerate}[label=\textup{(F\arabic*)}]
\item $J\ne\Q[y]$;
\item no element of degree at most $12$ of the reduced standard basis of $J$ becomes, under the substitution $y_j=1+z_j+\dots+z_n$, a nonzero polynomial whose nonzero coefficients all have one sign;
\item the affine-linear elements of this basis have a common zero in $\Sigma_F$, which we decide by Fourier--Motzkin elimination over $\Q$ (class (C2)).
\end{enumerate}
The substitution maps $(\R_{>0})^n$ onto $\Sigma_F$, and a nonzero polynomial whose coefficients have one sign does not vanish on $(\R_{>0})^n$. Hence an ideal has no zero in $\Sigma_F$ if it fails a filter.

\emph{Decompositions.} For ideals $P$ and $B$ of $\Q[y]$ we compute a list $L(P,B)$ of ideals. Each member contains $P+B$ and passes the filters, and every zero of $P+B$ in $\Sigma_F$ is a zero of some member. If $P+B$ is zero-dimensional or $B$ contains $P$, the list consists of $P+B$, unless $P+B$ fails a filter and the list is empty. Otherwise we add the generators of $B$ to $P$ one at a time, in order of increasing degree. After each step we replace each ideal of positive dimension by its minimal primes, keep each zero-dimensional ideal as it is and drop the ideals failing (F1) or (F2). If a decomposition does not finish within a time limit, we redo it with the saturation by $y_1\cdots y_n$ in place of the minimal primes of each ideal of positive dimension. Saturation keeps all zeros in $\Sigma_F$, since no coordinate vanishes there. At the end we drop each member failing (F3) or containing another one. The members are prime ideals, zero-dimensional ideals or saturated ideals. We do not use that they are prime.

\emph{Components and flats.} For a label $\ell$ of the kinds $\lab E$, $\lab R$, $\lab H$ the \emph{components} of $\ell$ on $\Sigma_F$ are the members of the lists $L((0),B)$, $B$ a branch ideal of $\ell$. A \emph{flat} of $F$ is an ideal $P$ of $\Q[y]$ with a set $G(P)$ of labels of these kinds, each of which holds at every zero of $P$ in $\Sigma_F$. The \emph{root} is $P=(0)$ with $G(P)=\emptyset$. Let $P$ be a flat and $\ell\notin G(P)$ a label of these kinds. The \emph{children} of $P$ via $\ell$ are the members other than $P$ of the lists $L(P,C)$, $C$ a component of $\ell$. If $P$ itself is such a member, we add $\ell$ to $G(P)$. Each child $P'$ is a flat. The set $G(P')$ consists of the labels whose $x$-polynomials lie in $P'$, together with $G(P)$ and $\ell$. We keep $G(P')\supseteq G(P)\cup\{\ell\}$ when $G(P)$ grows. The zero set of a child lies in those of $P$ and of $C$. So these labels hold at its zeros in $\Sigma_F$.

\begin{lemma}\label{lem:Scover}
Let $P$ be a flat of $F$, let $m\in\Sigma_F$ be a zero of $P$, and let $\ell\in\Lambda(m)\setminus(Q(F)\cup G(P))$.
\begin{enumerate}[label=\textup{(\alph*)}]
\item The point $m$ is a zero of $P+C$ for a component $C$ of $\ell$, and $P+C$ passes the filters.
\item If the children of $P$ via $\ell$ are computed, $m$ is a zero of one of them.
\end{enumerate}
\end{lemma}

\begin{proof}
(a) Since $\ell$ holds at $m$, the point $m$ is a zero of a branch ideal $B$ of $\ell$, hence of a member $C$ of $L((0),B)$. So $m$ is a zero of $P+C$ in $\Sigma_F$, and $P+C$ passes the filters. (b) The point $m$ is a zero of a member of $L(P,C)$. This member is not $P$, since otherwise $\ell$ would lie in $G(P)$. So it is a child of $P$ via $\ell$.
\end{proof}

\emph{Expansion and closure.} We call a label $\ell\notin G(P)$ of the kinds $\lab E$, $\lab R$, $\lab H$ \emph{possible} on $P$ if $Q(F)\cup G(P)\cup\{\ell\}$ contains no curve set and $P+C$ passes the filters for some component $C$ of $\ell$. To \emph{expand} $P$, we first find its possible labels. If the set $U_A(P)$ of the labels of $Q(F)$ and $G(P)$ and of the possible labels is certified, we call $P$ \emph{closed} and do not compute its children. Otherwise we compute the children of $P$ via all possible labels. If then the set $U(P)$ of the labels of $Q(F)$ and $G(P)$ and of the labels via which $P$ has a child is certified, we call $P$ closed as well. We call $P$ \emph{expanded} if its children via all possible labels are computed.

\begin{lemma}\label{lem:Sclosed}
Let $P$ be a closed flat of $F$, and let $m\in\Sigma_F$ be a zero of $P$. Then $\Lambda(m)$ is certified or contains a curve set.
\end{lemma}

\begin{proof}
The labels of $Q(F)$ and $G(P)$ lie in $\Lambda(m)$. Suppose that $\Lambda(m)$ contains no curve set, and let $\ell\in\Lambda(m)\setminus(Q(F)\cup G(P))$. Then $Q(F)\cup G(P)\cup\{\ell\}$ contains no curve set, and $\ell$ is possible on $P$ by Lemma~\ref{lem:Scover}(a). If the children of $P$ are computed, $P$ has a child via $\ell$, by Lemma~\ref{lem:Scover}(b). Hence $\Lambda(m)$ lies in $U_A(P)$, and in $U(P)$ if the children of $P$ are computed. Since $P$ is closed, $U_A(P)$ is certified, or the children of $P$ are computed and $U(P)$ is certified. A subset of a certified set is certified (\S\ref{ss:certs}).
\end{proof}

\subsection{The computation}\label{ss:Scomp}

On each face $F$ we expand the root and then every flat that we obtain, except the flats $P$ for which $Q(F)\cup G(P)$ contains a curve set and the \emph{covered} flats, the children of closed or covered flats. We decide whether a label set is certified in the sense of \S\ref{ss:nearlabels} by the search of \S\ref{ss:lexexact}. We check every certificate exactly as there. A label set for which the search finds no certificate counts as uncertified. This only adds work. Because the choice of the decompositions that are redone with saturations depends on the time limit and on the load of the machine, the proof concerns the computed flats, which are stored in the supplementary material.

\begin{proposition}\label{prop:Stree}
For each face $F$ the computation gives finitely many flats. Table~\ref{tab:Sflats} gives their numbers summed over the faces. Every flat $P$ of $F$ satisfies one of the following:
\begin{enumerate}[label=\textup{(\alph*)}]
\item $Q(F)\cup G(P)$ contains a curve set;
\item $P$ is closed or covered;
\item $Q(F)\cup G(P)$ is certified, and $P$ is expanded;
\item $P$ is zero-dimensional and expanded, and each of its zeros in $\Sigma_F$ is $\mathbf q_*$ or $\mathbf p_*$.
\end{enumerate}
\end{proposition}

\begin{table}[ht]
\centering
\small
\begin{tabular}{@{}lrrrrr@{}}
\toprule
dimension & flats & (a) & (b) & (c) & (d)\\
\midrule
$4$ & $1$ & & & $1$ & \\
$3$ & $24$ & & & $24$ & \\
$2$ & $177$ & & $14$ & $163$ & \\
$1$ & $538$ & $6$ & $126$ & $406$ & \\
$0$ & $1466$ & $12$ & $1377$ & & $77$\\
\midrule
all & $2206$ & $18$ & $1517$ & $594$ & $77$\\
\bottomrule
\end{tabular}
\caption{The flats of the sixteen faces, by dimension and by the first of the cases (a)--(d) of Proposition~\ref{prop:Stree} that they satisfy. Of the $1517$ flats of (b), $1510$ are closed. The empty face has $1979$ flats.}\label{tab:Sflats}
\end{table}

\begin{proof}
The flats, their sets $G$, their lists of children and the closures are the outcome of the computation. The proof of Theorem~\ref{thm:S} uses the lists of children of $35545$ pairs $(P,\ell)$, many of them empty. On the empty face $22$ of the $55$ labels have a component: the ten labels $\lab E$, the ten labels $\lab R(ab;s)$ with $s<a<b$, and $\lab H(25|34)$ and $\lab H(15|34)$. On this face $110$ decompositions are redone with saturations, and $60$ flats are saturated ideals. They all have dimension~$1$ and satisfy (b) or (c). The computation decides $2599$ label sets: $1840$ are certified, $703$ have an exact witness in the sense of \S\ref{ss:T458}, and $56$ remain undecided.

The $77$ flats outside (a)--(c) are zero-dimensional and expanded. Of them $73$ belong to the empty face, three to $12|3|4|5$ and one to $1|2|3|45$. We decompose each of them into primes. We add the squarefree parts of its univariate eliminants, which gives its radical (Seidenberg's lemma). We take a linear form whose minimal polynomial on the radical has as many roots as the radical has zeros, and we factor this polynomial over $\Q$. This gives $334$ primes. Of them $310$ fail (F2) and $18$ fail (F3). For each of the other six we compute a shape basis for a separating linear form and isolate the real roots of its univariate polynomial exactly (class (C2)). Two of these roots give points of $\Sigma_F$: the point $\mathbf q_*$, a zero of a prime of degree $48$ of the empty face, and the point $\mathbf p_*$ of $1|2|3|45$. The other four primes have no real zero in $\Sigma_F$.
\end{proof}

\subsection{The two exceptional points}\label{ss:Spoints}

\begin{proposition}\label{prop:Spoints}
\begin{enumerate}[label=\textup{(\alph*)}]
\item Exactly one mass vector with $0<m_1<m_2<m_3<m_4<m_5=1$, the point $\mathbf q_*$, has $\lab E(125)$ and $\lab H(25|34)$ in its label set. Moreover $\Lambda(\mathbf q_*)=\{\lab E(125),\lab H(25|34)\}$ and
\[
\Lambda(\mathbf p_*)=\{\lab E(124),\lab E(125),\lab E(234),\lab E(235),\lab E(345),\lab R(45;3),\lab{EQ}(4,5)\}.
\]
\item The hull $\VN$ is valid at $\mathbf q_*$ and at $\mathbf p_*$.
\item Let $S$ be $\Lambda(\mathbf q_*)$ or $\Lambda(\mathbf p_*)$. Then $S$ has a certificate relative to $\VN$ for the list $\cA_{\VN}(S)$ without the attained faces of the pieces $\pc{d2}[T]$, $\pc{d9}(T;ax,by)$, $\pc{d14b}$ and $\pc{d15}$.
\end{enumerate}
\end{proposition}

\begin{proof}
(a) The mass vectors in question form the stratum $\Sigma_{1|2|3|4|5}$. There $\lab E(125)$ has one component, the hyperplane $y_1=y_2+1$, and $\lab H(25|34)$ has one component, a curve. Every point of $\Sigma_{1|2|3|4|5}$ at which both labels hold is a zero of each component, hence of their sum. This sum is a prime ideal with $48$ zeros, and exactly one of them lies in $\Sigma_{1|2|3|4|5}$ (proof of Proposition~\ref{prop:Stree}). The $x$-polynomials of $\lab E(125)$ and $\lab H(25|34)$ lie in this ideal. For every other label a ball enclosure at $\mathbf q_*$ of one of its $x$-polynomials excludes $0$. At $\mathbf p_*$ we have $x=(4,3,2,1,1)$, and the label set follows by exact evaluation.

(b) Both points pass the check of \S\ref{ss:checks}. At $\mathbf p_*$ the triangle $\{3,4,5\}$ has the masses $(\frac14,1,1)$ and is in the centred case.

(c) The search of \S\ref{ss:lexexact} finds these certificates, and we check them exactly as there.
\end{proof}

Exact witnesses as in \S\ref{ss:T458} show that neither $\Lambda(\mathbf q_*)$ nor $\Lambda(\mathbf p_*)$ is certified; but the certificates of (c) keep all pieces $\pc{d14a}$, so the two points need no point certificate of \S\ref{sec:points}.

\subsection{Proofs of Theorem~\ref{thm:S} and Corollary~\ref{cor:S}}\label{ss:proofS}

\begin{proof}[Proof of Theorem~\ref{thm:S}]
The labels are homogeneous in $m$. So we may scale $m$ to $m_5=1$. Then $m$ lies in $\Sigma_F$ for one face $F$. We construct flats $P_0,P_1,\dots$ of $F$ with the zero $m$ such that $G(P_0)\subsetneq G(P_1)\subsetneq\cdots$. Let $P_0$ be the root, and let $P=P_r$ be constructed. The labels of $Q(F)$ and $G(P)$ lie in $\Lambda(m)$. We apply Proposition~\ref{prop:Stree} to $P$.
\begin{itemize}
\item In case (a), $\Lambda(m)$ contains a curve set.
\item In case (b), $P$ is closed or covered. The zero set of a child lies in that of its parent, and a child strictly contains its parent. So, following parents, we reach a closed flat with the zero $m$. Lemma~\ref{lem:Sclosed} gives (i) or (ii).
\item In cases (c) and (d), suppose first that $\Lambda(m)=Q(F)\cup G(P)$. Then $\Lambda(m)$ is certified in case (c), and $m$ is $\mathbf q_*$ or $\mathbf p_*$ in case (d). Otherwise pick $\ell\in\Lambda(m)\setminus(Q(F)\cup G(P))$. If $Q(F)\cup G(P)\cup\{\ell\}$ contains a curve set, so does $\Lambda(m)$. If not, $\ell$ is possible on $P$, by Lemma~\ref{lem:Scover}(a). Since $P$ is expanded, $m$ is a zero of a child $P'$ of $P$ via $\ell$, by Lemma~\ref{lem:Scover}(b). Then $G(P')$ contains $G(P)\cup\{\ell\}$, and we put $P_{r+1}=P'$.
\end{itemize}
The sets $G(P_r)$ increase strictly, and they contain at most $55$ labels. Hence the construction ends in one of the cases above.
\end{proof}

\begin{proof}[Proof of Corollary~\ref{cor:S}]
Positive scaling and relabelling of the masses preserve finiteness and settledness, by the isomorphisms in the proof of Lemma~\ref{lem:slice}. So we may assume $m=\hat m$. We apply Theorem~\ref{thm:S} to $m$.
\begin{enumerate}[label=(\roman*)]
\item If $\Lambda(m)$ is certified, Proposition~\ref{prop:cert} settles $m$ (\S\ref{ss:nearlabels}).
\item Let $\Lambda(m)$ contain $\sigma(D_T)$ for a name $T$ of Table~\ref{tab:curvesets} and a relabelling $\sigma$. If $m$ is a point of $T_9$ up to scaling and relabelling, Proposition~\ref{prop:farT9} settles $m$. Otherwise Theorem~\ref{thm:curves} applies, since $m$ lies on no special ray. So $m$ has finitely many central configurations up to similarity, and it is settled unless it is a counted point.
\item If $m$ is $\mathbf q_*$ or $\mathbf p_*$, then $\VN$ is valid at $m$, and $\Lambda(m)$ has a certificate relative to $\VN$ as in Proposition~\ref{prop:Spoints}(c). We may leave out the faces of the pieces left out there, by \S\ref{ss:certs}. So Proposition~\ref{prop:cert} settles $m$.
\end{enumerate}
A settled mass vector has finitely many central configurations up to similarity (Corollary~\ref{cor:U}).
\end{proof}

%% file: sec/s_reduction.tex
\subsection{Sorted masses and the reduction}\label{ss:reduction}

For $m\in(\R_{>0})^5$ let $m_{(1)}\le\dots\le m_{(5)}$ be its coordinates in increasing order. The \emph{sorted scaled image} of $m$ is
\[
\hat m=\big(m_{(1)},m_{(2)},m_{(3)},m_{(4)},m_{(5)}\big)/m_{(5)}.
\]
The sorted scaled images are the mass vectors with $0<m_1\le m_2\le m_3\le m_4\le m_5=1$. Among them are the points $\bone=(\tfrac14,1,1,1,1)$ and $\btwo=(\tfrac14,\tfrac14,1,1,1)$.

A \emph{face} $F$ is a set of consecutive pairs $(i,i+1)$, $1\le i\le4$. We write it as the partition of $12345$ into the runs of bodies joined by it, its \emph{blocks}. For instance $1|23|45$ is the face $\{(2,3),(4,5)\}$ and $1|2|3|4|5$ the empty face. The \emph{stratum} of $F$ is
\[
\Sigma_F=\{m:\ 0<m_1\le\dots\le m_5=1,\ m_i=m_{i+1}\text{ for }(i,i+1)\in F,\ m_i<m_{i+1}\text{ for }(i,i+1)\notin F\}.
\]
At a point of $\Sigma_F$ two bodies have equal masses if and only if they lie in one block of $F$. The sorted scaled images form the disjoint union of the sixteen strata. The point $\bone$ lies in $\Sigma_{1|2345}$ and $\btwo$ in $\Sigma_{12|345}$.

\begin{lemma}\label{lem:R}
Let $m\in(\R_{>0})^5$.
\begin{enumerate}[label=\textup{(\alph*)}]
\item The planar five-body problem with masses $m$ has finitely many central configurations up to similarity if and only if it has finitely many with masses $\hat m$.
\item $\hat m=\bone$ if and only if $m$ is a positive multiple of a permutation of $(1,4,4,4,4)$, and $\hat m=\btwo$ if and only if $m$ is a positive multiple of a permutation of $(1,1,4,4,4)$.
\end{enumerate}
\end{lemma}

\begin{proof}
(a) This is Lemma~\ref{lem:slice}(c), since $\hat m$ is obtained from $m$ by a permutation and a positive scaling.

(b) The sorted scaled image of a positive multiple of a permutation of $(1,4,4,4,4)$ is $(\tfrac14,1,1,1,1)=\bone$. Conversely, if $\hat m=\bone$, then $m$ is a permutation of $m_{(5)}\bone=\tfrac{m_{(5)}}4(1,4,4,4,4)$. The same argument applies to $(1,1,4,4,4)$ and $\btwo$.
\end{proof}

Theorem~\ref{thm:S} below concerns sorted masses. On the rays of Lemma~\ref{lem:R}(b) we will count the central configurations directly in \S\ref{sec:rays}.

%% file: sec/s8_curves.tex
\section{Curve families and the RL line}\label{sec:curves}

Case (ii) of Theorem~\ref{thm:S} concerns the mass vectors whose label set contains a curve set, that is, one of the nine sets $D_T$ of Table~\ref{tab:curvesets} after a relabelling. Each $D_T$ defines a curve of mass vectors up to scaling (Lemma~\ref{lem:DT}), along which a generic label set $S_C$ holds identically. The pieces enabled by $S_C$ include the pieces $\pc{d14a}$, $\pc{d14b}$ and $\pc{d15}$ of the clusters. We settle the points of these curves by excluding such pieces point by point along each curve. In this section we do so along the whole curves. We prove that every point of $T_9$ is settled (Proposition~\ref{prop:farT9}), and that every point of these curves other than the points of $T_9$ and of the special rays has finitely many central configurations (Theorem~\ref{thm:curves}).

\subsection{The families}\label{ss:families}
Table~\ref{tab:families} lists the families of mass vectors defined by the sets $D_T$, with their parametrizations and generic label sets. The labels follow the supplementary files. The line $\mathrm{RL}$ plays the role of a sixth family. We define the family $T_9$, with the constant $\rho_0$, in \S\ref{sec:T9} and treat it there.

\begin{table}[ht]
\centering
\small
\begin{tabular}{@{}lp{0.47\textwidth}p{0.40\textwidth}@{}}
\toprule
& masses $(m_1,\dots,m_5)$; parameter & generic label set $S_C$\\
\midrule
$T_1$ & $(t^3,\,(1+t^2)^2,\,t^2(1+t^2)^2,\,t^2(1+t^2)^2,\,(1+t^2)^2)$; $t>0$ & $\lab{R}(23;1)$ $\lab{R}(24;1)$ $\lab{R}(35;1)$ $\lab{R}(45;1)$ $\lab{EQ}(2,5)$ $\lab{EQ}(3,4)$\\
$T_2$ & $(t^3,\,t^2(1+t^2)^2,\,(1+t^2)^2,\,(1+t^2)^2,\,(1+t^2)^2)$; $t>0$ & $\lab{R}(23;1)$ $\lab{R}(24;1)$ $\lab{R}(25;1)$ $\lab{EQ}(3,4)$ $\lab{EQ}(3,5)$ $\lab{EQ}(4,5)$\\
$T_3$ & $(a,a,a,a,1)$; $a>0$ & $\lab{EQ}(i,j)$ $(1\le i<j\le4)$\\
$T_4$ & $(yr^5B,\,r^4g^2,\,rg^2B,\,rg^2B,\,r^5g^2B)$; $r>0$, $y=f^{1/2}$ & $\lab{R}(23;1)\,\lab{R}(24;1)\,\lab{R}(25;1)\,\lab{EQ}(3,4)$\\
$T_5$ & $(q^6A^2,\,(1+q^4)^2A^2,\,q^6(1+q^4)^2,\,q^4(1+q^4)^2A^2,\,q^4(1+q^4)^2A^2)$; $q\neq0$ & $\lab{E}(123)\,\lab{R}(24;1)\,\lab{R}(25;1)\,\lab{EQ}(4,5)$\\
$T_7$ & $((1-s)^2,\,s^2(1-s)^2,\,s^2,\,s^2,\,s^2(1-s)^2)$; $s\ne0,1$ & $\lab{E}(123)$ $\lab{E}(124)$ $\lab{E}(135)$ $\lab{E}(145)$ $\lab{EQ}(2,5)$ $\lab{EQ}(3,4)$\\
$T_8$ & $(q^2r^4fC^2,\,r^4g^2C^2,\,q^2r^4fg^2,\,fg^2C^2,\,r^4fg^2C^2)$; $r>0$, $q^4=f$ & $\lab{E}(123)\,\lab{R}(24;1)\,\lab{R}(25;1)$\\
$\mathrm{RL}$ & $(4,1,4,4,\mu)$; $\mu>0$ & $\lab{E}(123)$ $\lab{R}(13;2)$ $\lab{E}(124)$ $\lab{R}(14;2)$ $\lab{E}(234)$ $\lab{R}(34;2)$ $\lab{EQ}(1,3)$ $\lab{EQ}(1,4)$ $\lab{EQ}(3,4)$\\
\bottomrule
\end{tabular}
\caption{The families. Here $A=1+q^3+q^4$, $B=1+r+r^2$, $f=r^3+r^2+r$, $g=(1+r)(1+r^2)=1+f$ and $C=g+q^3$, with $q$ real. On $T_5$ and $T_8$ both signs of $q$ occur. Each $S_C$ also contains the labels $\lab{D6}(ab|cd)$ with $\lab{EQ}(a,b)$ and $\lab{EQ}(c,d)$ in $S_C$ (on $T_1$, $T_3$ and $T_7$), as $S(m)$ does (\S\ref{ss:pieces}).}\label{tab:families}
\end{table}

All masses in Table~\ref{tab:families} are positive: $A\ge\frac{229}{256}$ for real $q$, and on $T_8$, $C=1+f+q^3\ge1+f-f^{3/4}>0$. The \emph{transversal cluster} of a family is $Q_1=\{2,3,4,5\}$ for $T_1$, $T_2$, $T_4$, $T_5$, $T_7$ and $T_8$, with masses proportional to
\begin{gather*}
(1,c,c,1),\ c=t^2\ \text{on $T_1$};\qquad (t^2,1,1,1)\ \text{on $T_2$};\qquad (r^3,B,B,r^4B)\ \text{on $T_4$};\\
(A^2,q^6,q^4A^2,q^4A^2)\ \text{on $T_5$};\qquad (1,c,c,1),\ c=(1-s)^{-2}\ \text{on $T_7$};\\
(r^4C^2,\,q^6r^4,\,q^4C^2,\,r^4q^4C^2)\ \text{on $T_8$}.
\end{gather*}
On $T_3$ the clusters are $Q_k$, $k\le4$, with masses proportional to $(1,1,1,1/a)$, and $Q_5$ with masses $(1,1,1,1)$. On $\mathrm{RL}$ it is $Q_2=\{1,3,4,5\}$ with masses $(4,4,4,\mu)$. Two cluster families recur, up to relabelling and scaling: $F_A(\nu)=(1,1,1,\nu)$ and $F_B(c)=(1,c,c,1)$.

\begin{lemma}\label{lem:DT}
Let $T$ be one of the nine names of Table~\ref{tab:curvesets}, and let $m\in(\R_{>0})^5$ satisfy every label of $D_T$. Then a positive multiple of $m$ is a point of the family $T$, or of $T_2$ if $T=T_4$, or of $T_5$ if $T=T_8$. If $T=\mathrm{RL}$, a positive multiple of $m$ is the image of a point of $\mathrm{RL}$ under the exchange of the bodies $1$ and $2$. If $T=T_9$, it is a point of $T_9$ or such an image.
\end{lemma}

\begin{proof}
With $x_i$ as in \S\ref{ss:conditions}, for positive masses $\lab{R}(ab;s)$ says $x_s=(x_a^2+x_b^2)/(x_ax_b)^{1/2}$, and $\lab{E}(abc)$ says that one of $x_a,x_b,x_c$ is the sum of the other two. The case $T_3$ is immediate.

$T_1$ and $T_2$: we scale so that $m_2=m_5=1$ and $m_3=m_4=t^2$ with $t>0$ ($T_1$), resp.\ $m_3=m_4=m_5=1$ and $m_2=t^2$ ($T_2$). Then $\lab{R}(23;1)$ gives $m_1=t^3/(1+t^2)^2$. So $(1+t^2)^2m$ is the point $t$ of the family.

$T_7$: scale so that $x_2=x_5=1$. Also $x_3=x_4$. If $x_1+x_3=1$ or $x_1=1+x_3$, put $s=x_1$. If $x_3=1+x_1$, put $s=-x_1$. Then $s\ne0,1$, $x_1=|s|$ and $x_3=|1-s|$. So $s^2(1-s)^2m$ is the point $s$ of $T_7$.

$T_4$: we scale so that $x_5=1$, and write $x_3=x_4=r^2$, $x_2=y$ with $r,y>0$. Then $\lab{R}(23;1)$ and $\lab{R}(25;1)$ give $x_1=(y^2+r^4)/(y^{1/2}r)=(y^2+1)/y^{1/2}$, that is, $y^2(1-r)=r(1-r)(1+r+r^2)$. If $r=1$, then $m_3=m_4=m_5$ and $m$ satisfies $D_{T_2}$. Otherwise $y^2=f(r)$ and $m$ is a multiple of the point $r$ of $T_4$.

$T_5$ and $T_8$: scale so that $x_5=1$ and write $x_2=q^2$, $x_4=r^2$ with $q,r>0$. On $T_5$, $\lab{EQ}(4,5)$ gives $r=1$. On $T_8$, $\lab{R}(24;1)$ and $\lab{R}(25;1)$ give $(q^4+r^4)/(qr)=(q^4+1)/q$, that is, $q^4(1-r)=r(1-r)(1+r+r^2)$. Hence $r=1$, and $m$ satisfies $D_{T_5}$, or $q^4=f(r)$. In both cases $x_1=(1+q^4)/q$ and $x_1-x_2=(1-q^3+q^4)/q>0$, since $1-u^3+u^4\ge\frac{229}{256}$ for $u>0$. Hence $\lab{E}(123)$ means $x_3=(1+q^4\pm q^3)/q$. Replacing $q$ by $-q$ in the case of the minus sign, $x_2=q^2$, $x_1=(1+q^4)/|q|$ and $x_3=(1+q^3+q^4)/|q|$ with $q\in\R\setminus\{0\}$. Hence $m$ is a multiple of the point $q$ of $T_5$, resp.\ of the point $(r,q)$ of $T_8$.

$\mathrm{RL}$: since $x_2=x_3=x_4$, $\lab{E}(123)$ forces $x_1=2x_2$. Hence $m$ is a multiple of $(1,4,4,4,\mu)$; the exchange of $1$ and $2$ maps the point $\mu$ of $\mathrm{RL}$ to this vector.

$T_9$: since $x_2=x_3$, $\lab{E}(123)$ forces $x_1=2x_2$. Scale so that $x_4=1$ and put $x_2=u^2$, $u>0$. Then $\lab{R}(24;1)$ reads $2u^2=(u^4+1)/u$, that is, $(u-1)(u^3-u^2-u-1)=0$. For $u=1$, $m$ is a multiple of $(1,4,4,4,\mu)$ as for $\mathrm{RL}$. Otherwise $u=\theta_0$, the real root of $u^3=u^2+u+1$, and $m=(\theta_0^{-4}/4,\theta_0^{-4},\theta_0^{-4},1,m_5)$ with $\theta_0^{-4}=\rho_0$.
\end{proof}

\subsection{Settling a point of a family}\label{ss:howsettled}
A piece is \emph{excluded} at $m\in(\R_{>0})^5$ if no singular place $P$ of a nontrivial curve $C\subseteq Y_m$ has its type $x(P)$ in the attained set of the piece. Let $S$ be a set of labels, $E$ a set of pieces and $V$ a D5 hull. The pair $(S,E)$ is \emph{certified} relative to $V$ if there are $l_1,\dots,l_r\in\Z^{10}$ that satisfy (a) and (b) of Definition~\ref{def:cert} for the list $\cA_V(S)$ without the attained faces of the pieces of $E$ and of all pieces $\pc{d2}$ and $\pc{d9}$. A certificate for $(S,E)$ is one for every $(S',E')$ with $S'\subseteq S$ and $E'\supseteq E$ (\S\ref{ss:certs}). A \emph{minimal exclusion set} of $S$ is a set $E$ of pieces $\pc{d14a}$, $\pc{d14b}$, $\pc{d15}$, minimal under inclusion, such that $(S,E)$ is certified. In this section all certificates are relative to the hull $\Vtwo^{\mathrm w}$ or to $\VN$ (Remark~\ref{rem:hulls}).

\begin{lemma}\label{lem:criterion}
Let $m\in(\R_{>0})^5$, let $V$ be a D5 hull valid at $m$, let $S\supseteq S(m)$, and let $E$ be a set of pieces each of which is excluded at $m$. If $(S,E)$ is certified relative to $V$, then $m$ is settled.
\end{lemma}

\begin{proof}
This is the proof of Proposition~\ref{prop:cert}, applied to $\cA_V(S)$ without the faces of the pieces of $E$ and of the pieces $\pc{d2}$ and $\pc{d9}$, which are excluded at every positive mass (\S\ref{ss:certs}): by Theorem~\ref{thm:types}, the validity of $V$ and the definition of an excluded piece, at each singular place of a curve as there, the type lies in the relative interior of a face in this list.
\end{proof}

We exclude pieces $\pc{d14a}$, $\pc{d14b}$ and $\pc{d15}$ by the following results. We carry out the computations behind them for solutions of the rows over the field of Puiseux series with prescribed valuations. The expansions of the coordinates at a place are such solutions.
\begin{enumerate}[label=\textup{(X\arabic*)}]
\item By Corollary~\ref{cor:Kdrop}, if $\mathrm{K14}(m_{Q_j})$ has no solution, $\pc{d14a}(j)$ and all $\pc{d14b}(j,k)$ are excluded. At a mass point given exactly over a number field we decide this system by Theorem~\ref{thm:KA}, one run of which covers all conjugates of the point.
\item By Theorem~\ref{thm:D14}, each $\pc{d14b}(j,k)$ is excluded at every $m\in(\R_{>0})^5$.
\item At a place of type $\pc b_k(j)$, the leading coefficients of the cluster form a $\mathrm{K15}(m_{Q_j},k)$ solution with a bad shape (Lemma~\ref{lem:d15lead} and Theorem~\ref{thm:order3}). Thus $\pc{d15}(j,k)$ is excluded at $m$ if no $\mathrm{K15}(m_{Q_j},k)$ solution has a bad shape. If the three masses of $Q_j\setminus\{k\}$ are equal, $\pc{d15}(j,k)$ is excluded (Corollary~\ref{cor:d15equal}).
\item By Corollary~\ref{cor:nod15}, each $\pc{d15}(j,k)$ is excluded at every $m\in(\R_{>0})^5$.
\end{enumerate}
For a family with transversal cluster $Q_1$ we write $\pi(1,k)=\{\pc{d14b}(1,k),\pc{d15}(1,k)\}$ and $\Pi(1)$ for the nine pieces $\pc{d14a}(1)$, $\pc{d14b}(1,k)$, $\pc{d15}(1,k)$, $k\in Q_1$. Similarly for other clusters. By (X2) and (X4), each pair $\pi(j,k)$ is excluded at every $m\in(\R_{>0})^5$. A \emph{special point} of a family is a point at which $\mathrm{K14}$ has a solution for its transversal cluster (on $T_3$, for one of $Q_1,\dots,Q_4$).

\begin{theorem}\label{thm:SJ}
Let $T$ be one of $T_1$, $T_2$, $T_3$, $T_4$, $T_5$, $T_7$, $T_8$, and read $\lab{H}$ as $\cD_3$ \textup{(}Theorem~\ref{thm:D3}\textup{)}. On each branch of $T$, $S(m)=S_C$ except at the \emph{jump points}, where a label $\lab{E}$, $\lab{R}$, $\lab{H}$ or $\lab{EQ}$ outside $S_C$ holds. At each jump point a label set containing $S(m)$ is determined. The numbers of the jump points on each family are those of Table~\ref{tab:curvecounts}.
\end{theorem}

\begin{table}[ht]
\centering
\small
\begin{tabular}{@{}lrrrrrrr@{}}
\toprule
& $T_1$ & $T_2$ & $T_3$ & $T_4$ & $T_5$ & $T_7$ & $T_8$\\
\midrule
jump points (Theorem~\ref{thm:SJ}) & $7$ & $3$ & $2$ & $8$ & $11$ & $9$ & $18$\\
\quad points of \S\ref{sec:points} & $2$ & $0$ & $1$ & $0$ & $2$ & $2$ & $0$\\
\quad points of Table~\ref{tab:noK} & $1$ & $2$ & $1$ & $0$ & $2$ & $7$ & $0$\\
\quad other jump points & $4$ & $1$ & $0$ & $8$ & $7$ & $0$ & $18$\\
special points & $16$ & $6$ & $6$ & & & $32$ & \\
sets in the list $\mathcal B$ (Theorem~\ref{thm:options}) & & & & $1992$ & $108$ & $9$ & $108$\\
\bottomrule
\end{tabular}
\caption{Counts along the families: the jump points, split by their treatment in Proposition~\ref{prop:exceptional}; the special points; and the numbers of sets with a witness in the proof of Theorem~\ref{thm:options}.}\label{tab:curvecounts}
\end{table}

\begin{proof}
Each label is a disjunction of polynomial conditions (Appendix~\ref{app:labels}). We compose every polynomial exactly with the masses of the family. On $T_4$ and $T_8$ we reduce it modulo the relation $y^2=f$, resp.\ $q^4=f$, and replace it by its norm over the conjugates of $y$, resp.\ $q$, a nonzero polynomial in $r$ because the relation is irreducible over $\Q(r)$. We factor the result over $\Q$. We isolate the real roots of the factors in ball arithmetic at $400$ bits or more and show that each root ball lies on one side of every boundary point of a branch, unless the boundary point is itself a root. On $T_4$ and $T_8$ a root $r_0$ is a jump point only if the label polynomial vanishes at the conjugate of the branch. We decide this by balls at the conjugates or, where they do not separate, by an exact gcd over $\Q(r_0)$. We decide an $\lab{H}$ label at a root of its polynomial by its two conditions $\mathrm{Tri}$ in ball arithmetic, and the $\lab{H}$ labels whose polynomial vanishes identically on a family through the exact gcd of their two $\mathrm{Tri}$ conditions. No label outside $S_C$ holds identically on a branch: we show each of them false in ball arithmetic at two sample points of every branch. The identifications of jump points with the points of \S\ref{sec:points} and with the special rays are exact: the mass ratios have the same minimal polynomials, with a common isolating ball.
\end{proof}

\subsection{The exceptional points}\label{ss:exceptional}
Table~\ref{tab:noK} lists the jump points that we treat individually. We will count three of them directly (\S\ref{sec:rays}).

\begin{table}[ht]
\centering
\small
\begin{tabular}{@{}llp{0.25\textwidth}l@{}}
\toprule
point & masses, up to scaling and relabelling & on the families & finiteness\\
\midrule
$\bone$ & $(1,4,4,4,4)$ & $T_1$ $t=1$; $T_2$ $t=1$; $T_3$ $a=4$; $T_5$ $q=-1$; $T_7$ $s=2$; $\mathrm{RL}$ $\mu=4$ & Theorem~\ref{thm:count1}\\
$\btwo$ & $(1,1,4,4,4)$ & $T_7$ $s=\frac12,-1$; $\mathrm{RL}$ $\mu=1$ & Theorem~\ref{thm:count2}\\
$P_1$ & $(9,36,36,36,4)$ & $T_5$ $q=1$; $\mathrm{RL}$ $\mu=\frac49$ & Theorem~\ref{thm:countP1}\\
$P_2$ & $(9,36,36,4,9)$ & $T_7$ $s=\frac32,3$ & settled\\
$\hone$ & $((1+t_*)^{-2},t_*^{-2},t_*^{-2},1,1)$ & $T_7$ $s\approx1.5548,\ 2.8025$ & settled\\
$\htwo$ & $(\frac{25}8,\frac{25}2,\frac{25}2,\frac{25}2,1)$ & $T_2$ $t=\frac12$; $\mathrm{RL}$ $\mu=\frac8{25}$ & settled\\
\bottomrule
\end{tabular}
\caption{The jump points treated individually. Here $t_*\approx1.8024986256$ is the root greater than $1$ of $2t^8+t^6-8t^5-6t^4-8t^3+t^2+2$.}\label{tab:noK}
\end{table}

At $P_2$ the label set contains $\lab{H}(12|35)$ and $\lab{H}(13|25)$, whose conditions $\mathrm{Tri}$ fail. So Theorem~\ref{thm:D3} removes the two pieces $\pc{d3}$. It also enables $\pc{d6}(145|234)$, which is removed by Proposition~\ref{prop:d6P2}. $\mathrm{K14}$ of the cluster $Q_4$, with masses $(9,36,36,9)$, has no solution, the pieces $\pc{d15}(4,2)$ and $\pc{d15}(4,3)$ are excluded by (X4), and the remaining pair is certified relative to $\VN$. At $\hone$ the label set contains two labels $\lab{H}$ whose conditions $\cD_3$ hold. So Theorem~\ref{thm:D3} does not remove the two pieces $\pc{d3}$. Proposition~\ref{prop:d3tstar} does. $\mathrm{K14}$ of the cluster of the four heavier bodies has no solution, the four pieces $\pc{d15}$ of this cluster are excluded by (X4), and the remaining pair is certified. $\htwo$ is the point $\mu=8/25$ of $\mathrm{RL}$ after the exchange of the bodies $1$ and $2$. It is settled by Theorem~\ref{thm:RL}.

\begin{proposition}\label{prop:exceptional}
On $T_1$, $T_2$, $T_3$, $T_4$, $T_5$, $T_7$ and $T_8$, every jump point of Theorem~\ref{thm:SJ} has finitely many central configurations, and every such point is settled except the points $\bone$, $\btwo$ and $P_1$ of Table~\ref{tab:noK}.
\end{proposition}

\begin{proof}
Table~\ref{tab:curvecounts} splits the jump points into three kinds. The points of \S\ref{sec:points} are K-certified and settled there. On $T_3$ this is the point of equal masses. We have treated the points of Table~\ref{tab:noK} above. Every other jump point carries a point certificate: in an exact description over a number field, $\mathrm{K14}$ of the transversal cluster has no solution (Theorem~\ref{thm:KA}), the pieces $\pc{d15}$ needed of this cluster are excluded by (X4), and the pair formed by the label set of the point and the pieces so excluded is certified. On $T_8$ the cluster is $Q_1$ and the pair is $(S,\Pi(1))$, with $S$ the label set of the point. Thus Lemma~\ref{lem:criterion} applies. Each certificate in this proof holds relative to $\Vtwo^{\mathrm w}$ or to $\VN$ (\S\ref{ss:checks}).
\end{proof}

\subsection{Four-body configurations at the masses $F_A$ and $F_B$}\label{ss:FAFB}

\begin{lemma}\label{lem:K14FAFB}
$\mathrm{K14}(1,1,1,\nu)$, $\nu>0$, has a solution for exactly six values of $\nu$, and $\mathrm{K14}(1,c,c,1)$, $c>0$, for exactly sixteen values of $c$. They are listed in Table~\ref{tab:K14values}.
\end{lemma}

\begin{proof}
We enumerate the configurations exactly: the components of the K14 system in Dziobek form are eliminated by resultants and norms over their coefficient fields, all zeros are enclosed by Krawczyk boxes, each rejected candidate is shown to be degenerate or to give a value of $\nu$, resp.\ $c$, that is not positive real, and each value found is identified as a root of an exact polynomial.
\end{proof}

\begin{table}[ht]
\centering
\small
\begin{tabular}{@{}lp{0.30\textwidth}p{0.52\textwidth}@{}}
\toprule
& values (to $4$ decimals) & polynomial\\
\midrule
$F_A$, $\nu$ & $0.2462$, $12.1840$ (collinear) & $x^8-116x^6-4810x^4-1044x^2+81$\\
& $0.2540$, $19.6837$ (collinear) & $x^8-32x^7+300x^6-864x^5-5098x^4-4320x^3+7500x^2-4000x+625$\\
& $1/\sqrt3$ & $3x^2-1$\\
& $6.0193$ & $3x^6-106x^4-97x^2-25$\\
\midrule
$F_B$, $c$ & $0.1125$, $3.2140$ (collinear) & $x^8+16x^6-40x^5-2x^4-512x^3+192x^2-24x+1$\\
& $0.1387$, $7.2098$ & $x^4-52x^2+1$\\
& $0.1801$, $5.5533$ (collinear) & degree $16$\\
& $0.2490$; $4.0154$; $0.6012$; $1.6633$ (collinear) & $x^4+4x-1$; $x^4-4x^3-1$; $x^4+4x^3-1$; $x^4-4x-1$\\
& $0.2510$, $3.9843$ (collinear) & $x^{16}-16x^{14}-4x^{12}+528x^{10}-6906x^8+528x^6-4x^4-16x^2+1$\\
& $0.3111$, $8.8913$ (collinear) & $x^8-24x^7+192x^6-512x^5-2x^4-40x^3+16x^2+1$\\
& $0.4429$; $2.2579$ & $x^4+10x^2-2$; $2x^4-10x^2-1$\\
\bottomrule
\end{tabular}
\caption{The positive masses $\nu$ and $c$ at which $\mathrm{K14}(1,1,1,\nu)$, resp.\ $\mathrm{K14}(1,c,c,1)$, has a solution, with their minimal polynomials. The entry ``collinear'' marks the values at which the configurations are collinear.}\label{tab:K14values}
\end{table}

The special points of $T_1$, $T_2$, $T_3$ and $T_7$ are therefore the points $t=\sqrt c$, $t=\sqrt\nu$, $a=1/\nu$ and $s=1\pm c^{-1/2}$, respectively, with $\nu$ and $c$ from Table~\ref{tab:K14values}. Table~\ref{tab:curvecounts} gives their numbers.

\begin{proposition}\label{prop:d14a}
At each of the $32$ special points of $T_7$ the piece $\pc{d14a}(1)$ is excluded, and at each of the $6$ special points of $T_3$ every $\pc{d14a}(k)$, $k\le4$, is excluded.
\end{proposition}

\begin{proof}
Since the masses of $T_3$ are symmetric in the bodies $1,\dots,4$, it suffices to treat $\pc{d14a}(1)$. At a place of type $\pc u_0(1)$, after a rotation and possibly an exchange of the colours, $\val z_{il}=\val w_{il}=1$ inside $Q=Q_1$ and $\val z_{1l}=\alpha$, $\val w_{1l}=-\alpha$ for one $\alpha\in[0,1]$. The leading coefficients of the cluster form a K14 configuration in a gauge, together with explicit data for body $1$ with signs $s_l=\pm1$ of the leading $p_{1l}$. The configurations at the special values are known completely (Lemma~\ref{lem:K14FAFB}). In a weighted form of the rows, the Jacobian $J$ of the cluster rows has rank $10$ at every configuration (a nonzero $10\times10$ minor in ball arithmetic). The first-order forcing must lie in the image of $J$. This image is described by two linear conditions, the pairings with the two cluster scalings, which lie in the kernel of the Hessian of the cluster potential (Remark~\ref{rem:d14a}). For $\alpha<1$ this forces, according to the sign pattern and to whether $\alpha=0$, the vanishing of a signed dipole $\sum_is_im_i\tilde\zeta_i$, the quadrupole $\sum_im_i\tilde\zeta_i^2$ or $A^2Q_{zz}Q_{ww}-B^2I^2$, where $\tilde\zeta$, $\tilde\omega$ are the centred leading positions, $Q_{zz}=\sum m_i\tilde\zeta_i^2$, $Q_{ww}=\sum m_i\tilde\omega_i^2$, $I=\sum m_i\tilde\zeta_i\tilde\omega_i$, $A=3m_1$ and $B=3m_1+2M_Q$. In $400$-bit ball arithmetic each of these quantities is nonzero at all configurations of all special values, except the dipole of the sign class $(+,-,-,+)$ at sixteen centrally symmetric configurations of $T_7$, where it vanishes by the symmetry (proved by a Krawczyk uniqueness argument), and the quadrupole and the dipole of $(+,+,+,-)$ at $a=\sqrt3$, where the configuration is an equilateral triangle with its centre. In these cases, and for $\alpha=1$ at all points, we compute the obstruction series of a Lyapunov--Schmidt reduction to higher order in ball arithmetic at all leading points, with the vanishing of the lower orders proved exactly. A nonzero coefficient excludes each case.
\end{proof}

\subsection{The family $T_3$}\label{ss:T3}

\begin{theorem}\label{thm:T3}
Every point of $T_3$ has finitely many central configurations, and every point other than $a=4$ is settled.
\end{theorem}

\begin{proof}
For each $k\in\{1,2,3,4\}$, the pair $(S_C,\{\pc{d14a}(k)\}\cup\Pi(5))$ is certified. The cluster $Q_5$ has the masses $(1,1,1,1)$, at which K14 has no solution (\S\ref{sec:points}). The pieces $\pc{d15}(5,k)$ are excluded by (X4). Hence $\Pi(5)$ is excluded at every point. Off the six special points $\mathrm{K14}(m_{Q_k})$ has no solution. At the special points $\pc{d14a}(k)$ is excluded by Proposition~\ref{prop:d14a}. By Theorem~\ref{thm:SJ} and Lemma~\ref{lem:criterion} each point with label set $S_C$ is settled. Proposition~\ref{prop:exceptional} covers the other points.
\end{proof}

\subsection{The families $T_1$ and $T_7$}\label{ss:T1T7}
On $T_1$ and $T_7$ the transversal cluster is $F_B(c)$, a relabelling of $(1,1,c,c)$. By (X4) all pieces $\pc{d15}(1,k)$ are removed on these two families.

\begin{theorem}\label{thm:T1T7}
Every point of $T_1$ and of $T_7$ has finitely many central configurations. Every point of $T_1$ other than $t=1$, and every point of $T_7$ other than $s\in\{2,\frac12,-1\}$, is settled.
\end{theorem}

\begin{proof}
On $T_7$, off the sixteen special values of $c$, $\mathrm{K14}(m_{Q_1})$ has no solution. Thus (X1) excludes $\pc{d14a}(1)$. At the special values $\pc{d14a}(1)$ is excluded by Proposition~\ref{prop:d14a}. Thus, by \S\ref{ss:howsettled}, $\Pi(1)$ is excluded at all points of $T_7$, and $\pi(1,2)\cup\pi(1,5)$ at all points of $T_1$. The pairs $(S_C,\Pi(1))$ on $T_7$ and $(S_C,\pi(1,2)\cup\pi(1,5))$ on $T_1$ are certified, and $\Pi(1)$ is the only minimal exclusion set of $S_C$ on $T_7$ (Theorem~\ref{thm:options}). We conclude as in the proof of Theorem~\ref{thm:T3}.
\end{proof}

\subsection{The family $T_2$}\label{ss:T2}

\begin{theorem}\label{thm:T2}
Every point of $T_2$ has finitely many central configurations, and every point other than $t=1$ is settled. Every piece $\pc{d15}(1,k)$, $k\in Q_1$, is excluded at every point of $T_2$.
\end{theorem}

\begin{proof}
The pair $(S_C,\pi(1,3)\cup\pi(1,4))$ is certified. Each pair $\pi(1,k)$, $k\in Q_1$, is excluded at every point of $T_2$ (\S\ref{ss:howsettled}). We conclude as in the proof of Theorem~\ref{thm:T3}.
\end{proof}

\subsection{The families $T_4$, $T_5$ and $T_8$}\label{ss:T458}
The following two results describe, for the families $T_4$, $T_5$, $T_7$ and $T_8$, all ways of settling a point with label set $S_C$ by Lemma~\ref{lem:criterion} with $V=\Vtwo^{\mathrm w}$ and pieces $\pc{d14a}$, $\pc{d14b}$, $\pc{d15}$. A \emph{witness} for $(S,E)$ is a nonempty set $W$ of faces in $\cA_V(S)$ outside the pieces of $E$ and outside the pieces $\pc{d2}[T]$ and $\pc{d9}(T;ax,by)$, together with rationals $c_v>0$ on the vertices $v$ of the faces in $W$ such that $\sum_vc_vv=0$. Any functional that is nonnegative on these vertices vanishes on all of them. So, by induction on $i$, every functional $l_i$ of a certificate for $(S,E)$ vanishes on them, against (b) of Definition~\ref{def:cert} for the faces of $W$. Hence $(S,E)$ is not certified relative to $V$, and neither is $(S,E')$ for any $E'\subseteq E$.

\begin{theorem}\label{thm:options}
For each of the families $T_4$, $T_5$, $T_7$ and $T_8$, the minimal exclusion sets of $S_C$ relative to $\Vtwo^{\mathrm w}$, among the $45$ pieces $\pc{d14a}(j)$, $\pc{d14b}(j,k)$, $\pc{d15}(j,k)$, are the following.
\begin{enumerate}[label=\textup{(\roman*)}]
\item $T_4$ \textup{(}$19$ sets\textup{)}: $\pi(1,A)\cup\pi(1,B)$ for $A<B$ in $\{2,3,4,5\}$; $\pi(1,A)\cup\pi(A,1)$ for $A\in\{3,4,5\}$; $\pi(1,A)\cup\pi(2,B)\cup\pi(2,C)$ for $\{A,B,C\}=\{3,4,5\}$; $\Pi(2)$; and $\pi(1,A)\cup\pi(2,B)\cup\Pi(C)$ for $(A,B,C)$ a permutation of $(3,4,5)$.
\item $T_5$ and $T_8$ \textup{(}$4$ sets\textup{)}: $\pi(1,4)$; $\pi(1,5)$; $\pi(1,3)\cup\Pi(2)$; $\Pi(2)\cup\Pi(3)$.
\item $T_7$ \textup{(}$1$ set\textup{)}: $\Pi(1)$.
\end{enumerate}
\end{theorem}

\begin{proof}
Each listed set $E$ has a certificate for $(S_C,E)$. For each family we give a list $\mathcal B$ of sets $A$ of pieces (Table~\ref{tab:curvecounts}), each with a witness for $(S_C,\text{all pieces}\setminus A)$. The listed sets are exactly the minimal transversals of $\mathcal B$. A set $E$ with $(S_C,E)$ certified meets each $A\in\mathcal B$ and therefore contains a minimal transversal. Each listed set is certified, and each of its proper subsets misses some $A$. The face systems of $T_5$ and $T_8$ coincide. The certificates and witnesses also hold with the labels of $T_8$ (\S\ref{ss:checks}).
\end{proof}

\begin{proposition}\label{prop:N}
For each of the families $T_4$, $T_5$, $T_7$ and $T_8$, the pair formed by $S_C$ and the $25$ pieces $\pc{d14a}$, $\pc{d14b}$ of all clusters has a witness relative to $\Vtwo^{\mathrm w}$. Consequently every minimal exclusion set contains a piece $\pc{d15}$.
\end{proposition}

Lemma~\ref{lem:criterion} with $V=\Vtwo^{\mathrm w}$ therefore settles a point with label set $S_C$ on these families only if some piece $\pc{d15}$ is excluded there.

\begin{proposition}\label{prop:T458}
\begin{enumerate}[label=\textup{(\alph*)}]
\item Every jump point of $T_4$, $T_5$ and $T_8$ \textup{(}Theorem~\ref{thm:SJ}\textup{)} has finitely many central configurations, and every such point other than $q=1$ \textup{(}$P_1$\textup{)} and $q=-1$ \textup{(}$\bone$\textup{)} on $T_5$ is settled.
\item The minimal exclusion sets $\pi(1,2)\cup\pi(1,3)$ of $T_4$ and $\pi(1,4)$ of $T_5$ are excluded at every point, so every point of $T_4$ and of $T_5$ is settled, except $q=\pm1$ on $T_5$, where finiteness holds by counting.
\item The minimal exclusion set $\pi(1,5)$ of $T_8$ is excluded at every point, so every point of $T_8$ is settled.
\end{enumerate}
\end{proposition}

\begin{proof}
(a) is Proposition~\ref{prop:exceptional}. (b) and (c) follow from \S\ref{ss:howsettled}, Theorem~\ref{thm:SJ}, Theorem~\ref{thm:options} and Lemma~\ref{lem:criterion}, and from Proposition~\ref{prop:exceptional} for the other points.
\end{proof}

\subsection{The RL line}\label{ss:RL}
On $\mathrm{RL}$, $m=(4,1,4,4,\mu)$, the cluster is $Q_2=\{1,3,4,5\}$ with masses $(4,4,4,\mu)$.

\begin{theorem}\label{thm:RL}
Let $\mu>0$ and $m=(4,1,4,4,\mu)$. The planar five-body problem with masses $m$ has finitely many central configurations up to similarity, and $m$ is settled unless $\mu\in\{1,4,\frac49\}$.
\end{theorem}

The proof uses three results valid for every $\mu>0$: Proposition~\ref{prop:nod2} excludes the piece $\pc{d2}[134]$, and (X2) and (X4) exclude every $\pc{d14b}(2,k)$ and every $\pc{d15}(2,k)$, $k\in Q_2$, respectively.

\begin{proof}[Proof of Theorem~\ref{thm:RL}]
Each label polynomial restricted to the line is a polynomial in $\mu$. The label polynomials that vanish identically give $S_C$. The other nonzero restrictions have $6$ irreducible factors over $\Q$ with a positive root (Table~\ref{tab:RLfactors}). We isolate these roots in ball arithmetic. Off the positive roots of these factors $S(m)\subseteq S_C$. At a root of a factor $g$, $S(m)\subseteq S(g)$, the union of $S_C$ and the labels with a polynomial divisible by $g$. Let $U=\pi(2,1)\cup\pi(2,3)\cup\pi(2,4)\cup\pi(2,5)\cup\{\pc{d2}[134]\}$, excluded at all $\mu>0$ by the results above. We treat the factors as follows.
\begin{enumerate}[label=\textup{(\roman*)}]
\item The pair $(S_C,U)$ is certified, and so is $(S(g),U)$ for the factors $g$ of the first row of Table~\ref{tab:RLfactors}.
\item At the root $\mu=\frac8{25}$ of the factor in the second row, K14 has no solution for any of the five clusters (Theorem~\ref{thm:KA}), the pieces $\pc{d15}$ are removed by (X4), and the pair formed by $S(g)$ and all pieces $\pc{d14a}$, $\pc{d14b}$, $\pc{d15}$ together with $\pc{d2}[134]$ is certified.
\item The roots of the factors of the third row are $1$, $4$ and $\frac49$. At $\mu=1$ and $\mu=4$, $m$ is a permutation of $(1,1,4,4,4)$, resp.\ of $(1,4,4,4,4)$. At $\mu=\frac49$, $9m$ is a permutation of $P_1$. Theorems~\ref{thm:count2}, \ref{thm:count1} and~\ref{thm:countP1} apply.
\end{enumerate}
By Lemma~\ref{lem:criterion}, $m$ is settled at every $\mu>0$ other than $1$, $4$ and $\frac49$.

\begin{table}[ht]
\centering
\small
\begin{tabular}{@{}lrl@{}}
\toprule
treatment & factors & degrees\\
\midrule
(i) $(S(g),U)$ certified & $2$ & $1$ (root $\mu=16$), $3$\\
(ii) all K14 systems fail & $1$ & $1$ (root $\mu=\frac8{25}$)\\
(iii) counted & $3$ & $1$, $1$, $1$ ($\mu-1$, $\mu-4$, $9\mu-4$)\\
\bottomrule
\end{tabular}
\caption{The factors with a positive root of the label polynomials restricted to $\mathrm{RL}$, by their treatment in the proof of Theorem~\ref{thm:RL}.}\label{tab:RLfactors}
\end{table}
\end{proof}

\subsection{The family $T_9$}\label{sec:T9}\label{ss:farstrata}

Theorem~\ref{thm:curves} treats the mass vectors whose label set contains a curve set, except the points of the special rays and of the family $T_9$ below. In this subsection we settle the points of $T_9$ (Proposition~\ref{prop:farT9}). For $T_9$ we need Theorem~\ref{thm:KApar}. Table~\ref{tab:farclose} collects the numbers of the computation.

\paragraph{The curve $T_9$.} Let $\rho_0=0.0873780\ldots$ be the real root of $p^3+5p^2+11p-1$, and let $T_9$ be the curve $m(s)=(\rho_0/4,\rho_0,\rho_0,1,s)$, $s>0$. The cluster $Q_1$ has masses $(\rho_0,\rho_0,1,s)$, with body $b$ in position $b-1$. The cluster $Q_5$ has the constant masses $(\rho_0/4,\rho_0,\rho_0,1)$.

\begin{proposition}\label{prop:farT9}
Every point of $T_9$ is settled.
\end{proposition}

\begin{proof}
\emph{Labels.} We reduce each label polynomial composed with $m(s)$ modulo the minimal polynomial of $\rho_0$, isolate the positive roots of its norm in ball arithmetic, and decide at each of them, by a ball or by an exact gcd over $\Q(\rho_0)$, whether the polynomial vanishes there. As in the proof of Theorem~\ref{thm:SJ}, a label $\lab H$ holds only where its two conditions $\mathrm{Tri}$ hold as well, which we decide in the same way. At each $s>0$ the label set is $S_C=\{\lab E(123),\lab R(23;1),\lab R(24;1),\lab R(34;1),\lab{EQ}(2,3)\}$, except at $13$ jump points, all in $[0.0030920,2091.64]$. They fall into $10$ classes by their label sets.

\emph{Pieces removed at every $s$.} Theorem~\ref{thm:D14} removes the pieces $\pc{d14b}(1,k)$. $\mathrm{K14}$ fails at the masses of $Q_5$ (Theorem~\ref{thm:KA}) and Corollary~\ref{cor:nod15} removes the pieces $\pc{d15}(5,k)$. Hence no piece of the cluster $Q_5$ is attained.

\emph{The generic set.} Two certificates for $S_C$ relative to $\Vtwo^{\mathrm w}$ remove $\pc{d14b}(1,k)$, $k=2,\dots,5$, and in addition $\pc{d15}(1,5)$, respectively $\pc{d15}(1,2)$, $\pc{d15}(1,3)$, $\pc{d15}(1,4)$. By (X2) and (X4) these pieces are removed at all $s$, so that both certificates apply. Hence $m(s)$ is settled by Proposition~\ref{prop:cert} with these pieces removed, as in the proof of Proposition~\ref{prop:Kcert}.

\emph{The jump points.} Each class has a certificate whose removed pieces are removed at its points: $\pc{d14b}(1,k)$ and $\pc{d15}(1,k)$ by (X2) and (X4), the pieces of $Q_5$ as above, and, in the classes whose certificate removes it, $\pc{d14a}(1)$ because $\mathrm{K14}(m_{Q_1})$ has no solution at the roots of their jump factors, by Theorem~\ref{thm:KApar}(b). We check the hypothesis of that theorem by a unit standard basis with the factor adjoined, or through a polynomial in $s$ alone in the ideal of the full system that is prime to the factor. At $s=\rho_0/4$, where $m_1=m_5$, the certificate also removes the piece $\pc{d6}(145|234)$. After relabelling to the masses $(1,1,4,4,4/\rho_0)$, we carry out the test for diagram $6$ used at $P_2$ in \S\ref{ss:exceptional} with $\Q(4/\rho_0)$ in place of $\Q$, with the standard basis and the closed forms exact over $\Q(4/\rho_0)$: the leading system is zero-dimensional, and each of its solutions has a nonzero obstruction coefficient; the leading solutions and their obstructions are certified in ball arithmetic. Hence no place of this diagram exists. The certificates at $s=\rho_0/4$ and at the two jump points $s\approx0.0520549$ and $s\approx0.176101$, where $\lab E(245)$ and $\lab E(345)$ hold, are relative to $\VN$, the others relative to $\Vtwo^{\mathrm w}$.
\end{proof}

\begin{table}[ht]
\centering
\small
\begin{tabular}{@{}lp{0.62\textwidth}r@{}}
\toprule
& quantity & number\\
\midrule
$T_9$ & jump points; classes & $13$; $10$\\
& certificates for $T_9$ & $12$\\
& solutions of the leading system of diagram $6$ at $s=\rho_0/4$ & $5184$\\
\bottomrule
\end{tabular}
\caption{The computation along $T_9$ in numbers.}\label{tab:farclose}
\end{table}

\subsection{The curve sets}\label{ss:HC}

\begin{theorem}\label{thm:curves}\label{prop:HC}
Let $m\in(\R_{>0})^5$ satisfy every label of $\sigma(D_T)$ for one of the nine names $T$ of Table~\ref{tab:curvesets} and a relabelling $\sigma$. Suppose that $m$ is not a positive multiple of a permutation of $(1,4,4,4,4)$ or of $(1,1,4,4,4)$, and, if $T=T_9$, that $m$ is not a point of the family $T_9$ up to scaling and relabelling.
Then the planar five-body problem with masses $m$ has finitely many central configurations up to similarity, and $m$ is settled unless it is a positive multiple of a permutation of $P_1$ \textup{(}\S\ref{ss:twopoints}\textup{)}.
\end{theorem}

\begin{proof}
By Lemma~\ref{lem:DT}, applied to $\sigma^{-1}m$, a positive multiple of $m$ is obtained by a relabelling from a point on one of the following families:
\begin{enumerate}[label=\textup{(\roman*)}]
\item $T$ itself, if $T$ is $T_1$, $T_2$, $T_3$, $T_5$ or $T_7$;
\item $T_4$ or $T_2$, if $T=T_4$, and $T_8$ or $T_5$, if $T=T_8$;
\item $\mathrm{RL}$, if $T$ is $\mathrm{RL}$ or $T_9$. For $T_9$ the other case of Lemma~\ref{lem:DT}, a point of $T_9$, is excluded by hypothesis.
\end{enumerate} As in the proof of Corollary~\ref{cor:S}, positive scaling and relabelling of the masses preserve finiteness and settledness. Theorems~\ref{thm:T3}, \ref{thm:T1T7}, \ref{thm:T2} and~\ref{thm:RL} and Proposition~\ref{prop:T458}, which hold on the whole families, give the claim. The points of these families that are only counted are the points of the special rays, excluded here, and the positive multiples of permutations of $P_1$, which are counted points.
\end{proof}

%% file: sec/s9_rays.tex
\section{The special rays}\label{sec:rays}

Recall the points $\bone$ and $\btwo$ of \S\ref{ss:reduction}. At these two points the separation of \S\S\ref{sec:strata}--\ref{sec:curves} fails, and at $\bone$ it has to fail (\S\ref{sec:intro}). On the rays through these two mass vectors and their relabellings we therefore count the real central configurations directly, following Moczurad and Zgliczy\'nski \cite{MZ19,MZ26}: a branch and bound in a reduced system with interval arithmetic, exclusion tests derived from exact identities, and Krawczyk's test \cite{Kra69,Moo77,Neu90}. With the same method we also count the central configurations at one further mass point needed in \S\ref{sec:curves} (\S\ref{ss:twopoints}).

\subsection{Statements}\label{ss:raystatements}
In this section a configuration is $q=(q_1,\dots,q_5)\in(\R^2)^5$ outside the collision set $\Delta$, with $q_i=(\xi_i,\eta_i)$, $r_{ij}=|q_i-q_j|$, $M=\sum_im_i$ and
\[
F_i(q)=q_i-\sum_{j\ne i}m_j\,\frac{q_i-q_j}{r_{ij}^3}\qquad(i=1,\dots,5).
\]
A \emph{normalized central configuration} is a zero of $F=(F_1,\dots,F_5)$, that is, a solution of \eqref{eq:cc} with $\lambda=1$ and $c=0$. As in the proof of Lemma~\ref{lem:realpos}, every central configuration is similar, by an orientation-preserving similarity, to a normalized one, and two central configurations are orientation-preserving similar if and only if their normalizations lie in the same $SO(2)$-orbit. The group $SO(2)$ acts freely on normalized central configurations. Thus the central configurations up to orientation-preserving similarity correspond to the $SO(2)$-orbits of normalized central configurations. The reflection $(\xi,\eta)\mapsto(\xi,-\eta)$ acts on these orbits as an involution whose fixed orbits are the collinear ones. Hence, with $N$ orbits of which $N_c$ are collinear, there are $(N+N_c)/2$ central configurations up to similarity.

Rotations give $DF(q)\,Jq=0$ at a normalized central configuration $q$, where $J$ rotates each $q_i$ by $\pi/2$. Hence $\rank DF(q)\le9$. We call $q$ \emph{nondegenerate} if $\rank DF(q)=9$. Since $DF(q)$ is, up to the mass matrix, the Hessian of $V=\tfrac12\sum_im_i|q_i|^2+\sum_{i<j}m_im_j/r_{ij}$, this is the usual notion \cite{Moe14}: $q$ is a nondegenerate critical point of the potential on the normalized configurations, modulo rotations.

For a non-collinear configuration let $P$ be the convex hull of $q_1,\dots,q_5$. It is \emph{strictly convex} if all bodies are vertices of $P$, and of type $(4,1)$ or $(3,2)$ if exactly four, resp.\ three, bodies are vertices of $P$ and the others are interior points of $P$.

\begin{theorem}\label{thm:count1}
Let $m$ be a positive multiple of a permutation of $(1,4,4,4,4)$.
\begin{enumerate}[label=\textup{(\roman*)}]
\item There are exactly $306$ central configurations up to orientation-preserving similarity, and exactly $183$ up to similarity.
\item Every central configuration is nondegenerate.
\item Of the $306$, exactly $60$ are collinear, $24$ strictly convex, $102$ of type $(4,1)$ and $120$ of type $(3,2)$. In every non-collinear central configuration each body is a vertex or an interior point of the convex hull, and each of the $24$ oriented cyclic orders of the five bodies is the counterclockwise order of exactly one strictly convex one.
\item There are exactly $612$ positive normalized central configurations in the sense of \cite[Def.~2]{AK12}.
\end{enumerate}
\end{theorem}

\begin{theorem}\label{thm:count2}
Let $m$ be a positive multiple of a permutation of $(1,1,4,4,4)$. Then statements \textup{(i)--(iv)} of Theorem~\ref{thm:count1} hold with the numbers $294$ and $177$ in \textup{(i)}; $60$, $24$, $90$ and $120$ in \textup{(iii)}; and $588$ in \textup{(iv)}.
\end{theorem}

The $60$ collinear configurations agree with Moulton's count \cite{Mou10}. Statement (iv) follows from (i) by Lemma~\ref{lem:realpos}. By (iv), the positive solutions of the system of \cite{AK12} at $(4,4,4,4,1)$ are finite in number, although the system has a continuum of real solutions there (\S\ref{sec:intro}). At this mass, finiteness therefore rests on the signs of the distances, as anticipated in \S\ref{sec:intro}.

\subsection{The reduced system and the a priori bounds}\label{ss:raysmethod}
Every normalized central configuration satisfies the following identities. Summing $m_iF_i$ gives $\sum_im_iq_i=0$. Taking the scalar product of $F_i=0$ with $m_iq_i$ and summing gives $I=U$, where $I=\sum_im_i|q_i|^2$ and $U=\sum_{i<j}m_im_j/r_{ij}$. With Lagrange's identity $MI=\sum_{i<j}m_im_jr_{ij}^2$, this is the virial identity $\sum_{i<j}m_im_j(r_{ij}^2/M-1/r_{ij})=0$. For a nonempty proper subset $C$ of the bodies, summing $m_iF_i$ over $C$ gives the cluster identity $\sum_{i\in C}m_iq_i=\sum_{i\in C,\,j\notin C}m_im_j(q_i-q_j)/r_{ij}^3$. Let $c_C$ be the centre of mass of $C$. The scalar products of $m_iF_i=0$ with $q_i-c_C$, summed over $C$, give the virial identity of the cluster $C$.

\begin{lemma}\label{lem:bounds}
Let $q$ be a normalized central configuration and $R=\max_i|q_i|$.
\begin{enumerate}[label=\textup{(\alph*)}]
\item Some $r_{ij}$ is at least $M^{1/3}$. Hence $R\ge M^{1/3}/2$.
\item For every $\varepsilon>0$, $R<\max\{4\varepsilon,\ \max_{1\le s\le4}[(s-1)\varepsilon+W_s/\varepsilon^2]\}$, where $W_s$ is $M$ minus the sum of the $s$ smallest masses.
\item $r_{ij}\ge m_im_j/(MR^2)$ for all $i\ne j$.
\end{enumerate}
\end{lemma}

\begin{proof}
(a) If all $r_{ij}<M^{1/3}$, all terms of the virial identity are negative. (b) Let $G$ be the graph with an edge $ij$ when $r_{ij}\le\varepsilon$, and $C$ the component of a body $i_0$ with $|q_{i_0}|=R$, with $s$ bodies. If $s=5$, then $MR=|\sum_jm_j(q_{i_0}-q_j)|<4M\varepsilon$. Otherwise the scalar product of the cluster identity for $C$ with $q_{i_0}/R$ is at least $M_C(R-(s-1)\varepsilon)$ on the left and less than $M_C(M-M_C)/\varepsilon^2$ on the right. (c) $m_im_j/r_{ij}\le U=I\le MR^2$.
\end{proof}

Part (b) sharpens \cite[Thm.~11]{MZ19}: at $(1,4,4,4,4)$ that theorem gives about $10.11$, whereas ours gives $6.352$. Since the boxes $B_k$ below have sides set by this bound, the sharper bound shrinks the region searched. Table~\ref{tab:raycount} lists the values of $\varepsilon$ that we use and the resulting bounds.

Fix real constants $c_1,\dots,c_5$. A normalized central configuration lies in \emph{region} $k$ if $|q_k|^2+c_k\ge|q_i|^2+c_i$ for all $i$. Every one lies in some region. The constants only break ties: a symmetric configuration often has two bodies at the same distance from the centre of mass. With $c=0$ it would lie on the boundary of two regions. By Lemma~\ref{lem:bounds}(a), in region $k$ we have $|q_k|^2\ge\mu_k:=M^{2/3}/4+\min_ic_i-c_k$, which is positive for the constants we use. Hence $q_k\neq0$ and exactly one rotation puts $q_k$ on the positive $\xi$-axis. For a second body $e\ne k$ and the other three bodies $f_1,f_2,f_3$, the chart $\Phi_{k,e}\colon\R^7\to(\R^2)^5$, $v\mapsto q$ with
\[
q_k=(v_0,0),\qquad q_{f_a}=(v_{2a-1},v_{2a})\ (a=1,2,3),\qquad q_e=-\frac1{m_e}\sum_{i\ne e}m_iq_i,
\]
is an isomorphism onto $\{\sum_im_iq_i=0,\ \eta_k=0\}$, and the \emph{reduced system} $RS(v)$ consists of the seven components $F_{f_1},F_{f_2},F_{f_3}$ and the first component of $F_k$ at $\Phi_{k,e}(v)$ \cite[Thm.~2]{MZ26}, \cite[Thm.~3]{MZ22}.

\begin{lemma}\label{lem:RS}
Let $q=\Phi_{k,e}(v)\notin\Delta$.
\begin{enumerate}[label=\textup{(\alph*)}]
\item If $RS(v)=0$ and $\xi_k\ne\xi_e$, then $q$ is a normalized central configuration. Conversely, every normalized central configuration with $\eta_k=0$ is $\Phi_{k,e}(v)$ for a zero $v$ of $RS$.
\item If $q$ is a normalized central configuration with $\xi_k\ne0$ and $DRS(v)$ is invertible, then $q$ is nondegenerate.
\end{enumerate}
\end{lemma}

\begin{proof}
(a) For any configuration $\sum_im_iF_i=\sum_im_iq_i$ and $\sum_im_i\,q_i\times F_i=0$, with $a\times b=a_1b_2-a_2b_1$. If $RS(v)=0$, then $F_i=0$ for $i\ne k,e$. The first identity then gives $m_kF_k+m_eF_e=0$, and the second $m_k(q_k-q_e)\times F_k=0$. As $q_k-q_e=(\xi_k-\xi_e,-\eta_e)$ and the first component of $F_k$ vanishes, this is $m_k(\xi_k-\xi_e)F_{k,2}=0$. Hence $F_k=F_e=0$. The converse holds because the centre of mass of a normalized central configuration is $0$. (b) Differentiating the first identity gives $\sum_im_i(DF(q)u)_i=\sum_im_iu_i$ for all $u$. Since $\xi_k\neq0$, the vector $Jq$ does not lie in the image $L$ of $\Phi_{k,e}$. Hence $\R^{10}$ is the direct sum of $L$, the line spanned by $Jq$ and the translations. Let $u=\ell+\alpha Jq+\tau\in\ker DF(q)$ accordingly. Then $0=DF(q)\ell+\tau$, because $DF(q)Jq=0$ and $DF(q)\tau=\tau$. Taking the mass-weighted sum, we obtain $M\tau=0$. Thus $DF(q)\ell=0$, hence $DRS(v)\,\Phi_{k,e}^{-1}(\ell)=0$ and $\ell=0$. Therefore $\ker DF(q)$ is spanned by $Jq$.
\end{proof}

Consequently, in region $k$ the chart point of the rotated configuration lies in the box $B_k=[R_{\min,k},R_{\max}]\times[-R_{\max},R_{\max}]^6$, with $R_{\min,k}\le\sqrt{\mu_k}$ positive and $R_{\max}$ the bound of Lemma~\ref{lem:bounds}(b), and it is a zero of $RS$.

\subsection{The search and the count}\label{ss:raysearch}
We search the boxes $B_k$ by branch and bound in outward-rounded interval arithmetic: every operation of IEEE~754 binary64 arithmetic is computed in round-to-nearest, after which the lower endpoint is moved down and the upper endpoint up by one unit in the last place, and no multiplication is fused with an addition. Every computed interval therefore contains the exact value of the corresponding real expression on the box. We discard a box if one of the following tests shows that it contains no zero of $RS$ that is a normalized central configuration in region $k$:
\begin{enumerate}[label=\textup{(T\arabic*)},start=0]
\item an enclosure of $|q_i|^2+c_i-|q_k|^2-c_k$ is positive for some $i$;
\item the enclosure of the virial sum excludes $0$;
\item for some proper subset $C$ all of whose cross pairs are separated on the box, an enclosure of the difference between the two sides of the cluster identity excludes $0$;
\item for such a $C$ with $|C|\ge2$, the enclosures of the two sides of the virial identity for $C$ are disjoint;
\item[(MV)] for some component $F_r$ of $F\circ\Phi_{k,e}$, the mean-value enclosure $F_r(c)+\sum_t\partial_t(F_r\circ\Phi_{k,e})(X)\,(X_t-c_t)$ at the midpoint $c$ excludes $0$.
\end{enumerate}
Since the Newtonian kernel $(\xi,\eta)\mapsto\xi/(\xi^2+\eta^2)^{3/2}$ has no critical point, on a rectangle not containing the origin its extrema are attained at the corners, at the points with $\eta=0$ of its vertical sides, or at the critical points $(\pm|\beta|/\sqrt2,\beta)$ of its restrictions to the horizontal sides $\eta=\beta\ne0$, where it takes the values $\pm\frac{2}{3\sqrt3}\beta^{-2}$. The enclosure on a rectangle is the hull of its enclosed values at those of these points that lie in the rectangle. On a box on which all pairs are separated we apply Krawczyk's test, with the interval Jacobian of $RS$ and a floating-point approximate inverse $C$ of the Jacobian at the midpoint (\S\ref{ss:masscert}). If one component of the operator is disjoint from the box, the box contains no zero. If all components lie in the interior of the box, $RS$ has exactly one zero $v^*$ in it and $DRS(v^*)$ is invertible \cite{Kra69,Moo77,Neu90}. Boxes that are neither discarded nor certified are contracted by the operator or bisected, at an off-centre cut so that symmetric configurations do not lie on cut faces. We tighten a certified box by further Krawczyk steps. On the tightened box the enclosure of $\xi_k-\xi_e$ excludes $0$, and the enclosures of $|q_i|^2+c_i-|q_k|^2-c_k$ are negative for all $i\ne k$.

\begin{proposition}\label{prop:raycount}
Let $m$ be $(1,4,4,4,4)$ or $(1,1,4,4,4)$. Each certified box contains exactly one zero $v^*$ of $RS$, and $\Phi_{k,e}(v^*)$ is a nondegenerate normalized central configuration lying in region $k$ and in no other region. The map from certified boxes to $SO(2)$-orbits of normalized central configurations is a bijection. There are $306$ certified boxes for $(1,4,4,4,4)$ and $294$ for $(1,1,4,4,4)$, distributed over the regions as in Table~\ref{tab:raycount}.
\end{proposition}

\begin{proof}
For each mass vector we search the regions in subproblems whose initial boxes tile $B_k$. We verify the tiling, the inequalities $R_{\min,k}^2\le\mu_k$ and the bound $R_{\max}$ in exact rational arithmetic. In all subproblems each leaf box is discarded or certified, and no subproblem ends abnormally (\S\ref{ss:checks}). A box is discarded only if it contains no zero of $RS$ that is a normalized central configuration in region $k$, contractions keep every zero of $RS$, and bisections cover the box. Hence every chart point of a normalized central configuration in region $k$ lies in a certified box. By Krawczyk's test the certified box contains exactly one zero $v^*$, in its interior. It is a normalized central configuration by Lemma~\ref{lem:RS}(a), since $\xi_k\ne\xi_e$ on the tightened box, and nondegenerate by Lemma~\ref{lem:RS}(b), since $\xi_k\ge R_{\min,k}>0$. It lies strictly in region $k$, hence in no other region. Each orbit is thus the image of a certified box (surjectivity). Two certified boxes with the same orbit come from the same region and, since $q_k$ lies on the positive $\xi$-axis in both, give the same chart point, which lies in the interior of both boxes. The boxes of one subproblem are leaves in a tree of bisections and contractions. Boxes from different subproblems lie in different pieces of the tiling. The two boxes therefore coincide (injectivity). The counts are those of the search (Table~\ref{tab:raycount}).
\end{proof}

\begin{table}[ht]
\centering
\small
\begin{tabular}{@{}lllrrrrrrrr@{}}
\toprule
& & & \multicolumn{5}{c}{region} & & & boxes\\
\cmidrule(lr){4-8}
masses & $\varepsilon$ & bound & $1$ & $2$ & $3$ & $4$ & $5$ & total & collinear & examined\\
\midrule
$(1,4,4,4,4)$ & $\frac{397}{250}$ & $\frac{794}{125}$ & $168$ & $28$ & $38$ & $18$ & $54$ & $306$ & $60$ & $211{,}263{,}506$\\
$(1,1,4,4,4)$ & $\frac{1587}{1000}$ & $\frac{5332323001}{839523000}$ & $90$ & $150$ & $18$ & $24$ & $12$ & $294$ & $60$ & $123{,}212{,}100$\\
\bottomrule
\end{tabular}
\caption{The searches of Proposition~\ref{prop:raycount}: the value of $\varepsilon$ and the bound of Lemma~\ref{lem:bounds}(b) at it, which the binary64 number $R_{\max}$ is at least; the certified boxes per region; the collinear boxes; and the number of boxes examined.}\label{tab:raycount}
\end{table}

We certify the shapes in ball arithmetic (Arb, at $200$ bits, through python-flint 0.9.0) on the tightened boxes, after checking that all balls in these computations are finite, that all divisors exclude $0$ and that all square-root arguments are positive. We enclose the orientations of all triples of bodies on the box. If some orientation has a sign, the configuration is not collinear, and each body passes one of two tests: a body is a vertex of the convex hull if the orientations certify that the two hull edges proposed at it leave all other bodies strictly on their left. It is an interior point if they certify that it lies in the open triangle of three other bodies, or in the interior of a quadrilateral formed by the four others, as when the four bodies of mass $4$ form a square with the body of mass $1$ at its centre. If no orientation has a sign, we enlarge the box to one that is symmetric under the reflection $(v_2,v_4,v_6)\mapsto(-v_2,-v_4,-v_6)$ of the chart, which maps the zero set of $RS$ to itself, and Krawczyk's test succeeds on the enlarged box. Its unique zero then coincides with its mirror image. So all five bodies lie on the $\xi$-axis. In this way we show that every certified box is collinear, or non-collinear with a certified shape. Table~\ref{tab:raycount} gives the numbers of collinear boxes.

\begin{proof}[Proof of Theorems~\ref{thm:count1} and~\ref{thm:count2}]
By Proposition~\ref{prop:raycount} there are $306$, resp.\ $294$, orbits, all nondegenerate. With the $60$ collinear ones, we obtain $183$ and $177$ up to similarity. The shape counts are those of the classification. The reflection pairs the non-collinear orbits with equal shapes and reversed cyclic orders. As in the proof of Lemma~\ref{lem:slice}, positive scaling of the masses multiplies the normalized central configurations by $t^{1/3}$, and permutations of the masses relabel them. Since both preserve nondegeneracy, collinearity and shapes, the statements hold on the whole rays and for all relabellings.
\end{proof}

\begin{remark}
By nondegeneracy and the implicit function theorem, the numbers in Theorems~\ref{thm:count1} and~\ref{thm:count2} persist on open neighbourhoods of the two mass vectors that are invariant under positive scaling.
\end{remark}

\subsection{A further point}\label{ss:twopoints}
We treat the mass point $P_1=(9,36,36,36,4)$ of \S\ref{sec:curves} by the same count instead of the separation.

\begin{theorem}\label{thm:countP1}
Let $m$ be a positive multiple of a permutation of $P_1$. Then statements \textup{(i)--(iv)} of Theorem~\ref{thm:count1} hold with the same numbers.
\end{theorem}

The proof of Theorem~\ref{thm:countP1} is that of Theorem~\ref{thm:count1}. Here some coefficients of the chart, $\pm1/9$ and $\pm10/9$, are not binary64 numbers. So we enclose them by intervals. One certified box turns out to have its zero outside its region. This does not affect the count: we do not count that box, since the surjectivity in the proof of Proposition~\ref{prop:raycount} counts every orbit by a box of a region that contains it.

\begin{remark}\label{rem:P1hand}
The method of \cite{Tom26} also gives finiteness at $P_1$ without the count of Theorem~\ref{thm:countP1}. At every positive multiple $m$ of a permutation of $P_1$, the function $U$ vanishes identically on every irreducible curve in the slice $Y_m$ of \S\ref{ss:slice}. So $m$ is settled, and it has finitely many central configurations up to similarity. This is part of \cite{Tom26}.
\end{remark}

\subsection{The complex solution sets on the rays}\label{ss:raysAB}
\begin{remark}\label{rem:AB}
At $m=(1,4,4,4,4)$ the slice $Y_m$ of \S\ref{ss:slice} is the union of $36$ pairwise disjoint irreducible curves and a finite set. Through the correspondence of Lemma~\ref{lem:slice}(a) between the points of $Y_m$ and the solutions of \eqref{eq:AK4}, the curves contain the modified Roberts continuum of \S\ref{sec:intro}, carried to these masses by the exchange of the bodies $1$ and $5$ (Lemma~\ref{lem:slice}(b)), and no point of the curves corresponds to a positive normalized central configuration. At $m=(1,1,4,4,4)$ the slice $Y_m$ is finite. Both statements are proved in \cite{Tom26}. They give a second proof, by a different method, that every point of the special rays has finitely many central configurations up to similarity, but they do not give the counts.
\end{remark}

%% file: sec/s10_proof.tex
\section{Proof of Theorem~\ref{thm:main}}\label{sec:proof}

Let $m\in(\R_{>0})^5$. If $m$ lies on a special ray, Theorems~\ref{thm:count1} and~\ref{thm:count2} give exactly $306$ or $294$ central configurations up to orientation-preserving similarity. Otherwise Corollary~\ref{cor:S} gives finitely many up to similarity. In its proof we apply Theorem~\ref{thm:S} to the sorted scaled image $\hat m$ of $m$:
\begin{enumerate}[label=\textup{(\roman*)}]
\item if $\Lambda(\hat m)$ is certified in the sense of \S\ref{ss:nearlabels}, Proposition~\ref{prop:cert} with $V=\Vtwo^{\mathrm w}$ settles $\hat m$, the pieces $\pc{d2}$, $\pc{d9}$, $\pc{d14b}$ and $\pc{d15}$ being removed by Proposition~\ref{prop:nod2}, Theorem~\ref{thm:D9}, Theorem~\ref{thm:D14} and Corollary~\ref{cor:nod15};
\item if $\Lambda(\hat m)$ contains a curve set, Proposition~\ref{prop:farT9} or Theorem~\ref{thm:curves} applies, and at the counted points the latter rests on Theorem~\ref{thm:countP1};
\item if $\hat m$ is $\mathbf q_*$ or $\mathbf p_*$, Proposition~\ref{prop:cert} with $V=\VN$ settles $\hat m$, by Proposition~\ref{prop:Spoints}. \qed
\end{enumerate}

\medskip
Section~\ref{sec:verification} lists the computations on which the proof rests and how they are checked.

%% file: sec/s11_verification.tex
\section{Verification}\label{sec:verification}

The proof of Theorem~\ref{thm:main} consists of the arguments in the text and the computations they invoke, classified in \S\ref{ss:checks}. It does not use the cross-checks of \S\ref{ss:crosschecks} or Remarks~\ref{rem:P1hand} and~\ref{rem:AB}, nor any re-check these arguments do not invoke. The certificates, the programs that check them, their logs and the supplementary notes are in the supplementary material, with a checksum for each file. It is available at \url{https://doi.org/10.5281/zenodo.23202861}.

Some objects have other names in the supplementary material. The points $\hone$, $\htwo$ and $\hthree$ (Tables~\ref{tab:points} and~\ref{tab:noK}) are \texttt{X52}, \texttt{X29} and $48_1$ there, and the points of Table~\ref{tab:points} appear under the point records in its last column. The relations (R1), (R2) and (R3) of \S\ref{ss:d5} are called gen, det and tilde there and are proved as Theorems G, D and T of its module \texttt{d5fam}; the model of the centred case (\S\ref{ss:d5}) is that of its Theorem~C0. Theorem~\ref{thm:gate} is Theorem G of its module \texttt{nod2route}. Theorem~\ref{thm:S} is Theorem S of its module \texttt{strata}, in which $\mathbf q_*$ and $\mathbf p_*$ are zeros of the flats \texttt{F89} of the face $1|2|3|4|5$ and \texttt{F49} of the face $1|2|3|45$. The points $\mathbf q_*$ and $\mathbf p_*$ also appear under the point records $52_0$ and $47_0$. The jump points of $T_9$ (Proposition~\ref{prop:farT9}) fall there into numbered classes: the point $s=\rho_0/4$ is class $11$, and the two points $s\approx0.0520549$ and $s\approx0.176101$ form class $14$. The table of re-checks is the file \texttt{CHECKS.md}.

\subsection{Literature}\label{ss:literature}
We list the inputs from \cite{AK12} in \S\ref{ss:akinputs}.

\subsection{Classes of computations}\label{ss:checks}
\begin{enumerate}[label=(\arabic*)]
\item \emph{Exact rational arithmetic.} The lexicographic certificates (\S\ref{sec:lex}) are lists of integer functionals. We recompute the attained faces of the pieces in rational arithmetic and check every inequality in integers. The checks take the hulls of \S\ref{ss:d5} from files. We verify in exact arithmetic that each polytope in these files is the polytope of Table~\ref{tab:d5hulls} for its pair of modes, or its image under Lemma~\ref{lem:d5sym}. We make the exact algebraic decisions of class (C2) (\S\ref{ss:masscert}) in exact arithmetic over $\Q$ and over number fields $\Q[\theta]/(f)$.
\item \emph{Ball arithmetic.} We show that polynomials, such as those of labels, do not vanish at exact mass points, and we isolate real roots. We certify isolated zeros of polynomial systems by Krawczyk's test \cite{Kra69,Moo77,Neu90}. We use Arb in FLINT 3.6.0, through python-flint 0.9.0.
\item \emph{Standard bases.} We prove that a polynomial system has no solution by a standard basis over $\Q$ or $\Q(\jj)$ (Singular 4.4.1), or modulo large primes together with a lifting argument (\S\ref{sec:points}). In Theorem~\ref{thm:typeII} we bound the number of solutions of a system by a standard basis over $\Q$. In \S\ref{sec:strata}, by standard bases, minimal primes and saturations over $\Q$, we describe the points of the strata at which given labels hold. We use msolve to search for solutions, version 0.9.4 for every run whose output enters a proof. Where its output enters a proof we verify it: we check the rational parametrizations against the input polynomials and prove their completeness separately.
\item \emph{The special rays.} We obtain the counts of \S\ref{sec:rays} from a branch and bound in double precision with outward rounding.
\end{enumerate}
We verify each lexicographic certificate of the paper with a program different from the one that found it. The \emph{independent checker} of lexicographic certificates is a program written from the definitions and the tables of this paper; it shares no code with the search programs. It re-verifies in particular the certificates of the label sets in the proof of Theorem~\ref{thm:S} and the two certificates of Proposition~\ref{prop:Spoints}(c). The re-checks, with their counts, are tabulated in the supplementary material. They cover the computations of the proof except those in Table~\ref{tab:notrechecked}.

\begin{table}[ht]
\centering
\small
\begin{tabular}{@{}p{0.36\textwidth}p{0.58\textwidth}@{}}
\toprule
computation & not re-checked\tabularnewline
\midrule
\raggedright the flats of the computation (Proposition~\ref{prop:Stree}) & \raggedright a second enumeration of the flats from scratch\tabularnewline[2pt]
\raggedright the conditions on the primes, where (P1) is re-checked for the runs of Theorem~\ref{thm:KA} & \raggedright the other conditions, which are checked within the runs\tabularnewline[2pt]
\raggedright the searches of Proposition~\ref{prop:raycount} and Theorem~\ref{thm:countP1} & \raggedright the upper bounds, which rest on the exclusion tests and on the outward rounding of the search\tabularnewline
\bottomrule
\end{tabular}
\caption{Parts of the computations of the proof that the re-checks leave out.}\label{tab:notrechecked}
\end{table}

The following results rest also on computations stated in full in their proofs: Proposition~\ref{prop:eqfull}, Theorems~\ref{thm:W}, \ref{thm:typeII} and~\ref{thm:col}, Lemma~\ref{lem:Pdata}, Theorems~\ref{thm:points} and~\ref{thm:gate}, Propositions~\ref{prop:Stree}, \ref{prop:Spoints} and~\ref{prop:farT9}, Theorem~\ref{thm:SJ}, Lemma~\ref{lem:K14FAFB}, Proposition~\ref{prop:d14a}, Theorem~\ref{thm:RL}, and Theorems~\ref{thm:count1}, \ref{thm:count2} and~\ref{thm:countP1}. For the following results we list the computations here.
\begin{itemize}
\item The relations (R1)--(R3) of \S\ref{ss:d5}: an elimination whose generator has only positive coefficients for (R1), $36$ eliminations by standard bases over $\Q$ for (R2), and a polynomial identity for (R3).
\item Proposition~\ref{prop:d6P2}: the checks (i)--(iv) of its proof, made as stated there in exact arithmetic, by a standard basis over $\Q$ and in ball arithmetic.
\item Proposition~\ref{prop:d3tstar}: the checks (i)--(iii) of its proof, by exact rational bisection and in ball arithmetic.
\end{itemize}

We apply the certificates that use the hull $\VN$ of \S\ref{ss:d5} only at points where $\VN$ is valid. Eleven of the certificates in the proof of Proposition~\ref{prop:exceptional}, used at ten jump points of $T_1$, $T_4$ and $T_5$ and at ten jump points of $T_8$, are relative to $\VN$. All others there are relative to $\Vtwo^{\mathrm w}$. At these jump points, at the two point records of Table~\ref{tab:points} whose final checks use $\VN$, and at $\mathbf q_*$ and $\mathbf p_*$, we check in ball arithmetic that the point lies off the exceptional hypersurface of \S\ref{ss:d5} and that $m_o\ne m_a+m_b$ on every triangle $T=\{o,a,b\}$ with $\lab{E}(T)$ in the label set. We also show in ball arithmetic that no body of such a triangle is at its centre of mass, except at the triangles in the centred case, whose masses are exactly proportional to $(\frac14,1,1)$. Each factor of the hypersurface needed at these triangles is enclosed in a ball that excludes $0$.

This point check covers $88$ applications of $\VN$, including all those used in the proof. Each of them lies off the factors of the hypersurface needed at its triangles (in the centred case of \S\ref{ss:d5} only those of (R2) for the two other bodies of $T$, by the model of that case in the supplementary material) and satisfies the second condition. Thirteen of them have a triangle of masses $(\mu/4,\mu,\mu)$, in the centred case.

The enumerations of the cells of the prevarieties of \S\ref{ss:Ksystems} make $13609$ linear-programming decisions for $\mathrm{K14}$ and $32010$ for each system $\mathrm{K15}$. The final checks of the point records of Table~\ref{tab:points} concern $4$ distinct pairs $(S,T)$.

\subsection{Cross-checks}\label{ss:crosschecks}
\begin{itemize}
\item \emph{Pieces $\pc{d15}$, independently of Corollary~\ref{cor:nod15}.}
\begin{enumerate}[label=\textup{(\alph*)}]
\item At the points of Table~\ref{tab:points}, and at the other mass points of \S\ref{sec:curves} at which $\mathrm{K14}$ is used in the proof, the system $\mathrm{K15}(m_{Q_j},k)$ has no solution, by runs of Theorems~\ref{thm:KA} and~\ref{thm:KApar}(b), for each piece $\pc{d15}(j,k)$ removed by the certificate of the point. At the points of Table~\ref{tab:points} these runs and those of $\mathrm{K14}$ exclude all pieces $\pc{d14a}$, $\pc{d14b}$ and $\pc{d15}$ of the clusters of the transversals from the cluster systems alone, independently of \S\ref{sec:energy}.
\item The system $\mathrm{K15}((1,1,1,\nu),k)$ in the normalization of \S\ref{ss:Ksystems}, with $\nu$ as an unknown, has exactly $768$ solutions when $k$ is the body of mass $\nu$ and exactly $1448$ when $k$ has mass $1$. They lie over the roots of irreducible polynomials of degrees $128$ and $724$, each with exactly two positive roots. The proof counts the solutions modulo a prime and excludes the open cells of the prevariety by a local analysis, at two primes.
\item At the four resulting values of $\mu$ on the line $\mathrm{RL}$, and at the two values of $t$ on $T_2$ at which $t^2$ is a positive root of the second polynomial, a nonzero coefficient of order at most $3$ in a weighted form of the rows excludes a place of type $\pc b_k(j)$ at each leading point, $288$ leading points when $k$ is the body of mass $\nu$ and $96$ otherwise, at three primes. With the previous item, this excludes the pieces $\pc{d15}(2,k)$ on $\mathrm{RL}$ and $\pc{d15}(1,k)$, $k=3,4,5$, on $T_2$.
\item No solution of $\mathrm{K15}((1,1,t,t),k)$, $t\in\C^*$, has a bad shape. For $k=1$, all solutions of this system with $t$ as an unknown, in a normalization, have unit coordinates at the prime $2147483629$, by its prevariety and a local analysis at two of its cells. Two standard bases modulo that prime, of the system with the condition of a bad shape added, are the unit ideal. The other values of $k$ reduce to $k=1$. This excludes the pieces $\pc{d15}$ of the transversal cluster on $T_1$ and on $T_7$.
\item Along $T_9$, let $Z_k$ be the set of $s$ at which $\mathrm{K15}(m_{Q_1},k)$ has a torus zero. Then $Z_2=Z_1$, by the exchange of the two bodies of mass $\rho_0$, and $Z_4\cap Z_1=Z_4\cap Z_3=\emptyset$, by the cells of the three systems modulo $2147483629$ and by polynomials $u_4$, $u_1$, $u_3$ in $s$ in their ideals with $\gcd(u_4,u_1)=\gcd(u_4,u_3)=1$. Hence at every point of $T_9$ one of the two certificates for the generic label set applies without Corollary~\ref{cor:nod15}.
\end{enumerate}
\item \emph{Proposition~\ref{prop:eqfull}.} A standard basis over $\Q(\jj)$ of the system of Proposition~\ref{prop:eqfull}, with $\nu$ as an unknown and saturated by $\prod p_{il}\cdot\nu(3+\nu)$, is $\{1\}$. This confirms the proposition whenever $\nu(3+\nu)\ne0$. An independent encoding of the system gives the same result. With $c=d=0$ it finds the expected solutions of $\mathrm{K14}$.
\item \emph{Theorem~\ref{thm:D14} as used in Proposition~\ref{prop:T458}.} On the families $T_4$, $T_5$ and $T_8$ we confirm the exclusion of the pieces $\pc{d14b}(1,k)$ of the transversal cluster $Q_1$, for $k=2,\dots,5$ on $T_4$ and $k=4,5$ on $T_5$ and $T_8$, with Lemma~\ref{lem:d14blead}, the description of the configurations that leads to \eqref{eq:tsys}, and Lemma~\ref{lem:typeIII}, but without the rest of the proof of Theorem~\ref{thm:D14}. These are the pieces $\pc{d14b}$ of the minimal exclusion sets of \S\ref{ss:T458} that use only $Q_1$. By Lemma~\ref{lem:d14blead}, at a place with type in the attained set of $\pc{d14b}(1,k)$ the leading coefficients of the cluster form a $\mathrm{K14}(m_{Q_1})$ configuration that satisfies $\mathrm{X3}_k$ (\S\ref{ss:d14}), after the exchange of the colours if necessary: body $k$ lies at the $z$- or the $w$-centre of mass of $Q_1$. We show that at no point of these families does a $\mathrm{K14}(m_{Q_1})$ configuration satisfy $\mathrm{X3}_k$ in $z$ or in $w$. With $d_{il}=z_{il}w_{il}$ and $M_Q$ the mass of the cluster, the parallel-axis identity shows that $\mathrm{X3}_k$, in $z$ or in $w$, is equivalent to $M_Q\sum_lm_ld_{kl}=\sum_{i<l}m_im_ld_{il}$. By Lemma~\ref{lem:typeIII}, a non-collinear configuration at positive masses has a pattern of type I or II. By complex conjugation four of these seven classes suffice. At positive masses $\kappa\ne0$ in \eqref{eq:tsys} and the masses are then rational functions of $t$. A collinear configuration solves one of $32$ tournament systems $\sum_{l\ne i}m_ls_{il}/z_{il}^2=0$, $s_{il}=\pm1$, linear in the masses. In suitable labels $T_4$ and $T_5$ lie in the plane $m_3=m_4$ of the cluster masses. An explicit resultant $P_T$, of degree $6$ for $T_4$ and $12$ for $T_5$, vanishes on each family. For each of the four classes, every needed $k$ and each tournament we compute with Singular~4.4.1, in characteristic $0$, a standard basis over $\Q(\jj)$ of the ideal $J$ of the configurations with $\mathrm{X3}_k$ in this plane, and show that $P_T$ is a unit modulo $J$. We establish each unit statement by an exact standard basis of $J+P_T$ in characteristic $0$. Those for the four classes we establish also modulo the prime $p'=536870879$, by the following modular test: if a Gr\"obner basis of $J$ has $p'$-integral coefficients \textup{(H2)} and leading coefficients that are nonzero modulo $p'$ \textup{(H3)}, then its reduction is a Gr\"obner basis of the reduced ideal with the same staircase, the multiplication matrices reduce accordingly, and a unit modulo $p'$ is a unit. On $T_8$, where no two masses are identically equal, every point is a zero of an explicit polynomial $P_{T_8}$ in the masses of the bodies $3$ and $5$ and the parameter $r$. We write the systems in the same way, with the mass relations of $T_8$ in place of the plane. The ideal $J$ is then not saturated, so that $V(J)$ may contain points that are not configurations. For the collinear configurations and for three of the four classes, $P_{T_8}$ is a unit modulo $J$, by a standard basis in characteristic $0$ followed by the modular test at $p'$ and, except for two of the classes with $\mathrm{X3}_4$, at the prime $p=536870909$. For the fourth class $P_{T_8}$ is not a unit modulo either prime, but it vanishes on $V(J)$ only where a coordinate $t_i$ vanishes. At each configuration $t_i\ne0$. For this class we use the following extension of the modular test, with $f=t_i$. Let $A=\Q(\jj)[x]/J$, and let $\bar a$ be the dimension of the generalized eigenspace of multiplication by $\bar f$ for the eigenvalue $0$ on the reduction of $A$. If \textup{(F1)} $\dim A/f^NA=\bar a$ for some $N$ and \textup{(F2)} $\bar f$ is nilpotent modulo $\bar J+(\bar P_{T_8})$, then $P_{T_8}$ has no zero on $V(J)\setminus V(f)$. Indeed, the characteristic polynomial of multiplication by $f$ is $\lambda^a\chi_1$ with integral coefficients and $\chi_1(0)\ne0$. Condition \textup{(F1)} gives $a=\bar a$, so that $\chi_1(0)$ is a unit at the prime. So the idempotent of $A$ on which $f$ is invertible is integral and reduces to the corresponding idempotent of the reduction, on which $\bar P_{T_8}$ is invertible by \textup{(F2)}. Hence $P_{T_8}$ is invertible on the part of $A$ where $f$ is invertible. Here \textup{(F1)} holds with $\bar a=6$ and $N=3$, for $\mathrm{X3}_4$ and for $\mathrm{X3}_5$. Conditions \textup{(F2)}, \textup{(H2)} and \textup{(H3)} are checked at both primes. Each unit statement that we establish modulo a prime has a modular test in which we check \textup{(H2)} and \textup{(H3)} element by element. We derive the systems twice, independently. The two derivations agree.
\item \emph{Lemma~\ref{lem:typeIII}.} We check the identity for $h_{abc}$ in the proof of Lemma~\ref{lem:typeIII} by machine for each of the four sets $\{a,b,c\}$ of three bodies. A second proof of the lemma, by an explicit parametrization of the solutions of \eqref{eq:tsys} of type III, is in the supplementary material. We also ran the computations of the previous item for the classes of type III on $T_4$ and $T_5$, and on $T_8$ for $\mathrm{X3}_5$. They agree with the lemma. On $T_4$ and $T_5$, $P_T$ is a unit modulo $J$. On $T_8$, for $\mathrm{X3}_5$, the case $D=0$ for a polynomial $D$ of the class gives a unit. In the case $D\ne0$, where one unknown is eliminated through $D\ne0$, the extension of the modular test applies with $f=D$, $\bar a=192$ and $N=4$.
\item \emph{Theorem~\ref{thm:D14} as used in Theorems~\ref{thm:T1T7}, \ref{thm:T2}, \ref{thm:RL} and Proposition~\ref{prop:farT9}.} With Lemma~\ref{lem:d14blead} and the condition $\mathrm{X3}_k$, in $z$ or in $w$, but without the rest of the proof of Theorem~\ref{thm:D14}, we exclude the pieces $\pc{d14b}$ of the transversal cluster used in Theorems~\ref{thm:T1T7}, \ref{thm:T2} and~\ref{thm:RL} and Proposition~\ref{prop:farT9}. On $T_1$ and $T_7$, off the sixteen special values of $c$, $\mathrm{K14}(m_{Q_1})$ has no solution. At every configuration of Lemma~\ref{lem:K14FAFB} for $F_B$ we evaluate $\mathrm{X3}_k$ in $z$ and in $w$ in ball arithmetic for the four positions $k$; no ball contains $0$. If a $\mathrm{K14}(1,1,1,\nu)$ configuration with $\nu\ne0$ satisfies $\mathrm{X3}_k$, then $k$ is the body of mass $\nu$ and $3\nu^2=\pm1$. Indeed, in a normalization of the gauge either $z=w$ or the matrix $A$ of the derivation of \eqref{eq:tsys} has rank one. In each case $\mathrm{X3}_k$ adds one equation. We compute the eliminations with Singular~4.4.1, over $\Q(\jj)$ for the $2\times9\times2$ systems of the second case and over $\Q$ for the $2\times32$ of the first, and recompute all of them with Macaulay2~1.25.11. Only $k=4$ with $3\nu^2=1$ (one system) and $3\nu^2=-1$ (four systems) survive. This excludes $\pc{d14b}(1,k)$, $k=3,4,5$, on $T_2$, and all $\pc{d14b}(2,k)$ on $\mathrm{RL}$ except $\pc{d14b}(2,5)$ at $\mu=\mu^*=4/\sqrt3$. There $\mathrm{K14}$ at $(4,4,4,\mu^*)$ has exactly two configurations up to the gauge, the equilateral triangle of the bodies $1$, $3$, $4$ with body $5$ at its centre and its mirror image. The places on the segment of $\pc{d14b}(2,5)$ have exactly $96$ leading points, $48$ over each configuration. The exchange of the bodies $3$ and $4$ maps one half to the other. A weighted form of the rows that does not depend on the parameter $\sigma\in(1,2)$ of the segment, followed by a Lyapunov--Schmidt reduction, gives at each of the $48$ leading points over one configuration a series $g(\delta)$ that must vanish. Modulo two primes, with all inverses checked to be units, $g$ has a nonzero coefficient of order at most $5$. Along $T_9$, $\mathrm{K14}(m_{Q_1})$ together with the condition of Lemma~\ref{lem:d14blead} for body $k$ has no torus zero for any $s\in\C^*$, by Theorem~\ref{thm:KApar}(a) (four systems, two primes).
\item \emph{The inclusions in the proof of Theorem~\ref{thm:types}, all labellings.} For all $120$ labellings, each type set of Table~\ref{tab:types} lies in the attained set of its piece in Table~\ref{tab:pieces}, with diagram~$7$ refined by Table~\ref{tab:d7}. Three programs check this in exact rational arithmetic: for diagrams $1$, $4$, $7$, $8$ and $10$--$13$, for the refined diagram~$7$, and for diagrams $2$, $3$, $6$, $9$ and $14$--$16$.
\item \emph{Theorem~\ref{thm:hulls}.} Exact computations confirm Theorem~\ref{thm:hulls}. For each assertion of Propositions~\ref{prop:d5empty}--\ref{prop:d5VN}, and for each main step of their proofs, a list of relations of the model that contains those used in it, together with its negated conclusion, forms a system of linear relations and disjunctions. Two programs show that each of these systems is infeasible, by exact certificates. One program takes the relations from the programs that define the models; the other builds the models anew and uses its own exact linear programming. Independently, for every pair of modes of $\Vtwo^{\mathrm w}$, $\Vtwo$ and $\VN$, a complete branching over the disjunctions of the model, with an exact certificate at every node, shows that the projection of the model lies in the polytope of Table~\ref{tab:d5hulls}, or that the model is empty. The same holds for the model of the centred case on its four pairs of modes. A separate checker extracts each model again and verifies all these certificates. Moreover, every vertex $v$ of each polytope of Table~\ref{tab:d5hulls} is a limit of projections of model points: exact points $x_s$ of the model and $x_c$ of its closure, with $x_c$ projecting to $v$, have the half-open segment $(x_c,x_s]$ in the model. So each polytope is the convex hull of the closure of the projection of its model.
\end{itemize}

\subsection{Formalization}\label{ss:formal}
Theorem~\ref{thm:main} is proved in Lean~4 with Mathlib (project \texttt{lean\_full}, commit \texttt{52858e1}; toolchain v4.35.0-rc3, Mathlib commit \texttt{c55e6e7}). The formal statement asserts that for every $m\in(\R_{>0})^5$ there is a finite set of configurations in $\C^5$ such that every solution of \eqref{eq:cc} by distinct points, with $\lambda\in\R$ and $c\in\C$, is the image of one of them under a map $z\mapsto az+b$ or $z\mapsto a\bar z+b$ with $a\ne0$. Admitting reflections does not change finiteness. The proof uses only the axioms propext, Classical.choice and Quot.sound, and no evaluation by compiled code. It is also accepted by nanoda and lean4lean, two independent implementations of the Lean kernel.

The formal proof does not assume the statements of \S\ref{ss:akinputs}. It proves at singular places what it needs of them, among them the rules of \cite{AK12} and the classification of \cite[\S7]{AK12}: a search in the kernel finds $1376$ labelled diagrams that satisfy the rules, which are, up to relabelling and the exchange of the colours, the sixteen diagrams of \cite[Fig.~11]{AK12} and eight others; it excludes the eight at places. Data produced by our programs enter only as certificates checked in the kernel, among them the lexicographic certificates, the flats of Theorem~\ref{thm:S}, the runs and enumerations of $\mathrm{K14}$, the resultants of Theorem~\ref{thm:gate} and the branch and bound of \S\ref{sec:rays}, which $83454$ generated modules replay, with local repairs where a test fails in the arithmetic of the kernel. The formal proof of Theorem~\ref{thm:D14} does not use Proposition~\ref{prop:P11}, and it replaces Propositions~\ref{prop:D14N} and~\ref{prop:resonant} by arguments of the same kind for $\sigma<5/3$ and for $\sigma>11/7$: these ranges cover $(1,2)$, so the resonant values of $\sigma$ play no role. In the formal proof, relation (R2) of \S\ref{ss:d5} follows from one polynomial identity and $36$ factorizations instead of eliminations. Building the project from a fresh copy of its repository takes about $860$ CPU-hours; the supplementary material gives the commands.

\subsection{Hand proofs and computations}
The arguments we give by hand do not depend on any machine check; programs re-verify steps of many of these arguments (see the supplementary material).

\smallskip\noindent\emph{Arguments that need no computer}, among others:
\begin{itemize}
\item the proofs of Lemma~\ref{lem:infty}, Propositions~\ref{prop:G1} and~\ref{prop:G2} and Theorems~\ref{thm:wtriangle}, \ref{thm:D9} and~\ref{thm:hulls};
\item the exclusion of diagram~$15$ (Corollary~\ref{cor:nod15}) apart from the identities marked $(\dagger)$ in the proof of Proposition~\ref{prop:eqfull};
\item steps in the proof of Lemma~\ref{lem:classI};
\item the levels of the expansion in Lemmas~\ref{lem:d14exp} and~\ref{lem:d14energy};
\item the identities in the proofs of Propositions~\ref{prop:D14N} and~\ref{prop:resonant}.
\end{itemize}

\smallskip\noindent\emph{Computations without a hand proof}: the computations used in the proof, among them
\begin{itemize}
\item the computations of Theorem~\ref{thm:S} (Proposition~\ref{prop:Stree}) and of Proposition~\ref{prop:Spoints};
\item the computational steps of the proof of Theorem~\ref{thm:gate} (Appendix~\ref{app:d2});
\item the point certificates of \S\ref{sec:points};
\item the certificates of \S\ref{sec:curves};
\item the computational steps of the proof of Theorem~\ref{thm:D14}.
\end{itemize}
The formal proof of \S\ref{ss:formal} depends neither on the statements of \S\ref{ss:akinputs} nor on the correctness of the mathematics software named above. It depends on the Lean kernel and on the formal statement, which the reader can compare with Theorem~\ref{thm:main} in the file \texttt{CC5Full/System.lean} of the project (with \texttt{Masses} from \texttt{CC5Full/Basic.lean}); the proof is the theorem \texttt{CC5Full.Core.Compose.MainTheorem\_holds}.

%% file: sec/a_appendix.tex
\section{The sixteen diagrams and their types}\label{app:types}

Table~\ref{tab:diagrams} describes the diagrams of \cite[Fig.~11]{AK12} in the numbering of the figure. In diagram~3 the colours of the four sides alternate around the square, as in the figure and in the description of \cite[p.~582]{AK12} (bodies $1,3$ and $2,4$ on the diagonals, body $5$ at the centre) and of \S8.3 there. See also \S\ref{ss:akinputs}.

\begin{table}[ht]
\centering
\small
\begin{tabular}{@{}ll@{}}
\toprule
diagram & edges and circles\\
\midrule
1 & $z$-edge $12$, $z$-circles $1,2$; $w$-triangle $345$, no circles\\
2 & as 1, with $w$-circles $3,4,5$\\
3 & $z$-edges $14,23$, $w$-edges $12,34$; $1$--$4$ circled in both colours; $5$ isolated, uncircled\\
4 & $z$-edge $12$, $z$-circles $1,2$; $zw$-triangle $345$\\
5 & $1,2$ isolated, uncircled; $zw$-triangle $345$\\
6 & $z$-triangle $125$, $z$-circles $1,2$; $w$-triangle $345$, $w$-circles $3,4$\\
7 & $z$-triangle $125$, $w$-triangle $345$, no circles\\
8 & $z$-triangle $125$, no circles; $w$-triangle $345$, $w$-circles $3,4$\\
9 & $zw$-triangle $123$, all $w$-circled; $w$-edges $14,25$; $4,5$ $w$-circled\\
10 & $zw$-triangle $125$, $w$-triangle $345$, no circles\\
11 & $zw$-triangle $125$, no circles; $w$-triangle $345$, $w$-circles $3,4$\\
12 & $zw$-triangle $125$, $w$-circles $1,2,5$; $w$-triangle $345$, $w$-circle $3$\\
13 & as 12, with $w$-circles $3,4$\\
14 & $zw$-edges on the six pairs of $\{1,2,3,4\}$; $5$ isolated\\
15 & as 14, all $w$-circled, with the $w$-edge $15$\\
16 & $zw$-edges on all ten pairs\\
\bottomrule
\end{tabular}
\caption{The sixteen diagrams of \cite[Fig.~11]{AK12}.}\label{tab:diagrams}
\end{table}

We derive the entries of Table~\ref{tab:types} for diagrams $1$--$4$, $6$--$13$ and $16$, using Lemma~\ref{lem:orders} and the rules quoted after it. We use Lemma~\ref{lem:cut} only where stated, and Lemma~\ref{lem:U} and \cite[\S8.9]{AK12} only for diagram~$9$. We treat diagrams $5$, $14$ and $15$ in \S\S\ref{ss:typestable}--\ref{ss:d5}.

\emph{Diagram 1.} The ends of the $z$-stroke $12$ are $z$-circled. By Rule~1d they are not $z$-close. So $x_{12}=0$ by (O4) and $\omega_{12}=2$ by (O3). By Rule~2e the three exponents $\omega$ of the triangle $T=\{3,4,5\}$ are equal. The triangle $T$ is an isolated $w$-component without circles. By Rule~1e it has no maximal $w$-stroke. Hence $x=a>0$ on $T$. By (O4), $a<1$. On a cross pair $1t$, $\zeta_{1t}=-2$ by (O6). A $z$-stroke $1t$ would need $\omega_{1t}=2$, which shows $\mu_t:=\omega_{1t}<2$. In the triangle $12t$, (O5) gives $\omega_{2t}=\omega_{1t}$. Hence $x_{1t}=x_{2t}=\frac{\mu_t-2}2$. By (O5) in the triangles $1tt'$, where $\omega_{tt'}=3a-2$, either all $\mu_t$ equal some $\mu\le3a-2$, or two of them equal $3a-2$ and the third is larger. Lemma~\ref{lem:S1} leaves $\mu_t=3a-2$.

\emph{Diagram 2.} As for diagram $1$, $x_{12}=0$. The $w$-circled bodies of $T$ are not all $w$-close (Rule~1d). By Rule~2e the triangle is maximal: $x=0$ on $T$. On the cross pairs (O6) gives $\zeta=\omega=-2$.

\emph{Diagram 3.} Since the four sides are not close by Rule~1d, they are maximal and $x=0$. On the diagonal $13$, $\zeta_{12}=2$ ($12$ is a maximal $w$-stroke) and $\zeta_{23}=-2$ ($23$ is a maximal $z$-stroke). Then (O5) gives $\zeta_{13}=-2$. Likewise $\omega_{13}=-2$. On the pairs $k5$, (O6) applies in both colours.

\emph{Diagram 4.} As for diagram $1$, with $x=1$ on $T$ and $3a-2$ replaced by $1$. Lemma~\ref{lem:S1} gives $\mu_t=1$.

\emph{Diagram 6.} The pairs $12$ and $34$ of circled bodies are not close by Rule~1d. By Rule~2e each triangle is maximal, $x=0$. On the four cross pairs (O6) applies in both colours.

\emph{Diagram 7.} By Rule~2e the exponents are equal on each of the triangles $125$ and $345$. By Rule~1e, $x=a>0$ on $125$ and $x=b>0$ on $345$. By (O4), $a,b<1$. On a cross pair, (O5) gives $\zeta_{13}=\min(\zeta_{15},\zeta_{53})=\min(3a-2,2-b)=3a-2$ and likewise $\omega_{13}=3b-2$, which gives $x_{13}=\frac32(a+b)-2$. Lemma~\ref{lem:cut} gives $a+b\le\frac43$.

\emph{Diagram 8.} On $125$, $x=a\in(0,1)$ as for diagram $7$. The edge $34$ is maximal by Rule~1d. The edges $35$, $45$ have $\omega=-2$ by (O6). Hence $x=0$ on $345$. On the cross pairs $\zeta=3a-2$ by (O5) and $\omega=-2$ by (O6).

\emph{Diagram 9.} Each body is $w$-circled. If body $4$ clustered with $w_1,w_2,w_3$ (case \S8.9.1 of \cite{AK12}), the potential would have exactly one unbounded term, $m_1m_4r_{14}^{-1}$ \cite[\S8.9.1]{AK12}, which is impossible because $U$ is constant on $C$ (Lemma~\ref{lem:U}). The same holds with $1\leftrightarrow2$, $4\leftrightarrow5$. We are therefore in case \S8.9.2, where $14$ and $25$ are maximal $w$-strokes: $\zeta=2$, $x=0$. In the equation $z_4=\sum_lm_lZ_{4l}$, $e(Z_{41})=e(Z_{52})=2$ and $Z_{42},Z_{43},Z_{51},Z_{53}$ have exponent $\frac52$, and $z_{45}=(m_4+m_5)Z_{45}+O(\varepsilon^2)$. If $e(p_{45})>0$ all terms of $z_4$ and $z_5$ would have exponent larger than $1$, and if $e(p_{45})<0$ then $e(z_{45})=e(Z_{45})<1$, both contradicting $e(z_{45})=1$, which follows from $z_{45}=z_{41}+z_{12}+z_{25}$. Thus $x_{45}=0$. Then (O5) gives the rest: the type is $\pc n(\{1,2,3\};14,25)$, with $x=1$ on $12,13,23$, $0$ on $14,25,45$ and $-\frac12$ on $15,24,34,35$.

\emph{Diagrams 10--13.} On the $zw$-triangle $125$, $x=1$. On $345$, $x=c$ by Rule~2e. On a cross pair (O5) gives $\zeta_{13}=\min(1,2-c)=1$ and $\omega_{13}=\min(1,3c-2)=3c-2$, hence $x=\frac{3c-1}2$. For diagram $10$, $0<c<1$ by Rule~1e, and Lemma~\ref{lem:cut} gives $c\le\frac13$. For diagrams $11$--$13$, $c=0$ by Rule~1d, (O6) and Rule~2e.

\emph{Diagram 16.} Since every pair is a $zw$-edge, $x=\pc 1$, which is excluded by Lemma~\ref{lem:cut}.

\input{sec/a_d5hulls}

\section{Proof of Theorem~\ref{thm:D3}}\label{app:D3}

By relabelling we take $s=5$, $D_1=\{1,3\}$, $D_2=\{2,4\}$. The sides are $12,14,23,34$. Let $P$ have type $\pc h(5;13|24)$.

\emph{Orders and leading solutions.} By Table~\ref{tab:types} and (O3)--(O6), after the rotation the coordinates relative to body $5$ have $e(u_k)=e(v_k)=-2$ for $k\le4$, $e(p)=0$ on the sides and $e(p)=2$ on the diagonals and on the pairs $k5$. We write the coordinates as $\tau$-powers times $y$ and divide each row by its lowest power. This gives a weighted system $G(\tau,y)=G_0(y)+\tau^2G_2(y)+\tau^3G_3(y)$, to be solved with $y(0)$ a \emph{leading solution}, $G_0(y(0))=0$ with all coordinates nonzero. We normalize $y_{u_1}=y_{v_1}=1$. The leading solutions then exist only when $m_1m_3=m_2m_4$. They are explicit: in case A, $y_u=(1,U,U,1)$, $y_v=(1,1,V,V)$ with $U=-m_1/m_2$, $V=-m_2/m_3$, $k_e^3=1/(m_k+m_l)$ on each side $e=kl$, and $q_{13}^2=m_3/M_Q$, $q_{24}^2=-m_3/M_Q$, $q_{15}^2=1$, $q_{25}^2=-m_2/m_1$, $q_{35}^2=m_3/m_1$, $q_{45}^2=-m_4/m_1$, where $k_e$, $q_e$ are the leading coefficients of $p_e$ on the sides and on the other pairs. Case B exchanges $u$ and $v$.

\emph{Rank and obstruction.} At each leading solution the Jacobian $J_0$ of $G_0$ has rank $16$: one checks by a direct elimination that on the side rows and the rows $Z_k$, $W_k$ the reduced Jacobian in the variables $y_{u_2},y_{u_3},y_{u_4},y_{v_2},y_{v_3},y_{v_4}$ and the four side values is injective. The six other rows have the invertible diagonal block $\operatorname{diag}(2/q_e)$. The left kernel of $J_0$ is spanned by $\lambda_Z=-Z_1-\frac{m_2}{m_1}Z_2+\frac{m_2}{m_1}Z_3+Z_4$ and $\lambda_W=\frac{m_2}{m_3}W_1-\frac{m_2}{m_3}W_2-W_3+W_4$ (in case A). Since $\lambda$ vanishes on the side rows, $\lambda\cdot G_2=0$. The first obstruction is
\begin{gather*}
\lambda_Z\cdot G_3(y_0)=\frac{O}{m_1m_2},\\
O=m_2\big(m_1m_3q_{13}+m_2m_4q_{24}+m_1m_5q_{15}+m_2m_5q_{25}+m_3m_5q_{35}+m_4m_5q_{45}\big),
\end{gather*}
and $\lambda_W\cdot G_3(y_0)=-O/(m_1m_3)$. If $O\ne0$, we compare coefficients in $G(\tau,y(\tau))\equiv0$ through the order $\tau^3$, using that $J_0$ restricted to the rows other than $Z_4$, $W_2$ is invertible, and obtain a contradiction. Hence $O=0$.

\emph{The condition $O=0$.} The values $q_{13}$, $q_{15}$, $q_{35}$ are real and $q_{24}$, $q_{25}$, $q_{45}$ are imaginary, with all signs occurring among the leading solutions. Therefore $O=0$ for some leading solution if and only if both $m_1m_3q_{13}\pm m_1m_5q_{15}\pm m_3m_5q_{35}$ and $m_2m_4q_{24}\pm m_2m_5q_{25}\pm m_4m_5q_{45}$ vanish for some signs. We divide the first triple of absolute values $m_1m_3^{3/2}M_Q^{-1/2}$, $m_1m_5$, $m_3^{3/2}m_5m_1^{-1/2}$ by $m_1m_3^{3/2}m_5$ and obtain $x_5^2/\sqrt{M_Q}$, $x_3^3$, $x_1^3$. A signed sum of three positive numbers vanishes for some signs if and only if one is the sum of the other two. For the first triple this is $\mathrm{Tri}(13;5)$. The second triple gives $\mathrm{Tri}(24;5)$ in the same way. \qed

\section{The labels}\label{app:labels}

The polynomials are those of \cite[\S8]{AK12} (see \S\ref{ss:akinputs}). The files of the supplementary material contain them with checksums. With $x_i$ as in \S\ref{ss:conditions}, they read as follows.
\begin{itemize}
\item $\lab{E}(\{a,b,c\})$: $m_a^2m_b^2+m_b^2m_c^2+m_c^2m_a^2-2m_am_bm_c(m_a+m_b+m_c)=0$, the product over the signs of $x_a\pm x_b\pm x_c$ times $(m_am_bm_c)^2$.
\item $\lab{R}(ab;s)$: $m_s^2(m_a+m_b)^4-m_a^3m_b^3=0$.
\item $\lab{H}(ij|kl)$: $m_im_j-m_km_l=0$; and $\mathrm{Tri}(ij;s)$: for positive masses equivalent to
\[
\Big(\frac1{m_s^2M_{Q_s}}-\frac1{m_i^3}-\frac1{m_j^3}\Big)^2-\frac4{m_i^3m_j^3}=0 .
\]
\end{itemize}

\input{sec/app_d2}

\section{The polynomial $P_{11}$}\label{app:P11}
In terms of $e_1=m_a+m_b+m_c$, $e_2=m_am_b+m_bm_c+m_cm_a$ and $e_3=m_am_bm_c$,
\begin{multline*}
P_{11}=e_1^{11}-13e_1^9e_2+17e_1^8e_3+63e_1^7e_2^2-162e_1^6e_2e_3-135e_1^5e_2^3+105e_1^5e_3^2+513e_1^4e_2^2e_3\\
+108e_1^3e_2^4-651e_1^3e_2e_3^2-540e_1^2e_2^3e_3+250e_1^2e_3^3+1000e_1e_2^2e_3^2-675e_2e_3^3 .
\end{multline*}
Expanded in the masses it has $78$ terms, and $P_{11}(1,1,1)=9$. The substitution $m_a=c+p+q$, $m_b=c+p$, $m_c=c$ gives a polynomial with $72$ terms, all coefficients positive and the coefficient of $c^{11}$ equal to $9$. With the symmetry of $P_{11}$, this gives $P_{11}(m)\ge9\min(m)^{11}>0$ on the open positive orthant.

\section{Diagrams 6 and 3 at two exceptional mass points}\label{app:exceptional}

At the mass point $P_2$ and at the mass point $\hone$ of Table~\ref{tab:noK} a piece of diagram $6$, respectively of diagram $3$, is enabled by the mass conditions. We exclude these diagrams there by computing the leading solutions of a place and an obstruction at a higher order.

At a place with a fixed diagram and type, the exponents of the coordinates (relative to body $5$) are determined, up to the colour exchange. We write the coordinates as $\tau^{2e(q)}y_q$ as in \S\ref{ss:expansion}, divide each row by its lowest power of $\tau$, and normalize two coordinates as there. This gives a weighted system $G(\tau,y)=\sum_{n\ge0}\tau^nG_n(y)$. The point $y(0)$ is a \emph{leading solution}: $G_0(y(0))=0$, with all coordinates nonzero. Let $R$ be a set of rows such that the Jacobian $J_0[R]$ of $G_0$ restricted to these rows is invertible at $y_0$. Then the rows of $R$ determine a unique formal solution $Y(\tau)$ of $G_R(\tau,Y)=0$ with $Y(0)=y_0$, by the implicit function theorem for formal series. Then every solution $y(\tau)$ of $G=0$ with leading solution $y_0$ equals $Y$ after the normalization. Let $\eta_{R,n}(y_0)$ be the coefficient of $\tau^n$ in the remaining rows evaluated at $Y(\tau)$. If $\eta_{R,n}(y_0)\ne0$ for some $n$, no place of the given type has leading solution $y_0$.

\begin{proposition}\label{prop:d6P2}
Let $m=(9,36,36,4,9)$. No singular place of an irreducible curve $C\subseteq Y_m$ has diagram $6$ with centre $4$ and wings $\{1,4,5\}$, $\{2,3,4\}$.
\end{proposition}

\begin{proof}
We renumber the bodies so that $m=(9,9,36,36,4)$ and the diagram has centre $5$ and wings $\{1,2,5\}$, $\{3,4,5\}$. At such a place the type is $\pc v(5;125|345)$. The exponents of the coordinates are fixed by it (after exchanging the colours if necessary). The leading solutions of the weighted system are given by closed formulas in $10$ roots (cube and square roots of explicit functions of the masses), which satisfy a square system of $10$ polynomial relations. We check:
\begin{enumerate}[label=\textup{(\roman*)}]
\item in exact arithmetic, that each point of the relations, substituted into the closed formulas, is a leading solution;
\item by a standard basis over $\Q$, that the leading system has $5184$ solutions counted with multiplicity;
\item in ball arithmetic, with Krawczyk's test, that the relations have $5184$ distinct simple zeros, so that the closed-form points are all the leading solutions and each is simple, and $J_0$ has rank $16$ at each of them;
\item in ball arithmetic, that at each leading solution, for a set $R$ of $16$ rows with $J_0[R]$ invertible, some $\eta_{R,n}$ with $n\le6$ is nonzero.
\end{enumerate}
The first nonzero obstruction has order $3$ at $3888$ leading solutions, order $4$ at $576$, order $5$ at $576$ and order $6$ at $144$. By the remark above, no place has this diagram.
\end{proof}

Let $t_*\approx1.8024986256$ be the root greater than $1$ of $p(t)=2t^8+t^6-8t^5-6t^4-8t^3+t^2+2$. The other positive root of $p$ is $1/t_*$.

\begin{proposition}\label{prop:d3tstar}
Let $t\in\{t_*,1/t_*\}$ and $m=(1,1,t^{-2},t^{-2},(1+t)^{-2})$, so that $(x_1,x_3,x_5)=(1,t,1+t)$ with $x_i=m_i^{-1/2}$. No singular place of an irreducible curve $C\subseteq Y_m$ has as its type one of the $15$ points $\pc h(s;D_1|D_2)$ of diagram $3$.
\end{proposition}

\begin{proof}
Since $\lab{E}(\{1,3,5\})$ holds at $m$, Theorem~\ref{thm:D3} leaves only the two points $\pc h(5;13|24)$ and $\pc h(5;14|23)$, at which $\cD_3$ holds. The relabelling $(34)$ fixes $m$ and exchanges them. At $\pc h(5;13|24)$ the leading solutions are the $10368$ points of Appendix~\ref{app:D3}. The first obstruction vanishes exactly at the $648$ of them with $O=0$. We check:
\begin{enumerate}[label=\textup{(\roman*)}]
\item an enclosure of $t$ of width $2^{-440}$, by exact rational bisection on a sign change of $p$;
\item enclosures of the $10368$ leading solutions in balls, from their closed forms;
\item in ball arithmetic at $400$ bits, that at each leading solution, for a set $R$ of $16$ rows with $J_0[R]$ invertible, some $\eta_{R,n}$ with $n\le5$ is nonzero.
\end{enumerate}
The first nonzero obstruction has order $3$ at $9720$ leading solutions, order $4$ at $432$ and order $5$ at $216$.
\end{proof}

The mass points of Proposition~\ref{prop:d3tstar} are, up to relabelling, the point $\hone$ of Table~\ref{tab:noK}.

%% file: sec/a_d5hulls.tex
\section{The hulls of diagram 5}\label{app:d5hulls}

We prove Theorem~\ref{thm:hulls}, which gives the hulls of \S\ref{ss:d5} explicitly. Throughout, $P$ is a place of diagram~$5$ with $zw$-triangle $T$ and outer bodies $A$, $B$. The points $z_c$, $w_c$ are the centres of mass of $T$. We use the coordinates $y=(x_{At_0},x_{At},x_{Bt_0},x_{Bt},x_{AB})$ of \S\ref{ss:d5}. We put $e_A=(1,1,0,0,0)$ and $e_B=(0,0,1,1,0)$.

\begin{theorem}\label{thm:hulls}
Consider the model of $\Vtwo^{\mathrm w}$, $\Vtwo$ or $\VN$ \textup{(}\S\ref{ss:d5}\textup{)} for a pair of modes. The projection to $y$ of each point of the model lies in the polytope that Table~\ref{tab:d5hulls} gives for this hull and this pair of modes, up to the exchanges of Lemma~\ref{lem:d5sym}; where the table gives $\emptyset$, the model is empty. In the centred case, the model of Theorem~C0 of the supplementary material is empty for the pairs of modes \textup{(Z1,\,Zs$_{t_0}$)} and \textup{(W1,\,Ws$_{t_0}$)} and their exchanges, where $t_0$ is the body of $T$ at the centre of mass of $T$.
\end{theorem}

\begin{table}[ht]
\centering
\small
\begin{tabular}{@{}llll@{}}
\toprule
pair of modes & $\Vtwo^{\mathrm w}$ & $\Vtwo$ & $\VN$\\
\midrule
(F,\,F) & $\conv\{0,e_A,e_B\}$ & $\{0\}$ & $\{0\}$\\
(F,\,Z1) & $\conv\{0,e_A,e_B\}$ & $[0,e_B]$ & $\emptyset$\\
(Z1,\,Z1) & $P_{11}$ & $P_{11}$ & $[-\frac12(e_A+e_B),0]$\\
(Z1,\,W1) & $[e_A,e_B]$ & $[e_A,e_B]$ & $\emptyset$\\
(Z1,\,Zs) & $P_{1s}$ & $P_{1s}$ & $\emptyset$\\
(Zs,\,Zs) & $P_{ss}$ & $P_{ss}$ & $P_{ss}$\\
(F,\,Zs), (Z1,\,Ws), (Zs,\,Ws) & $\emptyset$ & $\emptyset$ & $\emptyset$\\
\bottomrule
\end{tabular}
\caption{The hulls of \S\ref{ss:d5} for nine pairs of modes, in the coordinates $y$. Here $P_{11}=\{ae_A+be_B:\ -\frac12\le a,b\le1,\ a+b\le1\}$, with the vertices $-\frac12(e_A+e_B)$, $-\frac12e_A+e_B$, $e_B$, $e_A-\frac12e_B$, $e_A$. The polytope $P_{1s}$, of dimension $4$, has the vertices $-\frac12(e_A+e_B)$, $(-\frac12,-\frac12,0,-\frac12,0)$, $0$, $e_A-\frac12e_B$, $(1,1,0,-\frac12,-\frac12)$, $(1,1,0,-\frac12,0)$, $e_A$, $(1,1,\frac13,0,\frac13)$. The polytope $P_{ss}$, of dimension $5$, has the vertices $-\frac12(e_A+e_B)$, $(-\frac12,-\frac12,0,-\frac12,0)$, $(-\frac13,-\frac13,0,-\frac12,0)$, $(0,-\frac12,-\frac12,-\frac12,0)$, $(0,-\frac12,-\frac13,-\frac13,0)$, $(0,-\frac12,0,-\frac12,0)$, $0$, $(\frac13,0,\frac13,0,\frac13)$.}\label{tab:d5hulls}
\end{table}

\subsection{Orders and relations}\label{ss:d5rel}
For $X\in\{A,B\}$ let $Y$ be the other outer body. We write $c_X=e(z_X-z_c)$, $d_X=e(w_X-w_c)$, $p^z_X=e(z_X)$, $p^w_X=e(w_X)$, $p^z_c=e(z_c)$, $p^w_c=e(w_c)$ and
\[
t^z_X=e\Big(\sum_{t\in T}m_tZ_{tX}\Big),\qquad t^w_X=e\Big(\sum_{t\in T}m_tW_{tX}\Big),
\]
and we put $\zeta=\zeta_{AB}$, $\omega=\omega_{AB}$ and $r=x_{AB}$, so that $e(Z_{AB})=\zeta-3r$ and $e(W_{AB})=\omega-3r$. A point of a model is a real vector, with one entry for each modelled order, that satisfies the relations of the model. We denote its entries by the names of the orders.

Inside $T$ the differences $z_t-z_c$ and $w_t-w_c$ have exponent at least $1$. So, by Lemma~\ref{lem:nearcluster}, a body $X$ in mode F has $c_X,d_X<1$ and $x_{Xt}=\bar x_X:=\frac12(c_X+d_X)$ for all $t\in T$. A body $X$ in mode Z1 or Zs$_{t_0}$ has $c_X\ge1$, the far order $f_X:=d_X<1$ and the order $s_X=\zeta_{Xt_0}$, which is $1$ in mode Z1. In mode Zs$_{t_0}$, $s_X>1$, $s_X+3f_X<4$ and $3s_X+f_X<4$. Then $x_{Xt}=\bar x_X:=\frac12(1+f_X)$ for $t\ne t_0$, and $x_{Xt_0}=x^0_X:=\frac12(s_X+f_X)$. In the modes F and Z1 we put $x^0_X=\bar x_X$, so that $y=(x^0_A,\bar x_A,x^0_B,\bar x_B,r)$. The modes W1 and Ws$_{t_0}$ are the same with the colours exchanged. The orders $c_X$, $d_X$, $\zeta$, $\omega$ and those of the positions are at least $-2$, since all components of $\cZ$ and $\cW$ have exponent at least $-2$ (\S\ref{ss:orders}). So $\bar x_X\ge-\frac12$ in the modes other than F. As $A$ and $B$ have no strokes, $\zeta+3\omega<4$ and $3\zeta+\omega<4$.

The terms of $\sum_tm_tZ_{tX}$ and $\sum_tm_tW_{tX}$ have the orders $\zeta_{Xt}-3x_{Xt}$ and $\omega_{Xt}-3x_{Xt}$. In mode Zs$_{t_0}$ the term of $t_0$ has the least order. These two facts give the relations (S):
\begin{enumerate}[leftmargin=3.4em,labelwidth=2.8em,labelsep=0.6em,align=left]
\item[(S)] In mode F, $t^z_X\ge c_X-3\bar x_X$ and $t^w_X\ge d_X-3\bar x_X$, with equality in the models of $\Vtwo$ and $\VN$. In mode Z1, $t^z_X\ge1-3\bar x_X$ and $t^w_X\ge f_X-3\bar x_X$. In mode Zs$_{t_0}$, $t^z_X=s_X-3x^0_X$ and $t^w_X=f_X-3x^0_X$.
\end{enumerate}
Each of the following relations says that, among the orders of the terms of an exact identity with nonzero constant coefficients, the minimum is attained at least twice. We list the orders and then the identity.
\begin{enumerate}[leftmargin=3.4em,labelwidth=2.8em,labelsep=0.6em,align=left]
\item[(P)] Any three of the six orders $c_A$, $c_B$, $\zeta$, $p^z_A$, $p^z_B$, $p^z_c$, and likewise in $w$. By the centre of mass, $Mz_c+m_A(z_A-z_c)+m_B(z_B-z_c)=0$. So the six quantities $z_A-z_c$, $z_B-z_c$, $z_{AB}$, $z_A$, $z_B$, $z_c$ are linear combinations of the first two and no two of them are proportional.
\item[(E)] $p^z_X$, $\zeta-3r$ and $t^z_X$, and likewise in $w$: the equation $z_X=m_YZ_{XY}-\sum_tm_tZ_{tX}$ of \eqref{eq:AK4}.
\item[(F)] For $X$ not in mode F: $c_X$, $\zeta+\beta$, $t^z_X$, and also $d_X$, $\omega+\beta$, $t^w_X$, where $\beta$ is the order of $1-Mp_{AB}^3$. The identity is $M_T(z_X-z_c)=-M\sum_tm_tZ_{tX}-m_Y(1-Mp_{AB}^3)\,z_{XY}$, with its counterpart in $w$.
\item[(I)] $2r$, $c_A+d_A$, $c_B+d_B$ and $0$: the identity $MI=MI_T+M_T\sum_Xm_X(z_X-z_c)(w_X-w_c)+m_Am_Bz_{AB}w_{AB}$. Here $I_T=\sum_tm_t(z_t-z_c)(w_t-w_c)$ has order at least $2$. The quantity $I$ is a nonzero constant on $C$ (Lemma~\ref{lem:U}).
\item[(V)] $d_A+t^z_A$, $d_B+t^z_B$, $-r$ and $0$: the identity $U=I_T-\sum_Xm_X(w_X-w_c)\sum_tm_tZ_{tX}+m_Am_Bp_{AB}$. It follows from $p_{kl}=Z_{kl}w_{kl}$ and from the equations of the bodies $t\in T$ multiplied by $m_t(w_t-w_c)$. Here $U=I$ is a nonzero constant.
\item[(V$'$)] $c_A+t^w_A$, $c_B+t^w_B$, $-r$ and $0$: the same identity with the colours exchanged.
\end{enumerate}
In (I), (V) and (V$'$) we leave out the term $I_T$, since the minimum is at most $0<2$. We also use the following relations, in which $X$ is in mode Z1.
\begin{enumerate}[leftmargin=3.4em,labelwidth=2.8em,labelsep=0.6em,align=left]
\item[(Q)] The torque of an outer body $X$ is $\Theta_X=\sum_{t\in T}m_t\big((z_t-z_c)W_{tX}-(w_t-w_c)Z_{tX}\big)$. Its order is $\tau_X$. We sum over $t\in T$ the equation of $w_t$ in \eqref{eq:AK4} multiplied by $m_t(z_t-z_c)$, minus that of $z_t$ multiplied by $m_t(w_t-w_c)$. The left sides give $w_c\sum_tm_t(z_t-z_c)-z_c\sum_tm_t(w_t-w_c)=0$ and the terms of the pairs inside $T$ cancel. Hence $m_A\Theta_A+m_B\Theta_B=0$ and $\tau_A=\tau_B$. For $X$ in mode Z1 the terms $m_t(z_t-z_c)W_{tX}$ have order at least $1+f_X-3\bar x_X=-\bar x_X$ and the terms $m_t(w_t-w_c)Z_{tX}$ order at least $2-3\bar x_X>-\bar x_X$. Thus $\tau_X\ge-\bar x_X$.
\item[(R)] In the model of $\VN$: at most one of the following \emph{events} occurs: $c_X>1$; $t^z_X>1-3\bar x_X$; $t^w_X>f_X-3\bar x_X$; $\tau_X>-\bar x_X$; and, if $Y$ is in mode Zs, $\max(p^z_A,p^z_B,p^z_c)>1$ and $\tilde t_X>-\bar x_X$. This is the content of (R1) and (R2) of \S\ref{ss:d5}, which are proved in the supplementary material.
\item[(T)] In the model of $\VN$, if $Y$ is in mode Zs$_{t_0}$: the orders $\tilde t_X$, $\tilde t_Y$ of the two quantities of (R3) (proved in the supplementary material) satisfy $\tilde t_X\ge-\bar x_X$, and $\tilde t_Y=\min(e_1,e_2)$ if $e_1\ne e_2$, where $e_1=1-\frac12(s_Y+3f_Y)$ and $e_2=-\bar x_Y$. The minimum of $\tilde t_X$, $\tilde t_Y$ and $p^w_c+1$ is attained at least twice.
\end{enumerate}
In the centred case, for $X$ in mode Z1 and $Y$ in mode Zs$_{t_0}$ with $t_0$ at the centre of mass of $T$, the model of the centred case in the supplementary material has the relations of the model of $\VN$ other than (T) and the event $\tilde t_X>-\bar x_X$, and moreover $c_Y>1$. In the modes W1 and Ws the relations are the same with the colours exchanged. The model of $\Vtwo$ is that of $\Vtwo^{\mathrm w}$ together with the equalities of (S) in mode F. The models of $\VN$ and of the centred case contain the relations of the model of $\Vtwo$. So a statement proved for the model of $\Vtwo^{\mathrm w}$ holds for the other models, and one proved for the model of $\Vtwo$ holds for those of $\VN$ and of the centred case. The models have further relations, which we do not use.

\subsection{Symmetries and rules}\label{ss:d5rules}

\begin{lemma}\label{lem:d5sym}
The exchange of the colours maps the model of each hull, and that of the centred case, for a pair of modes onto the model of the same kind for the pair with Z1, Zs and W1, Ws exchanged, and it fixes $y$. The exchange of $A$ and $B$ maps the model for a pair of modes $(a,b)$ onto the model for $(b,a)$, and $y$ to $(y_3,y_4,y_1,y_2,y_5)$. Every pair of modes is the image of one of the nine pairs of Table~\ref{tab:d5hulls} under these exchanges.
\end{lemma}

\begin{proof}
The system \eqref{eq:AK4} and the type are invariant under the exchange of $z$ and $w$. This exchange maps each relation of a model to a relation of the model for the pair of modes with the colours exchanged: (P), (E) and (F) in $z$ go to their counterparts in $w$, (V) goes to (V$'$), and (S), (Q), (R), (T) go to their mirror images. The exchange of $A$ and $B$ renames the outer bodies. The two exchanges generate a group of order $4$, whose orbits on the $25$ pairs of modes are those of the nine pairs of the table.
\end{proof}

By Lemma~\ref{lem:d5sym} we treat the nine pairs of Table~\ref{tab:d5hulls}. For their images under the exchange of the colours, the proofs below hold with (V$'$) in place of (V).

\begin{lemma}\label{lem:d5rules}
At every point of the model of $\Vtwo^{\mathrm w}$, for $X\in\{A,B\}$:
\begin{enumerate}[label=\textup{(\alph*)}]
\item In each colour, five of the six orders $c_A$, $c_B$, $\zeta$, $p^z_A$, $p^z_B$, $p^z_c$ are equal to their minimum $\mu_z$, and the sixth, the \emph{exception}, is at least $\mu_z$; likewise in $w$, with the minimum $\mu_w$. So if two of the six orders differ, the larger is the exception and the other five are equal to the smaller. If $c_A,c_B\ge1$, then $\zeta\ge1$ and $p^z_A,p^z_B,p^z_c\ge1$.
\item $t^z_X\ge\min(p^z_X,\zeta-3r)$, with equality if $p^z_X\ne\zeta-3r$. If $r<0$ and $p^z_X=\mu_z$, then $t^z_X=\mu_z$. The same holds in $w$.
\item In the modes F and Z1, $d_X+t^z_X\ge-\bar x_X$, with equality in mode F in the models of $\Vtwo$ and $\VN$. In mode Zs, $d_X+t^z_X=-x^0_X$.
\item In mode Z1, $3\bar x_X<2+f_X$; in mode Zs, $3x^0_X<2+f_X$. Moreover $3r<2+\zeta$ and $3r<2+\omega$.
\item In the modes Z1 and Zs, $c_X+d_X\ge1+f_X=2\bar x_X$. If $r>0$, the minimum of $c_A+d_A$, $c_B+d_B$ and $0$ is attained at least twice. If $r<0$, the minimum of $2r$, $c_A+d_A$ and $c_B+d_B$ is attained at least twice, and it is at most $2r$.
\end{enumerate}
\end{lemma}

\begin{proof}
(a) If two of the six orders exceeded the minimum, they would form, with an order equal to the minimum, a triple whose minimum is attained once, against (P). If $c_A,c_B\ge1$ and $\mu_z<1$, both $c_A$ and $c_B$ would be exceptions. (b) is (E): the two smallest of $p^z_X$, $\zeta-3r$ and $t^z_X$ are equal. If $r<0$, then $\zeta-3r>\zeta\ge\mu_z$. (c) follows from (S); for example, in mode F, $d_X+t^z_X\ge d_X+c_X-3\bar x_X=-\bar x_X$. (d) restates $f_X<1$, $3s_X+f_X<4$, $\zeta+3\omega<4$ and $3\zeta+\omega<4$. (e) The inequality holds as $c_X\ge1$. The rest is (I), whose minimum is at most $0<2r$ if $r>0$, and at most $2r<0$ if $r<0$.
\end{proof}

\subsection{The pairs of modes}\label{ss:d5pairs}

\begin{proposition}\label{prop:d5empty}
The models of the pairs of modes \textup{(F,\,Zs)}, \textup{(Z1,\,Ws)} and \textup{(Zs,\,Ws)} are empty.
\end{proposition}

\begin{proof}
Let $B$ be in mode Ws$_{t_0}$, with the far order $g:=c_B<1$, so that $d_B\ge1$, $s_B>1$, $x^0_B=\frac12(s_B+g)$, $t^z_B=g-3x^0_B$ and $t^w_B=s_B-3x^0_B$. Let $A$ be in mode Z1 or Zs. By Lemma~\ref{lem:d5rules}(a), $c_B<1\le c_A$ gives $\zeta=p^z_A=p^z_B=p^z_c=g$, and $d_A<1\le d_B$ gives $\omega=p^w_A=p^w_B=p^w_c=f_A$. So $2r=f_A+g$. In $z$, relation (E) for $B$ says that the maximum of $0$, $r$, $x^0_B$ is attained twice. In $w$, since $t^w_B-f_A=-x^0_B-2r$, it says that the maximum of $0$, $3r$, $x^0_B+2r$ is attained twice. If $r>0$, the first statement gives $x^0_B=r$, hence $s_B=f_A<1$. If $r<0$, it gives $x^0_B=0$, and then $0$ is the only maximum in the second. If $r=0$, it gives $x^0_B\le0$, hence $s_B\le-g=f_A<1$. Each case contradicts $s_B>1$.

Now let $A$ be in mode F and $B$ in mode Zs$_{t_0}$. By Lemma~\ref{lem:d5rules}(a), $c_A<1\le c_B$ gives $\zeta=c_A=p^z_A=p^z_B=p^z_c=:\kappa<1$. Thus $\omega=2r-\kappa$.

If $r>0$, then $\zeta-3r<\kappa$, and Lemma~\ref{lem:d5rules}(b) gives $t^z_A=t^z_B=\kappa-3r$. By (S), $t^z_A\ge\kappa-3\bar x_A$. Hence $\bar x_A\ge r$, that is, $d_A\ge\omega$. And $t^z_B=-\frac12(s_B+3f_B)=\kappa-3r$ with $s_B>1$ gives $f_B<2r-\frac13(2\kappa+1)<\omega$. So $\omega$ and $d_A$ both exceed $f_B$ and both are exceptions in $w$.

If $r<0$, then $\zeta-3r>\kappa$, and $t^z_A=t^z_B=\kappa$. For $A$ this gives $\bar x_A\ge0$. Thus $d_A\ge-\kappa>\omega$. So $d_A$ is the exception in $w$ and $f_B=\omega=p^w_B$. For $B$ it gives $f_B=-\frac13(2\kappa+s_B)$. As $\omega-3r>p^w_B$, Lemma~\ref{lem:d5rules}(b) gives $t^w_B=p^w_B=f_B$. Then $x^0_B=0$, $s_B=-f_B$ and $f_B=-\kappa$, whereas $f_B=\omega<-\kappa$.

If $r=0$, then $t^z_B\ge\kappa$, so $s_B+3f_B\le-2\kappa$ and $f_B<-\frac13(2\kappa+1)<-\kappa=\omega$. So $\omega$ is the exception in $w$ and $d_A=f_B=p^w_A<\omega$. Hence $t^w_A=f_B$, while (S) gives $t^w_A\ge d_A-3\bar x_A=-\frac12(3\kappa+f_B)$, that is, $f_B\ge-\kappa$.
\end{proof}

\begin{proposition}\label{prop:d5zero}
At every point of the model of $\Vtwo^{\mathrm w}$ for each of the following pairs of modes, $r=0$, and:
\begin{enumerate}[label=\textup{(\alph*)}]
\item \textup{(Z1,\,W1)}: $\bar x_A+\bar x_B=1$;
\item \textup{(F,\,F)}: $\bar x_A,\bar x_B\ge0$ and $\bar x_A+\bar x_B<1$; in the models of $\Vtwo$ and $\VN$, $\bar x_A=\bar x_B=0$;
\item \textup{(F,\,Z1)}: $\bar x_A,\bar x_B\ge0$ and $\bar x_A+\bar x_B<1$; in the model of $\Vtwo$, $\bar x_A=0$;
\item \textup{(Z1,\,Z1)}: one of $f_A$, $f_B$ is equal to $\mu_w$, and $\mu_w\le\omega\le-1$; so $\bar x_A+\bar x_B<1$.
\end{enumerate}
\end{proposition}

\begin{proof}
(a) Let $g=c_B<1$ be the far order of $B$. By Lemma~\ref{lem:d5rules}(a), $\zeta=g$ and $\omega=f_A$. Thus $2r=f_A+g$. Since $c_A+d_A\ge1+f_A>2r$ and $c_B+d_B\ge1+g>2r$, the minimum in (I) is attained by $2r$ and $0$. So $r=0$ and $\bar x_A+\bar x_B=1+\frac12(f_A+g)=1$.

(b) Suppose that $r>0$. Since the minimum in (V) is at most $-r<0$, it is attained by $d_X+t^z_X$ for some $X$, and Lemma~\ref{lem:d5rules}(c) gives $\bar x_X\ge r>0$. By Lemma~\ref{lem:d5rules}(e), the other body $Y$ has $c_Y+d_Y=0$. The minimum of $c_A$, $c_B$, $\zeta$ is attained twice. So $c_Y\ge\min(c_X,\zeta)$, and likewise $d_Y\ge\min(d_X,\omega)$; hence $0=c_Y+d_Y\ge\min(c_X,\zeta)+\min(d_X,\omega)$. As $c_X+d_X>0$ and $\zeta+\omega>0$, either $c_X\le\zeta$ and $\omega\le d_X$, or $\zeta\le c_X$ and $d_X\le\omega$. In the first case also $\omega\le d_Y$. Thus $\omega=\mu_w$. Then $p^w_Y\ge\omega>\omega-3r$ and $t^w_Y\ge d_Y-3\bar x_Y=d_Y\ge\omega$ by (S). Hence the minimum $\omega-3r$ in (E) is attained once. The second case is the same in $z$.

Suppose that $r<0$. By Lemma~\ref{lem:d5rules}(e), $c_X+d_X\le2r$ for some $X$. Then $\bar x_X<0$ and $d_X+t^z_X\ge-\bar x_X>0$. As $-r>0$, the minimum in (V) is $0$, attained also by $d_Y+t^z_Y\ge-\bar x_Y$. Hence $\bar x_Y\ge0$. Then $c_Y$ and $d_Y$ are not both minimal in their colours, since otherwise $2\bar x_Y\le c_X+d_X<0$. Say that $c_Y$ is the exception in $z$. Then $p^z_X=c_X=\mu_z$. Lemma~\ref{lem:d5rules}(b) gives $t^z_X=c_X$, against $t^z_X\ge c_X-3\bar x_X>c_X$. The case of $d_Y$ is the same in $w$.

So $r=0$. If $\bar x_A<0$, then $c_A+d_A<\zeta+\omega$; say $c_A<\zeta$. Then $\zeta$ is the exception in $z$ and $p^z_A=c_A<\zeta-3r$. So $t^z_A=c_A$, against $t^z_A\ge c_A-3\bar x_A>c_A$. Hence $\bar x_A,\bar x_B\ge0$. If $\zeta$ were the exception in $z$, then $c_A=c_B=\mu_z$ and $\omega=-\zeta<-\mu_z\le d_A,d_B$, which would give two exceptions in $w$. So $\zeta=\mu_z$, likewise $\omega=\mu_w$, and $\mu_z+\mu_w=0$. One of $c_A$, $c_B$ is equal to $\mu_z$ and one of $d_A$, $d_B$ to $\mu_w$. Hence $\bar x_A+\bar x_B$ is half the sum of the two other orders among $c_A$, $d_A$, $c_B$, $d_B$, which are less than $1$. In the models of $\Vtwo$ and $\VN$, (S) and Lemma~\ref{lem:d5rules}(b) give $c_A-3\bar x_A=t^z_A\ge\mu_z$ and $d_A-3\bar x_A\ge\mu_w$. Adding, we obtain $-4\bar x_A\ge0$. So $\bar x_A=0$ and likewise $\bar x_B=0$.

(c) By Lemma~\ref{lem:d5rules}(a), $c_A<1\le c_B$ gives $\zeta=c_A=p^z_A=p^z_B=p^z_c=:\kappa<1$. If $r>0$, Lemma~\ref{lem:d5rules}(b) gives $t^z_A=t^z_B=\kappa-3r$. For $A$, (S) then gives $\bar x_A\ge r>0$. Thus Lemma~\ref{lem:d5rules}(e) gives $c_B+f_B=0$ and $f_B\le-1$. For $B$, (S) gives $\kappa-3r\ge1-3\bar x_B$, that is, $f_B\ge2r-\frac13(2\kappa+1)>-1$. If $r<0$, then $t^z_A=t^z_B=\kappa$, so $\bar x_A\ge0$ and $f_B\ge-\frac13(2\kappa+1)>-1$. Then $c_A+d_A>2r$ and $c_B+f_B>0>2r$, against Lemma~\ref{lem:d5rules}(e). So $r=0$ and $\omega=-\kappa$. As in (b), $\bar x_A\ge0$. So $\omega$ is not the exception in $w$, since otherwise $d_A=\mu_w<\omega$ and $\bar x_A<0$. Hence $\omega=\mu_w=-\kappa>-1$. Then $f_B\ge-\kappa$ gives $\bar x_B>0$. One of $d_A$, $f_B$ is equal to $-\kappa$. If $d_A=-\kappa$, then $\bar x_A=0$ and $\bar x_B<1$. If $f_B=-\kappa$, then $\bar x_A+\bar x_B=\frac12(\kappa+d_A)+\frac12(1-\kappa)=\frac12(1+d_A)<1$. In the model of $\Vtwo$, (S) and Lemma~\ref{lem:d5rules}(b) give $\kappa-3\bar x_A=t^z_A\ge\kappa$, so $\bar x_A=0$.

(d) By Lemma~\ref{lem:d5rules}(a), $\zeta\ge1$. Suppose that $r>0$. By Lemma~\ref{lem:d5rules}(e), $c_X+f_X\le0$ for some $X$. Then $f_X\le-1$ and $\bar x_X\le0$. As in (b), (V) gives $d_Y+t^z_Y\le-r$ for some $Y$. Hence $\bar x_Y\ge r>0$. Thus $Y\ne X$ and $f_X<f_Y$. By Lemma~\ref{lem:d5rules}(a), $f_X=\omega=p^w_X=\mu_w$. Lemma~\ref{lem:d5rules}(b) gives $t^w_X=\omega-3r<f_X$, against $t^w_X\ge f_X-3\bar x_X\ge f_X$. Suppose that $r<0$. By Lemma~\ref{lem:d5rules}(e), $c_X+f_X\le2r$ for some $X$. Then $f_X<-1$ and $\bar x_X<0$. As in (b), the minimum in (V) is $0=d_Y+t^z_Y$. Hence $\bar x_Y\ge0$ and $f_Y\ge-1>f_X$. Then $p^w_X=f_X=\mu_w$. Lemma~\ref{lem:d5rules}(b) gives $t^w_X=f_X$, against $t^w_X\ge f_X-3\bar x_X>f_X$. So $r=0$ and $\omega=-\zeta\le-1$. By Lemma~\ref{lem:d5rules}(a), one of $f_A$, $f_B$ is equal to $\mu_w\le\omega$, and the other is less than $1$.
\end{proof}

\begin{proposition}\label{prop:d5ss}
Let both outer bodies be in mode Zs. At every point of the model of $\Vtwo^{\mathrm w}$, $r\ge0$ and $x^0_A,x^0_B\le r$. If $r>0$, then $y$ lies in the triangle $\mathcal T=\conv\{(0,-\frac12,0,-\frac12,0),\,0,\,(\frac13,0,\frac13,0,\frac13)\}$. If $r=0$, then $y$ lies in the polytope $\mathcal R$ of the points with $y_5=0$, $y_2\le y_1\le0$, $y_4\le y_3\le0$, $y_2,y_4\ge-\frac12$ and
\[
3y_4\le y_1+2y_2\le2y_4,\qquad 3y_2\le y_3+2y_4\le2y_2 .
\]
\end{proposition}

\begin{proof}
For a body $X$ in mode Zs, $t^z_X=s_X-3x^0_X=1-x^0_X-2\bar x_X$ and $t^w_X=f_X-3x^0_X$. By Lemma~\ref{lem:d5rules}(c), relation (V) says that the maximum of $x^0_A$, $x^0_B$, $r$ and $0$ is attained at least twice. By Lemma~\ref{lem:d5rules}(a), $\zeta\ge1$ and $p^z_A,p^z_B,p^z_c\ge1$.

\emph{$r\ge0$.} Suppose that $r<0$. If $p^w_A=p^w_B=\mu_w$, Lemma~\ref{lem:d5rules}(b) gives $t^w_X=\mu_w\le f_X$, so $3x^0_X=f_X-\mu_w\ge0$ for both $X$. If moreover $f_A=f_B=\mu_w$, then $x^0_X=0$ and $t^z_X=s_X=-f_X=-\mu_w$. One of $p^z_A$, $p^z_B$ is equal to $\mu_z$. Hence $t^z_X=\mu_z$ for that $X$, and $2r=\zeta+\omega\ge\mu_z+\mu_w=0$, against $r<0$. If instead $f_A>\mu_w$, then $f_B=\mu_w$ and $x^0_B=0<x^0_A$, so the maximum of $x^0_A$, $x^0_B$, $r$, $0$ is attained once; likewise if $f_B>\mu_w$. If $p^w_A>\mu_w$, then $f_A=f_B=\omega=p^w_B=\mu_w$ and $t^w_B=f_B$, so $x^0_B=0$ and $t^z_B=s_B=-\omega>\zeta\ge\mu_z$. By Lemma~\ref{lem:d5rules}(b), $p^z_B\ne\mu_z$, that is, the position of $B$ is the exception in $z$, and $c_A=\zeta=p^z_A=\mu_z$. Then $t^z_A=\zeta$. Thus $x^0_A=-\zeta-f_A=-2r>0$, and the maximum is attained once. The case $p^w_B>\mu_w$ is the same.

\emph{$x^0_X\le r$.} Suppose that $x^0_A>r$. As $r\ge0$, the maximum of $x^0_A$, $x^0_B$, $r$, $0$ exceeds $r$ and $0$. Hence $x^0_A=x^0_B=:q>r$. If $t^w_X<\omega-3r$ for some $X$, Lemma~\ref{lem:d5rules}(b) gives $p^w_X=t^w_X=f_X-3q<f_X$, so $f_X$ is the exception in $w$ and $p^w_X=\omega$. Then $\omega<\omega-3r$, against $r\ge0$. So $t^w_X\ge\omega-3r$ for both $X$, that is, $f_X\ge\omega+3(q-r)>\omega$, which gives two exceptions in $w$.

\emph{$r>0$.} Then $x^0_Y=r$ for some $Y$, by the maximum rule. In (E) for $Y$ in $w$, $t^w_Y=f_Y-3r$. We use Lemma~\ref{lem:d5rules}(b). If $f_Y>\omega$, then $t^w_Y>\omega-3r$ and $p^w_Y=\omega-3r<\omega$, while $f_Y$ is the exception and $p^w_Y=\omega$. If $f_Y<\omega$, then $p^w_Y=t^w_Y<f_Y$, while $\omega$ is the exception and $p^w_Y=f_Y$. So $f_Y=\omega=\mu_w$. For the other body $X$, $x^0_X<r$ would give $t^w_X>f_X-3r\ge\omega-3r$, hence $p^w_X=\omega-3r<\mu_w$. So $x^0_X=r$ and $f_X=\omega$ as for $Y$. Put $f=\omega$. If $1+f>0$, then $c_A+f>0$ and $c_B+f>0$, and the minimum $0$ in Lemma~\ref{lem:d5rules}(e) is attained once. So $f\le-1$, that is, $\bar x_A=\bar x_B=\frac12(1+f)\le0$. By Lemma~\ref{lem:d5rules}(d), $3r<2+\omega=1+2\bar x_A$. So $y=(r,\bar x_A,r,\bar x_A,r)$ with $0<3r<1+2\bar x_A$ and $\bar x_A\le0$, which lies in $\mathcal T$.

\emph{$r=0$.} Put $L=\zeta=-\omega$, so that $L\ge1$, $\zeta-3r=L$ and $\omega-3r=-L$. As $x^0_X\le0$, $t^w_X\ge f_X$ and $t^z_X\ge s_X>1$. By Lemma~\ref{lem:d5rules}(a) and (b), in $z$ either (z1) $\zeta$ is the exception, $c_A=c_B=p^z_A=p^z_B=p^z_c=\mu_z<L$ and $t^z_A=t^z_B=\mu_z$, or (z2) $\zeta=\mu_z=L$ and $t^z_X\ge L$ for both $X$. In $w$ there are four cases:
\begin{enumerate}[leftmargin=3.4em,labelwidth=2.8em,labelsep=0.6em,align=left]
\item[(w1)] $\omega$ is the exception, $f_A=f_B=p^w_A=p^w_B=\mu_w<-L$; then Lemma~\ref{lem:d5rules}(b) gives $t^w_X=f_X$, so $x^0_X=0$ and $t^z_X=s_X=-f_X$;
\item[(w2)] $f_A$ is the exception, and $f_B=\omega=-L<f_A$;
\item[(w3)] the same with $A$ and $B$ exchanged;
\item[(w4)] $f_A=f_B=\omega=-L$.
\end{enumerate}
In the cases (w2)--(w4), (z1) is impossible: for $X$ with $f_X=-L$, $t^z_X=\frac12(3L-s_X)=\mu_z$ would give $x^0_X=\frac12(s_X-L)=L-\mu_z>0$, against $x^0_X\le0$. So $t^z_X\ge L$ for both $X$ in these cases. We show $y\in\mathcal R$. The inequalities $y_2\le y_1\le0$, $y_4\le y_3\le0$ and $y_2,y_4\ge-\frac12$ hold since $s_X>1$, $x^0_X\le0$ and $\bar x_X\ge-\frac12$. As $t^z_X=1-x^0_X-2\bar x_X$ and $f_Y=2\bar x_Y-1$, the other four inequalities say that
\[
-f_Y\le t^z_X\le-\tfrac12(1+3f_Y)\qquad\text{for }\{X,Y\}=\{A,B\}.
\]
The lower bound holds with equality in (w1). In (w2), $t^z_A\ge L=-f_B$ and $t^z_B\ge L>-f_A$; (w3) is the same; and in (w4), $t^z_X\ge L=-f_Y$. For the upper bound, $s_X>1$ gives $t^z_X<-\frac12(1+3f_X)$, which suffices if $f_X\ge f_Y$. It remains to treat (w2) with $X=B$, and (w3) with $X=A$, which is the same. In (F) for $A$ in $w$, the terms are $f_A$, $-L+\beta$ and $t^w_A\ge f_A$. Hence $-L+\beta\ge f_A>-L$, that is, $\beta>0$. In (F) for $B$ in $w$, the terms are $f_B=-L$, $-L+\beta>-L$ and $t^w_B\ge-L$. So $t^w_B=-L=f_B$. Hence $x^0_B=0$, $s_B=L$ and $t^z_B=L$. Finally, $t^z_A\ge L$ gives $3f_A\le-2L-s_A<-2L-1$, that is, $t^z_B=L<-\frac12(1+3f_A)$.
\end{proof}

\begin{proposition}\label{prop:d5zs}
Let $A$ be in mode Z1 and $B$ in mode Zs. At every point of the model of $\Vtwo^{\mathrm w}$, $y$ lies in $\mathcal S^0$ if $r=0$, in $\mathcal S^+$ if $r>0$, and in $\mathcal S^-$ if $r<0$, where
\begin{align*}
\mathcal S^0&=\{y:\ y_1=y_2,\ y_5=0,\ y_4\le y_3\le0,\ -\tfrac12\le y_4\le y_1\le1\},\\
\mathcal S^+&=\{y:\ y_1=y_2,\ y_3=y_5\ge0,\ -\tfrac12\le y_4\le0,\ 3y_5\le y_1+2y_4,\ y_1\le1\},\\
\mathcal S^-&=\{y:\ y_1=y_2,\ y_3=0,\ y_4\le y_5\le0,\ y_1+2y_4\ge0,\ y_1\le1\}.
\end{align*}
Moreover $x^0_B\le\max(r,0)$. If $r>0$, then $f_B=\omega\le-1$ and $t^w_A=\omega-3r$. If $r=0$, then $f_B\le f_A$.
\end{proposition}

\begin{proof}
Here $y=(\bar x_A,\bar x_A,x^0_B,\bar x_B,r)$, with $\bar x_A=\frac12(1+f_A)<1$, $\bar x_B\ge-\frac12$ and $x^0_B>\bar x_B$. By Lemma~\ref{lem:d5rules}(a), $\zeta\ge1$ and $p^z_A,p^z_B,p^z_c\ge1$. By Lemma~\ref{lem:d5rules}(c), (V) says that the minimum of $f_A+t^z_A$, $-x^0_B$, $-r$ and $0$ is attained at least twice, and $f_A+t^z_A\ge-\bar x_A$.

\emph{Step 1: $x^0_B\le\max(r,0)$.} Suppose not. Then the minimum in (V) is attained by $-x^0_B$ and $f_A+t^z_A$. Hence $\bar x_A\ge x^0_B>0$ and $f_A>-1$. If $t^w_B<\omega-3r$, Lemma~\ref{lem:d5rules}(b) gives $p^w_B=t^w_B<f_B$, so $f_B$ is the exception in $w$, $p^w_B=\omega=f_A$ and $\omega<\omega-3r$. Then $r<0$ and $\zeta=2r-\omega<1$, against $\zeta\ge1$. Otherwise $f_B\ge\omega+3(x^0_B-r)>\omega$. Thus again $f_B$ is the exception and $f_A=\omega$. If $r\le0$, then $\zeta=2r-\omega<1$ as before. If $r>0$, then $\zeta-3r=-r-\omega>-x^0_B-\omega=t^z_A$, and Lemma~\ref{lem:d5rules}(b) gives $p^z_A=t^z_A=-x^0_B-\omega$. As $p^z_A\ge1$, we obtain $\omega<-1$, against $\omega=f_A>-1$.

\emph{Step 2: if $r>0$, then $x^0_B=r$.} If $x^0_B<r$, the minimum in (V) is attained by $-r$ and $f_A+t^z_A$, so $\bar x_A\ge r>0$. By Lemma~\ref{lem:d5rules}(e), $c_B+f_B=0$. Thus $f_B\le-1<f_A$. So $f_A$ is the exception in $w$ and $f_B=\omega=p^w_B$. In (E) for $B$ in $w$, the term $\omega-3r$ is less than $p^w_B=\omega$ and than $t^w_B=\omega-3x^0_B$. Hence the minimum is attained once.

\emph{Step 3: if $r>0$, then $f_B=\omega\le-1$, $t^w_A=\omega-3r$ and $3r\le\bar x_A+2\bar x_B$.} As $x^0_B=r$, (E) for $B$ in $w$ gives $f_B=\omega=\mu_w$, as in the proof of Proposition~\ref{prop:d5ss}. If $c_B+f_B>0$, Lemma~\ref{lem:d5rules}(e) gives $c_A+f_A=0$, so $f_B=\mu_w\le f_A\le-1$; otherwise $f_B\le-c_B\le-1$. In (F) for $B$ in $w$, the terms are $\omega$, $\omega+\beta$ and $t^w_B=\omega-3r$. Hence $\beta=-3r$. In (F) for $A$ in $w$, the terms are $f_A\ge\omega>\omega-3r=\omega+\beta$ and $t^w_A$. Thus $t^w_A=\omega-3r$. By (S), $t^w_A\ge f_A-3\bar x_A=-\frac12(3+f_A)$. Hence $6r\le3+f_A+2f_B$, which is $3r\le\bar x_A+2\bar x_B$.

\emph{Step 4: if $r<0$, then $x^0_B=0$, $\bar x_B\le r$ and $\bar x_A+2\bar x_B\ge0$.} Since $x^0_B\le0$ by step 1, $t^w_B\ge f_B$. If $p^w_B=\mu_w$, Lemma~\ref{lem:d5rules}(b) gives $t^w_B=\mu_w\le f_B\le t^w_B$, so $x^0_B=0$ and $f_B=\mu_w$. If $p^w_B>\mu_w$, then $f_A=f_B=\omega=p^w_A=\mu_w$, and Lemma~\ref{lem:d5rules}(b) gives $t^w_A=f_A$. By (S), $f_A\ge-\frac12(3+f_A)$. Thus $\omega=f_A\ge-1$ and $\zeta=2r-\omega<1$, which is impossible. So $x^0_B=0$ and $f_B=\mu_w$.

Next, $c_B+f_B\le2r$. Otherwise, by Lemma~\ref{lem:d5rules}(e), $c_A+f_A=2r$. Hence $f_A<-1$, $\bar x_A<0$ and $t^w_A\ge f_A-3\bar x_A>f_A$. Then $p^w_A\ne\mu_w$ by Lemma~\ref{lem:d5rules}(b). So $f_A=f_B=\omega$. Now $c_B>c_A$ makes $c_B$ the exception in $z$. Hence $p^z_B=\mu_z=\zeta$, and Lemma~\ref{lem:d5rules}(b) gives $t^z_B=\zeta$. As $x^0_B=0$, this gives $s_B=t^z_B=\zeta$ and $x^0_B=\frac12(\zeta+\omega)=r<0$, a contradiction. So $f_B\le2r-c_B\le2r-1$, that is, $\bar x_B\le r$.

Finally, if $p^w_A=\mu_w$, then Lemma~\ref{lem:d5rules}(b) gives $t^w_A=\mu_w=f_B$, and (S) gives $f_B\ge-\frac12(3+f_A)$, that is, $\bar x_A+2\bar x_B\ge0$. If $p^w_A>\mu_w$, then $f_A=f_B=\omega$, and $s_B=-f_B$ gives $t^z_B=s_B=-\omega=\zeta-2r$. As $\zeta-2r<\zeta-3r$, (E) for $B$ in $z$ and Lemma~\ref{lem:d5rules}(b) give $p^z_B=t^z_B>\zeta$. So the position of $B$ is the exception in $z$ and $p^z_A=\zeta=\mu_z$. Then $t^z_A=\zeta$. By (S), $\zeta\ge-\frac12(1+3\omega)$, that is, $\omega\ge-1-4r$. But $\omega=f_B\le2r-1<-1-4r$.

\emph{Step 5: if $r=0$, then $f_B\le f_A$.} Otherwise $f_A=\omega=\mu_w$. In (F) for $B$ in $w$, the terms are $f_B$, $\omega+\beta$ and $t^w_B\ge f_B$. Hence $\beta\ge f_B-\omega>0$. In (F) for $A$ in $w$, the terms are $\omega$, $\omega+\beta>\omega$ and $t^w_A$. Thus $t^w_A=\omega=f_A$, and (S) gives $f_A\ge-1$. Then $\bar x_B>\bar x_A\ge0$. Hence $x^0_B>\bar x_B>0$, against step 1.

Steps 1--5, with $\bar x_A<1$, $\bar x_B\ge-\frac12$ and $x^0_B>\bar x_B$, give the three cases.
\end{proof}

\begin{proposition}\label{prop:d5VN}
\begin{enumerate}[label=\textup{(\alph*)}]
\item The models of $\VN$ for \textup{(Z1,\,W1)}, \textup{(F,\,Z1)} and \textup{(Z1,\,Zs)} are empty, and so is the model of Theorem~C0 of the supplementary material for \textup{(Z1,\,Zs$_{t_0}$)} in the centred case.
\item At every point of the model of $\VN$ for \textup{(Z1,\,Z1)}, $\bar x_A=\bar x_B\le0$.
\end{enumerate}
\end{proposition}

\begin{proof}
Since these models contain the relations of the model of $\Vtwo$, Propositions~\ref{prop:d5zero} and~\ref{prop:d5zs} apply to them. By (R), a body $X$ in mode Z1 has at most one event. For each of the other quantities of (R) the opposite inequality holds. With (S), (Q) and (T), this means that $c_X=1$, that $t^z_X$ and $t^w_X$ take their lower bounds $1-3\bar x_X$ and $f_X-3\bar x_X$, that $\tau_X=-\bar x_X$ and, if $Y$ is in mode Zs, that $p^z_A,p^z_B,p^z_c\le1$ and (in the model of $\VN$) that $\tilde t_X=-\bar x_X$, except for the event.

(a) \emph{(Z1,\,W1).} By Proposition~\ref{prop:d5zero}(a) and its proof, $r=0$, $\zeta=g:=c_B$, $\omega=f_A$ and $f_A+g=0$. Lemma~\ref{lem:d5rules}(b) gives $t^z_B\ge g$ and $t^w_B\ge f_A$. If $t^z_B$ took its lower bound $-\frac12(3+g)$, then $g\le-1$. If $t^w_B$ took its lower bound $-\frac12(1+3g)$, then $-\frac12(1+3g)\ge-g$, and again $g\le-1$. As $g=-f_A>-1$, both are events of $B$, against (R).

\emph{(F,\,Z1).} By Proposition~\ref{prop:d5zero}(c) and its proof, $r=0$, $\bar x_A=0$, $\zeta=c_A=:\kappa<1$, $\omega=\mu_w=-\kappa$ and $f_B\ge-\kappa$. Lemma~\ref{lem:d5rules}(b) gives $t^z_B\ge\kappa$ and $t^w_B\ge-\kappa$. If $t^z_B=-\frac12(1+3f_B)$, then $f_B\le-\frac13(2\kappa+1)<-\kappa$. So $t^z_B$ is the event of $B$. Then $t^w_B=-\frac12(3+f_B)\ge-\kappa$ gives $f_B\le2\kappa-3<-1$, against $f_B\ge-\kappa>-1$.

\emph{(Z1,\,Zs).} Let $A$ be in mode Z1 and $B$ in mode Zs$_{t_0}$, in the model of $\VN$ or in that of the centred case. By Proposition~\ref{prop:d5zs}, $y$ lies in $\mathcal S^0$, $\mathcal S^+$ or $\mathcal S^-$. We say that \emph{det holds} if $p^z_A,p^z_B,p^z_c\le1$; it fails only if this is the event of $A$. As these positions are at least $1$, det means that they are equal to $1=\mu_z$. If det fails, then $c_A=1=\mu_z$.

If $r=0$, then $x^0_B\le0$ and $t^z_B\ge s_B>1$. If det holds, (E) for $B$ in $z$ has the terms $p^z_B=1$, $\zeta\ge1$ and $t^z_B>1$, so $\zeta=1$. If det fails, $\zeta>1$ would make $\zeta$ the exception in $z$, and Lemma~\ref{lem:d5rules}(b) would give $t^z_B=p^z_B=1$; so again $\zeta=1$ and $\omega=-1$. If $\omega$ were the exception in $w$, then $p^w_A=f_A<-1$, and Lemma~\ref{lem:d5rules}(b) would give $t^w_A=f_A$, against $t^w_A\ge-\frac12(3+f_A)>f_A$. So $\omega=\mu_w=-1$ and one of $f_A$, $f_B$ is equal to $-1$. As $f_B\le f_A$, this gives $f_B=-1$. Then $\bar x_B=0<x^0_B$, against $x^0_B\le0$.

If $r<0$, then $x^0_B=0$ and $t^z_B=s_B>1$. If det holds, then $p^z_B=1=\mu_z$, and Lemma~\ref{lem:d5rules}(b) gives $t^z_B=1$, against $t^z_B>1$. If det fails, $p^z_B=\mu_z$ would give the same, so the position of $B$ is the exception in $z$ and $c_A=\zeta=p^z_A=1$. Then $t^z_A=1$ by Lemma~\ref{lem:d5rules}(b). As $t^z_A$ takes its lower bound $-\frac12(1+3f_A)$, $f_A=-1$, that is, $\bar x_A=0$. Then $\bar x_A+2\bar x_B\ge0$ gives $\bar x_B\ge0>r$, against $\bar x_B\le r$.

If $r>0$, then $x^0_B=r$ and $f_B=\omega\le-1$, so $s_B=2r-\omega=\zeta>1$ and $t^z_B=\zeta-3r$. In the centred case $c_B>1$. If det holds, $\zeta>1=\mu_z$ makes $\zeta$ the exception in $z$, so $c_B=1$. If det fails, $c_A=1=\mu_z$ makes $c_B$ the exception, so $\zeta=1$. Both are impossible. In the model of $\VN$, suppose first that det holds and that $t^z_A$ and $t^w_A$ take their lower bounds. By (E) for $B$ in $z$, with the terms $p^z_B=1$ and $t^z_B=\zeta-3r$, we have $\zeta-3r\le1$. The equality $t^w_A=\omega-3r=-\frac12(3+f_A)$ gives $f_A=6r-3-2\omega$. If $\zeta-3r<1=p^z_A$, Lemma~\ref{lem:d5rules}(b) gives $t^z_A=\zeta-3r=-r-\omega$, and $t^z_A=-\frac12(1+3f_A)$ then gives $\omega=2r-1$ and hence $r=\frac12(1+\omega)\le0$. If $\zeta-3r=1$, then $\omega=-1-r$, $f_A=8r-1$, and $t^z_A\ge1$ gives $f_A\le-1$ and hence $r\le0$. Hence the event of $A$ is the failure of det, or $t^z_A$, or $t^w_A$. In each case $\tilde t_A=-\bar x_A$. In (T), $e_1-e_2=\frac12(3-2r-\omega)>0$, as $\omega\le-1$ and $r<1$. Hence $\tilde t_B=e_2=-\frac12(1+\omega)$. If $t^w_A$ takes its lower bound, then $\bar x_A=3r-1-\omega$ and $\tilde t_A=1+\omega-3r<\tilde t_B$, so (T) gives $p^w_c+1=\tilde t_A$, that is, $p^w_c=\omega-3r<\mu_w$, which is impossible. If $t^z_A$ takes its lower bound and the event is $t^w_A$, then det holds, and as above $t^z_A=\zeta-3r$ if $\zeta-3r<1$ and $t^z_A\ge1$ if $\zeta-3r=1$. In the first case $\bar x_A=\frac13(1+r+\omega)$. Then $3r\le\bar x_A+2\bar x_B$ gives $8r\le4+4\omega\le0$. In the second case $\bar x_A\le0$ and $\bar x_B=\frac12(1+\omega)=-\frac12r$. Then $3r\le\bar x_A+2\bar x_B$ gives $4r\le\bar x_A\le0$. Both contradict $r>0$.

(b) By Proposition~\ref{prop:d5zero}(d), $r=0$, and one of $f_A$, $f_B$ is equal to $\mu_w\le-1$. By (Q) and (R), $\tau_A=\tau_B$ and $\tau_X\ge-\bar x_X$, with equality unless $\tau_X$ is the event of $X$. Suppose that $f_B<f_A$. Then $\tau_A=\tau_B\ge-\bar x_B>-\bar x_A$. So $\tau_A$ is the event of $A$. Hence $c_A=1$ and $t^z_A$, $t^w_A$ take their lower bounds. By Lemma~\ref{lem:d5rules}(a), $f_B=\omega=-\zeta$ and $\mu_z=1$. Lemma~\ref{lem:d5rules}(b) gives $t^z_A\ge1$, that is, $f_A\le-1$. If $\zeta>1$, then $\zeta$ is the exception in $z$, and Lemma~\ref{lem:d5rules}(b) gives $t^z_A=t^z_B=1$. So $f_A=-1$. Then $t^z_B=1\ge-\frac12(1+3f_B)$ contradicts $f_B<-1$. If $\zeta=1$, then $f_B=-1<f_A\le-1$. By symmetry $f_A=f_B$ and both are equal to $\mu_w\le-1$, that is, $\bar x_A=\bar x_B\le0$.
\end{proof}

\begin{proof}[Proof of Theorem~\textup{\ref{thm:hulls}}]
By Lemma~\ref{lem:d5sym} it suffices to treat the nine pairs of Table~\ref{tab:d5hulls}. Proposition~\ref{prop:d5empty} gives the empty pairs. For (Z1,\,W1), (F,\,F), (F,\,Z1) and (Z1,\,Z1), Proposition~\ref{prop:d5zero} gives $y=\bar x_Ae_A+\bar x_Be_B$ with the stated conditions, and $\bar x_A,\bar x_B\in[-\frac12,1]$; these points lie in the polytopes of the table. For (Zs,\,Zs), Proposition~\ref{prop:d5ss} gives $y\in\mathcal R\cup\mathcal T$. For (Z1,\,Zs), Proposition~\ref{prop:d5zs} gives $y\in\mathcal S^0\cup\mathcal S^+\cup\mathcal S^-$. Proposition~\ref{prop:d5VN} gives the column of $\VN$ and the centred case; for (F,\,F) and (Zs,\,Zs) the model of $\VN$ lies in that of $\Vtwo$. We find the vertices of $\mathcal R$, $\mathcal S^0$, $\mathcal S^+$ and $\mathcal S^-$ by the following remark. Let $\Delta$ be a polytope and $q_1,\dots,q_k$ affine maps on $\Delta$. If $\delta=\sum_v\lambda_vv$ is a convex combination of the vertices $v$ of $\Delta$ and $u=\sum_j\mu_jq_j(\delta)$ is a convex combination, then $(\delta,u)=\sum_{v,j}\lambda_v\mu_j\,(v,q_j(v))$. So every point $(\delta,u)$ with $\delta\in\Delta$ and $u\in\conv\{q_1(\delta),\dots,q_k(\delta)\}$ lies in the convex hull of the points $(v,q_j(v))$.

On $\{y_5=0\}$ we put $u_1=y_1+2y_2$ and $u_2=y_3+2y_4$. Since $y_2\le y_1\le0$ means $3y_2\le u_1\le2y_2$, and likewise for $y_3$ and $y_4$, the polytope $\mathcal R$ is the set of the points with $y_2,y_4\ge-\frac12$ and $3\max(y_2,y_4)\le u_1,u_2\le2\min(y_2,y_4)$. Over the triangle $-\frac12\le y_2\le y_4\le\frac23y_2$, with the vertices $(y_2,y_4)=(-\frac12,-\frac12)$, $(-\frac12,-\frac13)$ and $(0,0)$, the bounds are $3y_4$ and $2y_2$, and $(u_1,u_2)$ lies in the square spanned by the four pairs of bounds; over the triangle with $y_2$ and $y_4$ exchanged the same holds. By the remark, $\mathcal R$ lies in the convex hull of the points with $(y_2,y_4)=(-\frac12,-\frac12)$ and $u_1,u_2\in\{-\frac32,-1\}$, the points with $(y_2,y_4)=(-\frac12,-\frac13)$ or $(-\frac13,-\frac12)$ and $u_1=u_2=-1$, and $0$. These seven points lie in $\mathcal R$. They are the seven vertices of $P_{ss}$ with $y_5=0$. The triangle $\mathcal T$ is the convex hull of two of them and $(\frac13,0,\frac13,0,\frac13)$. So $P_{ss}=\conv(\mathcal R\cup\mathcal T)$, a pyramid over $\mathcal R$ with apex $(\frac13,0,\frac13,0,\frac13)$.

The points of $\mathcal S^0$, $\mathcal S^+$ and $\mathcal S^-$ have $y_1=y_2$ and $-\frac12\le y_4\le0$. They are determined by $y_4$ and the pair $(y_1,y_3)$, $(y_1,y_5)$, resp.\ $(y_1,y_5)$. For fixed $y_4$ this pair ranges over the rectangle $[y_4,1]\times[y_4,0]$ for $\mathcal S^0$ (where $y_5=0$), over the triangle with the vertices $(-2y_4,0)$, $(1,0)$ and $(1,\frac13(1+2y_4))$ for $\mathcal S^+$ (where $y_3=y_5$), and over the rectangle $[-2y_4,1]\times[y_4,0]$ for $\mathcal S^-$ (where $y_3=0$). By the remark, applied over the segment $-\frac12\le y_4\le0$, $\mathcal S^0$ lies in the convex hull of $-\frac12(e_A+e_B)$, $(-\frac12,-\frac12,0,-\frac12,0)$, $0$, $e_A-\frac12e_B$, $(1,1,0,-\frac12,0)$ and $e_A$, $\mathcal S^+$ in that of $0$, $(1,1,0,-\frac12,0)$, $e_A$ and $(1,1,\frac13,0,\frac13)$, and $\mathcal S^-$ in that of $0$, $(1,1,0,-\frac12,-\frac12)$, $(1,1,0,-\frac12,0)$ and $e_A$. These points lie in the respective sets. Together they are the eight vertices of $P_{1s}$. So $P_{1s}=\conv(\mathcal S^0\cup\mathcal S^+\cup\mathcal S^-)$.
\end{proof}

%% file: sec/app_d2.tex
\section{Diagram 2 beyond the torque}\label{app:d2}

In this appendix we prove Lemma~\ref{lem:d2gate} and Theorem~\ref{thm:gate}. Let $P$ be a singular place of diagram $2$ of a curve $C$ as in Theorem~\ref{thm:types}. Exchanging the colours and relabelling the bodies (with their masses) we may assume that the segment is $\{1,2\}$, $z$-coloured, and the triangle $\{3,4,5\}$, $w$-coloured. Then the ends of the segment are $z$-circled and the vertices of the triangle are $w$-circled. Put $M_t=m_3+m_4+m_5$, $\rho=m_2/m_1$, and fix $\jj_\rho$ with $\jj_\rho^2=-\rho$.

\subsection{The weighted system}\label{ss:d2weights}
By Table~\ref{tab:types} the type is $\pc a_1(\{3,4,5\})$: $x_{kl}=0$ inside the segment and the triangle and $x_{kl}=-2$ across. By (O3) and (O6) of Lemma~\ref{lem:orders}, $\zeta_{12}=-2$, $\omega_{12}=2$, $\omega_{kl}=-2$ and $\zeta_{kl}=2$ on the sides of the triangle, and $\zeta=\omega=-2$ across. Hence, in the coordinates $u_k=z_k-z_5$, $v_k=w_k-w_5$ of \S\ref{ss:slice}, the exponents of Lemma~\ref{lem:place} are $e(u)=(-2,-2,2,2)$, $e(v)=(-2,-2,-2,-2)$, $e(p_{kl})=0$ for $kl\in\{12,34,35,45\}$ and $e(p_{kl})=2$ for the six cross pairs. Let $w=e/2$ be the \emph{weights}. For each of the eighteen rows \eqref{eq:rows}, we substitute $q=T^{w_q}y_q$ for each coordinate $q$ and divide by the lowest power of $T$. This gives $G_r(T,y)=\sum_{i\ge0}T^iG_{r,i}(y)$. Since its coefficients are $\pm1$, $\pm m_l$ and $-(m_k+m_5)$, the same monomials occur at all positive masses. A \emph{leading solution} is a zero $y_0$ of $G_0=(G_{r,0})_r$ with $y_{u_1}=y_{v_1}=1$ and all coordinates nonzero. The matrix $J_0$ is the $18\times16$ Jacobian matrix of $G_0$ at $y_0$ in the other sixteen unknowns.

\begin{lemma}\label{lem:d2branch}
After a rotation depending on the local parameter and a change of the parameter, the coordinates near $P$ are $q=T^{w_q}y_q(\xi)$ with $T=\xi^L$ for an integer $L\ge1$, where the $y_q$ are convergent power series in $\xi$ with $y_{u_1}\equiv y_{v_1}\equiv1$, $G(\xi^L,y(\xi))\equiv0$, and $y(0)$ is a leading solution.
\end{lemma}

\begin{proof}
By Lemma~\ref{lem:place}, after the rotation there, each coordinate is $\vartheta^{d\,e_q}\tilde y_q(\vartheta)$ with $\tilde y_q(0)\ne0$. We put $L=2d$ (a positive integer) and $T=\vartheta^L$, so that $q=T^{w_q}\tilde y_q$. A further rotation by the unit $\tilde y_{u_1}^{-1}$ makes $y_{u_1}\equiv1$ and multiplies all other $\tilde y_q$ by units. The zero set of $G_0$ is invariant under $y_q\mapsto\lambda^{e'_q}y_q$, with $e'_q=w_q-w_{u_1}$, $w_q+w_{u_1}$ or $w_q$ for $q$ a $u_k$, a $v_k$ or a $p_{kl}$, which comes from $T\mapsto\lambda T$ combined with a rotation. In particular, $e'_{u_1}=0$ and $e'_{v_1}=-2$. Replacing $T$ by $T'\phi^L$ and rotating accordingly turns $y_q$ into $\phi^{Le'_q}y_q$. The equation $\phi^{-2L}y_{v_1}(\xi\phi)=1$ is solved by a convergent $\phi(\xi)$, $\phi(0)^{2L}=y_{v_1}(0)$, by the implicit function theorem. Since all rows vanish on $X_m$, $G(\xi^L,y(\xi))\equiv0$. At $\xi=0$ this gives $G_0(y(0))=0$.
\end{proof}

\subsection{The leading solutions and the torque}\label{ss:d2torque}
For $h\in\C^{\{3,4,5\}}$ let $F'_k(h)=h_k-\sum_{l\in\{3,4,5\}\setminus k}m_lp_{kl}^3(h_k-h_l)$, the rows of the three bodies of the triangle alone. The rows of $G_0$ are the least-weight terms of the rows \eqref{eq:rows}.

\begin{lemma}\label{lem:d2lead}
At a leading solution: $u_2=-m_1/m_2$, $p_{12}^3=1/(m_1+m_2)$, $v_2=v_c:=(m_3v_3+m_4v_4)/M_t=1$; with $\omega_k=v_k-v_c$ for $k\in\{3,4,5\}$ (so $\omega_5=-1$ and $\sum_km_k\omega_k=0$), $p_{1k}^2=-1/\omega_k$ and $p_{2k}=\sigma_k\jj_\rho p_{1k}$ with $\sigma_k=\pm1$; and the triangle coordinates $u_3,u_4,v_3,v_4,p_{34},p_{35},p_{45}$ satisfy $F'_3=F'_4=F'_5$ for $h=u$ and for $h=v$ and $p_{kl}^2u_{kl}v_{kl}=1$, so they form a complex central configuration of three bodies: of Lagrange type ($p_{kl}^3=1/M_t$) or of Euler type ($u$ parallel to $v$).
\end{lemma}

\begin{proof}
The least-weight parts of $Z_1$, $Z_2$ are $u_1-m_2p_{12}^3(u_1-u_2)$ and its mirror, which give $u_2$ and $p_{12}$. That of $D_{12}$ gives $v_2=v_1=1$. Those of $W_1$ and of $m_3W_3+m_4W_4$ give $v_2=\bar v$ and $\bar v=v_c$, where $\bar v=m_3v_3p_{35}^3+m_4v_4p_{45}^3$. Those of $D_{1k}$, $D_{2k}$ are $p_{1k}^2(1-v_k)-1$ and $p_{2k}^2u_2(1-v_k)-1$. The remaining seven rows are the three-body rows of the triangle. For the dichotomy let $\Psi(h)_k=\sum_lm_lp_{kl}^3(h_k-h_l)$. By antisymmetry $\Psi$ maps the plane $\sum_km_kh_k=0$ to itself. The vectors $u-u_{\mathrm{cm}}\one$ and $\omega$ are fixed vectors of $\Psi$ in it. If they are independent, $\Psi$ is the identity on the plane, and comparing the quadratic forms $\sum_km_kh_k\Psi(h)_k$ and $\sum_km_kh_k^2$, we obtain $p_{kl}^3=1/M_t$. Otherwise $\omega$ is a multiple of $u-u_{\mathrm{cm}}\one$.
\end{proof}

A \emph{triple point} is a leading solution whose triangle is of Lagrange type with $u=\Lambda v$. Since the shape of a Lagrange triangle with $u$ parallel to $v$ is equilateral, triple points have an equilateral $\omega$.

Let $\ell(y)\in\C^{18}$ have the entries $\ell_{Z_k}=m_k(M_tv_k-s_v)$ and $\ell_{W_k}=-m_k(M_tu_k-s_u)$ for $k\in\{3,4\}$, with $s_u=m_3u_3+m_4u_4$, $s_v=m_3v_3+m_4v_4$, and $0$ elsewhere. Put $\lambda=\ell(y_0)$, $\tau(y)=\ell(y)\cdot G_1(y)$ (the \emph{torque}) and $\sigma_3(y)=\ell(y)\cdot G_3(y)$. At a leading solution put, for $k\in\{3,4,5\}$,
\[
\tau_k=m_kp_{1k}(m_1+\sigma_k\jj_\rho m_2),\qquad \theta_k=\frac{\tau_k}{m_k\omega_k}=\frac{m_1p_{1k}(1+\sigma_k\jj_\rho\rho)}{\omega_k}.
\]
We say that $\theta$ is \emph{constant} if $\theta_3=\theta_4=\theta_5$. Since $|m_1+\sigma_k\jj_\rho m_2|^2=m_1^2(1+\rho^3)$ and $|p_{1k}|=|\omega_k|^{-1/2}$, $|\theta_k|=m_1(1+\rho^3)^{1/2}|\omega_k|^{-3/2}$. Thus constant $\theta$ forces $|\omega_3|=|\omega_4|=|\omega_5|$.

\begin{lemma}\label{lem:d2rot}
\begin{enumerate}[label=\textup{(\alph*)}]
\item $\ell(y)\cdot G_0(y)=0$ identically.
\item $G_1$ vanishes on every row except $Z_3$, $Z_4$, where $G_{1,Z_k}=\sum_{j=1,2}m_ju_j(p_{kj}^3-p_{5j}^3)$. Moreover, $\ell\cdot G_2\equiv0$ and $\ell\cdot G_n\equiv0$ for $n\ge4$.
\item At a leading solution, $G_{1,Z_k}=-(\theta_k-\theta_5)$ and $\tau=-M_t\sum_{k\in\{3,4,5\}}\tau_k$.
\item $\ell(y)\cdot G(T,y)=T\tau(y)+T^3\sigma_3(y)$, so along the series of Lemma~\ref{lem:d2branch}, $\tau(y(\xi))=-T^2\sigma_3(y(\xi))$. In particular $\tau(y_0)=0$.
\end{enumerate}
\end{lemma}

\begin{proof}
(a) The rows $Z_k$, $W_k$ of $G_0$, $k\in\{3,4\}$, are $F'_k-F'_5$ at $u$ and at $v$ (Lemma~\ref{lem:d2lead}). By antisymmetry $\sum_km_kF'_k(u)=s_u$, $\sum_km_kF'_k(v)=s_v$ and $\sum_km_k\big(v_kF'_k(u)-u_kF'_k(v)\big)=0$. Inserting these into $\ell\cdot G_0$ gives $0$. Note that this is the conservation of angular momentum of the triangle. (b) Each row is a polynomial in $T$. In $Z_k$, $k\in\{3,4\}$, the terms with a factor $p_{kj}$ or $p_{5j}$, $j\in\{1,2\}$, have weight $2$ (those with $u_j$) or $4$, one or three more than the least weight $1$. Hence $Z_k$ has only the relative exponents $0$, $1$, $3$. In $W_k$ these terms have weight $2$, three more than $-1$. So $W_k$ has only $0$ and $3$. The other rows have the relative exponents $0$, $2$, $3$ and $5$ ($Z_1$, $Z_2$), $0$ and $3$ ($W_1$, $W_2$), $0$ and $2$ ($D_{12}$ and $D_{jk}$, $j\in\{1,2\}$, $k\in\{3,4\}$) and $0$ ($D_{15}$, $D_{25}$ and the sides of the triangle). Since $\ell$ is supported on $Z_3,Z_4,W_3,W_4$, the coefficients of $T^2$ and of $T^n$, $n\ge4$, vanish there. (c) With $u_1=1$, $u_2=-m_1/m_2$, $p_{2k}^3=-\rho\sigma_k\jj_\rho p_{1k}^3$ and $p_{1k}^3=-p_{1k}/\omega_k$ we get $m_1(p_{1k}^3-p_{2k}^3)=-\theta_k$. Thus $G_{1,Z_k}=-(\theta_k-\theta_5)$. With $\ell_{Z_k}=M_tm_k\omega_k$ and $\sum_km_k\omega_k=0$, we obtain $\tau=-M_t\sum_km_k\omega_k\theta_k$. (d) follows from (a), (b), and $G(\xi^L,y(\xi))\equiv0$.
\end{proof}

\begin{lemma}\label{lem:d2theta}
At a leading solution with $\tau=0$, $\theta_3=\theta_4$ holds if and only if $\theta$ is constant, if and only if $G_1(y_0)=0$.
\end{lemma}

\begin{proof}
With $a_k=m_k\omega_k\ne0$, $\sum_ka_k=0$ and $\sum_ka_k\theta_k=0$ by Lemma~\ref{lem:d2rot}(c). Hence $a_4(\theta_3-\theta_4)=a_5(\theta_5-\theta_3)$ and $a_4(\theta_4-\theta_5)=a_3(\theta_5-\theta_3)$. The last equivalence is Lemma~\ref{lem:d2rot}(b), (c).
\end{proof}

\subsection{Simple points}\label{ss:d2simple}

\begin{lemma}\label{lem:d2simple}
Let $y_0$ be a leading solution with $\tau(y_0)=0$ and $\operatorname{rank}J_0=16$.
\begin{enumerate}[label=\textup{(\alph*)}]
\item The left kernel of $J_0$ is spanned by $\lambda$ and $l_0=e_{W_1}-e_{W_2}$, and $J_0y_1=-G_1(y_0)$ has exactly one solution $y_1$.
\item $y_1$ vanishes on $u_2,p_{12},v_2$; $(y_1)_{p_{jk}}=-p_{jk}(y_1)_{v_k}/(2\omega_k)$ for $j\in\{1,2\}$, $k\in\{3,4\}$, and $(y_1)_{p_{j5}}=0$; and $m_3(y_1)_{v_3}+m_4(y_1)_{v_4}=0$. Write $(y_1)_{v_3}=t/m_3$, $(y_1)_{v_4}=-t/m_4$.
\item A series as in Lemma~\ref{lem:d2branch} with $y(0)=y_0$ exists only if
\[
\eta_2:=D\tau(y_0)\cdot y_1=\tfrac12M_t\,t\,(\theta_3-\theta_4)=0 .
\]
\item Suppose that $\theta$ is constant. Then $y_1=0$, and, with $y_2$ the solution of $J_0y_2=-G_2(y_0)$, a series as in \textup{(c)} exists only if
\[
\Omega_3:=\lambda\cdot\big(DG_1(y_0)\,y_2+G_3(y_0)\big)=D\tau(y_0)\cdot y_2+\sigma_3(y_0)=0 .
\]
\end{enumerate}
\end{lemma}

\begin{proof}
(a) $\lambda J_0=0$ by Lemma~\ref{lem:d2rot}(a), and $l_0J_0=0$ because the rows $W_1$, $W_2$ of $G_0$ differ by $(1-v_2)(1-(m_1+m_2)p_{12}^3)$, whose gradient vanishes at a leading solution. The two vectors are independent ($\lambda_{Z_k}=M_tm_k\omega_k\ne0$). Thus they span the left kernel, of dimension $2$. Since $\lambda\cdot G_1=\tau=0$ and $l_0\cdot G_1=0$, a solution $y_1$ exists. It is unique by the rank.
(b) The rows $Z_1$, $Z_2$, $D_{12}$ of $G_0$ involve only $u_2,p_{12},v_2$, with an invertible Jacobian there, and zero right side. The rows $D_{jk}$ give the across entries. The row $W_1$ and the identity $m_3W_3+m_4W_4=s_v-M_t\bar v$ of the rows of $G_0$ give $D s_v\cdot y_1=0$.
(c) Since $J_0$ is injective, the series is a power series in $T$ with first coefficient $y_1$ (the coefficient of a non-integral power, if any, would be in the kernel of $J_0$). So $\tau(y(T))=T\,D\tau(y_0)y_1+O(T^2)$, which is $O(T^2)$ by Lemma~\ref{lem:d2rot}(d). For the formula: $\tau=\sum_{k=3,4}\ell_{Z_k}G_{1,Z_k}$. By (b), $DG_{1,Z_k}\cdot y_1=\frac{3(y_1)_{v_k}}{2\omega_k}\theta_k$ and $D\omega_k\cdot y_1=(y_1)_{v_k}$. Hence we obtain $D\tau\cdot y_1=M_t\sum_{k=3,4}m_k\big(-(\theta_k-\theta_5)+\frac32\theta_k\big)(y_1)_{v_k}=\frac12M_t\sum_{k=3,4}m_k\theta_k(y_1)_{v_k}$, using $\sum_{k=3,4}m_k(y_1)_{v_k}=0$.
(d) Since $G_1(y_0)=0$ by Lemma~\ref{lem:d2theta}, $y_1=0$. The equation at order $T^2$ is solvable because $\lambda\cdot G_2\equiv0$ and $G_2$ vanishes on the rows $W$. At order $T^3$ the equation is $J_0y_3=-(DG_1(y_0)y_2+G_3(y_0))$. Its $l_0$-component vanishes (it is the cross part of $F'_1(v)-F'_2(v)$, a multiple of $\tau$). Its $\lambda$-component is $\Omega_3$. The second expression for $\Omega_3$ is Lemma~\ref{lem:d2rot}(d).
\end{proof}

\subsection{Triple points}\label{ss:d2triple}

\begin{lemma}\label{lem:d2triple}
Let $y_0$ be a triple point, $u=\Lambda v$ on the triangle, with $\tau(y_0)=0$.
\begin{enumerate}[label=\textup{(\alph*)}]
\item $\ker J_0=\C\,\kappa$, where $\kappa$ has $\delta v=(1,-m_3/m_4)$ on $(v_3,v_4)$, $\delta u=-\Lambda\,\delta v$ on the triangle, $\delta p_{kl}=0$ on the triangle, $0$ on $u_2,p_{12},v_2$, and across entries as in Lemma~\ref{lem:d2simple}(b). So $\operatorname{rank}J_0=15$. The rows $R$ other than $Z_3,Z_4,W_2$ and the columns other than $v_3$ give an invertible $15\times15$ minor.
\item A series as in Lemma~\ref{lem:d2branch} with $y(0)=y_0$ exists only if $\theta$ is constant.
\item Every triple point has the equilateral shape of $\omega$, so constant $\theta$ forces $m_3=m_4=m_5$. Then, if $\tau=0$, $\sigma$ is constant \textup{(}Lemma~\ref{lem:d2equi}(c)\textup{)}. Let $y_0$ have $\theta$ and $\sigma$ constant, let $y(T,s)$ solve the rows $R$ with $y_{v_3}=v_3+s$, and put $f_3(T,s)=Z_3(T,y(T,s))$ and $F(T,s)=\tau(y(T,s))+T^2\sigma_3(y(T,s))$. Then $f_3=b_3sT+O(|(T,s)|^3)$ and $F=F_{ss}s^2+F_{sT}sT+F_{TT}T^2+O(|(T,s)|^3)$, with $b_3=\frac{3\theta_3}{2\omega_3}$, $F_{ss}=-\frac38M_tm_1(1+\sigma\jj_\rho\rho)\,c\sum_km_k(\delta\omega_k)^2/\omega_k$ \textup{(}$c=p_{1k}/\omega_k$, $\delta\omega=(1,-m_3/m_4,0)$\textup{)}, and $F_{TT}$ the value of \eqref{eq:Omega3} below. A series as above exists only if $b_3F_{ss}F_{TT}=0$.
\end{enumerate}
\end{lemma}

\begin{proof}
(a) For $J_0x=0$, (b) of Lemma~\ref{lem:d2simple} holds with zero right sides. At a Lagrange triangle the triangle rows of $J_0$ involve only $\delta P_{kl}$, $P=p^3$, and read $\sum_lm_lu_{kl}\delta P_{kl}=0=\sum_lm_lv_{kl}\delta P_{kl}$. For $u=\Lambda v$ they coincide and force $\delta P=\kappa'\,\delta^0$ with $\delta^0=(m_5v_{35}v_{45},-m_4v_{34}v_{45},m_3v_{34}v_{35})$ on $(34,35,45)$. The rows $D_{kl}$, times $\Lambda v_{kl}$, read $\delta u_{kl}+\Lambda\delta v_{kl}=-\frac23M_t\kappa'\Lambda v_{kl}\delta^0_{kl}$. When we sum over the cycle of the three sides, the left sides cancel and the right side is $\pm\frac23M_t\kappa'\Lambda M_tv_{34}v_{35}v_{45}$. Hence $\kappa'=0$, $\delta u=-\Lambda\delta v$, and the row $W_1$ gives $m_3\delta v_3+m_4\delta v_4=0$. The left kernel of $J_0$ is spanned by $l_0=e_{W_1}-e_{W_2}$, $e_{Z_3}-\Lambda e_{W_3}$ and $e_{Z_4}-\Lambda e_{W_4}$, since $Z_k=\Lambda W_k$ as rows of $J_0$. Removing one row from the support of each and the column $v_3$, where $\kappa$ does not vanish, leaves an invertible minor.
(b) The rows $R$ have a solution $y(T,s)$ with $y_{v_3}=v_3+s$ and no $T^1$ term, since $G_1$ lives on $Z_3,Z_4$. A series gives $s(T)\to0$ with $y=y(T,s(T))$. Here $y(T,s)=y_0+s\kappa+O(s^2+T^2)$. The proof of Lemma~\ref{lem:d2simple}(c) applied to $\kappa$ gives $D\tau\cdot\kappa=\frac12M_tm_3(\theta_3-\theta_4)$. If $\theta_3\ne\theta_4$, Lemma~\ref{lem:d2rot}(d) forces $s(T)=O(T^2)$, and then the rows $Z_3,Z_4$ give $G_1(y_0)=0$, that is $\theta$ constant (Lemma~\ref{lem:d2theta}), a contradiction.
(c) The rows $Z_3$ and $\tau$ restricted to $y(T,s)$ satisfy $\ell_{Z_3}f_3+\ell_{Z_4}f_4=TF$. Hence a series gives a curve germ on which $f_3=F=0$. The coefficient of $T$ in $f_3$ is $G_{1,Z_3}(y_0)=0$. That of $T^2$ vanishes because $Z_3=\Lambda W_3$ as rows of $J_0$ and $G_2$ vanishes on $W_3$. That of $s^2$ vanishes because along $\kappa$, $\delta P=0$ and the rows are bilinear in $(P,h)$. The coefficient of $sT$ is $b_3=DG_{1,Z_3}\cdot\kappa$. At $T=0$ the across rows hold exactly and $F(0,s)=-M_tm_1(1+\sigma\jj_\rho\rho)\sum_km_kp_{1k}$, whose first-order terms are multiples of $\sum_km_k\delta\omega_k=0$, which gives $F_{ss}$. The number $F_{TT}$ is the coefficient of $T^2$ in $F(T,0)$, computed as in Lemma~\ref{lem:d2omega3}. The tangent of the germ is a common zero of the quadratic parts. The equation $b_3sT=0$ leaves the two axes, which are excluded by $F_{TT}\ne0$ and $F_{ss}\ne0$.
\end{proof}

\subsection{The order-three obstruction and the equilateral realizations}\label{ss:d2equi}

\begin{lemma}\label{lem:d2omega3}
In Lemma~\ref{lem:d2simple}(d) let $\sigma$ be constant, let $A=p_{1k}^3$ \textup{(}independent of $k$\textup{)} and $J_{uv}=\sum_km_k\omega_k(u_k-u_{\mathrm{cm}})$, $u_{\mathrm{cm}}=\sum_km_ku_k/M_t$. Then $\sigma_3(y_0)=0$ and
\begin{equation}\label{eq:Omega3}
\Omega_3=\tfrac32M_tm_1A\Big[(1-\sigma\rho^2\jj_\rho)\,J_{uv}+\frac{M_t\rho^2m_1p_{12}(1-\sigma\jj_\rho)}{1+\rho}\Big].
\end{equation}
For a triangle of Lagrange type, $J_{uv}=M_t^{-1/3}\sum_{k<l}m_km_l\zeta_{kl}$, where $p_{kl}=M_t^{-1/3}\zeta_{kl}$ and $\zeta_{kl}^3=1$.
\end{lemma}

\begin{proof}
By the weights, $G_2(y_0)$ is $-1$ on $D_{12}$, $-p_{jk}^2u_k(v_j-v_k)$ on $D_{jk}$ ($j\in\{1,2\}$, $k\in\{3,4\}$), $-\Phi$ with $\Phi=m_3P_{35}u_3+m_4P_{45}u_4$ on $Z_1$ and $Z_2$, and $0$ on the other rows. We solve $J_0y_2=-G_2$: on the segment $\delta p_{12}=0$, $\delta u_2=\Phi(m_1+m_2)/m_2$ and $\delta v_2=-m_2p_{12}$. The triangle part is $\kappa''$ times the rotation, with $\kappa''=m_2^2p_{12}/(m_1+m_2)$ from the row $W_1$. The rows $D_{jk}$ give the across entries. Since $\theta$ is constant, $p_{1k}^3=A$ and $p_{2k}^3=-\sigma\rho\jj_\rho A$ for all $k$. Hence the terms in $\delta u_2$ cancel in $D\tau\cdot y_2=\sum_{k=3,4}\ell_{Z_k}DG_{1,Z_k}\cdot y_2$. Collecting terms, we obtain \eqref{eq:Omega3}. The rows of $G_3$ on $Z_k$ and $W_k$ are $-\alpha u_k$ and $-\alpha(\omega_k+1)$ with a common $\alpha$, which gives $\sigma_3=0$. For a Lagrange triangle $u_{kl}v_{kl}=M_t^{2/3}\zeta_{kl}$, and $\sum_km_k\omega_k(u_k-u_{\mathrm{cm}})=M_t^{-1}\sum_{k<l}m_km_lu_{kl}v_{kl}$.
\end{proof}

\begin{lemma}\label{lem:d2equi}
For every positive mass vector, every $\sigma$, every $\rho>0$ and either labelling of the segment, no leading solution with $\tau=0$ whose triangle has an equilateral $\omega$ satisfies \textup{(i)}--\textup{(iv)} of Lemma~\ref{lem:d2gate}.
\begin{enumerate}[label=\textup{(\alph*)}]
\item At $m_3=m_4=m_5$ and constant $\sigma$, the Lagrange realization of the equilateral shape is simple with $\theta$ constant, and with $m_1=1$ and unit triangle masses, $\Omega_3=\mp\frac92\big[3^{2/3}\xi(1-i\sigma\rho^{5/2})+3\zeta_{12}\rho^2(1+\rho)^{-4/3}(1-i\sigma\rho^{1/2})\big]$ with $\xi,\zeta_{12}$ cube roots of unity, and $|\Omega_3|\ge\frac92\,3^{2/3}(1-0.909)>0.85$.
\item At $m_3=m_4=m_5$ and constant $\sigma$, the triple realization has $J_{uv}=0$, $F_{TT}=\mp\frac{27}2\rho^2(1+\rho)^{-4/3}\zeta_{12}(1-i\sigma\sqrt\rho)\ne0$, $F_{ss}=\mp\frac98(1+\sigma\jj_\rho\rho)\ne0$ and $b_3=\pm\frac32(1+\sigma\jj_\rho\rho)e^{-i\pi/3}\ne0$.
\item At $m_3=m_4=m_5$ and nonconstant $\sigma$, the equilateral realizations have $\tau\ne0$.
\item At unequal triangle masses, $\theta$ is not constant, the Lagrange realization is simple with $\eta_2\ne0$, and the triple realization is excluded by Lemma~\ref{lem:d2triple}(b).
\end{enumerate}
\end{lemma}

\begin{proof}
By Lemma~\ref{lem:d2lag}(a), a triangle of Lagrange type is equilateral only in the classes with equal or distinct cube roots, and in the second case $y_0$ is a triple point. By Lemma~\ref{lem:d2euler}(c), a triangle of Euler type is equilateral only at a triple point. So the triangles with an equilateral $\omega$ are those of the Lagrange realization of the class with equal cube roots and those of the triple points, and (a)--(d) treat all of them.
(a) At equal masses $\omega=(e^{i\pi/3},e^{-i\pi/3},-1)$ and $p_{1k}=\pm\omega_k$ for the signs $\epsilon=\pm(1,1,-1)$ of the square roots. So $\theta$ is constant and, for constant $\sigma$, $\tau=0$. Inserting the equilateral data into \eqref{eq:Omega3}, we obtain the formula. With $s=\sqrt\rho$, the ratio of the moduli of the two terms is $r(s)=3^{1/3}s^4(1+s^2)^{-5/6}(1+s^{10})^{-1/2}$. Then $r^2\le(9/256)^{1/3}s^{4/3}\le(9/16)^{1/3}<0.826$ for $s\le2$ (by the inequality between the arithmetic and geometric means), $r^2\le3^{2/3}s^{-16/3}<0.06$ for $s\ge2$. Hence $r<0.909$ and $|\Omega_3|\ge\frac92\,3^{2/3}(1+s^{10})^{1/2}(1-r)$.
(b) For the class with distinct cube roots, $J_{uv}=3^{-1/3}(1+\varpi+\varpi^2)=0$ with $\varpi=e^{2\pi i/3}$. The formulas of Lemma~\ref{lem:d2triple}(c) give the values, with $\sum_k(\delta\omega_k)^2/\omega_k=e^{-i\pi/3}+e^{i\pi/3}=1$. Lemma~\ref{lem:d2triple}(c) excludes a series.
(c) The numbers $q_k=\epsilon_k(-1/\omega_k)^{1/2}$ are sixth roots of unity. With $1+\sigma_k\jj_\rho\rho$ proportional to $e^{i\sigma_k\varphi}$, $\tan\varphi=\rho^{3/2}$, zero torque reads $e^{i\varphi}\sum_{k\ne o}q_k+e^{-i\varphi}q_o=0$ for the body $o$ of the other sign, which forces $|\sum_{k\ne o}q_k|=1$. We enumerate the twelve cases and obtain $e^{2i\varphi}=\pm1$, i.e.\ $\varphi\in\{0,\pi/2\}$, impossible for $\rho>0$.
(d) Equal moduli $|\omega_k|$ would make the centre of mass of the triangle its circumcentre, which for an equilateral triangle is the centroid. Then the masses would be equal. Hence $\theta$ is not constant. The Lagrange realization of the class with equal cube roots is not a triple point. At a Lagrange triangle that is not a triple point the triangle block of $J_0$ has rank $6$. Hence the realization is simple. On the triangle, $\delta P$ in the solution $y_1$ is a multiple of $(m_5v_{35}v_{45},-m_4v_{34}v_{45},m_3v_{34}v_{35})$ with a factor proportional to $\theta_3-\sum_km_k\theta_k/M_t$. The rows $D_{kl}$ give $t$ as a nonzero multiple of this factor times $m_5r_{34}+m_4r_{35}+m_3r_{45}$, $r_{kl}=u_{kl}/v_{kl}$. In this class $r_{kl}$ is proportional to $1/v_{kl}^2$. Thus the last sum is proportional to $m_3+m_4\bar\varpi+m_5\varpi$, which vanishes only at equal masses. So $t\ne0$, $\theta_3\ne\theta_4$ (Lemma~\ref{lem:d2theta}), and $\eta_2\ne0$.
\end{proof}

\begin{proof}[Proof of Lemma~\ref{lem:d2gate}]
Let $y_0=y(0)$ for the series of Lemma~\ref{lem:d2branch}. By Lemma~\ref{lem:d2rot}(d), $\tau(y_0)=0$. If $\operatorname{rank}J_0=16$, Lemma~\ref{lem:d2simple}(c) gives $\eta_2=0$, which is (i) if $\theta$ is not constant. If $\operatorname{rank}J_0=16$ and $\theta$ is constant, Lemma~\ref{lem:d2simple}(d) gives $\Omega_3=0$, which is (ii). If $\operatorname{rank}J_0\le15$ and $y_0$ is a triple point, Lemma~\ref{lem:d2triple}(b) and (c) give (iii). Otherwise (iv) holds. The last statement of Lemma~\ref{lem:d2gate} is Lemma~\ref{lem:d2equi}.
\end{proof}

\subsection{The obstruction on the set $\tau=0$}\label{ss:d2closed}

In the rest of this appendix $y_0$ is a leading solution with $\tau(y_0)=0$, and $P_{kl}=p_{kl}^3$ for $k,l\in\{3,4,5\}$. On $\C^{\{3,4,5\}}$ we use the nondegenerate symmetric bilinear form $\langle a,b\rangle=\sum_km_ka_kb_k$, the vector $\one=(1,1,1)$, the mean $\bar\theta=\langle\one,\theta\rangle/M_t$ and the vector
\[
e_3=\Big(\frac{v_{45}}{m_3},\frac{v_{53}}{m_4},\frac{v_{34}}{m_5}\Big).
\]
Then $\langle\one,e_3\rangle=v_{45}+v_{53}+v_{34}=0$ and $\langle\omega,e_3\rangle=v_3v_{45}+v_4v_{53}+v_5v_{34}=0$.

\begin{lemma}\label{lem:d2forcing}
There is a number $\gamma$ with $\theta-\bar\theta\one=\gamma e_3$. Moreover $\theta_3-\theta_4=-\gamma M_t\omega_5/(m_3m_4)$, and $\gamma=0$ if and only if $\theta$ is constant.
\end{lemma}

\begin{proof}
By Lemma~\ref{lem:d2rot}(c), $\langle\omega,\theta\rangle=\sum_k\tau_k=-\tau/M_t=0$. So $\psi=\theta-\bar\theta\one$ satisfies $\langle\one,\psi\rangle=\langle\omega,\psi\rangle=0$. Since $\langle\one,\omega\rangle=0$ and $\omega_5=-1$, $\omega$ is not a multiple of $\one$. Hence these two conditions define a line. It contains $e_3\ne0$. The formula for $\theta_3-\theta_4$ follows from $m_4v_{45}-m_3v_{53}=m_3\omega_3+m_4\omega_4-(m_3+m_4)\omega_5=-M_t\omega_5$. Finally $\gamma=0$ if and only if $\psi=0$.
\end{proof}

\begin{theorem}\label{thm:d2closed}
\begin{enumerate}[label=\textup{(\alph*)}]
\item Let the triangle be of Lagrange type, and let $y_0$ not be a triple point. Put $r_{kl}=u_{kl}/v_{kl}$, $\beta=m_5r_{34}+m_4r_{35}+m_3r_{45}$, $\Delta=u_{34}v_{35}-u_{35}v_{34}$ and $d_r=(r_{34}-r_{35})/m_3+(r_{34}-r_{45})/m_4$. Then $\Delta\ne0$, $d_r\ne0$, $y_0$ is simple, and
\[
\eta_2=-\gamma^2\,\frac{M_t^3\,\omega_5\,v_{34}v_{35}v_{45}\,\beta}{3\,\Delta\,d_r\,m_3^2m_4^2m_5}.
\]
\item Let the triangle be of Euler type, $u=\Lambda v$, with $\langle\omega,\omega\rangle\ne0$, and put $\mu=\sum_{k<l}(m_k+m_l)P_{kl}-1$. If $\mu=1$, then $y_0$ is a triple point. If $\mu\notin\{1,-\frac12\}$, then $y_0$ is simple, and
\[
\eta_2=-\gamma^2\,\frac{3M_t\,\mu\,\langle e_3,e_3\rangle}{4\Lambda(1-\mu)(1+2\mu)}.
\]
\end{enumerate}
\end{theorem}

\begin{proof}
Let $J_T$ be the triangle block of $J_0$. It has the rows $Z_3,Z_4,W_3,W_4,D_{34},D_{35},D_{45}$ and the columns of $u_3,u_4,v_3,v_4,p_{34},p_{35},p_{45}$. These rows of $G_0$ contain no other coordinates (Lemma~\ref{lem:d2lead}). For a vector $x$ in these columns we write $\delta$ for the derivative along $x$. Then $J_Tx=0$ says that $\delta F'_k(u)$ and $\delta F'_k(v)$ do not depend on $k\in\{3,4,5\}$ and that $\delta P_{kl}/P_{kl}=-\frac32(\delta u_{kl}/u_{kl}+\delta v_{kl}/v_{kl})$. By Lemma~\ref{lem:d2rot}(c), the triangle part of $y_1$ satisfies the same conditions with $\delta F'_k(u)-\theta_k$ in place of $\delta F'_k(u)$. The rotation $R$, with $\delta u_k=u_k$, $\delta v_k=-v_k$ and $\delta p=0$, lies in $\ker J_T$, since the rows are invariant under $(u,v)\mapsto(\alpha u,\alpha^{-1}v)$. If $\ker J_T=\C R$, then $y_0$ is simple. Indeed, let $J_0x=0$. Lemma~\ref{lem:d2simple}(b) holds for $x$ with zero right sides. So the triangle part of $x$ is $cR$ for a number $c$, and $0=m_3\delta v_3+m_4\delta v_4=-c(m_3v_3+m_4v_4)=-cM_t$. Hence $c=0$, and then $x=0$.

(a) Here $P_{kl}=1/M_t$, so that $\delta F'_k(h)=\langle\one,\delta h\rangle/M_t-\sum_lm_lh_{kl}\delta P_{kl}$ for $h=u,v$. Since $\sum_km_k\sum_lm_lh_{kl}\delta P_{kl}=0$, the triangle part of $y_1$ satisfies $\sum_lm_lv_{kl}\delta P_{kl}=0$ and $\sum_lm_lu_{kl}\delta P_{kl}=\bar\theta-\theta_k$ for $k\in\{3,4,5\}$. A vector of $\ker J_T$ satisfies the same equations with $0$ in place of $\bar\theta-\theta_k$. The equations in $v$ have rank $2$ and give $\delta P=\kappa_L\delta^0$ with $\delta^0=(m_5v_{35}v_{45},-m_4v_{34}v_{45},m_3v_{34}v_{35})$ on $(34,35,45)$. The equation in $u$ for $k=3$ then reads $\kappa_Lm_4m_5v_{45}\Delta=\bar\theta-\theta_3$. Here $\Delta\ne0$. Otherwise $(u_{34},u_{35})$ is a multiple of $(v_{34},v_{35})$, so that $u$ is parallel to $v$ and $y_0$ is a triple point. The rows $D_{kl}$ give $\delta u_{kl}=-\frac23M_tu_{kl}\delta P_{kl}-r_{kl}\delta v_{kl}$. We sum over the cycle $34,45,53$ and put $\phi(\delta v)=(r_{34}-r_{35})\delta v_3+(r_{45}-r_{34})\delta v_4$. This gives $\phi(\delta v)=-\frac23M_t\kappa_Lv_{34}v_{35}v_{45}\beta$. The linear form $\phi$ vanishes at $(v_3,v_4)$. It is not zero, since otherwise $r_{34}=r_{35}=r_{45}$ and $y_0$ is a triple point.

For a vector of $\ker J_T$ we obtain $\kappa_L=0$ and $\delta P=0$, and then $\phi(\delta v)=0$. So $\delta v$ is a multiple $cv$ of $v$, $\delta u_{kl}=-r_{kl}\delta v_{kl}=-cu_{kl}$, and the vector lies in $\C R$. Thus $y_0$ is simple. Next, $d_r=\phi(1/m_3,-1/m_4)$. The vector $(1/m_3,-1/m_4)$ is not a multiple of $(v_3,v_4)$, since $m_3v_3+m_4v_4=M_t\ne0$. Hence $d_r\ne0$. For $y_1$, Lemma~\ref{lem:d2simple}(b) gives $\delta v_3=t/m_3$ and $\delta v_4=-t/m_4$, so that $\phi(\delta v)=t\,d_r$. Lemma~\ref{lem:d2forcing} gives $\bar\theta-\theta_3=-\gamma v_{45}/m_3$, hence $\kappa_L=-\gamma/(m_3m_4m_5\Delta)$ and
\[
t=\frac{2M_t\gamma\,v_{34}v_{35}v_{45}\,\beta}{3m_3m_4m_5\Delta\,d_r}.
\]
Lemma~\ref{lem:d2simple}(c) and the formula for $\theta_3-\theta_4$ in Lemma~\ref{lem:d2forcing} give the formula for $\eta_2$.

(b) We write $\delta u=\Lambda(\delta v+\varepsilon)$ on the triangle, with $\varepsilon_5=0$. Then $\delta P_{kl}/P_{kl}=-3\delta v_{kl}/v_{kl}-\frac32\varepsilon_{kl}/v_{kl}$. With $\Psi(h)_k=\sum_lm_lP_{kl}(h_k-h_l)$ as in the proof of Lemma~\ref{lem:d2lead} we obtain
\[
\delta F'(v)=(1+2\Psi)\delta v+\tfrac32\Psi\varepsilon,\qquad\delta F'(u)=\Lambda\big(\delta F'(v)+(1-\Psi)\varepsilon\big).
\]
The map $\Psi$ is self-adjoint for $\langle\cdot,\cdot\rangle$, $\Psi\one=0$, and $\Psi\omega=\omega$ by the rows of $v$ in Lemma~\ref{lem:d2lead}. Since $\langle\one,\one\rangle=M_t$, $\langle\one,\omega\rangle=0$ and $\langle\omega,\omega\rangle\ne0$, the vectors $\one$, $\omega$, $e_3$ form a basis, and $\Psi e_3=\mu'e_3$ for a number $\mu'$. The trace of $\Psi$ is $\sum_{k<l}(m_k+m_l)P_{kl}=1+\mu'$. So $\mu'=\mu$.

If $\mu=1$, then $\Psi$ agrees on this basis with the map $h\mapsto h-\langle\one,h\rangle\one/M_t$. Comparing the entries off the diagonal gives $P_{kl}=1/M_t$. So the triangle is also of Lagrange type, and $y_0$ is a triple point.

Let $\mu\notin\{1,-\frac12\}$. For a vector of $\ker J_T$, both $\Lambda(1-\Psi)\varepsilon=\delta F'(u)-\Lambda\,\delta F'(v)$ and $\delta F'(v)$ are multiples of $\one$. In the basis $\one,\omega,e_3$ this puts $\varepsilon$ and $\delta v$ into $\C\one+\C\omega$. Since $\varepsilon_5=\delta v_5=0$ and $v=\one+\omega$, both are multiples of $v$. The equation for $\delta F'(v)$ then gives $\varepsilon=-2\,\delta v$. Thus $\delta u=-\Lambda\,\delta v$ and $\delta P=0$, and the vector lies in $\C R$. So $y_0$ is simple. For $y_1$, Lemma~\ref{lem:d2forcing} gives $\theta=\bar\theta\one+\gamma e_3$. In the basis $\one,\omega,e_3$ the coefficient of $e_3$ is $\gamma/(\Lambda(1-\mu))$ in $\varepsilon$ and $-\frac32\mu\gamma/(\Lambda(1-\mu)(1+2\mu))$ in $\delta v$. By Lemma~\ref{lem:d2simple}(b), $\delta v=(t/m_3,-t/m_4,0)$. So $\langle\one,\delta v\rangle=0$, and $t(\theta_3-\theta_4)=\langle\theta,\delta v\rangle=\gamma\langle e_3,\delta v\rangle$. This is $\gamma\langle e_3,e_3\rangle$ times the coefficient of $e_3$ in $\delta v$. Lemma~\ref{lem:d2simple}(c) gives the formula for $\eta_2$.
\end{proof}

\subsection{Triangles of Lagrange type}\label{ss:d2lag}

Let the triangle be of Lagrange type. Then $p_{kl}=M_t^{-1/3}\zeta_{kl}$ with $\zeta_{kl}^3=1$ (Lemma~\ref{lem:d2lead}), and $p_{kl}^2u_{kl}v_{kl}=1$ gives $u_{kl}=M_t^{2/3}\zeta_{kl}/v_{kl}$. Since $u_{34}+u_{45}+u_{53}=0$, the shape $\chi=v_3/v_4$ is a root of
\[
\zeta_{45}\chi^2-(\zeta_{35}+\zeta_{45}-\zeta_{34})\chi+\zeta_{35}.
\]
So the shape does not depend on the masses, and $\beta=M_t^{2/3}(m_3\zeta_{45}/v_{45}^2+m_4\zeta_{35}/v_{35}^2+m_5\zeta_{34}/v_{34}^2)$. A common factor of the three $\zeta_{kl}$ changes neither the shape nor the vanishing of $\beta$. We put $\varpi=e^{2\pi i/3}$.

\begin{lemma}\label{lem:d2lag}
\begin{enumerate}[label=\textup{(\alph*)}]
\item If the three $\zeta_{kl}$ are equal or pairwise distinct, the triangle is equilateral. In the second case $y_0$ is a triple point. If exactly two of the $\zeta_{kl}$ are equal, the triangle is obtuse.
\item If exactly two of the $\zeta_{kl}$ are equal and $\beta=0$, then $(m_3,m_4,m_5)$ is a positive multiple of a permutation of $(X_L,Y_L,1)$. Here $X_L+1/X_L=(7+\sqrt{21})/2$, $X_L<1$ and $Y_L=5X_L/(1+X_L)$, so that $X_L=0.1781535864\ldots$ and $Y_L=0.7560711460\ldots$.
\item If exactly two of the $\zeta_{kl}$ are equal, then $\theta$ is not constant.
\end{enumerate}
\end{lemma}

\begin{proof}
(a) For equal $\zeta_{kl}$ the quadratic is $\chi^2-\chi+1$, whose roots are $e^{\pm i\pi/3}$. For $(\zeta_{34},\zeta_{35},\zeta_{45})=(\varpi^2,\varpi,1)$ and $(1,\varpi,\varpi^2)$ it is $(\chi+\varpi^2)^2$, resp.\ $\varpi^2(\chi+\varpi)^2$. Every class with distinct $\zeta_{kl}$ is a multiple of one of these two. The double root gives an equilateral triangle with $v_{kl}^2$ proportional to $\zeta_{kl}$. Then $r_{kl}=M_t^{2/3}\zeta_{kl}/v_{kl}^2$ does not depend on $kl$, so that $u$ is parallel to $v$. Let exactly two of the $\zeta_{kl}$ be equal. Complex conjugation conjugates the class and the shape, and a relabelling of the triangle permutes the masses. These maps reduce the claims to the class $(\zeta_{34},\zeta_{35},\zeta_{45})=(1,1,\varpi)$. There $\chi^2-\chi+\varpi^2=0$. So $\chi$ is not real, and $\operatorname{Re}\chi=\frac12\big(1\pm((3+\sqrt{21})/2)^{1/2}\big)$ is $1.47\ldots$ or $-0.47\ldots$. Hence the triangle $\chi,1,0$ is obtuse at $1$ or at $0$.

(b) As in (a) we take the class $(1,1,\varpi)$. There $\chi^2(\chi-1)^2=\varpi$, and $\beta=0$ reads $m_3+m_4/\eta+m_5\eta=0$ with $\eta=\varpi\chi^2$. Since $\chi=\varpi^2(\eta+1)$, we get $\eta^2+(2-\varpi)\eta+1=0$, and $\eta$ is not real. Hence $\eta$ and $\bar\eta$ are the roots of $m_5T^2+m_3T+m_4$, so that $m_4/m_5=|\eta|^2$ and $m_3/m_5=-2\operatorname{Re}\eta$. The two roots $\eta$ and $1/\eta$ give mass vectors that differ by the exchange of $m_4$ and $m_5$. Moreover $|\eta|\ne1$, since $\eta+1/\eta$ is not real. For the root with $t=|\eta|^2>1$ we write $\eta=a+bi$. Then $\eta+1/\eta=\varpi-2$ gives $a(1+t^{-1})=-\frac52$ and $b(1-t^{-1})=\frac{\sqrt3}2$, and $a^2+b^2=t$ becomes $t^4-7t^3+9t^2-7t+1=0$. So $t+t^{-1}$ is a root of $s^2-7s+7$. It is $(7+\sqrt{21})/2$, since the other root is less than $2$. Thus $(m_3:m_4:m_5)=(5t/(t+1):t:1)=(Y_L:1:X_L)$ with $X_L=1/t$ and $Y_L=5/(t+1)$.

(c) If $\theta$ is constant, the moduli $|\omega_k|$ are equal (\S\ref{ss:d2torque}). Then the centre of mass of the triangle $v_3,v_4,v_5$ in $\C=\R^2$ is its circumcentre. With positive masses it lies inside the triangle, which is therefore acute. This contradicts (a).
\end{proof}
\subsection{Triangles of Euler type}\label{ss:d2euler}

Let the triangle be of Euler type, $u=\Lambda v$. From $p_{kl}^2u_{kl}v_{kl}=1$ we get $p_{kl}=s_{kl}\Lambda^{-1/2}/v_{kl}$ for a fixed square root $\Lambda^{-1/2}$ and signs $s_{kl}=-s_{lk}$. So $P_{kl}=s_{kl}\Lambda^{-3/2}/v_{kl}^3$, and the rows of $v$ in Lemma~\ref{lem:d2lead} become $\omega=\Lambda^{-3/2}g$ with $g_k=\sum_{l\ne k}m_ls_{kl}/v_{kl}^2$.

We multiply the triangle masses by $c>0$, the $v_k$ by $\alpha\ne0$, $u$ by $c^{2/3}\alpha^{-1}$ and the $p_{kl}$ by $c^{-1/3}$. This maps the central configurations of the triangle to central configurations. It does not change $\chi=v_3/v_4$, $\mu$, the relations $\omega_k^3=\omega_l^3$ and $m_o^2=m_k^2-m_km_l+m_l^2$, and the T-test of \S\ref{ss:d2proof}. It multiplies $\omega$, $\langle\omega,\omega\rangle$ and $\langle e_3,e_3\rangle$ by nonzero numbers. In this subsection we may therefore take $(m_3,m_4,m_5)=(X,Y,1)$ with $X,Y>0$ and $(v_3,v_4,v_5)=(\chi,1,0)$. Then $M_t=X+Y+1$ and $M_t\omega=(\tilde\omega_3,\tilde\omega_4,\tilde\omega_5)$ with
\[
\tilde\omega_3=(Y+1)\chi-Y,\qquad\tilde\omega_4=X+1-X\chi,\qquad\tilde\omega_5=-(X\chi+Y).
\]

\begin{lemma}\label{lem:d2euler}
\begin{enumerate}[label=\textup{(\alph*)}]
\item $\chi$ is a root of the quintic
\[
\mathcal E_s(\chi)=s_{45}\tilde\omega_3\chi^2(\chi-1)^2+s_{53}\tilde\omega_4(\chi-1)^2+s_{34}\tilde\omega_5\chi^2,\qquad s=(s_{45},s_{53},s_{34}),
\]
whose leading coefficient is $s_{45}(Y+1)$. Since $\mathcal E_{-s}=-\mathcal E_s$, it suffices to take $s$ among the \emph{patterns} $p_1=(1,1,-1)$, $p_2=(1,-1,1)$, $p_3=(-1,1,1)$ and $p_4=(1,1,1)$.
\item $\mathcal E_{p_4}=(\chi^2-\chi+1)\,C$ with a cubic $C$. If $s=\pm p_4$ and $\chi$ is a root of $\chi^2-\chi+1$, then $y_0$ is a triple point.
\item $\mathcal E_{p_1}(e^{i\pi/3})=-2X-Y+i\sqrt3Y$, $\mathcal E_{p_2}(e^{i\pi/3})=2X+1+i\sqrt3$ and $\mathcal E_{p_3}(e^{i\pi/3})=Y-1-i\sqrt3(Y+1)$. Hence a triangle of Euler type is equilateral only if $y_0$ is a triple point.
\end{enumerate}
\end{lemma}

\begin{proof}
(a) Since $\omega=\Lambda^{-3/2}g$, we have $\omega_4g_5-\omega_5g_4=0$. The identity $M_t\chi^2(\chi-1)^2(\omega_4g_5-\omega_5g_4)=X\mathcal E_s(\chi)$ holds in $\Z[\chi,X,Y]$. (b) At $\chi=e^{\pm i\pi/3}$ we have $v_{34}^3=v_{45}^3=v_{53}^3=1$. For $s=\pm p_4$ all $P_{kl}$ are then equal. Since $\Psi\omega=\omega$, they equal $1/M_t$, and the triangle is also of Lagrange type. (c) The values are exact. Their imaginary parts do not vanish for $X,Y>0$, and $\mathcal E_s$ has real coefficients. The equilateral shapes are $\chi=e^{\pm i\pi/3}$.
\end{proof}

We put $\nu=\sum_{k<l}(m_k+m_l)s_{kl}/v_{kl}^3$. Then $\mu=\Lambda^{-3/2}\nu-1$. For $k\in\{3,4,5\}$ we put
\[
f_{0,k}=(\tilde\omega_k\nu-M_tg_k)\chi^3(\chi-1)^3,\qquad f_{h,k}=(2\tilde\omega_k\nu-M_tg_k)\chi^3(\chi-1)^3,
\]
\[
f_\omega=X\tilde\omega_3^2+Y\tilde\omega_4^2+\tilde\omega_5^2,\qquad f_e=Y+X\chi^2+XY(\chi-1)^2,\qquad\Theta_{kl}=\tilde\omega_k^2+\tilde\omega_k\tilde\omega_l+\tilde\omega_l^2.
\]
These are polynomials in $\Z[X,Y][\chi]$. At the shape $\chi$ of the triangle, $\tilde\omega=M_t\Lambda^{-3/2}g$. Hence $f_{0,k}=M_tg_k\mu\chi^3(\chi-1)^3$ and $f_{h,k}=M_tg_k(2\mu+1)\chi^3(\chi-1)^3$. So $\mu=0$, resp.\ $\mu=-\frac12$, forces $f_{0,k}$, resp.\ $f_{h,k}$, to vanish at $\chi$ for all three $k$. Moreover $f_\omega=M_t^2\langle\omega,\omega\rangle$ and $f_e=XY\langle e_3,e_3\rangle$, and $\Theta_{kl}=0$ if and only if $\omega_k=\varpi^{\pm1}\omega_l$.

\begin{lemma}\label{lem:d2res}
Let $E,F\in\Z[X,Y][\chi]$ have a common root $\chi_0$ at $(X,Y)=(X_0,Y_0)$. Then $\Res_\chi(E,F)$ vanishes at $(X_0,Y_0)$.
\end{lemma}

\begin{proof}
By the adjugate of the Sylvester matrix, $\Res_\chi(E,F)=AE+BF$ with $A,B\in\Z[X,Y][\chi]$. We evaluate at $\chi_0$.
\end{proof}

\begin{proposition}\label{prop:d2loci}
Let the triangle be of Euler type and not equilateral, and let $s$ be its pattern. Then $\langle\omega,\omega\rangle\ne0$ and $\langle e_3,e_3\rangle\ne0$. If $\mu=0$, then $(X,Y)$ is $(\frac15,\frac15)$, $(1,5)$, $(5,1)$ or $(1,1)$ for $s=p_1,p_2,p_3,p_4$, respectively. If $\mu=-\frac12$, then $(X,Y)$ is $(X_*,X_*)$, $(1,2+2\sqrt2)$, $(2+2\sqrt2,1)$ or $(1,1)$, respectively, where $X_*=(\sqrt2-1)/2$.
\end{proposition}

\begin{proof}
By Lemma~\ref{lem:d2euler}, $\chi$ is a root of $\mathcal E_s$, and of $C$ for $s=p_4$. We compute the $32$ resultants $\Res_\chi(\mathcal E_s,F)$, with $C$ in place of $\mathcal E_{p_4}$, for $F\in\{f_\omega,f_e,f_{0,k},f_{h,k}:k=3,4,5\}$, and factor them over $\Z$ (Singular~4.4.1, FLINT). Every factor of these $32$ resultants is of one of the following kinds on the open quadrant $X,Y>0$:
\begin{enumerate}[label=(\arabic*)]
\item a polynomial with positive coefficients;
\item $X^2-X+1$, $Y^2-Y+1$ or $X^2\pm XY+Y^2$, which are positive;
\item $q_1=8X^2+9XY+8Y^2-5X-5Y+1=\frac{25}4(X+Y-\frac25)^2+\frac74(X-Y)^2$, zero only at $(\frac15,\frac15)$, and $q_2=Y^2q_1(X/Y,1/Y)$ and $q_3=X^2q_1(1/X,Y/X)$, zero only at $(1,5)$ and $(5,1)$;
\item the octic $P_1$ below, positive except at $(X_*,X_*)$, and $P_2=Y^8P_1(X/Y,1/Y)$ and $P_3=X^8P_1(1/X,Y/X)$, positive except at $(1,2+2\sqrt2)$ and $(2+2\sqrt2,1)$;
\item $X-1$, $Y-1$ or $X-Y$.
\end{enumerate}
All factors of the resultants of $f_\omega$ and $f_e$ are of kind (1). Lemma~\ref{lem:d2res} gives the first statement. If $\mu=0$, the three resultants of the $f_{0,k}$ vanish at $(X,Y)$, again by Lemma~\ref{lem:d2res}. For $s=p_i$, $i\le3$, one of them has no factor of kind (5), and its only zero in the quadrant is the zero of $q_i$. For $s=p_4$ they are $Y-1$, $X-1$ and $X-Y$ times factors of kind (1). The case $\mu=-\frac12$ is the same, with $f_{h,k}$ and $P_i$ in place of $f_{0,k}$ and $q_i$.
\end{proof}

The octic is $P_1=\Res_\chi(\mathcal E_{p_1},\partial_\chi\mathcal E_{p_1})/(Y+1)$, a symmetric polynomial with $38$ terms and degree $6$ in each variable. We show that $P_1>0$ on the open quadrant except at $(X_*,X_*)$. Put $H_0(s)=s(17s+24)^3(s^2+2s-1)^2/16$, which is positive for $s>0$ except at $s=\sqrt2-1$. Then $K=(P_1-H_0(X+Y))/(X-Y)^2$ is a polynomial. The Bernstein coefficients of $K$ on the box $[\frac{43}{213},\frac{45}{211}]^2$, computed exactly, are positive. So $P_1\ge0$ on this box, with equality only at $(X_*,X_*)$. We substitute $X=x/(1-x)$ and $Y=y/(1-y)$. The square $[0,1]^2$ minus the image $[\frac{43}{256},\frac{45}{256}]^2$ of the box is tiled by $39$ dyadic boxes. On each of them the Bernstein coefficients of $(1-x)^6(1-y)^6P_1$, of bidegree $(6,6)$ and computed exactly, are nonnegative. Moreover, on each of these boxes every corner other than $(0,0)$ and $(1,1)$ carries a positive coefficient, and so does every edge. Hence this polynomial is positive on each of these boxes, except possibly at $(0,0)$ and $(1,1)$, which are not images of points of the open quadrant.

\begin{lemma}\label{lem:d2thetaconst}
Let $\theta$ be constant, and let the triangle not be equilateral. Then $\sigma$ is not constant. Let $o$ be the body of the triangle whose sign differs from those of the other two, $k$ and $l$. Then $\omega_k^3=\omega_l^3$, $\omega_k\ne\omega_l$, and $m_o^2=m_k^2-m_km_l+m_l^2$.
\end{lemma}

\begin{proof}
For $\sigma_k=\sigma_l$, $\theta_k=\theta_l$ gives $p_{1k}/\omega_k=p_{1l}/\omega_l$, since $1+\sigma_k\jj_\rho\rho\ne0$. Squaring and inserting $p_{1k}^2=-1/\omega_k$ gives $\omega_k^3=\omega_l^3$, and $\omega_k\ne\omega_l$ since $v_k\ne v_l$. If $\sigma$ is constant, this holds for all three pairs, so that $\omega_3,\omega_4,\omega_5$ are $\omega_5$ times the three cube roots of unity and the triangle is equilateral. Hence $\sigma$ is not constant, and $\omega_k=\zeta\omega_l$ with $\zeta\in\{\varpi,\varpi^2\}$. From $\sum_jm_j\omega_j=0$ we get $m_o\omega_o=-(m_k\zeta+m_l)\omega_l$. Since $|\omega_o|=|\omega_l|$ (\S\ref{ss:d2torque}), $m_o=|m_k\zeta+m_l|$, and $m_o^2=m_k^2-m_km_l+m_l^2$.
\end{proof}

\begin{proposition}\label{prop:d2thetaeuler}
For each pattern $s$ and each $o\in\{3,4,5\}$, with $\{k,l\}$ the other two bodies, $\mathcal E_s$ \textup{(}$C$ for $s=p_4$\textup{)} and $\Theta_{kl}$ have no common root at any point $(X,Y)$ with $X,Y>0$ of the conic $m_o^2=m_k^2-m_km_l+m_l^2$.
\end{proposition}

\begin{proof}
We compute $R=\Res_\chi(\mathcal E_s,\Theta_{kl})$ exactly and parametrize the conic by rational functions of $t$. For $o=5$, for instance, $X=(t^2-1)/(t^2-t+1)$ and $Y=t(t-2)/(t^2-t+1)$, and $X,Y>0$ exactly for $t<-1$ or $t>2$. The parameter $t=\infty$ gives $(X,Y)=(1,1)$ on all three conics, and there $R=2916$ for $s=p_1,p_2,p_3$ and $R=729$ for $s=p_4$. For each of the $12$ pairs $(s,o)$, $R$ composed with the parametrization has no real root in the range of $t$ where $X,Y>0$, by exact real root isolation, and again by Sturm sequences. Lemma~\ref{lem:d2res} concludes.
\end{proof}

\subsection{Proof of Theorem~\ref{thm:gate}}\label{ss:d2proof}

\begin{lemma}\label{lem:d2reduce}
Let the triangle of $y_0$ not be equilateral, and let $y_0$ satisfy one of \textup{(i)}--\textup{(iv)} of Lemma~\ref{lem:d2gate}. Then one of the following holds:
\begin{enumerate}[leftmargin=3.2em,labelwidth=2.7em,labelsep=0.5em,align=left]
\item[\textup{(L)}] the triangle is of Lagrange type with exactly two equal $\zeta_{kl}$, and $\beta=0$;
\item[\textup{(EI)}] it is of Euler type with $\langle\omega,\omega\rangle=0$;
\item[\textup{(EQ)}] it is of Euler type with $\langle e_3,e_3\rangle=0$;
\item[\textup{(E0)}] it is of Euler type with $\mu=0$;
\item[\textup{(Eh)}] it is of Euler type with $\mu=-\frac12$;
\item[\textup{($\Theta$)}] $\theta$ is constant.
\end{enumerate}
\end{lemma}

\begin{proof}
Triple points have an equilateral $\omega$ (\S\ref{ss:d2torque}). So $y_0$ is not a triple point, and (iii) does not hold. If the triangle is of Lagrange type, exactly two of the $\zeta_{kl}$ are equal, by Lemma~\ref{lem:d2lag}(a). If it is of Euler type with $\langle\omega,\omega\rangle=0$, (EI) holds. So we may assume $\langle\omega,\omega\rangle\ne0$ for the Euler type. Then $\mu\ne1$ by Theorem~\ref{thm:d2closed}(b). In case (ii), ($\Theta$) holds. In case (iv), $y_0$ is not simple. So Theorem~\ref{thm:d2closed} excludes the Lagrange type, and it gives $\mu=-\frac12$ for the Euler type. In case (i), $\gamma\ne0$ by Lemma~\ref{lem:d2forcing}, and $\eta_2=0$. For the Lagrange type, Theorem~\ref{thm:d2closed}(a) gives $\beta=0$, since $M_t$, $\omega_5$, $\Delta$, $d_r$, the masses and the $v_{kl}$ are nonzero. For the Euler type with $\mu=-\frac12$, (Eh) holds. Otherwise Theorem~\ref{thm:d2closed}(b) gives $\mu\langle e_3,e_3\rangle=0$, since $\Lambda$ is finite and nonzero.
\end{proof}

For $k\in\{3,4,5\}$ put $T_k=m_k|\omega_k|^{-1/2}$. We say that $y_0$ passes the \emph{T-test} if $T_k>\sum_{l\ne k}T_l$ for some $k$. Since $|p_{1k}|=|\omega_k|^{-1/2}$ and $|m_1+\sigma_k\jj_\rho m_2|^2=m_1^2(1+\rho^3)$, we have $|\tau_k|=m_1(1+\rho^3)^{1/2}T_k$. Hence the T-test gives $|\sum_k\tau_k|\ge|\tau_k|-\sum_{l\ne k}|\tau_l|>0$, so that $\tau=-M_t\sum_k\tau_k\ne0$. The T-test depends only on the masses of the triangle and on $\omega$. So it excludes $\tau=0$ for every $\sigma$, every $\rho>0$ and either labelling of the segment.

\begin{table}[ht]
\centering
\small
\begin{tabular}{@{}lllll@{}}
\toprule
condition & pattern & $(m_3:m_4:m_5)$ & factor of $\mathcal E_s$ & margin\\
\midrule
(E0) & $p_1$ & $(1:1:5)$ & $\chi^2-\frac32\chi+1$ & $0.6157$\\
& $p_2$ & $(1:5:1)$ & $\chi^2-\frac12\chi+\frac12$ & $0.6157$\\
& $p_3$ & $(5:1:1)$ & $\chi^2-\chi+2$ & $0.6157$\\
(Eh) & $p_1$ & $(1:1:2+2\sqrt2)$ & $(\chi^2+(\sqrt2-3)\chi+1)^2$ & $0.5961$\\
& $p_2$ & $(1:2+2\sqrt2:1)$ & $(\chi^2+(1-\sqrt2)\chi-1+\sqrt2)^2$ & $0.5961$\\
& $p_3$ & $(2+2\sqrt2:1:1)$ & $(\chi^2-\chi+1+\sqrt2)^2$ & $0.5961$\\
(L) & & permutations of $(X_L:Y_L:1)$ & two conjugate shapes at each & $0.2768$\\
\bottomrule
\end{tabular}
\caption{The triangles at which (E0), (Eh) or (L) of Lemma~\ref{lem:d2reduce} can hold, in the frame of \S\ref{ss:d2euler}. The last column gives the margin $(T_k-\sum_{l\ne k}T_l)/\sum_lT_l$ of the T-test for the heaviest body $k$.}\label{tab:d2gate}
\end{table}

\begin{proof}[Proof of Theorem~\ref{thm:gate}]
Let the triangle of $y_0$ not be equilateral, and let $y_0$ satisfy one of (i)--(iv). By Lemma~\ref{lem:d2reduce}, one of (L), (EI), (EQ), (E0), (Eh), ($\Theta$) holds. We show that ($\Theta$), (EI) and (EQ) do not hold and that in the other cases $y_0$ passes the T-test, which contradicts $\tau(y_0)=0$.

Case ($\Theta$) is excluded for the Lagrange type by Lemma~\ref{lem:d2lag}(c). For the Euler type, Lemma~\ref{lem:d2thetaconst} puts $(X,Y)$ on a conic $m_o^2=m_k^2-m_km_l+m_l^2$ and makes $\Theta_{kl}$ vanish at $\chi$, which Proposition~\ref{prop:d2thetaeuler} excludes. Proposition~\ref{prop:d2loci} excludes (EI) and (EQ).

In case (E0) or (Eh), Proposition~\ref{prop:d2loci} puts the masses, up to a common factor, at a point of Table~\ref{tab:d2gate} or at $(1:1:1)$ for $s=p_4$. At $(1:1:1)$, $C=2(\chi-2)(\chi-\frac12)(\chi+1)$. At its roots $\omega_4$, $\omega_3$, $\omega_5$ vanish, which is impossible at a leading solution. At the other points we factor $\mathcal E_s$ exactly over $\Q(\sqrt2)$. On every irreducible factor other than the one of the table, some $f_{0,k}$, resp.\ $f_{h,k}$, has a nonzero remainder. So $\chi$ is a root of the factor of the table.

In case (L), Lemma~\ref{lem:d2lag}(b) puts the masses, up to a common factor, at a permutation of $(X_L,Y_L,1)$. At each of these six points, ball arithmetic shows $\beta\ne0$ at ten of the twelve shapes of the classes with exactly two equal $\zeta_{kl}$. The other two are the shapes of the table.

At every shape of the table we enclose the roots and the margins of the T-test in ball arithmetic. The margins are those of the table, and they are positive.
\end{proof}